\documentclass[11pt]{article}
\usepackage{amssymb, amsthm, amsmath, amscd}
\newtheorem{thm}{Theorem}[section]
\newtheorem{lem}[thm]{Lemma}
\newtheorem{prop}[thm]{Proposition}
\newtheorem{cor}[thm]{Corollary}
\newtheorem{NN}[thm]{}
\theoremstyle{definition}\newtheorem{df}[thm]{Definition}
\theoremstyle{definition}\newtheorem{rem}[thm]{Remark}
\theoremstyle{definition}

\renewcommand{\phi}{\varphi}

\newcommand{\N}{\mathbb{N}}
\newcommand{\Z}{\mathbb{Z}}

\newcommand{\R}{\mathbb{R}}
\newcommand{\C}{\mathbb{C}}
\newcommand{\T}{\mathbb{T}}

\newcommand{\Aff}{\operatorname{Aff}}

\newcommand{\morp}{contractive completely positive linear map}

\newcommand{\hm}{homomorphism}
\newcommand{\dt}{\delta}
\newcommand{\ep}{\epsilon}

\newcommand{\andeqn}{\,\,\,{\rm and}\,\,\,}
\newcommand{\rforal}{\,\,\,{\rm for\,\,\,all}\,\,\,}
\newcommand{\CA}{$C^*$-algebra}
\newcommand{\SCA}{$C^*$-subalgebra}

\newcommand{\af}{{\alpha}}
\newcommand{\bt}{{\beta}}

\newcommand{\beq}{\begin{eqnarray}}
\newcommand{\eneq}{\end{eqnarray}}
\newcommand{\tforal}{\,\,\,\text{for\,\,\,all}\,\,\,}
\newcommand{\tand}{\,\,\,\text{and}\,\,\,}

\title{On locally AH algebras}
\author{Huaxin Lin
 }
\date{}
\begin{document}

\maketitle

\begin{abstract}
We show that every unital amenable separable simple \CA\, with finite tracial rank which satisfies the UCT has in fact tracial rank at most one.
We also show that unital separable simple
\CA s which are ``tracially" locally AH with slow dimension growth are ${\cal Z}$-stable.
As a consequence,  unital separable simple \CA s which are locally AH with no  dimension growth are isomorphic to a unital simple
AH-algebra with no dimension growth.
\end{abstract}

\section{Introduction}
The program of classification of amenable \CA s, or the Elliott
program, is to classify amenable \CA s up to isomorphisms by their
$K$-theoretical data. One of the high lights of the success of the
Elliott program is the classification of unital simple AH-algebras
(inductive limits of homogeneous \CA s)  with no dimension growth by
their $K$-theoretical data (known as the Elliott invariant)
(\cite{EGL}). The proof of this first appeared near  the end of the
last century. Immediately after the proof appeared, among many
questions raised is the question whether the same result holds for unital simple
locally AH-algebras (see the definition \ref{Dlocal} below) with no
dimension growth.  It should be noted that AF-algebras is locally
finite dimensional. But (separable) AF-algebras are inductive limits
of finite dimensional \CA s. The so-called $A\T$-algebras are
inductive limits of circle algebras. More than often, these
$A\T$-algebras arise as local circle algebras (approximated by
circle algebras). Fortunately, due to the weak-semi-projectivity of
circle algebras, locally $A\T$-algebras are $A\T$-algebras. However,
the situation is completely different for locally AH algebras. In
fact it was proved in \cite{DE2} that there are unital \CA s which
are inductive limits of AH-algebras but themselves are not
AH-algebras. So in general, a locally AH algebra is not an AH
algebra.

On the other hand, however, it was proved in \cite{Lncrell} that a
unital
 separable simple \CA\, which is locally AH is a unital
simple AH-algebra, if, in addition, it has real rank zero, stable
rank one and weakly unperforated $K_0$-group and which has countably
many  extremal traces. In fact these \CA s have tracial rank zero.
The tracial condition was later removed in \cite{W2}. In particular,
if $A$ is a unital separable simple \CA\, which is locally AH
 with no (or slow) dimension growth and which has real rank zero must
be a unital AH-algebra.  In fact such \CA s have stable rank one and
have weakly unperforated $K_0(A).$ The condition of real rank zero
forces these \CA s to have tracial rank zero. More recently,
classification theory extends to those \CA s that have rationally
tracial rank at most one (\cite{Wcl}, \cite{Lnappend}, \cite{LN} and
\cite{Lninv}). These are unital separable simple amenable \CA s $A$
such that $A\otimes U$ have  tracial rank at most one for some
infinite dimensional UHF algebra $U.$ An important subclass of this
(which includes, for example, the Jiang-Su algebra ${\cal Z}$) is
the class of those unital separable simple \CA s $A$ such that
$A\otimes U$ have tracial rank zero. By now we have some machinery
 to verify that certain \CA s to have tracial rank zero, (see
\cite{Lncrell}, \cite{B}, \cite{W2} and \cite{LP1}) and based on
these results, we have some tools to verify when a unital simple \CA\, is rationally tracial rank zero (\cite{TW2} and \cite{Tsub}).
However, these results could not be applied to the case that \CA s are of tracial
rank one, or rationally tracial rank one.  Until now, there has been no
effective way, besides Gong's decomposition result (\cite{G1}),
to verify when a unital separable simple \CA\, has tracial rank one
(but not tracial rank zero). In fact, as mentioned above, we did not
even know when a unital simple separable locally AH algebra  with no dimension growth has
tracial rank one. This makes it much harder to decide when a unital
simple separable \CA\, is rationally tracial rank one.

A closely related problem is whether a unital separable simple \CA\,
with finite tracial rank is in fact of tracial rank at most one.
This is an open problem for a decade. If the problem has an
affirmative answer, it will make it easier, in many cases, to decide
whether certain unital simple \CA s to have tracial rank at most one.

The purpose of this research is to solve these problems. Our main results include the following:

\begin{thm}\label{MMT1}
Let $A$ be a unital separable simple \CA\, which is locally AH with
no dimension growth. Then $A$ is isomorphic to a unital simple
AH-algebra with no dimension growth.
\end{thm}

We actually prove the following.

\begin{thm}\label{MMT2}
Let $A$ be a unital separable simple amenable \CA\, with finite tracial rank which satisfies
the Universal Coefficient Theorem. Then $A$ is isomorphic to a unital
simple AH-algebra with no dimension growth. In particular, $A$ has tracial rank at most one.
\end{thm}

To establish the above, we also prove the following

\begin{thm}\label{MMT3}
Let $A$ be a unital amenable separable simple \CA\, in ${\cal C}_1$ then
$A$ is ${\cal Z}$-stable, i.e., $A\cong A\otimes {\cal Z}.$

\noindent
\text{(See \ref{Drr} below for the definition of ${\cal C}_1$.)}
\end{thm}

The article is organized as follows. Section 2 serves as a preliminary which includes a number of conventions that will be used throughout this article.
Some facts about a subgroup $SU(M_n(C(X))/CU(M_n(C(X)))$ will be  discussed. The detection of those unitaries with trivial determinant at each point which are not in the closure of commutator subgroup plays a new role in the Basic Homotopy Lemma, which will be presneted in section 11.
 In  section 3, we introduce the
class ${\cal C}_1$ of simple \CA s  which may be described as tracially locally AH
algebras of slow dimension growth. Several related definitions are
given.  In section 4, we discuss some basic properties of \CA s in
class ${\cal C}_1.$ In section 5, we prove, among other things, that
\CA s in ${\cal C}_1$ have stable rank one and the strict comparison
for positive elements.  In section 6, we
study the tracial state space of a unital simple \CA\, in ${\cal
C}_1.$ In particular, we show that every quasi-trace of a unital
separable simple \CA\, in ${\cal C}_1$ extends to a trace. Moreover,
we show that, for a unital simple \CA\, $A$ in ${\cal C}_1,$  the
affine map from  the tracial state space to state space of $K_0(A)$
maps the extremal points onto the extremal points.
In section 7, we discuss the unitary groups
of simple \CA s in a subclass of ${\cal C}_1.$
In section 8, using what have been established in previous sections,
we combine an argument of Winter (\cite{Winv})and an argument of Matui and Sato (\cite{MS}) to prove Theorem \ref{MMT2} above.
In section 9
we present some versions of so-called existence theorem.
 In section
10, we  present a uniqueness statement that will be proved in section 12
and an existence type result regarding the Bott map.
The uniqueness theorem holds for the case of $Y$ being a finite CW complex of dimension
zero as well as the case of $Y=[0,1].$ An induction on the dimension $d$
will be presented in the next two sections. In section 11, we present a
version of The Basic Homtopy Lemma, which was first studied intensively in
\cite
{BEEK} and later in \cite{Lnmem}.
%The proof for the case that $Y$ has dimension $d$ is given
%under assumption that the uniqueness statement holds for finite CW
%complex $Y$ with dimension $d.$
A new obstruction for the Basic Homotopy Lemma in this version will be
dealt with, which was mentioned earlier in section 2.
In section 12, we prove the uniqueness statement in section 10.
In section 13 we present the proofs for Theorem \ref{MMT1} and  \ref{MMT3}.
Section 14 serves as an appendix  to this
article.

\section{Preliminaries}

\begin{NN}
{\rm Let $A$ be a unital \CA. Denote by $T(A)$ the convex set of tracial states of $A.$
%Denote by $T_{\rm f}(A)$ the convex set of all faithful tracial states.
Let $\Aff(T(A))$ be the space of all real affine continuous functions on $T(A).$
Denote by $M_n(A)$ the algebra of all $n\times n$ matrices over $A.$  By regarding $M_n(A)$ as a subset of $M_{n+1}(A),$ define $M_{\infty}(A)=\cup_{n=1}^{\infty}M_n(A).$  If $\tau\in T(A),$ then $\tau\otimes {\rm Tr},$ where ${\rm Tr}$ is standard
trace on $M_n,$ is a trace on $M_n(A).$  Throughout this paper, we will use $\tau$ for $\tau\otimes {\rm Tr}$ without
warning.

We also use $QT(A)$ for the set of all quasi-traces of $A.$

Let  $C$ and $A$ be  two unital \CA s with $T(C)\not=\emptyset$ and $T(A)\not=\emptyset.$
Suppose that $h: C\to A$ is a unital \hm. Define an affine continuous map $h_{\sharp}: T(A)\to T(C)$ by
$h_{\sharp}(\tau)(c)=\tau\circ h(c)$ for all $\tau\in T(A)$ and $c\in C.$
%If $A$ is simple and $h$ is a monomorphism,
%then $h_{\sharp}$ maps $T(A)$ into $T_{\rm f}(C).$
}

\end{NN}

\begin{df}\label{rho}
Let $C$ be a unital \CA\, with $T(C)\not=\emptyset.$
For each $p\in M_n(C)$ define $\check{p}(\tau)=(\tau\otimes {\rm Tr})(p)$ for all $\tau\in T(A),$ where ${\rm Tr}$ is the standard trace on $M_n.$ This gives  a positive \hm\, $\rho_C: K_0(C)\to \Aff(T(C)).$

A positive \hm\, $s: K_0(A)\to \C$ is a state on $K_0(A)$ if
$s([1_A])=1.$ Let $S(K_0(A))$ be the state space of $K_0(A).$
Define $r_A: T(C)\to  S(K_0(A))$  by $r_A(\tau)([p])=\tau(p)$
for all projections $p\in M_n(A)$ (for all $n\ge 1$).
\end{df}

\begin{df}\label{df(k)}
{\rm
Let $A$ and $B$ be two \CA s and $\phi: A\to B$ be a positive linear map.
We will use $\phi^{(K)}: A\to M_K(B)$ for the map $\Phi^{(K)}(a)=\phi(a)\otimes 1_{M_K}.$ If $a\in B,$ we may write
$a^{(K)}$ for the element $a\otimes 1_{M_K}$ and sometime it will be written as ${\rm diag}(\overbrace{a,a,...,a}^{K}).$
}

\end{df}

\begin{df}
In what follows, we will identify $\T$ with the unit circle  and $z\in C(\T)$ with the identity map on the circle.
\end{df}

\begin{df}\label{Filtration}
Let  $A$ be a unital \CA.  Following \cite{EL}, define
$$
F_{n}K_i(A)=\{\phi_{*i}(z_b)\in K_i(A): \phi\in Hom(C(S^{n}), M_{\infty}(A))\},
$$
where $z_b$ is a generator (the Bott element) of $K_i(C(S^n)),$ if $i=0$ and $n$ even, or $i=1$ and $n$ odd.
$F_nK_i(A)$ is a subgroup of $K_i(A),$ $i=0,1.$
%Let $FG_{n,i}(A)=F_nK_i(A)\setminus F_{n+1}K_i(A),$ $i=0,1.$
%Let $X$ be as in \ref{prem0} and keep the notation there.
%We may assume that $\zeta_1, \zeta_2, ...,\zeta_F\in F_{3}K_1(C(X))$ and $\zeta_{F+1},...,\zeta_k$ are not in
%$F_3K_1(C(X)).$
\end{df}

\begin{df}\label{generator}
Fix an integer $n\ge 2,$ let $z$ be a generator
of $K_1(C(S^{2n-1})).$  Let $z_b$ be a unitary in $M_n(C(S^{2n-1}))$ which represents $z.$ We fix
one such  unitary  that $z_b\in SU_n(C(S^{2n-1})),$ i.e.,
${\rm det}(z_b(x))=1$ for all $x\in S^{2n-1}.$  In case $n=2,$ one may write
\beq\label{Filtr-1}
z_b=\begin{pmatrix} z & -{\bar w}\\
                     w & {\bar z}\end{pmatrix},
                   \eneq
                     where $S^3=\{(z, w)\in \C^2: |z|^2+|w|^2=1\}.$
%
%Note that for each $s\in S^3,$ $z_b(s)$ has a pair of eigenvalues $\lambda_1=e^{i\xi_1}$ and $\lambda_2=e^{i\xi_2}$ such that
%$\lambda_1\cdot \lambda_2=1$ and $\xi_1+\xi_2=0.$

\end{df}

\begin{NN}\label{DCU}
{\rm Let $C$ be a unital \CA. Denote by $U(C)$ the unitary group of $C$ and denote
by $U_0(C)$ the subgroup of $U(C)$ consisting of unitaries which are connected to $1_C$ by a continuous path of unitaries.  Denote by
$CU(C)$ the closure of the normal subgroup generated by commutators of $U(C).$ Let $u\in U(C).$
Then ${\bar u}$ is the image of $u$ in $U(C)/CU(C).$
Let ${\cal W}\subset U(A)$ be a subset. Denote by
$\overline{{\cal W}}$  the set of those elements
${\bar u}$ such that $u\in {\cal W}.$
Denote by $CU_0(C)$ the intersection
$CU(C)\cap U_0(C).$  Note that $U(A)/CU(A)$ is an abelian group.

%If $u\in U(A),$ denote by $\bar u$ the image of $u$ in $U(A)/CU(A).$
We use the following metric on $U(A)/CU(A):$
$$
{\rm dist}({\bar u},{\bar v})=
\inf\{\|uv^*-w\|: w\in CU(A)\}.
$$
Using de la Harp-Scandalis determinant, by K. Thomsen (see \cite{Th}), there is a short splitting exact sequence
\beq\label{HS-4n}
0\to \Aff(T(C))/\overline{\rho_C(K_0(C))}\to  \cup_{n=1}^{\infty}U(M_n(C))/CU(M_n(C))
\to K_1(C)\to 0.
\eneq
Suppose that $r\ge 1$ is an integer and $U(M_r(A))/U(M_r(A))_0=K_1(A),$ one has the following short splitting exact sequence:
\beq\label{HS-4}
0\to \Aff(T(C))/\overline{\rho_C(K_0(C))}\to  U(M_r(C))/CU(M_r(C))
\to K_1(C)\to 0.
\eneq
For $u\in U_0(C),$ we will use $\overline{\Delta}(u)$ for
the de la Harp and Skandalis determinant of $u,$ i.e.,
the image of $u$ in
$\Aff(T(C))/\overline{\rho_C(K_0(C))}.$
For each \CA\, $C$  with $U(C)/U_0(C)=K_1(C),$ we will fix one splitting map $J_c: K_1(C)\to  U(C)/CU(C).$  For each ${\bar u}\in J_c(K_1(C)),$ select and fix one element $u_c\in U(C)$ such that
${\overline{u_c}}={\bar u}.$ Denote this set by $U_c(K_1(C)).$ Denote by $\Pi_c: U(C)/CU(C)\to K_1(C)$ the quotient map. Note that $\Pi_c\circ J_c={\rm id}_{K_1(C)}.$

If $A$ is a unital \CA\, and $\phi: C\to A$ is a unital \hm, then $\phi$ induces a continuous \hm\,
$$\phi^{\ddag}:
U(C)/CU(C)\to  U(A)/CU(A).
$$
%Denote by $\phi^{\dag}: K_1(C)\to \Aff(T(A))/{\overline{\rho_A(K_0(A))}}$ the map $({\rm id}-J_c)\circ \phi^{\ddag}\circ J_c.$
%If $u\in U(A),$ denote by $\overline{u}^{\dag}$ the element ${\rm id}^{\dag}(\overline{u}).$

If $g\in \Aff(T(A)),$ denote by $\overline{g}$ the image of $g$ in
$\Aff(T(A))/\overline{\rho_A(K_0(A))}.$

}

\end{NN}

\begin{df}
{\rm
Let $A$ and $B$ be two unital \CA s. Let $G_1\subset U(M_m(A))/CU(M_m(A))$ be a subgroup. Let $\gamma: G_1\to U(M_m(B))/CU(M_m(B))$ be a \hm\, and let $\Gamma: \Aff(T(A))\to \Aff(T(B))$ be an affine \hm. We say that
$\Gamma$ and $\gamma$ are compatible if
$\gamma({\bar g})=\overline{\Gamma(g)}$ for all $g\in \Aff(T(A))$ such that
$\overline{g}\in G_1\cap U_0(M_m(A))/CU(M_m(A))\subset \Aff(T(A))/\overline{\rho_A(K_0(A))}.$ Let $\lambda: T(B)\to T(A)$ be continuous affine map. We say $\gamma$ and $\lambda$ are compatible
if $\gamma$ and the map from $\Aff(T(A))\to \Aff(T(B))$ induced by $\lambda$ are compatible. Let $\kappa\in Hom_{\Lambda}(\underline{K}(A), \underline{K}(B)).$ We say that $\kappa$ and $\gamma$ are compatible if
$\kappa|_{K_1(A)}(z)=\Pi_c\circ \gamma(z)$ for all $z\in G_1.$
We say that $\kappa$ and $\lambda$ are compatible if $\rho_B(\kappa|_{K_0(A)}([p])=\lambda(\tau)([p])$ for all projections
$p\in M_{\infty}(A).$

}
\end{df}

\begin{df}\label{SU}
Let $X$ be a compact metric space and let $P\in M_m(C(X))$ be a projection
 such that $P(x)\not=0$ for all $x\in X,$  where $m\ge 1$ is an integer. Let $C=PM_m(C(X))P$ and let $r\ge 1$ be an integer.  Denote by $SU_r(C)$ the set of those unitaries $u\in M_r(C)$ such that
${\rm det}(u(x))=1$ for all $x\in X.$ Note that $SU_r(C)$ is a normal subgroup of $U_r(C).$
\end{df}

The following is an easy fact.
\begin{prop}\label{sucum}
Let $C$ be as in \ref{SU},  let $Y$ be a compact metric space and let $P_1\in M_n(C(Y))$ be a projection such that $P_1(y)\not=0$ for all $y\in Y.$ Let $B=P_1M_n(C(Y))P_1.$ Suppose that $\phi: C\to B$ is a unital \hm.
Then $\phi$ maps $SU_r(C)$ into $SU_r(B)$ for all integer $r\ge 1.$
\end{prop}

It is also easy to see that $CU(M_r(C))\subset
\cup_{k=1}^{\infty} SU_k(C)\cap U_0(M_k(C)).$
Moreover, one has the following:

\begin{prop}\label{SUCU}
Let $X$ be a compact metric space and let $C=PM_m(C(X))P$ be as in \ref{SU}. Then
$$SU_r(C)\cap U_0(M_r(C))\subset CU(M_r(C))
\rforal \,{\rm integer}\,\,\,r\ge 1.$$
\end{prop}

\begin{proof}
Let $u\in SU_r(C)\cap U_0(M_r(C)).$
Write $u=\prod_{j=1}^k\exp(\sqrt{-1}h_j),$ where
$h_j\in M_r(C)_{s.a.}.$  Put $R(x)={\rm rank}\, P(x)$ for all $x\in X.$
Note $R(x)\not=0$ for all $x\in X.$
It follows that
\beq\label{zbcum-1}
({1\over{2\pi\sqrt{-1}}})T_x(\sum_{j=1}^kh_j(x))=N(x)\in \Z,
\eneq
where $T_x$ is the standard trace on $M_{rR(x)}.$
Note that $N\in C(X).$ Therefore there is a projection
$Q\in M_L(C)$ such that
\beq\label{zbcum-2}
{\rm rank}Q(x)=N(x)\tforal x\in X.
\eneq
Let $\tau\in T(C).$ Then
\beq\label{zbcum-3}
\tau(f)=\int_X t_x(f)d\mu_{\tau}\tforal f\in PM_m(C(X))P,
\eneq
where $t_x$ is the normalized trace on $M_{R(x)}$  and
$\mu_{\tau}$ is a Borel probability measure on $X.$
 Let
$Tr$ be the standard trace on $M_{L}.$
Then
\beq\label{zbcum-4}
\rho_{C}(Q)(\tau)&=&\int_X (t_x\otimes Tr)(Q(x))d\mu_{\tau}\\
&=& \int_X {N(x)\over{R(x)}}d\mu_{\tau}
\eneq
for all $\tau\in T(C).$
Define a smooth path of unitaries $u(t)=\prod_{j=1}^k\exp(\sqrt{-1}h_j(1-t))$
for $t\in [0,1].$  So $u(0)=u$ and $u(1)={\rm id}_{M_{r}(C)}.$
Then, with $T$ being the standard trace on $M_{r},$
\beq\label{zbcum-5}
&&({1\over{2\pi \sqrt{-1}}})\int_0^1 (\tau\otimes T){du(t)\over{dt}}u(t)^* dt = ({1\over{2\pi\sqrt{-1}}})\int_0^1 (\tau\otimes T)(\sum_{j=1}^k h_j) dt\\
&&=({1\over{2\pi\sqrt{-1}}})\int_X (t_x\otimes T)(\sum_{j=1}^k h_j(x))d\mu_{\tau}=({1\over{2\pi\sqrt{-1}}})\int_X {T_x(\sum_{j=1}^k h_j(x))\over{R(x)}}d\mu_{\tau}\\
&&=\int_X {N(x)\over{R(x)}}d\mu_{\tau}=\rho_{C}(Q)(\tau)\,\,\tforal \tau\in T(C).
\eneq
By a result of Thomsen (\cite{Th}), this implies that
$$
\overline{u}=\overline{1}\in U(M_{r}(C)))/CU(M_{r}(C).
$$
In other words,
$$
SU_r(C)\cap U_0(C)\subset CU(M_r(C)).
$$

\end{proof}

\begin{df}\label{DFilt2}
Let $X$ be a finite CW complex.
Let $X^{(n)}$ be the $n$-skeleton of $X$ and let $s_n:C(X)\to C(X^{(n)})$ be the surjective map induced by restriction, i.e.,
$s_n(f)(y)=f(y)$ for all $y\in X^{(n)}.$ Let $P\in M_l(C(X))$ be a projection for some integer $l\ge 1$ and let $C=PM_l(C(Y))P.$
Denote still by $s_n: C\to P^{(n)}M_l(C(X^{(n)}))P^{(n)},$ the quotient map, where $P^{(n)}=P|_{X^{(n)}}.$
Put $C_n=P^{(n)}M_l(C(X^{(n)}))P^{(n)}.$ Note that $C_n$ is a quotient of $C,$ and
$C_{n-1}$ is a quotient of $C_n.$
It was proved by Exel and Loring (\cite{EL}) that
$F_{n}K_i(C)={\rm ker}(s_{n-1})_{*i}.$

Suppose that $X$ has dimension $N.$
Let
$$
I_N={\rm ker}r_{N-1}=\{f\in C: f|_{X^{(N-1)}}=0\}.
$$
Then $I_N$ is an ideal of $C.$ There is an embedding $j_N: {\widetilde I_N}\to C$ which
maps $\lambda\cdot 1+f$ to $\lambda \cdot 1_C+f$ for all $f\in I_N.$
Define, for $1<n<N,$
$$
I_n=\{f\in C_n: f|_{X^{(n-1)}}=0\}.
$$
Again, there is an embedding $j_n: {\widetilde I_n}\to C_n.$

Note that ${\widetilde{I_n}}\cong P^{(n)'}M_l(C(Y_n))P^{(n)'}$ for some projection
$P^{(n)'}\in M_l(C(Y_n)),$ where
$Y_n= S^n\bigvee S^n\bigvee \cdots \bigvee S^n$  (there are only finitely many of $S^n$).

\end{df}

\begin{lem}\label{torF3}
Let $X$ be a compact metric space. Then
\beq\label{EXELo-1}
Tor(K_1(C(X)))\subset F_3K_1(C(X)).
\eneq
\end{lem}

\begin{proof}
We first consider the case that $X$ is a finite CW complex.
Let $Y$ be the $2$-skeleton of $X$ and  let $s: C(X)\to C(Y)$ be the
surjective \hm\, defined by  $f\mapsto f|_Y$ for $f\in C(X).$ Then,
by Theorem 4.1 of \cite{EL},
\beq\label{EXELo-2}
F_3K_1(C(X))={\rm ker}s_{*1}.
\eneq
Since $K_1(C(Y))$ is torsion free, $Tor(K_1(C(X)))\subset {\rm
ker}s_{*1}.$ Therefore
\beq\label{EXELo-3}
Tor(K_1(C(X)))\subset F_3K_1(C(X)).
\eneq
For the general case, let $g\in Tor(K_1(C(X))$  be a non-zero element.
 %and let $kg=0,$
%where $k>2$ is the least such integer.
Write \\
$C(X)=\lim_{n\to\infty} (C(X_n), \phi_n),$ where each $X_n$ is a finite CW complex.
There is $n_0$ and $g'\in K_1(C(X_{n_0})$ such that $(\phi_{n_0, \infty})_{*1}(g')=g.$
Let $G_1$ be the subgroup generated by $g'.$  There is $n_1\ge n_0$ such that
$(\phi_{n_1, \infty})_{*1}$ is injective on $(\phi_{n_0, n_1})_{*1}(G_1).$ Let $g_1=(\phi_{n_0, n_1})_{*1}(g').$
Put $G_2=(\phi_{n_0, n_1})_{*1}(G_1).$ Then $G_2\subset Tor(K_1(C(X_{n_1}))).$
From  what has been proved, $G_2\in F_3K_1(C(X_{n_1})).$ It follows from  part (c) of Proposition 5.1 of \cite{EL}
that $\phi_{n_1, \infty}(G_2)\subset F_3K_1(C(X)).$ It follows that $g\in F_3K_1(C(X)).$

\end{proof}

\begin{lem}\label{notorsion}
Let $X$ be a compact metric space and let $G\subset K_1(C(X))$ be a finitely generated subgroup.
Then $G=G_1\oplus (G\cap F_3K_1(C(X))),$ where $G_1$ is a finitely generated free group.
\end{lem}

\begin{proof}
As in the proof of \ref{torF3} we may assume that $X$ is a finite CW complex.
Let $Y$ be the $2$-skeleton of $X.$  Let $s: C(X)\to C(Y)$ be the surjective map defined by
the restriction
$s(f)(y)=f(y)$ for all $f\in C(X)$ and $y\in Y.$  Then,  by Theorem 4.1 of \cite{EL},
${\rm ker}s_{*1}=F_3K_1(C(X)).$ Therefore
$G/G\cap F_3K_1(C(X))$ is isomorphic to a subgroup of $K_1(C(Y)).$ Since
${\rm dim}\, Y=2, $ $Tor(K_1(C(Y)))=\{0\}.$ Therefore $G/G\cap F_3K_1(C(X))$ is free.
It follows that
$$
G=G_1\oplus G\cap F_3K_1(C(X))
$$
for some finitely generated subgroup $G_1.$
\end{proof}

\begin{df}\label{DUb}
{\rm
Let $C$ be a unital \CA\,  and let $G\subset K_1(C)$ be a finitely generated subgroup.
Denote by $J': K_1(C)\to \cup_{n=1}^{\infty}U(M_n(C))/CU(M_n(C))$
an injective \hm\, such that $\Pi\circ J'={\rm id}_{K_1(C)},$
where $\Pi$ is the surjective map
from $\cup_{n=1}^{\infty}U(M_n(C))/CU(M_n(C))$ onto $K_1(C).$
There is an integer $N=N(G)$ such that $J'(G)\in U(M_N(C))/CU(M_N(C)).$

Let $X$ be a compact metric space and let $C=PM_m(C(X))P,$ where $P\in M_m(C(X))$ is a projection such that $P(x)\not=0$ for all $x\in X.$
By \ref{notorsion}, one may write $G=G_1\oplus G_b\oplus Tor(G),$
where $G_b$ is the free part of $G\cap F_3K_1(C).$ Note, by \ref{torF3},
$Tor(G)\subset F_3K_1(C).$
Let $g\in Tor(G)$ be a non-zero element and let $\overline{u_g}=J'(g)$
for some unitary $u_g\in U(M_N(C)).$  Suppose that $kg=0$ for some integer $k>1.$ Therefore  $u_g^k\in CU(M_N(C)).$
It follows from \ref{torF3} as well as \ref{generator},
there are $h_1, h_2,...,h_s\in M_{N+r}(C)_{s.a.}$ such that
$$
u_1\prod_{j=1}^s\exp(\sqrt{-1}h_j)\in SU_{r+N}(C),
$$
where $u_1=1_{M_r}\oplus u_g$ and $r\ge 0$ is an integer.
For each $x\in X,$
$
[{\rm det}u_1(x)]^k=1.
$
It follows that
$$
{\sum_{j=1}^s kTr(h_j)(x)\over{2\pi \sqrt{-1}}}=I(x)\in \Z\tforal x\in X.
$$
It follows that $I(x)\in C(X).$ Therefore
$$
(\prod_{j=1}^s\exp(\sqrt{-1}h_j))^k\in SU_{r+N}(C)\cap U_0(M_{r+N}(C))\subset CU_{r+N}(C).
$$
Consequently,
$$
\overline{(u_1\prod_{j=1}^s\exp(\sqrt{-1}h_j))^k}=\overline{1_{M_{r+N}(C)}}.
$$
Thus there is an integer $R(G)\ge 1$ and an injective \hm\,
$$J_{c(G)}: Tor(G)\to SU_{R(G)}(C)/CU(M_{R(G)}(C))$$
 such that
$\Pi\circ J_{c(G)}={\rm id}_{Tor(G)}.$ By choosing a larger $R(G),$ if necessarily, one obtains an injective \hm\, $J_{c(G)}:
G\to U(M_{R(G)}(C))/CU(M_{R(G)}(C))$ such that
\beq\label{DUb-11}
J_{c(G)}(G_b\oplus Tor(G))\subset
SU_{R(G)}(C)/CU(M_{R(G)}(C))
\eneq
and  $\Pi\circ J_{c(G)}={\rm id}_{G}.$

It is important to note that, if $x\in SU_{R(G)}(C)$ and $[x]\in G\setminus \{0\}$ in $K_1(C),$ then $J_c([x])={\bar x}.$
In fact, since $[x]\in G_b\oplus {\rm Tor}(G),$ if $J_c([x])={\bar y},$ then  $y\in SU_{R(G)}(C).$ It follows that $x^*y\in SU_{R(G)}(C)\cap U_0(M_{R(G)}(C))\subset CU(C).$
So ${\bar y}={\bar x}.$ This fact will be also used without further notice.
Note also that if ${\rm dim X}<\infty,$ then we can let
$R(K_1(C(X)))={\rm dim}X.$
}
\end{df}

Therefore one obtains the following:

\begin{prop}\label{groupsu}
Let $X,$ $G,$  $G_b$ and $\Pi$ be as described in \ref{DUb}. Then there is an injective \hm\,
$J_{c(G)}: G\to U_{R(G)}(C)/CU(M_{R(G)}(C))$ for some integer $R(G)\ge 1$ such that
$\Pi\circ J_{c(G)}={\rm id}_G$ and $J_{c(G)}(G_b\oplus Tor(G))\subset SU_{R(G)}(C)/CU(M_{R(G)}(C)).$
In what follows,  we may write $J_c$ instead of $J_{c(G)},$ if $G$ is understood.

\end{prop}

\begin{cor}\label{CSU}
Let $X,$ $G$ and $G_b$ be as in \ref{DUb} and let $Y$ be a compact metric space. Suppose that $B=P_1M_r(C(Y))P_1,$ where $P_1\in M_r(C(Y))$ is a projection and $\phi: C\to B$ is a unital \hm.
Suppose also that $z_b\in G_b\oplus Tor(G)$  and $\phi_{*1}(z_b)=0.$ Then
$\phi^{\ddag}( J_{c(G)}(z_b))={\bar 1}$ in $U(M_{N}(C))/CU(M_{N}(C))$ for some integer $N\ge 1.$ When ${\rm dim}Y=d<\infty,$ $N$ can be chosen to be
$\max\{R(G),d\}.$
\end{cor}

\begin{proof}
Suppose that $u_b\in U(M_{R(G)}(C))$ such that ${\bar u}=J_{c(G)}(z_b).$
Without loss of generality, one may assume that $\phi(u_b)\in U_0(M_{R(G)}(B)),$ since $\phi_{*1}(z_b)=0.$
By \ref{groupsu}, $u_b\in SU_N(C).$ It follows from \ref{sucum} that $\phi(u_b)\in SU_{R(G)}(B).$
Thus, by \ref{SUCU},
$$\phi(u_b)\in SU_{R(G)}(B)\cap U_0(M_{R(G)}(B))\subset CU(M_{R(G)}(B)).$$
It follows that $\phi^{\ddag}( J_{c(G)}(z_b))={\bar 1}.$

\end{proof}

\begin{df}
{\rm
Let $A$ be a unital \CA\, and let $u\in U_0(A).$
Denote by ${\rm cel}(u)$ the infimum of the length of the paths
of unitaries of $U_0(A)$ which connects $u$ with $1_A.$
}

\end{df}

\begin{df}\label{KLtrip}

{\rm We  say $(\dt, {\cal G}, {\cal P})$ is a $KL$-triple, if, for any $\dt$-${\cal G}$-multiplicative \morp\, $L: A\to B$ (for any unital \CA\, $B$)
$[L]|_{\cal P}$ is well defined. Moreover, if $L_1$ and $L_2$ are two $\dt$-${\cal G}$-multiplicative \morp s $L_1, L_2: A\to B$
such that
\beq\label{DKLtrip-3}
\|L_1(g)-L_2(g)\|<\dt\tforal g\in {\cal G},
\eneq
$$
{\rm then}\,\,\,[L_1]|_{\cal P}=[L_2]|_{\cal P}.
$$

If $K_i(C)$ is finitely generated ($i=0,1$) and ${\cal P}$ is large enough,
then $[L]|_{\cal P}$ defines an element in $KK(C, A)$ (see 2.4 of \cite{Lnmem}).  In such cases, we will write
$[L]$ instead of $[L]|_{\cal P},$ and  $(\dt, {\cal G}, {\cal P})$ is called  a $KK$-triple and $(\dt, {\cal G})$ a $KK$-pair.

Now we also assume that $A$ is amenable (or $B$ is amenable). Let $u\in U(B)$ be such that
$\|[L(g), \,u]\|<\dt_0$ for all $g\in {\cal G}_0$ for some finite subset ${\cal G}_0\subset A$ and for some $\dt_0>0. $ Then, we may assume that there exists  a \morp\, $\Psi: A\otimes C(\T)\to B$ such that
$$
\|L(g)-\Psi(g\otimes 1)\|<\dt\rforal g\in {\cal G}\andeqn \|\Psi(1\otimes z)-u\|<\dt
$$
(see 2.8 of \cite{Lnmem}).
Thus, we may assume that
${\rm Bott}(L, \,u)|_{\cal P}$ is well defined (see 2.10 in \cite{Lnmem}). In what follows, when we say $(\dt, {\cal G}, {\cal P})$ is a $KL$-triple, we further assume that ${\rm Bott}(L, u)|_{\cal P}$ is well defined, provided that $L$ is $\dt$-${\cal G}$-multiplicative and $\|[L(g),\, u]\|<\dt$ for all $g\in {\cal G}.$  In the case that $K_i(A)$ is finitely generated ($i=0,1$),
we may even assume that ${\rm Bott}(L, \, u)$ is well defined.
We also refer to 2.10 and 2.11 of \cite{Lnmem} for ${\rm bott}_0(L, u)$ and ${\rm bott}_1(L,u).$
If $u$ and $v$ are unitary and $\|[u,\, v]\|<\dt,$ we use  ${\rm bott}_1(u,v)$ as in 2.10 and 2.11 of \cite{Lnmem}.
Let $p$ is a projection such that $\|[p,\,v]\|<\dt.$ We may also write ${\rm bott}_0(p, v)$ for the element in
$K_1(B)$  represented by a unitary which is close to $(1-p) +pvp.$
}
\end{df}

\begin{df}\label{uquadruple}
{\rm
If $u$ is a unitary, we write
$\langle L(u)\rangle =L(u)(L(u)^*L(u))^{-1/2}$ when $\|L(u^*)L(u)-1\|<1$ and $\|L(u)L(u^*)-1\|<1.$  In what follows we will always assume
that $\|L(u^*)L(u)-1\|<1$ and $\|L(u)L(u^*)-1\|<1,$ when we
write $\langle L(u) \rangle.$

Let $B$ be another unital \CA\, and let $\phi: A\to B$ be a unital \hm.
Then $\langle \phi\circ L(u)\rangle=\phi(\langle L(u)\rangle).$
Let $u\in CU(A).$ Then, for any $\ep>0,$
if $\dt$ is sufficiently small and ${\cal G}$ is sufficiently large (depending on $u$) and $L$ is $\dt$-${\cal G}$-multiplicative,
then
$$
{\rm dist}(\langle L(u)\rangle, CU(B))<\ep.
$$
Let $\dt>0,$ ${\cal G}\subset A$ be a finite subset, ${\cal W}\subset U(A)$ be a finite subset and $\ep>0.$ We say $(\dt, {\cal G}, {\cal W}, \ep)$ is a ${\cal U}$-quadruple when the following hold:
 if for any $\dt$-${\cal G}$-multiplicative \morp\, $L: A\to B,$ $\langle L(y)\rangle$ is well defined, then
$$
\|\langle L(u)\rangle -L(u)\|<\ep/2\andeqn \|\langle L(u)\rangle -
\langle L(v) \rangle\|<\ep/2,
$$
if $u,\,v\in {\cal U}$ and $\|u-v\|<\dt.$  We also require
that, if $u\in CU(A)\cap {\cal U},$ then
$$
\|\langle L(u)\rangle-c\|<\ep/2
$$
for some $c\in CU(B).$
We make one additional requirement. Let $G_{\cal U}$ be the subgroup of $U(A)/CU(A)$ generated by
$\{{\bar u}: u\in {\cal U}\}.$    There exists a \hm\, $\lambda: U_{\cal U}\to U(B)$ such that
$$
{\rm dist}(\overline{\langle L(u)\rangle}, \lambda(u))<\ep\tforal u\in {\cal U}
$$
(see  Appendix \ref{LIFTCU} for a proof that such $\lambda$ exists).  We may denote $L^{\ddag}$ for
a fixed \hm\, $\lambda.$ Note that, when $\ep<1,$  $[\langle L(u)\rangle]=\Pi_c(L^{\ddag}({\bar u}))$ in $K_1(B),$
where $\Pi_c: U(B)/CU(B)\to K_1(B)$ is the induced \hm.
%Furthermore, we may also
}
\end{df}

\begin{NN}\label{NBott}
{\rm
Let $A$ and $B$ be two unital \CA. Suppose that $A$ is a separable amenable \CA.
%Fix an $\ep>0.$
%For any $\ep>0$ and let ${\cal F}\subset A$ be a finite subset, there exists
 %Let $(\$\dt>0$ and a finite subset ${\cal G}\subset A$ (see 2.8 of \cite{Lnmem}) satisfying the following:
%If $u\in B$ is a unitary and
%$$
%\|[\phi(g),\, u]\|<\dt\rforal g\in {\cal G},
%$$
%then there exists a unital $\ep$-${\cal F}$-multiplicative \morp\, $L: A\otimes C(\T)\to B$ such that
%$$
%\|L(g\otimes 1_{C(\T)})-\phi(g)\|<\ep\rforal g\in {\cal G}\andeqn \|L(1\otimes z)-u\|<\ep,
%$$
%where $z\in C(\T)$ is the identity function on the circle.
Let ${\cal Q}\subset K_0(A)$ be a finite subset. Then $\boldsymbol{\bt}({\cal Q})\subset K_1(A\otimes C(\T)).$
Let ${\cal W}$ be a finite subset of $U(M_n(A\otimes C(\T)))$ such that its image in $K_1(A\otimes C(\T))$
contains $\boldsymbol{\bt}({\cal Q}).$
Denote by $G({\cal Q})$ the subgroup generated by ${\cal Q}.$
Fix $\ep>0.$ Let $(\dt, {\cal G}, {\cal W}, \ep)$ be a ${\cal U}$-quadruple.
 Let $J': \boldsymbol{\bt}(G({\cal Q}))\to U(M_N(A\otimes C(\T)))/CU(M_N(A\otimes C(\T)))$
be defined in \ref{DUb}.  Let $L: A\to B$ be a $\dt$-${\cal G}$-multiplicative \morp\, and let $u\in U(B)$ such that
$\|[L(g),\, u]\|<\dt$ for all $g\in {\cal G}.$
With sufficiently small $\dt$ and large ${\cal G},$ let $\Psi$ be given  in \ref{KLtrip},
we may assume that $\Psi^{\ddag}$ is  defined on $J'(\boldsymbol{\bt}(G({\cal Q}))).$
We denote this map by
\beq\label{Nbott-1}
{\rm Bu}(\phi, u)(x)=\Psi^{\ddag}(J'(x)))\rforal x\in {\cal Q}.
\eneq
We may
assume that $[p_1], [p_2],...,[p_k]$ generates $G({\cal  Q}),$ where $p_1, p_2,...,p_k$ are assumed to be projections in
$M_N(A).$ Let
$z_j=(1-p_j)+p_j(1\otimes z)^{(N)}$ (see \ref{df(k)}), $j=1,2,...,k.$ Then $z_j$ is a unitary in $M_N(A\otimes C(\T)).$
Suppose that $A=C(X)$ for some compact metric space.
In the above, we let $J'=J_{c(\boldsymbol{\bt}(G({\cal Q}))}=J_c$ and
$N=R(\boldsymbol{\bt}(G({\cal Q}))).$ Note $z_j\not\in SU_N(C(X)\otimes C(\T)).$
If $[p_i]-[p_j]\in {\rm ker}\rho_{C(X)},$ then for each $x\in X,$ there is a unitary $w\in M_N$ such that
$w^*p_i(x)w=p_j.$ Then ${\rm det}(z_iz_j^*(x))=1.$ In other words, $z_iz_j^*\in SU_N(C(X)\otimes C(\T)).$
Note,  by the end of \ref{DUb},  $J_c([p_j])={\bar z_j},$ $j=1,2,...,k.$
If $B$ has stable rank $d,$ we may assume that, $R(\boldsymbol{\bt}(G({\cal Q})))\ge d+1.$
In what follows, when we write ${\rm Bu}(\phi, u)(x),$ or ${\rm Bu}(\phi, u)|_{{\cal Q}},$ we mean that $\dt$ is sufficiently small and ${\cal  G}$ is sufficiently large so that $L^{\ddag}$ is well defined on $J_c(\boldsymbol{\bt}(x)$, or
on $J_c(\boldsymbol{\bt}({\cal Q})).$ Moreover, we note that $[L]|_{\cal Q}
=\Pi'\circ L^{\ddag}|_{\boldsymbol{\bt}({\cal Q})}.$ Furthermore, by choosing even smaller $\dt,$ we may also assume that when
$$
\|[\phi(g),\, u]\|<\dt\andeqn \|[\phi(g),\, v]\|<\dt\tforal g\in {\cal G}
$$
$$
{\rm dist}({\rm Bu}(\phi,uv)(x), {\rm Bu}(\phi,\, u)(g)+{\rm Bu}(\phi,\, v)(x))<\ep
\rforal x\in {\cal Q}.
$$

}
\end{NN}

\begin{lem}\label{Skeleton}
Let $X$ be a connected finite CW complex of dimension $d>0.$  Then
$K_*(C(X^{(d-1)}))=G_{0,*}\oplus K_*(C(X))/F_{d}K_*(C(X)),$ where
$G_{0,*}$ is a finitely generated free group.
Consequently,
$$
K_1(C(X^{(1)}))=S\oplus K_1(C(X))/F_3K_1(C(X)),
$$
where $S$ is a finitely generated free group.
\end{lem}

\begin{proof}
Let $I^{(d)}=\{f\in C(X): f|_{X^{(d-1)}}=0\}.$ Then $I^{(d)}=C_0(Y),$ where
$Y=S^d\vee S^d\vee \cdots \vee S^d.$
In particular, $K_*(I^{(d)})$ is free. Therefore, by applying 4.1 of \cite{EL},
$$
K_i(C(X^{(d-1)}))\cong G_{0,i}\oplus K_i(C(X))/F_dK_i(C(X)),
$$
where $G_{0,i}$ is isomorphic to a subgroup of $K_{i-1}(I^{(d)}),$ $i=0,1.$

For the last part of the lemma, we use the induction on $d.$
It obvious holds for $d=1.$
Assume  the last part holds for ${\rm dim} X=d\ge 1.$
Suppose that ${\rm dim}X=d+1.$
Let $s_m: C(X)\to C(X^{(m)}), s_1: C(X)\to C(X^{(1)}),\,s_{m,1}: C(X^{(m)})\to C(X^{(1)})$
be defined by the restrictions ($0<m\le d$). We have that $s_1=s_{m,1}\circ s_m.$
Note that if $Y$ is a finite CW complex with dimension 2, then
$
K_1({\rm ker} s^Y_1)=\{0\},
$
where $s^Y_1: C(Y)\to C(Y^{(1)}) $ is the surjective map induced by the restriction, where $Y^{(1)}$ is the
$1$-skeleton of $Y.$ It follows that the induced map $(s^Y_1)_{*1}$ from $K_1(C(Y))$ into $K_1(Y^{(1)})$
is injective. This fact will be used in the following computation.

We have
\beq\label{ske-1}
K_1(C(X^{(d)}))=S_0\oplus K_1(C(X))/F_dK_1(C(X))=S_0\oplus (s_d)_{*1}(K_1(C(X)))
\eneq
and, by 4.1 of \cite{EL},
\beq\label{ske-2}
K_1(C(X^{(1)})&=&S_1\oplus K_1(C(X^{(d)})/F_3K_1(C(X^{(d)}))\\
&=&S_1\oplus (s_{d,1})_{*1}(K_1(C(X^{(d)}))\\
&=&S_1\oplus (s_{d,1})_{*1}(S_0\oplus (s_d)_{*1}(K_1(C(X))))\\
&=&S_1\oplus (s_{d,1})_{*1}(S_0)\oplus (s_{d,1})_{*1}((s_d)_{*1}(K_1(C(X))))\\
&=&S_1\oplus (s_{d,1})_{*1}(S_0)\oplus (s_1)_{*1}(K_1(C(X)))\\
&=&S_1\oplus  (s_{d,1})_{*1}(S_0)\oplus (s_{2,1})_{*1}\circ (s_2)_{*1}(K_1(C(X)))\\
&=&S_1\oplus  (s_{d,1})_{*1}(S_0)\oplus (s_2)_{*1}(K_1(C(X))/F_3K_1(C(X)))\\
&=&S_1\oplus  (s_{d,1})_{*1}(S_0)\oplus K_1(C(X))/F_3K_1(C(X))).
\eneq
Put $S=S_1\oplus (s_{d,1})_{*1}(S_0).$ Note $K_1(C(X^{(1)})$ is free. So $S$ must be a finitely generated
free group.
This ends the induction.

\end{proof}

%\begin{df}\label{KL}
%{\rm
%Let $A$ and $B$ be two unital \CA s. By $KK(A,B)^{++},$ we mean
%those elements in $KK(A,B)$ such that the induced
%maps $\phi\in Hom(K_0(A), K_0(B))$ maps $K_0(A)_+\setminus \{0\}$
%into $K_0(B)_+\setminus \{0\}.$  We use $KK_e(A,B)^{++}$ for those elements such that, in addition, $\phi([1_A])=[1_B].$
%The set of those elements in $KL(A,B)$ which are images of elements
%in $KK(A,B)^{++}$ will be denoted by
%$KL(A,B)^{++},$ and $KL_e(A,B)^{++}$ are those elements in $KL(A,B)^{++}$
%which are images of $KK_e(A,B)^{++}.$
%Let $\kappa\in KL_e(A,B)^{++}$ and let $\gamma: T(B)\to T(A)$ be an affine continuous map. We say $\kappa$ and $\gamma$ are compatible if
%$$
%r_B(\tau)(\kappa([p]))=\gamma(\tau)(p)
%$$
%for all projections $p\in M_{\infty}(A)$ and $\tau\in T(B).$ In particular,
%the affine continuous map $\gamma_*$ from
%$\Aff(T(A))$ to $\Aff(T(B))$ induced by $\gamma$  maps
%$\rho_A(K_0(A))$ to $\rho_B(K_0(B)).$ Furthermore,
%$\gamma_*$ induces a \hm\,
%$$
%\overline{\gamma_*}: \Aff(T(A))/\overline{\rho_A(K_0(A))}\to
%\Aff(T(B))/\overline{\rho_B(K_0(B))}.
%$$
%
%Let $\lambda: U(A)/CU(A)\to U(B)/CU(B)$ be a continuous \hm. We say
%$\kappa,$ $\gamma$ and $\lambda$ are compatible, if $\kappa$ and
%$\gamma$ are compatible,
%$$
%\lambda|_{\Aff(T(A))/\overline{\rho_A(K_0(A))}}=\overline{\gamma_*}
%$$
%and $\lambda$ induces a \hm\, from $K_1(A)$ into $K_1(B)$ which
%coincides with $\kappa|_{K_1(A)}.$
%}

%\end{df}

\section{Definition of ${\cal C}_g$}

\begin{df}\label{Cuntz}
{\rm Let $A$ be a \CA\, and let $a, b\in A_+.$ Recall that we write
$a\lesssim b$ if there exists a sequence $x_n\in A_+$ such that
$$
\lim_{n\to\infty}\|x_n^*bx_n-a\|=0.
$$
If $a=p$ is a projection  and $a\lesssim b,$ there is a projection $q\in Her(b)$ and a partial isometry
$v \in A$ such that $vv^*=p$ and $v^*v=q.$
}
\end{df}

\begin{df}\label{fep}
{\rm
Let $0<d<1.$ Define $f_d\in C_0((0,\infty])$ by
$f_d(t)=0$ if $t\in [0,d/2],$ $f_d(t)=1$ if $t\in [d, \infty),$ and $f_d$ is linear in $(d/2, d).$
}
\end{df}

%\begin{len}\label{LIFTCU}
%Let $A$ be a unital separable \CA,\,let ${\cal U}\subset U(A)$ be a finite subset and let $r\ge 1$ be an integer.
%Then, for any $\ep>0,$ there exists $\dt>0$ and a finite subset ${\cal G}\subset A$ satisfying the following: if $B$ is a unital \CA\, of stable rank one and with  exponential rank no more than $r,$ and if $L: A\to B$ is a $\dt$-${\cal G}$-multiplicative linear map, there
%exists a \hm\, $\lambda : G_u\to U(B)/CU(B)$ such that
%\beq\label{LIFTCU-1}
%{\rm dist}(\overline{\langle L(u)\rangle}, \lambda(\bar{u}))<\ep
%\eneq

\begin{df}\label{IK}
{\rm
Denote by ${\cal I}^{(0)}$ the class of finite dimensional \CA s and denote by ${\cal I}^{(1,0)}$ the class of \CA s
with the form $C([0,1])\otimes F,$ where $F\in {\cal I}^{(0)}.$
For an integer $k\ge 1,$ denote by ${\cal I}^{(k)}$ the class of \SCA\, with the form
$PM_r(C(X))P,$ where $r\ge 1$ is an integer, $X$ is a finite CW complex of covering dimension at most $k$ and $P\in M_r(C(X))$ is a projection.
}
\end{df}

\begin{df}\label{Ik}
{\rm
Denote by ${\cal I}_k$ the class of those \CA s which are quotients of \CA s in ${\cal I}^{(k)}.$
Let $C\in {\cal I}_k.$ Then $C=PM_r(C(X))P,$ where $X$ is a compact subset of a finite CW complex, $r\ge 1$ and $P\in M_r(C(X))$ is a projection.
Furthermore, there exists a finite CW complex $Y$ of dimension $k$ such that $X$ is a compact subset of $Y$ and there is a projection $Q\in M_r(C(Y))$ such that
$\pi(Q)=P,$ where $\pi: M_r(C(Y))\to M_r(C(X))$ is the quotient map defined by
$\pi(f)=f|_X.$

}
\end{df}

\begin{df}\label{Dlocal}
{\rm Let $A$ be a unital \CA. We say that $A$ is a locally AH-algebra, if for any
finite subset ${\cal F}\subset A$ and any $\ep>0,$ there exists
a \SCA\, $C\in {\cal I}_k$ (for some $k\ge 0$) such that
$$
{\rm dist}(a, C)<\ep\rforal a\in {\cal F}.
$$

$A$ is said to be locally AH-algebra with no dimension growth, if there exists an integer $d\ge 0,$
such that, for any finite subset ${\cal F}\subset A$ and any $\ep>0,$
there exists a \SCA\, $C\subset A$  with the form
$C=PM_r(C(X))P\in {\cal I}_d$ such that
\beq
{\rm dist}(a, C)<\ep\rforal a\in {\cal F}.
\eneq
%$A$ is said to be a locally AH-algebra with slow dimension growth, if
%for any finite subset ${\cal F}\subset A,$ any $\ep>0$ and any $\eta>0,$
%there exists a \SCA\, $C\subset A$  with the form
%$C=PM_r(C(X))P\in {\cal I}_d$ such that
%\beq
%{\rm dist}(a, C)<\ep\rforal a\in {\cal F}\andeqn\\\label{slow}
%{d+1\over{{\rm rank}(P(x))}}<\eta\rforal x\in X.
%\eneq
}
 \end{df}

\begin{df}\label{Drr}
{\rm
Let $g: \N\to \N$ be a nondecreasing map.
Let $A$ be a unital simple \CA. We say that $A$ is in ${\cal C}_g$  if the following holds:
For any finite subset ${\cal F}\subset A,$ any $\ep>0,$ any $\eta>0$ and  any
$a\in A_+\setminus \{0\},$
there is a projection $p\in A$ and
a \SCA\, $C=PM_r(C(X))P\in {\cal I}_d$ with $1_C=p$ such that
\beq\label{Drr-1}
\|pa-ap\|&<&\ep\rforal a\in {\cal F},\\\label{Drr-2}
{\rm dist}(pap, C)&<&\ep\rforal a\in {\cal F},\\\label{Drr-3}
{d+1\over{{\rm rank}(P(x))}}&<&{\eta\over{g(d)+1}}\rforal x\in X\andeqn\\\label{Drr-4}
1-p &\lesssim & a.
\eneq

If $g(d)=1$ for all $d\in \N,$ we say $A\in {\cal C}_1.$

Let ${\cal B}$ be a class of unital \CA s.
Let $A$ be a unital simple \CA. We say that $A$ is  tracially  in ${\cal B}$  if the following holds:
For any finite subset ${\cal F}\subset A,$ any $\ep>0$ and  any
$a\in A_+\setminus \{0\},$
there is a projection $p\in A$ and
a \SCA\, $ B\in {\cal B}$ with $1_B=p$ such that
\beq\label{Trinb-1}
\|pa-ap\|&<&\ep\rforal a\in {\cal F},\\\label{Trinb-2}
{\rm dist}(pap, B)&<&\ep\rforal a\in {\cal F}\andeqn\\\label{Trunb-3}
1-p&\lesssim & a.
\eneq

We write $A\in {\cal C}_{g,\infty},$ if $A$ is tracially ${\cal I}_d$ for some integer $d\ge 0.$

Using the fact that $A$ is unital simple, it  is easy to see  that if $A\in {\cal C}_{g, \infty}$ then $A\in {\cal C}_g.$ Moreover, ${\cal C}_g\subset
{\cal C}_1$ for all $g.$

}
\end{df}

%\begin{df}\label{DC11}
%Let $A$ be a unital simple \CA\, in $C_1.$ If, (\ref{Drr-3}) can be replaced by
%\beq\label{Drr-3+}
%{d+1\over{{\rm rank }P(x)}}<{\eta d\over{K(d)}}\tforal x\in X,
%\eneq
%then we say $A$ is in ${\cal C}_{1,1}.$
%
%A locally AH-algebra is said to have extremely slow dimension growth, if  in the definition
%of \ref{Dlocal}, (\ref{slow}) is replaced by
%\beq\label{extremslow}
%{d+1\over{{\rm rank}\, P(x)}}<{\eta d\over{K(d)}}\tforal x\in X.
%\eneq
%\end{df}
%\begin{rem}\label{Extremslow}
%It is likely that every unital separable simple \CA\,  in ${\cal C}_1$ is actually in ${\cal C}_{1,1}.$
%Please see  \ref{FT1} and Appendix \ref{Zstable}.
%  In fact, from  the work of  N. C. Phillips (\cite{Ph1}, \cite{Ph2}, \cite{Ph3}, \cite{Ph4} and \cite{Ph5}),
%the ratio $K(d)/d$ (described in \ref{DKd}) is likely bounded.  We will not make an effort
%to prove this here. It only causes a minor technical difficulty at the end.
%Our main result shows that a unital separable simple  locally AH \CA\,  with slow dimension growth
%is actually an AH-algebra with no dimension growth.
%\end{rem}

\section{\CA s in ${\cal C}_g$ }

\begin{prop}\label{herd}
Every unital hereditary \SCA\, of a unital simple \CA\, in ${\cal C}_g$ {\rm (}or in ${\cal C}_{g,\infty}${\rm )}
is in ${\cal C}_g${\rm (}or is in ${\cal C}_{g,\infty}${\rm )}.
\end{prop}

\begin{proof}
Let $A$ be a unital simple \CA\, in ${\cal C}_g$   and let $e\in A$ be a non-zero projection. Let $B=eAe$ be a unital hereditary \SCA\, of $A.$
To prove that $B$ is in  ${\cal C}_g,$  let ${\cal F}\subset B$ be a finite subset, let $1>\ep>0,$ $b\in B_+\setminus \{0\}$ and let $\eta>0.$

Since $A$ is simple, there are $x_1, x_2,...,x_m\in A$ such that
\beq\label{herd-1}
\sum_{i=1}^m x_i^*ex_i=1_A.
\eneq
Let ${\cal F}_1=\{e\}\cup {\cal F}\cup\{x_1, x_2,...,x_m, x_1^*, x_2^*,...,x_m^*\}.$
Put $M=\max\{\|a\|: a\in {\cal F}_1\}.$ Since
$A\in {\cal C}_g,$  there is a projection  $p\in A,$ a \SCA\,
$C\subset A$ with $1_C=p,$ $C=PM_r(C(X))P,$ where $X$ is a compact subset of a  finite CW complex with dimension $d,$ $r\ge 1$ is an integer and $P\in M_r(C(X))$ is a projection, such that
\beq\label{herd-2}
{d+1\over{{\rm rank}P(\xi)}}&<&(\eta /16(m+1))(1/(g(d)+1))\tforal \xi\in X,\\
\label{herd-3}
\|px-xp\|&<& {\ep\over{256(m+1)M}}\tforal x\in {\cal F},\\
{\rm dist}(pxp, C)&<& {\ep\over{256(m+1)M}}\tforal x\in {\cal F}\andeqn\\
1-p &\lesssim & b.
\eneq

There is a projection $q_1\in C,$ a projection $q_2\in B$
and $y_1,y_2,...,y_m\in C$ such that
\beq\label{herd-4}
&&\|q_1-pep\|<{\ep\over{128(m+1)M}},\\
&&\|q_2-epe\|< {\ep\over{128(m+1)M}}\andeqn
\sum_{i=1}^m y_i^*q_1y_i = p.
\eneq
It follows that $q_1(x)$ has rank at least ${\rm rank}\,P(x)/m$ for each $x\in X.$ One also has
\beq\label{herd-4-}
\|q_1-q_2\|<{\ep\over{64(m+1)M}}.
\eneq
There is a unitary $w\in A$ such that
\beq\label{herd-4++}
\|w-1\|<{\ep\over{32(m+1)M}}\andeqn w^*q_1w=q_2.
\eneq
Define $C_1=w^*q_1Cq_1w.$ Then $C_1\cong q_1Cq_1\in {\cal I}_d.$ Moreover,
\beq\label{herd-5}
{d+1\over{{\rm rank}(w^*q_1w)(x)}}<\eta /(g(d)+1)\tforal x\in X.
\eneq
For any $a\in {\cal F},$
\beq\label{herd-5-1}
\|q_2a-aq_2\|&\le & 2\|q_1-q_2\|\|a\|+\|q_1eae-eaeq_1\|\\
&<&\ep.
\eneq
Moreover,
if $c\in C$ such that
\beq\label{herd-5+1}
\|pap-c\|<{\ep\over{256(m+1)M}},
\eneq
then
\beq\label{herd-5+2}
\|q_1aq_1-q_1cq_1\|&\le & \|q_1aq_1-pepapep\|+\|pepapep-q_1cq_1\|\\\label{herd-5+2+}
&<&{2\ep\over{256(m+1)}}+{2\ep\over{256(m+1)}}={\ep\over{64(m+1)}}.
\eneq
It follows from  (\ref{herd-4}), (\ref{herd-4-}), (\ref{herd-4++}) and ( \ref{herd-5+2+}) that
\beq\label{herd-5+3}
\|q_2aq_2-w^*q_1cq_1w\|&=&\|q_2aq_2-w^*q_2waw^*q_2w\|\\
&+&\|w^*q_2waw^*q_2w-w^*q_1cq_1w\|\\
&<&{4\ep\over{32(m+1)M}}+\|q_2waw^*q_2-q_1cq_1\|\\
&<& {\ep\over{8(m+1)}}+{2\ep\over{32(m+1)}}+\|q_2aq_2-q_1cq_1\|\\
&<& {6\ep\over{32(m+1)}}+{2\ep\over{64(m+1)}}+\|q_1aq_1-q_1cq_1\|\\
&<&{7\ep\over{32(m+1)}}+{\ep\over{64(m+1)}}={15\ep\over{64(m+1)}}.
\eneq
Therefore, for all $a\in {\cal F},$
\beq\label{herd-5+4}
{\rm dist}(q_2aq_2, C_1)<\ep.
\eneq
Note that
\beq\label{herd-6}
\|(e-q_2)-(e-q_2)(1-p)(e-q_2)\|&\le &
\|(e-q_2)-e(1-p)e\|\\
&\le & \|q_2-epe\|<{\ep\over{128(m+1)M}}.
\eneq
So, in particular,
$(e-q_2)(1-p)(e-q_2)$ is invertible in $(e-q_2)A(e-q_2).$ Therefore
$[e-q_2]=[(e-q_2)(1-p)(e-q_2)]$ in the Cuntz equivalence.
But
\beq\label{herd-7}
(e-q_2)(1-p)(e-q_2)&\lesssim & (1-p)(e-q_2)(1-p)\lesssim (1-p)\\
&\lesssim  & b.
\eneq
We conclude that, in $B,$
\beq\label{herd-8}
(e-q_2)\lesssim b.
\eneq
It follows from (\ref{herd-5}), (\ref{herd-5-1}), (\ref{herd-5+4}) and (\ref{herd-8}) that
$B$ is in ${\cal C}_g.$

Since $C_1\in {\cal I}_d,$ if we assume that
$A$ is in ${\cal C}_{g,\infty}.$  Then the above shows that
 $B\in {\cal C}_{g, \infty}.$

\end{proof}

\begin{prop}\label{SP}
A unital simple \CA\, $A$  which satisfies  condition {\rm (\ref{Drr-1})}, {\rm (\ref{Drr-2})} and {\rm (\ref{Drr-3})} in the definition of \ref{Drr}  has property (SP), i.e., every non-zero hereditary \SCA\, $B$ of $A$ contains a non-zero projection.
\end{prop}

\begin{proof}
Let $A$ be a unital simple \CA\, satisfying  the conditions (\ref{Drr-1}), (\ref{Drr-2}) and (\ref{Drr-3}) in  \ref{Drr}.
Suppose that $B\subset A$ is a hereditary \SCA. Choose $a\in B_+\setminus\{0\}$ with $\|a\|=1.$
Choose  $1/4>\lambda>0.$
 Let $f_1\in C([0,1])$ be such that $0\le f_1\le 1,$ $f_1(t)=1$ for all $t\in [1-\lambda/4, 1]$ and
$f_1(t)=0$ for $t\in [0, 1-\lambda/2].$ Put $b=f_1(a).$ Let $f_2\in C([0,1])$ be such that $0\le f_2\le 1,$ $f_2(t)=1$
for all $t\in [1-\lambda/2, 1]$ and
$f_2(t)=0$ for $t\in [0,1-\lambda].$ Put $c=f_2(a).$
Then $b\not=0,$ $bc=b$ and $b, c\in B.$

Since $A$ is simple, there are $x_1,x_2,...,x_m\in A$ such that
\beq\label{SP-1}
\sum_{i=1}^m x_i^*bx_i=1.
\eneq
Let $M=\max\{\|x_i\|: 1\le i\le m\}.$

Let  ${\cal F}=\{a, b, c, x_1, x_2,...,x_m, x_1^*, x_2^*,...,x_m^*\}.$
For any $1/16>\ep>0,$ there is a projection $p\in A$ and a \SCA\, $C\subset A$ with $1_C=p,$ where
$C=PM_r(C(X))P,$ $r\ge 1$ is an integer, $X$ is a compact subset of a finite CW complex with dimension
$d(X),$ $P\in M_r(C(X))$ is a projection such that
\beq\label{SP-2}
{d(X)+1\over{{\rm rank}P(x)}}&<&1/4m,
\\\label{SP-3}
\|pf-fp\|&<&\ep/64(m+1)(M+1)\tforal f\in {\cal F},\\
\andeqn{\rm dist}(pfp, B) &<&\ep/64(m+1)(M+1).
\eneq

There is $a_1\in C$ such that
\beq\label{SP-3+1}
\|pap-a_1\|<\ep/64(m+1)(M+1).
\eneq
By choosing sufficiently small $\ep,$
we may assume that
\beq\label{SP-4}
\|f_1(a_1)-pf_1(a)p\|<1/64,\,\,\,\|f_2(a_1)-pf_2(a)p\|<1/64,
\eneq
and  there are $y_1, y_2,...,y_m\in C$ such that
\beq\label{SP-5}
\|\sum_{i=1}^m y_i^*f_1(pap)y_i-p\|<1/16\andeqn \|\sum_{i=1}^m y_i^*f_1(a)y_i-p\|<1/16.
\eneq
We may also assume that
\beq\label{SP-5+1}
\|f_2(pap)-f_2(a_1)\|<1/64\\
\|f_1(pap)-f_1(a_1)\|<1/64.
\eneq

Let $q$ be the open projection which is given by $\lim_{k\to\infty}(f_1(a_1))^{1/k}.$ Then by (\ref{SP-5}), $q(x)$ has rank at least
${\rm rank}P(x)/m.$
Thus, in $M_r(C(X)),$ $f_1(a_1)(x)$ has rank at least $4d(X)$ for all $x\in X,$  by (\ref{SP-2}). It follows from
Proposition 3.2 of \cite{DNNP} that there is a non-zero projection $e\in \overline{f_1(a_1)Cf_1(a_1)}.$
Note that $f_2(a_1)f_1(a_1)=f_1(a_1).$ Therefore
\beq\label{SP-6}
f_2(a_1)e=e.
\eneq
It follows  from that
\beq\label{SP-7}
\|ce-e\|&\le &\|pcpe-e\|=\|pf_2(a)pe-e\|\\
&\le &\|pf_2(a)pe-f_2(a_1)e\|
+\|f_2(a_1)e-e\|<1/64.\\
\eneq
Similarly,
\beq\label{SP-8}
\|ec-e\|<1/16.
\eneq
It follows from 2.5.4 of \cite{Lnbk} that there is a projection $e_1\in \overline{cAc}\subset B$ such that
$$
\|e-e_1\|<1.
$$
\end{proof}

\begin{prop}\label{norm}
Let $A$ be a unital simple \CA\, in ${\cal C}_g$ (or in ${\cal C}_{g, \infty}$).  Then in {\rm (\ref{Drr-2})}, we can also assume that, for any $\ep>0,$
$$
\|pxp\|\ge \|x\|-\ep \rforal x\in {\cal F}.
$$
\end{prop}

\begin{proof}
The proof of this is contained in that of \ref{SP}.  Note that in the proof of \ref{SP},
(\ref{SP-5}) implies that $f_1(pap)\not=0.$ This implies that $\|pap\|\ge 1- \lambda/2.$
With $\lambda=\ep,$ this will gives $\|pap\|\ge 1-\lambda.$ This holds for any finitely many given positive elements with $\|a\|=1.$ If $0\not=\|a\|\not=1,$ it is clear that, by considering elements $a/\|a\|,$ we can have
$\|pap\|\ge \|a\|-\ep.$  In general, when $x$ is not positive, we can enlarge the finite subset ${\cal F}$ so it also contains
$x^*x.$ We omit the full proof.

\end{proof}

\begin{prop}\label{matrix}
Let $A$ be a unital simple \CA.  Then
the following are equivalent:

{\rm (1)}
$A\in {\cal C}_g$ {\rm (}or in ${\cal C}_{g,\infty}${\rm )};

{\rm (2)}
for any integer $n\ge 1,$ $M_n(A)\in {\cal C}_g$ {\rm (}or in ${\cal C}_{g,\infty}${\rm )},

{\rm (3)} for some integer $n\ge 1,$ $M_n(A)\in {\cal C}_g$ {\rm (}or in ${\cal C}_{g,\infty}${\rm)}.

\end{prop}

\begin{proof}
It is clear that (2) implies (3).  That (3) implies (1) follows from \ref{herd}.
To prove (1) implies (2),  let $n \ge 1$ be an integer, let ${\cal F}\subset M_n(A)$ be a finite subset, $\ep>0,$ $a\in M_n(A)_+\setminus \{0\}$ and
$\sigma>0.$
We first consider the case that $A\in {\cal C}_1.$

Since $A$ is simple, so is $M_n(A).$  Let $\{e_{i,j}\}$ be a matrix unit for $M_n.$
We identify $A$ with $e_{11}Ae_{11}.$ Therefore, (by 3.3.4 of \cite{Lnbk}, for example), there
is a non-zero  element $a_1\in A_+$  such that $a_1\lesssim a.$ Since $A$ has property (SP), by 3.5.7 of \cite{Lnbk}, there
are $n$ mutually orthogonal and mutually equivalent nonzero projections $q_1,q_2,...,q_n\in \overline{a_1Aa_1}.$
Choose a finite subset ${\cal F}_1\subset A$ such that
\beq\label{matrix-1}
\{(a_{i,j})_{n\times n}: a_{i,j}\in {\cal F}_1\}\supset {\cal F}.
\eneq
Since we assume that $A\in {\cal C}_1,$ there exists a projection $q\in A,$ a \SCA\, $C\subset A$ with $1_C=q,$
$C=PM_r(C(X))P\in {\cal I}_k,$ where $r\ge 1$ is an integer, $X$ is a compact subset of a finite CW complex with dimension $d(X)=k$ and $P\in M_r(C(X))$
is a projection, such that
\beq\label{matrix-2}
\|qa-aq\|&<&\ep/n^2\tforal a\in {\cal F}_1\\
{\rm dist}(qaq, C)&<&\ep/n^2\tforal a\in {\cal F}_1,\\\label{matrix-2+n}
{d(X)+1\over{{\rm rank}P(x)}}&<&\sigma /(g(k)+1)\tforal x\in X\andeqn\\
1_A-p&\lesssim & q_1.
\eneq

Now let $p={\rm diag}(\overbrace{q,q,...,q}^n),$  $Q={\rm diag}(\overbrace{P,P,...,P}^n)$ and
$B=M_n(C).$  Then $B=QM_{nr}(C(X))Q.$ Moreover
\beq\label{matrix-3}
\|pa-ap\| &< &\ep\tforal a\in {\cal F},\\
{\rm dist}(pap, B)&<&\ep\tforal a\in {\cal F},\\
{d(X)+1\over{{\rm rank}Q(x)}}&<&\sigma /n(g(k)+1) \tforal x\in X.
\eneq
Furthermore,
\beq\label{matrix-4}
1_{M_n(A)}-q\lesssim {\rm diag}(q_1, q_2,...,q_n)\lesssim  a_1\lesssim a.
\eneq
Thus $M_n(A)\in {\cal C}_1.$

For the case that $A\in {\cal C}_{g,\infty},$ the proof is  the same with obvious modifications. We omit the details.

 % However, in stead of (\ref{matrix-2+n}), we require that
%\beq\label{matrix-5}
%{d(X)+1\over{{\rm rank}P(x)}}<{\sigma d(X)\over{K(d)+1}}
%\eneq
%for all $x\in X.$ The rest of the proof remain the same.
\end{proof}

The following is certainly known.

\begin{lem}\label{trivialpro}
Let $X$ be a compact subset of a finite CW complex, $r\ge 1$ be an integer and let $E\in M_r(C(X))$ be a projection. Then, there is a projection $Q\in M_k(EM_r(C(X))E)$ for some integer $k\ge 1$ such that $QM_k(EM_r(C(X))E)Q\cong M_l(C(X))$ and there is a unitary $W\in M_k(EM_r(C(X))E)$ such that $W^*EW\le Q.$
\end{lem}

\begin{proof}
There is a decreasing sequence of finite CW complexes $\{X_n\}$ of dimension $d$ (for some integer $d\ge 1$) such that
$X=\cap_{n=1}^{\infty} X_n.$  There is an integer $n\ge 1$ and a projection
$E'\in M_r(C(X_n))$ such that $E'|_X=E.$ There is an integer $k\ge 1$ and
a projection $Q'\in M_k(E'M_r(C(X_n))E')$ which is unitarily equivalent to the constant projection ${\rm id}_{M_l}$ for some $l\ge {\rm rank} E'+d.$
Note that there is a unitary $W'\in M_k(E'M_r(C(X_n))E')$ such that
$$
(W')^*E'W'\le Q'.
$$
Let $W=W'|_X.$ Then $W^*EW\le Q.$
Let $Z\in M_k(E'M_r(C(X_n))E')$ be a unitary such that
$$
Z^*Q'Z={\rm id}_{M_l}.
$$
Let $Q=Q'|_X.$ Then $QM_k(EM_r(C(X))E)Q\cong M_l(C(X)).$

\end{proof}

\begin{lem}\label{Lpopa}
Let $C=PM_r(C(X))P,$ where $r\ge 1$ is an integer, $X$ is a compact subset set of a finite CW complex and
$P\in M_r(C(X))$ is a projection.  Suppose that $A$ is a unital \CA\, with the property (SP) and suppose
that $\phi: C\to A$ is a unital \hm.  Then, for any $\ep>0$ and
any finite subset ${\cal F}\subset C,$   there exists a non-zero projection $p\in A$ and
a finite dimensional \SCA\, $B\subset A$ with $1_B=p$
% and
%$$
%{\rm dim}B\le (\max_{x\in X} {\rm rank }P(x)\})^2m
%$$
such that
\beq\label{Lpopa-1}
\|p\phi(f)-\phi(f)p\| &<&\ep\tforal f\in {\cal F},\\\label{Lpopa-2}
{\rm dist}(\phi(f),B)&<&\ep\\\label{Lpopa-3}
\andeqn \|p\phi(f)p\| &\ge & \|\phi(f)\|-\ep\tforal x\in {\cal F}.
\eneq

\end{lem}

\begin{proof}
We first prove the case that $C=M_r(C(X)).$ It is clear that this case can be reduced further to the case that
$C=C(X).$   Denote by $C_1=\phi(C(X))\cong C(Y),$ where $Y\subset X$ is a compact subset.
To simplify notation, by considering a quotient of $C,$ if necessarily, without loss of generality, we may assume that $\phi$ is a monomorphism.

Let $\ep>0$ and let
${\cal F}=\{f_1, f_2,...,f_m\}\subset C.$ There are
$x_1, x_2,...,x_m\in X$ such that
\beq\label{Lpopa-4}
|\phi(f_i)(x_i)|=\|\phi(f_i)\|,\,\,\,i=1,2,...,m.
\eneq
There exists $\dt>0$ such that
\beq\label{Lpopa-5}
|f_i(x)-f_i(x_i)|<\ep/2
\eneq
for all $x$ for which ${\rm dist}(x, x_i)<\dt,$  $i=1,2,...,m.$
We may assume that $\{x_1, x_2,...,x_{m'}\}=\{x_1, x_2,...,x_m\}$ as a ser with $m'\le m.$
There is $\dt_0>0$ such that ${\rm dist}(x_i,x_j)\ge \dt_0$ if $i\not=j$ and $i,j=1,2,...m'.$
Let $\dt_1=\min\{\dt,\dt_0\}.$
Let $g_i\in C(X)$ defined by
$g_i(x)=1$ if ${\rm dist}(x, x_i)<\dt_1/4$ and $g_i(x)=0$ if
${\rm dist}(x, x_i)\ge \dt_1/2,$ $i=1,2,...,m'.$
For each $i,$ by property (SP), there exists a non-zero projection
$e_i\in \overline{\phi(g_i)A\phi(g_i)},$ $i=1,2,...,m'.$
Put $p=\sum_{i=1}^{m'} e_i.$ Let $B$ be the \SCA\, generated by $e_1, e_2,...,e_{m'}.$ Then $B$ is isomorphic to a direct sum of $m$ copies of
$\C.$ In particular, $B$ is of  dimension $m'.$
As in Lemma 2 of \cite{Lnappnormal}, this implies that
\beq\label{Lpopa-6}
\|p\phi(f)-\phi(f)p\|<\ep\andeqn\\
\|p\phi(f)p-\sum_{j=1}^{m'} f(x_j)e_j\|<\ep
\eneq
for all $f\in {\cal F}.$
By (\ref{Lpopa-4}),
\beq\label{Lpopa-7}
\|p\phi(f_i)p\|&\ge & \|\sum_{j=1}^{m'} f_i(x_j)e_j\|-\ep\\
&\ge & |f_i(x_i)|-\ep=\|\phi(f_i)\|-\ep,
\eneq
$i=1,2,...,m.$
This prove the case that $C=C(X).$
For the case that
 $C=M_r(C(X)),$ we note  that $\phi(e_{11})A\phi(e_{11})$ has (SP), where
 $\{e_{i,j}\}$ is a matrix unit for $M_r.$ Thus this case follows from the case that $C=C(X).$  Note that, in this case
 ${\rm dim}B=r^2m'.$

Now we consider the general case. There is an integer $K\ge 1$ and a projection $Q\in M_K(C)$ such that
$QM_K(C)Q\cong M_r(C(X)).$ By  choosing a projection with larger rank, we may assume that ${\rm rank}Q\ge {\rm rank}P+2dim(X).$
By conjugating by a unitary, we may further assume that
$Q\ge 1_C.$ Define $\Phi=\phi\otimes {\rm id}_{M_K}: M_K(C)\to M_K(A)$ and define $\Psi: QM_K(C)Q\to A_1=\Phi(Q)M_K(A)\Phi(Q)$ by
$\Psi=\Phi|_{QM_K(C)Q}.$

Now let $\ep>0$ and let ${\cal F}\subset C$ be a finite subset.
Let ${\cal F}_1=\{Q, 1_C\}\cup\{Q((a_{i,j})_{K\times K})Q: a_{i,j}\in {\cal F}\}.$
Let $M=\max\{\|f\|: f\in {\cal F}_1\}.$
%Suppose that ${\cal F}_1$ has $m$ elements.
From what we have shown, there is a projection
$p_1\in M_K(A)$ and a finite dimensional \SCA\, $B_1\subset M_K(A)$ with $1_{B_1}=p_1$ such that
\beq\label{Lpopa-8}
\|p_1\Phi(f)-\Phi(f)p_1\|&<&\ep/8(M+1)\\\label{Lpopa-8+1}
{\rm dist}(\Phi(f), B_1)&<&\ep/8(M+1)\andeqn\\\label{Lpopa-8+2}
\|p_1\Phi(f)p_1\| &\ge & \|\Phi(f)\|-\ep/8(M+1)
\eneq
for all $f\in {\cal F}_1.$
There is a projection $e\in B_1$ such that
\beq\label{Lpopa-9}
\|p_1\Phi(1_C)p_1-e\|<\ep/4(M+1).
\eneq
Then, for $f\in {\cal F}\subset C,$
\beq\label{Lpopa-10}
\hspace{-0.2in}\|e\phi(f)-\phi(f)e\| &\le &2M\|e-p_1\Phi(1_C)p_1\|+\|p_1\Phi(1_C)p_1\Phi(f)-\Phi(f)p_1\Phi(1_C)p_1\|\\
&<& \ep/4+\ep/8(M+1)+\|p_1\Phi(1_C)\Phi(f)-\Phi(f)\Phi(1_C)p_1\|\\
&<& \ep/4+\ep/8(M+1)+\|p_1\Phi(f)-\Phi(f)p_1\|\\
&<&\ep/4+\ep/8(M+1)+\ep/8(M+1)<\ep.
\eneq
Let $B=eB_1e.$ Then we also have
\beq\label{Lpopa-11}
{\rm dist}(e\phi(f)e, B)={\rm dist}(e\phi(f)e, eB_1e)<\ep/8(M+1)<\ep\tforal f
\in {\cal F}.
\eneq
It follows from (\ref{Lpopa-8+2}) that
\beq\label{Lpopa-12}
\|e\phi(f)e\| &\ge & \|p_1\Phi(1_C)p_1\Phi(f)p_1\Phi(1_C)p_1\|-\ep/8(M+1)\\
&\ge & \|p_1\Phi(1_C)\Phi(f)\Phi(1_C)p_1\|-\ep/4(M+1)\\
&=& \|p_1\Phi(f)p_1\|-\ep/2(M+1)\\
&\ge &\|\Phi(f)\|-\ep/(M+1)=\|\phi(f)\|-\ep/(M+1)
\eneq
for all $f\in {\cal F}.$
This completes the proof.
\end{proof}

\begin{prop}\label{popa}
Let $A$ be a unital simple \CA\, with the property (SP).  Suppose that $A$ satisfies
the following conditions: For any $\dt>0$ and any finite subset ${\cal G}\subset A,$
there exits a projection $q\in A,$ a \SCA\, $C\in {\cal I}_k$ for some $k\ge 1$ with $1_C=p$ such that
\beq\label{popa-1}
\|qx-xq\| &<&\dt,\\
{\rm dist}(qxq, C) &<&\ep\andeqn\\
\|qxq\| &\ge & \|x\|-\ep \tforal x\in {\cal G}.
\eneq
Then A satisfies the Popa condition:
for any $\ep>0$ and any finite subset ${\cal F}\subset A,$ there is a projection $p\in A$ and a finite dimensional
\SCA\, $B\subset A$ with $1_B=p$ such that
\beq\label{popa-2}
\|px-xp\|<\ep,\,\,
{\rm dist}(pxp, B)<\ep\andeqn \|pxp\|\ge \|x\|-\ep
\eneq
for all $x\in {\cal F}.$

\end{prop}

\begin{proof}
Let $\ep>0$ and let ${\cal F}\subset A$ be a finite subset.  By the assumption,
there is a projection $q$ and a \SCA\, $C\subset A$ with $C\in {\cal I}_k$ and with $1_C=q$ such that
\beq\label{popa-3}
\|qx-xq\|<\ep/8,\,\,
{\rm dist}(qxq, C)<\ep/8\andeqn
\|qxq\|\ge \|x\|-\ep
\eneq
for all $x\in {\cal F}.$
For each $x\in {\cal F},$ let $c_x\in C$ such that
\beq\label{popa-4}
\|pxp-c_x\|<\ep/8.
\eneq
It follows from \ref{Lpopa} that, there is $p\in qAq$ and a finite dimensional \SCA\, $B\subset qAq$  with $1_B=p$ such that
\beq\label{popa-5}
\|pc_x-c_xp\|&<&\ep/8,\\
{\rm dist}(pc_xp, B)&<&\ep/8\andeqn\\
\|pc_xp\| &\ge & \|c_x\|-\ep/8
\eneq
for all $x\in {\cal F}.$
It follows that
\beq\label{popa-6}
\|px-xp\| &\le & \|pqx-pc_x\|+\|c_xp-xqp\|\\
&<&\ep/8+\|pqxq-pc_x\|+\ep/8+\|c_xp-qxqp\|\\
&<&\ep/8+\ep/8+\ep/8+\ep/8<\ep
\eneq
for all $x\in {\cal F}.$
Also
\beq\label{popa-7}
{\rm dist}(pxp, B)&\le & \|pxp-pc_xp\|+{\rm dist}(pc_xp, B)\\
&=& \|pqxqp-pc_xp\|+\ep/8<\ep
\eneq
for all $x\in {\cal F}.$ Moreover,
\beq\label{popa-8}
\|pxp\|&=&\|pqxqp\|\ge  \|pc_xp\|-\ep/8\\
&\ge & \|c_x\|-\ep/4\ge \|qxq\|-\ep/4-\ep/8\\
&\ge & \|x\|-\ep/8-\ep/4-\ep/8\ge \|x\|-\ep
\eneq
for all $x\in {\cal F}.$

\end{proof}

\begin{cor}\label{Cpopa}
Let $A$ be a unital simple \CA\, in ${\cal C}_1.$ Then $A$ satisfies the Popa condition.
\end{cor}

\begin{thm}\label{trace}
Let $A$ be a unital separable simple \CA\, in ${\cal C}_1.$
Then $A$ is MF, quasi-diagonal and $T(A)\not=\emptyset.$
\end{thm}

\begin{proof}
By 4.2 of \cite{Lncltr1}, every unital separable \SCA\, which satisfies the Popa condition is MF (\cite{BK1}).
To see $A$ is quasi-diagonal, let $\ep>0$ and let ${\cal F}\subset A$ be a finite subset.
Let ${\cal F}_1=\{a,b, ab: a, b\in {\cal F}\}.$ Since $A$ has the Popa condition,
there is a non-zero projection $e\in A$  and a finite dimensional \SCA\, $B\subset A$ with $1_B=e$ such that
$\|ex-xe\|<\ep/2$ for all $x\in {\cal F}_1,$  $\|exe\|>\|x\|-\ep$ and ${\rm dist}(x, B)<\ep$ for all $x\in {\cal F}_1.$ There is a \morp\, $L: eAe\to B$
such that $\|L(exe)-exe\|<\ep$ for all $x\in {\cal F}_1.$  Define $L_1: A\to B$ by $L_1(x)=L(exe)$ for all $x\in A.$
Then $\|L_1(x)\|\ge \|exe\|-\ep$ for all $x\in {\cal F}_1$ and $\|L_1(a)L_1(b)-L_1(ab)\|<\ep$ for all $a,b \in {\cal F}.$
It follows from Theorem 1 of \cite{Vo} that $A$ is quasi-diagonal.
Since it is MF, it has tracial state. It follows that $A$ is finite.
\end{proof}

\section{Regularity of \CA s in ${\cal C}_1$}

The following lemma is a variation of 3.3 of \cite{DNNP}.

\begin{lem}\label{rank1L}
Let $X$ be a compact metric space with covering dimension $d\ge 0.$ Let $r\ge 1$ be an integer and $P, p\in M_r(C(X))$ be  non-zero projections such that $p\le P.$ Suppose that $rank\, P-rank\, p$ is at least $3(d+1)$ at each $x\in X$ and $a\in pM_r(C(X))p.$  Then, for any $\ep>0,$ there exists
an invertible element  $x\in PM_r(C(X))P$ such that
$$
\|a-x\|<\ep.
$$
\end{lem}

\begin{proof}
Let $A=PM_r(C(X))P$ and let $k={\rm rank\, p}.$
We first consider the case that $X$ is a  finite CW complex with dimension $d.$
In this case, by consider each summand separately, without loss of generality,  we may assume that $X$ is  connected.

Let $q\in M_N(C(X))$ for some large $N$ such that ${\rm rank \,q}=k+d+1$ and $q$ is trivial. By 6.10.3 of \cite{blbk}, $p$ is unitarily equivalent to a subprojection of $q.$ Thus, we find a projection $q_1\in M_N(C(X))$ with rank $d+1$ such
that $p\oplus q_1$ is trivial.
Since ${\rm rank}(P-p)$ is at least $3(d+1),$
by 6.10.3 of \cite{blbk} , there exists a trivial projection $p_1\in (P-p)M_r(C(X))(P-p)$ with rank $2d+1$ and $q_1$ is unitarily equivalent to  a subprojection of $p_1.$ Therefore, we may assume that $q_1\in (P-p)M_r(C(X))(P-p).$

Put $B=(p+q_1)A(p+q_1).$ Then $B\cong M_{k+d+1}(C(X)).$ Note
that $P-(p+q_1)$ has rank at least $2d+2.$ As above, find a trivial projection $q_2\in (P-(p+q_1))A(P-(p+q_1))$ with rank at least $d+1.$
Then, by 3.3 of \cite{DNNP}, there is an invertible element
$z\in (p+q_1+q_2)A(p+q_1+q_2)$ such that
$$
\|a-z\|<\ep/2.
$$
Define $x=z+\ep/2(P-(p+q_1+q_1)).$ Then $x$ is invertible in $A$ and
$
\|a-x\|<\ep.
$
This proves the case that $X$ is a finite CW complex.  In general, there exists a sequence finite CW complexes
$X_n$ with dimension at most $d$ such that $ C(X)=\lim_{n\to\infty} C(X_n)$ and
$M_r(C(X))=\lim_{n\to\infty}M_r(C(X_n)).$  Let $\phi_n: M_r(C(X_n))\to M_r(C(X))$ be the \hm\, induced by the inductive limit system.
We may assume that $\phi_n$ is unital. There are, for some $n\ge 1,$  projections $Q, p'\in M_r(C(X_n))$ such that
$$
\|\phi_n(Q)-P\|<\ep/32,\,\,\,\|\phi_n(p')-p\|<\ep/32\andeqn p'\le Q
$$
(see the proof of 2.7.2 of \cite{Lnbk}).
Moreover there is a unitary $U\in M_r(C(X))$ such that $\|U-1\|<\ep/16$ such that
$$
U^*\phi_n(Q)U=P.
$$
We may also assume that there is $b\in QM_r(C(X_n))Q$ such that
$$
\|\phi_n(b)-a\|<\ep/16.
$$
Let $b'=p'bp'\in QM_r(C(X_n))Q.$ Then
$$
\|\phi_n(b')-a\|<\ep/8.
$$
By what we have proved, there is an invertible element $c\in QM_r(C(X_n))Q$ such that
$$
\|b'-c\|<\ep/4.
$$
Note that $A$ is a unital hereditary \SCA\, of $M_r(C(X)).$ Put $x=U^*\phi_n(b')U.$  Then  $x$ is an invertible element in $A.$ We also have
$$
\|a-x\|<\ep.
$$
\end{proof}

\begin{thm}\label{srk1}
Let $A$ be a unital simple \CA\, in ${\cal C}_1.$
Then $A$ has stable rank one.
\end{thm}

\begin{proof}
We may assume that $A$ is infinite dimensional.
Let $a\in A$ be a nonzero element. We will show that $a$ is a  norm limit of invertible elements.
So we may assume that $a$ is not invertible and $\|a\|=1.$ Since $A$ is finite (\ref{trace}), $a$ is not one-sided invertible. Let $\ep>0.$  By 3.2 of \cite{Rod1},  there is a zero divisor $b\in A$ such that
\beq\label{srk-1}
\|a-b\|<\ep/2.
\eneq
We may further assume that $\|b\|\le 1.$ Therefore, by \cite{Rod1}, there is a unitary $u\in A$ such that $ub$ is orthogonal
to a non-zero positive element $x\in A.$  Put $d=ub.$ Since $A$ has (SP) (by \ref{SP}), there exists $e\in A$ such that $de=ed=0.$
Since $A$ is also simple  (for example, by 3.5.7 of \cite{Lnbk}), we may write $e=e_1\oplus e_2$ with $e_2\lesssim e_1.$

Note that $d, e_2\in (1-e_1)A(1-e_1)$ and $(1-e_1)A(1-e_1)\in {\cal C}_1$ (by \ref{herd}).  Since $(1-e_1)A(1-e_1)$ is simple, there are $x_1,x_2,...,x_m\in (1-e_1)A(1-e_1)$ such that
\beq\label{srk-2}
\sum_{i=1}^m x_i^*e_2x_i=1-e_1.
\eneq

Let $\dt>0.$  Let ${\cal F}=\{d, e_2, x_1,x_2,...,x_m,x_1^*, x_2^*,...,x_m^*\}.$
There exists a projection $p\in (1-e_1)A(1-e_1)$ and a unital \SCA\, $C\subset A$ with
$1_C=p$ and with $C=PM_r(C(X))P,$ where $r\ge 1$ is an integer, $X$ is a compact subset of a finite CW complex with dimension $d(X)$ and $P\in M_r(C(X))$ is a projection such that
\beq\label{srk-3}
\|px-xp\|&<& \dt\tforal x\in {\cal F},\\\label{srk-3+1}
{\rm dist}(pxp, C)&<&\dt\tforal x\in {\cal F},\\\label{srk-3+2}
{d(X)+1\over{{\rm rank}P(t)}}&<&1/8(m+1)\tforal t\in X\andeqn\\\label{srk-3+3}
(1-e_1-p) &\lesssim & e_2\lesssim e_1.
\eneq
With sufficiently small $\dt,$ we may assume that
\beq\label{srk-4}
\|e_2-(e_2'+e_2'')\|<\ep/16\andeqn
\|d-(d_1+d_2)\|<\ep/16,
\eneq
where $e_2'\in C$ and $e_2''\in (1-e_1-p)A(1-e_1-p)$ are  nonzero projections,
$d_1\in (p-e_2')C(p-e_2')$ and $d_2\in (1-e_1-p)A(1-e_1-p).$ Moreover,
there are $y_1,y_2,...,y_m\in C$ such that
\beq\label{srk-5}
\|\sum_{i=1}^my_i^*e_2'y_i-p\|<\ep/16.
\eneq
By (\ref{srk-3+3}), there is a partial isometry $v\in A$ such that $v^*v=1-e_1-p$ and $vv^*\le e_1.$
Put
$$
e_1'=vv^*\andeqn d_2'=(\ep/8)(e_1-e_1')+(\ep/8)v+(\ep/8)v^*+d_2.
$$
Then $(\ep/8)v+(\ep/8)v^*+d_2$ has a matrix decomposition in $(e_1'+(1-e_1-p))A(e_1'+(1-e_1-p)):$
$$
\begin{pmatrix} 0 & \ep/8\\
                        \ep/8 & d_2\end{pmatrix}.
                        $$
                        It follows that $d_2'$ is invertible in $(1-p)A(1-p).$
                        Note also
  \beq\label{srk-6}
  \|d_2-d_2'\|<\ep/8.
  \eneq
In $C,$ we have $d_1e_2'=e_2'd_1=0.$  It follows from (\ref{srk-5}) and (\ref{srk-3+2}) that
$e_2'$ has rank at least
\beq\label{srk-7}
{\rm rank P}(x)/m\ge 8d(X)+1.
\eneq
It follows from
%  3.3 of \cite{DNNP}
\ref{rank1L} that  there exists an invertible element $d_1'\in C$ such that
$$
\|d_1-d_1'\|<\ep/16.
$$
Therefore $d'=d_1'+d_2'$ is invertible   in $A.$ However,
\beq\label{srk-8}
\|d-(d_1'+d_2')\| &\le&  \|d-(d_1+d_2)\|+\|d_1-d_1'\|+\|d_2-d_2'\|\\
&<& \ep/16+\ep/16+\ep/8=\ep/4.
\eneq
Thus
\beq\label{srk-9}
\|b-u^*d'\|=\|u^*u(b-u^*d')\|=\|ub-d'\|=\|d-d'\|<\ep/4.
\eneq
Finally
\beq\label{srk-10}
\|a-u^*d'\|\le \|a-b\|+\|b-u^*d'\|<\ep/2+\ep/4<\ep
\eneq
Note that $u^*d'$ is invertible.
\end{proof}

Some version of the following is known. The  relation of $\ep$ and $\ep^2/2^9$ in the statement is not sharp but
will be needed in the proof of \ref{scomp}.

\begin{lem}\label{apa}
Let $A$ be a \CA,\, $a\in A_+$ with $0\le a\le 1$ and let $p\in A$ be a projection.
Let $1>\ep>0.$
Then
\beq\label{apa-1}
f_{\ep}(a)\lesssim f_{\ep^2/2^9}(pap+(1-p)a(1-p)).
\eneq
\end{lem}

\begin{proof}
We will work in $M_2(A)$ and identify $A$ with $e_{11}M_2(A)e_{11},$ where
$\{e_{i,j}\}_{1\le i,j\le 2}$ is a matrix unit for $M_2.$

Let
$x=\begin{pmatrix} pa^{1/2} & 0\\
                                (1-p)a^{1/2} & 0\end{pmatrix}.
                                $
Then,
\beq\label{apa-2}
x^*x=\begin{pmatrix} a & 0\\
                                  0 &0\end{pmatrix}\andeqn xx^*=\begin{pmatrix} pap & pa(1-p)\\
                                                                                                          (1-p)ap & (1-p)a(1-p)\end{pmatrix}.
 \eneq
 We compute that
 \beq\label{apa-3}
&& \|(1-f_{\ep^2/2^9}(pap))pap\|<\ep^2/2^8,\\
 && \|(1-f_{\ep^2/2^9}((1-p)a(1-p)))(1-p)a(1-p)-(1-p)a(1-p)\|<\ep^2/2^8.
 \eneq
 Moreover,
 \beq\label{apa-4}
 \|(1-f_{\ep^2/2^9}(pap))pa^{1/2}\|^2 =
 \|(1-f_{\ep^2/2^9}(pap))pap(1-f_{\ep^2/2^9}(pap))\|<\ep^2/2^8.
 \eneq
 Therefore
 \beq\label{apa-5}
 \|(1-f_{\ep^2/2^9}(pap))pa(1-p)\|<\ep/2^4.
 \eneq
 Similarly
 \beq\label{apa-6}
 \|(1-f_{\ep^2/2^9}((1-p)a(1-p)))(1-p)ap\|<\ep/2^4.
 \eneq
 Put $b={\rm diag}(pap, (1-p)a(1-p)).$ Then
 \beq\label{apa-7}
 \|xx^*-f_{\ep^2/2^9}(b)xx^*f_{\ep^2/64}(b)\|<\ep/2
\eneq
It follows from Lemma 2.2 of \cite{RrUHF} that
\beq\label{apa-8}
f_{\ep}(xx^*) &\lesssim& f_{\ep^2/2^9}(b)xx^*f_{\ep^2/2^9}(b)\le f_{\ep^2/2^9}(b)\\
\eneq
This implies that, in $A,$
$$
f_{\ep}(a)\lesssim f_{\ep^2/2^9}(pap+(1-p)a(1-p)).
$$

\end{proof}

The following is a standard compactness fact. We include here for convenience.

\begin{lem}\label{Dini}
Let $X$ be a compact metric space and $g\in C(X)$ with $0\le g(x)\le 1$ for all $x\in X.$
Suppose that $f_0$ is a lower-semi continuous function on $X$ such that
$1\ge f_0(x)>g(x)$ for all $x\in X.$ Suppose also that $f_n\in C(X)$ is a sequence
of continuous functions with $f_n(x)>0$ for all $x\in X,$ $f_n(x)\le f_{n+1}(x)$ for all $x$ and $n,$ and
$\lim_{n\to\infty} f_n(x)=f_0(x)$ for all $x\in X.$ Then there is $N\ge 1$ such that
\beq\label{Dini-1}
f_n(x)>g(x)\rforal x\in X
\eneq
and for all $n\ge N.$
\end{lem}

\begin{proof}
For each $x\in X,$ there exists $N(x)\ge 1$ such that
\beq\label{Dini-2}
f_n(x)>g(x)\tforal n\ge N(x).
\eneq
Since $f_{N(x)}-g$ is continuous, there is $\dt(x)>0$ such that
\beq\label{Dini-3}
f_{N(x)}(y)>g(y)\tforal y\in B(x, \dt(x)).
\eneq
Then $\cup_{x\in X}B(x, \dt(x))\supset X.$ There are $x_1, x_2,...,x_m\in X$ such that
$\cup_{i=1}^m B(x_i, \dt(x_i))\supset X.$  Let $N=\max\{N(x_1), N(x_2),...,N(x_m)\}.$
Then, if $n\ge N,$ for any $x\in X,$ there exists $i$ such that $x\in B(x_i, \dt(x_i)).$ Then
\beq\label{Dini-4}
f_n(x)\ge f_{N(x_i)}(x)\ge g(x).
\eneq

\end{proof}

\begin{thm}\label{scomp}
If $A$ is a unital simple \CA\, in ${\cal C}_1,$  then $A$ has the strict comparison for positive elements in the following sense:
If $a, b\in A_+$ and
$$
d_\tau(a)<d_{\tau}(b)\tforal \tau\in T(A),
$$
then $a\lesssim b,$ where $d_\tau(a)=\lim_{\ep\to 0} \tau(f_{\ep}(a)).$

.
\end{thm}

\begin{proof}
Let $a,\,b\in A_+$ be two non-zero elements such that
\beq\label{scomp-1}
d_\tau(a)<d_\tau(b)\tforal \tau\in T(A).
\eneq
For convenience, we assume that $\|a\|, \|b\|= 1.$
Let $1/2>\ep>0.$   Put $c=f_{\ep/16}(a).$
%where, for $1>t>0,$
%$f_t\in C([0,1])_+$ with $\|f_t\|=1,$ $f_t(s)=0$ if $s\in [0,t/2]$ and $f_t(s)=1$ if
%$s\in [t, 1].$
If $c$ is Cuntz equivalent to $a,$ then
zero is an isolated point in $sp(a).$ So, $a$ is Cuntz equivalent to a projection.
Then $d_\tau(a)$  as a function on $T(A)$ is continuous on $T(A).$ Since $d_\tau(b)$ (as a function
on $T(A)$)  is lower-semi continuous on $T(A),$
the inequality (\ref{scomp-1}) implies that
\beq\label{scomp-2}
r_0=\inf\{d_\tau(b)-d_{\tau}(c): \tau\in T(A)\}>0.
\eneq
Otherwise, there is a nonzero element $c'\in \overline{aAa}_+$ such that
$c'c=cc'=0.$  Therefore
\beq\label{scomp-3}
\inf\{d_\tau(b)-d_{\tau}(c): \tau\in T(A)\}>0.
\eneq
So in either way, (\ref{scomp-2}) holds.

Put $c_1=f_{\ep/64}(a).$
It follows from \ref{Dini} that there is $1>\dt_1>0$ such that
\beq\label{scomp-3+1}
\tau(f_{\dt_1}(b))>\tau(c)\ge d_{\tau}(c_1)\tforal \tau\in T(A).
\eneq
Put $b_1=f_{\dt_1}(b).$
Then
\beq\label{scomp-3+2}
r=\inf\{\tau(b_1)-d_{\tau}(c_1): \tau\in T(A)\}
\ge \inf\{\tau(b_1)-\tau(c): \tau\in T(A)\}>0.
\eneq
Note that $\|b\|=1.$ Since $A$ is simple and has (SP) (by \ref{SP}),
there is a  non-zero projection $ e\in A$ such that $b_1e=e$ and
\beq\label{scomp-4}
\tau(e)<r/8\tforal \tau\in T(A).
\eneq
Let $r_1=\inf\{\tau(e): \tau\in T(A)\}.$  Note that, since $A$ is simple and $T(A)$ is compact, $r_1>0.$
Let $b_2=(1-e)b_1(1-e).$
Thus, there is $0<\dt_2<\dt_1/2<1/2$ such that
\beq\label{scomp-3+3}
7r/8<\inf\{\tau(f_{\dt_2}(b_2))-\tau(c_1):\tau\in T(A)\}<r-r_1.
\eneq

Since $f_{\dt_2}(b_2)f_{3/4}(b_2)=f_{3/4}(b_2)$ and since $\overline{f_{3/4}(b_2)Af_{3/4}(b_2)}$ is non-zero,
there is a non-zero projection $e_1\in A$ such that $e_1f_{\dt_2}(b_2)=e_1$ with
$\tau(e_1)<r/18$ for all $\tau\in T(A).$

There are $x_1, x_2,...,x_m\in A$ such that
\beq\label{scomp-5}
\sum_{i=1}^m x_i^*e_1x_i=1.
\eneq
Let $\sigma=\min\{\ep^2/2^{17}(m+1), \dt_1/8, r_1/2^7(m+1)\}.$

By \cite{CP}, there are $z_1, z_2,...,z_K\in A$  and $b'\in A_+$ such that
\beq\label{scomp-6-1}
\|f_{\ep^2/2^{12}}(c)-\sum_{j=1}^Kz_j^*z_j\|<\sigma/4\andeqn \|f_{\dt_2}(b_2)-(b'+e_1+\sum_{j=1}^Kz_jz_j^*)\|<\sigma/4.
\eneq

Let
 $$
 {\cal F}=\{a, b, c, e, c_1, b_1, b', e_1\}\cup\{x_i, x_i^*:1\le i\le m\}\cup\{z_j, z_j^*: 1\le j\le K\}.
 $$
Since $A\in {\cal C}_1,$ for any $\eta>0,$ there exists a projection $p\in A,$ a \SCA\, $C\subset A$  with $1_C=p$ and
with the form $C=PM_r(C(X))P,$ where $r\ge 1$ is an integer, $X$ is a compact subset of a finite CW complex with
dimension $d(X)$ and $P\in M_r(C(X))$ is a projection such that
\beq\label{scomp-6}
\|px-xp\|&<&\eta\tforal x\in {\cal F},\\
{\rm dist}(pxp, C)&<&\eta\tforal x\in {\cal F}\\\label{scomp-6+2}
{d(X)+1\over{\rm rank}P(\xi)}&<&1/256(m+1)\tforal \xi\in X\andeqn\\
1-p &\lesssim& e.
\eneq
By choosing sufficiently small $\eta,$ we obtain $b_3, c_2, b''\in C_+,$ a projection $q_1\in C,$
$y_1, y_2,...,y_m\in C,$ $ z_1', z_2',...,z_K\in C$ such that
\beq\label{scomp-7}
\|pcp-c_2\|<\sigma,\,\,\, \|f_{\ep^2/2^{14}}(pcp)-f_{\ep^2/2^{14}}(c_2)\|<\sigma,\\\label{scomp-7+3}
\|pb_2p-b_3\|<\sigma,\,\,\,\|f_{\dt_2}(pb_2p)-f_{\dt_1}(b_3)\|<\sigma,\,\,\,\|f_{\dt_2/4}(pb_2p)-f_{\dt_2/4}(b_3)\|<\sigma,
\\\label{scomp-8}
\|pe_1p-q_1\|<\sigma,\,\,\,\|\sum_{i=1}^m y_i^*q_1y_i-p\|<\sigma
\eneq
and,  (using (\ref{scomp-6-1})),  such that
\beq\label{scomp-9}
\|f_{\ep^2/2^{14}}(c_2)-\sum_{j=1}^K (z_j')^*z'_j\|<\sigma \andeqn\\
\|f_{\dt_2}(b_3)-(\sum_{j=1}^K z_j'(z_j')^*+q_1+b'')\|<\sigma.
\eneq
Note that, by (\ref{scomp-8}) and (\ref{scomp-6+2}),
\beq\label{scomp-10}
{\rm rank}(q_1)(\xi)\ge {\rm rank P}(\xi)/m\ge 256(d(X)+1)\tforal \xi\in X.
\eneq
Therefore
\beq\label{scomp-10+1}
t(q_1)\ge 1/m\tforal t\in T(C).
\eneq
It follows that
\beq\label{scomp-12}
t(q_1)-2\sigma> 9(d(X)+1)/m\tforal t\in T(C).
\eneq

Therefore, by (\ref{scomp-9}),
\beq\label{scomp-11}
d_t(f_{\ep^2/2^{13}}(c_2))  +9d(X)/m  &\le & t(f_{\ep^2/2^{14}}(c_2))\le \sigma+\sum_{j=1}^Kt((z_j')^*z_j') +9d(X)/m\\
&=& \sigma+9 d(X)/m+\sum_{j=1}^Kt(z_j'(z_j')^*)\\
&\le & (t(q_1)-\sigma)+\sum_{j=1}^Kt(z_j'(z_j')^*)\\
&\le & t(f_{\dt_1}(b_3))\le
d_t(f_{\dt_1/2}(b_3))
\eneq
for all $t\in T(C).$
It follows from  3.15 of \cite{Tom-2}
\beq\label{scomp-13}
f_{\ep^2/2^{13}}(c_2)\lesssim f_{\dt_1/2}(b_2).
\eneq
By (\ref{scomp-7+3}) and Lemma 2.2 of \cite{RrUHF},
\beq\label{scomp-14}
f_{\dt_1/2}(b_3)\le f_{\dt_1/4}(pb_2p)\le pb_2p.
\eneq
By (\ref{scomp-7}),
\beq\label{scomp-15}
f_{\ep^2/2^{11}}(pcp)\lesssim f_{\ep^2/2^{12}}(c_2)\le pb_2p.
\eneq
It follows from \ref{apa} and (\ref{scomp-15})  that
\beq\label{scomp-16}
f_{\ep/2}(c) &\lesssim & f_{\ep^2/2^{11}}(pcp+(1-p)c(1-p))\lesssim f_{\ep^2/2^{11}}(pcp)\oplus (1-p)\\
&\lesssim & pb_2p +e\lesssim b_2+e \lesssim b_1\lesssim b.
\eneq
We also have
\beq\label{scomp-17}
f_{\ep}(a)\lesssim f_{\ep/2}(f_{\ep/16}(a))=f_{\ep/2}(c)\lesssim b.
\eneq
Since this holds for all $1>\ep>0,$ by 2.4 of \cite{RrUHF}, we conclude that
$$
a\lesssim b.
$$

\end{proof}

\begin{thm}\label{weak}
If $A$ is a unital separable simple \CA\, in ${\cal C}_1,$  then $K_0(A)$ is weakly unperforated Riesz group.
\end{thm}

\begin{proof}
Note that, for each integer $n\ge 1,$ by \ref{matrix},
$M_n(A)\in {\cal C}_1.$ Suppose that $p, \, q\in M_n(A)$ are two
projections such that
\beq\label{weak-1}
\tau(p)>\tau(q)\tforal \tau\in T(A).
\eneq
Then, by \ref{scomp},
$
q\lesssim p.
$
Therefore, if $x\in K_0(A)$ with $nx>0$ for some integer $n\ge 1,$
then one may write $x=[p]-[q]$ for some projections
$p, q\in M_k(A)$ for some integer $k\ge 1.$
The fact that $nx>0$ implies that
\beq\label{weak-2}
n(\tau(p)-\tau(q))>0\tforal \tau\in T(A)
\eneq
which implies that
\beq\label{weak-3}
\tau(p)>\tau(q)\tforal \tau\in T(A).
\eneq
It follows that $[p]>[q].$ So $x>0.$ This shows
that $K_0(A)$ is weakly unperforated.

To show that $K_0(A)$ is a Riesz group, let $q\le p$ be two projections
in $M_n(A)$ such that
$p=p_1+p_2,$ where $p_1$ and $p_2$ are two mutually orthogonal projections in $M_n(A).$
We need to show that there are projections $q_1, q_2\in M_n(A)$ such that $q=q_1+q_2$ and $[q_1]\le [p_1]$ and
$[q_2]\le [p_2].$   Since, by \ref{matrix}, $M_n(A)\in {\cal C}_1,$ to simplify the notation, we may assume that
$p, q\in A.$

Since $q\le p,$  we may assume that
\beq\label{weak-4}
\tau(p)>\tau(q)\tforal \tau\in T(A).
\eneq
Therefore, (since $A$ is simple and has (SP)), we obtain two non-zero projections
$p_{0,1}\le p_1$ and $ p_{0,2}\le p_2,$ and another non-zero
projection $e_{00}\in A$ such that
\beq\label{weak-5}
\tau(p')>\tau(q)+\tau(e_{00})\tforal \tau\in T(A),
\eneq
where $p'=p-p_{0,1}-p_{0,2}.$
Put $p_1'=p_1-p_{0,1}$ and
$p_2'=p_2-p_{0,2}.$ So $p'=p_1'+p_2'.$
From what has been proved, we have
\beq\label{weak-6}
q\oplus e_{00}\lesssim p'=p_1'+p_2'.
\eneq
Let $v\in A$ such that
\beq\label{weak-6+1}
v^*v=q\oplus e_{00}\andeqn vv^*\le p_1'+p_2'.
\eneq

There are
$$
x_1, x_2,...,x_{m_1}, y_1,y_2,...,y_{m_2},z_1,z_2,...,z_{m_3}\in A
$$
such that
\beq\label{weak-6+}
\sum_{i=1}^{m_1} x_i^*e_{00}x_i=1,
\sum_{i=1}^{m_2} y_i^*p_1'y_i=1\andeqn
\sum_{i=1}
^{m_3}z_i^*p_2'z_i=1.
\eneq
Let
$$
{\cal F}_0=\{x_i, y_j, z_k: 1\le i\le m_1, 1\le j\le m_2,
1\le k\le m_3\}.
$$
Define
$$
{\cal F}=\{p,q, p_1', p_2', e_{00}, v, v^*\}\cup
{\cal F}_0.
$$
%Let $1/2>\ep>0.$

Fix $\eta>0.$
Since $A\in {\cal C}_1,$ there exist a projection $e\in A$ and  a
\SCA\, $C=PM_r(C(X))P\in {\cal I}_k$ with $1_C=e$ such that
\beq\label{weak-7}
\|ex-xe\|&<&\eta\tforal x\in {\cal F},\\
{\rm dist}(exe, C)&<&\eta\tforal x\in {\cal F},\\
{k+1\over{{\rm rank}P(\xi)}}&<&{1\over{64(m_1+m_2+m_3+1)}}\tforal\xi\in X\andeqn\\
1-e &\lesssim  & p_{0,1}.
\eneq
With sufficiently small $\eta,$ we may assume that there exist
projections $p'', p_1'', p_2'', q', e_{00}'\in C,$  there are
$x_i', y_j', z_k'\in C$ ($1\le i\le m_1,$ $1\le j\le m_2$ and
$1\le k\le m_3$)
and there are projections
$q'', p_0\in (1-e)A(1-e)$ such that
\beq\label{weak-8}
\|ep'e-p''\|<1/16,\,\,\, \|p_1''-ep_1'e\|<1/16,\\\label{weak-8+1}
\|p_2''-ep_2'e\|<1/16,\,\,\,\|q'-eqe\|<1/16,\\\label{weak-8+2}
\|q''-(1-e)q(1-e)\|<1/16,\,\,\, \|p_0-(1-e)p'(1-e)\|<1/16,\\\label{weak-8+3}
p''=p_1''+p_2'',\andeqn
\|\sum_{i=1}^{m_1} (x_i')^*e_{00}'x_i'-e\|<1/16,\\\label{weak-8+5}
\|\sum_{j=1}^{m_2}(y_j')^*p_1'y_j-e\|<1/16\andeqn
\|\sum_{k=1}^{m_3}(z_k)^* p_2'z_k'-e\|<1/16.
\eneq
Moreover,
\beq\label{weak-9}
q'\oplus e_{00} \lesssim p'' \,\,\,{\rm in}\,\,\, M_2(C).
\eneq
Note that
\beq\label{weak-10}
&&{\rm rank}(e_{00})(x)\ge {\rm rank}(P(x))/m_1\ge 64(k+1),\\
&&{\rm rank}(p_1'')(x)\ge 64(k+1)\andeqn
{\rm rank} (p_2'')(x)\ge 64(k+1)
\eneq
for all $x\in X.$
Suppose that $X$ is the disjoint union of compact subsets
$X_1, X_2,...,X_N$ such that
${\rm rank}(e_{00}),$ ${\rm rank}(p_1')$ and ${\rm rank}(p_2')$
are all constant on each $X_i,$ $i=1,2,..., N.$
On each $X_i,$ there are non-negative integers
$m_{0,1}, m_{0,2},$ $m_1,$ and $m_2$ such that
\beq\label{weak-11}
{\rm rank}(q')-k=m_{0,1}+m_{0,2},\\
{\rm rank}(p_1'')=m_1',\,\,\,{\rm rank}(p_2'')=m_2',\\
m_1'-10k>m_{0,1}, m_2'-10k>m_{0,2}.
\eneq
It follows from 6.10.3 of \cite{blbk}  that $q'|_{X_i}$ has a trivial subprojection
$q_{1,i}'\le q'|_{X_i}$ such that
${\rm rank}( q_{1,i}')=m_{0,1}.$
Thus, by 6.10.3 of \cite{blbk},
\beq\label{weak-12}
q_{1,i}'\lesssim p_1''|_{X_i}.
\eneq
Now
\beq\label{weak-13}
{\rm rank} (q')|_{X_i}-{\rm rank}(q_{1,i}')+9k<m_2={\rm rank}(p_2''|_{X_i}).
\eneq
It follows from 6.10.3 of \cite{blbk} again that
\beq\label{weak-14}
q'|_{X_i}-q_{1,i}'\lesssim p_2''|_{X_i}.
\eneq
Define projections $q_1', \, q_2'\in C$ such that
\beq\label{weak-15}
q_1'|_{X_i}=q_{1,i}'\andeqn q_2'|_{X_i}=q'|_{X_i}-q_{1,i}',
\eneq
$i=1,2,...,N.$
Then
\beq\label{weak-16}
q'=q_1'+q_2'\andeqn q_1'\lesssim p_1''\andeqn q_2'\lesssim p_2''.
\eneq
Note, from  (\ref{weak-8+1}),
%\beq\label{weak-17}
$p_1''\lesssim p_1'\andeqn p_2''\lesssim p_2'.$
%\eneq
We also have
\beq\label{weak-18}
q''\le (1-e)\lesssim p_{0,1}.
\eneq
Put $q_1''=q''+q_1'.$ Then
\beq\label{weak-19}
q_1'' \lesssim p_{0,1}+p_1'=p_1.
\eneq
By (\ref{weak-8+1}) and (\ref{weak-8+2}), there exists a unitary $v\in A$ such that
\beq\label{weak-20}
v^*(q''+q')v=q.
\eneq
Define
\beq\label{weak-21}
q_1=v^*(q''+q_1'')v\andeqn q_2=v^*(q_2')v.
\eneq
Then
$
q=q_1+q_2.
$
But we also have
\beq\label{weak-23}
q_1\lesssim q''+q_1''\lesssim p_1\andeqn q_2\lesssim q_2'\lesssim p_2'\le p_2.
\eneq
This ends the proof.

\end{proof}

\begin{prop}\label{Propdiv}
Let $A$ be a unital simple \CA\, in ${\cal C}_1.$ Then, for any non-zero projections
$p$ and $q,$ and any integer $n\ge 1,$ there are mutually orthogonal projections
$p_1, p_2,...,p_n, p_{n+1}\in pAp$ such that
$$
p=\sum_{i=1}^{n+1}p_i,\,\,\, [p_i]=[p_1],\,\,\,i=1,2,...,n,
$$
 $p_{n+1}\lesssim p_1$ and $p_{n+1}\lesssim q.$
\end{prop}

\begin{proof}
There are $v_1, v_2,...,v_K\in M_K(A)$  and a projection $p'\in A$ such that
\beq\label{Propdiv-0}
\sum_{i=1}^K v_i^*p'v_i=1\andeqn  p'\le p.
\eneq

Let $\eta=\inf\{\tau(q): \tau\in T(A)\}.$  Choose an integer
$m\ge 1$ such that $1/m<\eta/2.$

 Let $1/2>\dt>0$ and ${\cal F}=\{p,p', v_i, v_i^*, 1\le i\le K\}.$
Since $A$ is in ${\cal C}_1,$ there are a projection $e\in A$ and a \SCA\, $C\subset A$ with
$1_C=e$ such that
$C=PM_r(C(X))P,$ where $r\ge 1$ is an integer, $X$ is a compact subset of a finite CW complex
with dimension $k$ and $P\in M_r(C(X))$ is a projection, and
\beq\label{Propdiv-1}
\|ex-xe\|<\dt,\,\,\,
{\rm dist}(exe, C)<\dt \tforal x\in {\cal F},\\\label{div-1+1}
{k+1\over{{\rm rank}(P(\xi))}}<{1\over{4(n+1)(K+1)}}\tforal \xi\in X\andeqn\\
\tau(1-e)<\min\{\eta/2,1/8(nmK+2)\}\tforal \tau\in T(A).
\eneq
With sufficiently small $\dt,$ we may assume that
there are  projections $e',\, e_1'\in C$ such that
\beq\label{Provdiv-2}
&&\|epe-e'\|<1/16,\,\,\, \|(1-e)p(1-e)-(1-e')\|<1/16,\\
&&\|ep'e-e_1\|<1/16 \andeqn\\
&&K{\rm rank} (e'(\xi))\ge {\rm rank }(P(\xi))\tforal \xi\in X
\eneq
It follows from (\ref{div-1+1}) that
\beq\label{Propdiv-3}
{\rm rank}(e'(\xi))\ge 4m(n+1)(k+1)\tforal \xi\in X.
\eneq
There is a trivial projection $e''\le e'$ in $PM_r(C(X))P$ such that
\beq\label{Propdiv-4}
{\rm rank}(e''(\xi))\ge (4m(n+1)-1)(k+1)\andeqn {\rm rank}(e'-e'')(\xi)\le k+1
\eneq
for all $\xi\in X.$
It follows that there are mutually orthogonal and mutually equivalent projections $p'_1,p'_2,...,p'_{n}\in C$ such that
\beq\label{Propdiv-5}
\sum_{i=1}^n p_i'\le e''\andeqn   (n+1)[p_1']\ge [e''].
\eneq
and $e''-\sum_{i=1}^{n} p_i'$ has rank less than $n.$
This implies that
\beq\label{Propdiv-6}
\tau(p_1')>\tau(e'-e'')+\tau(1-e')\tforal \tau\in T(A).
\eneq
Put $p_{n+1}'=e'-e''+(1-e').$
Then
\beq\label{Propdiv-7}
[p]=[\sum_{i=1}^{n+1}p_i'].
\eneq
Therefore there are mutually orthogonal projections $p_1, p_2,...,p_{n+1}\in eAe$ such that
\beq\label{Propdiv-8}
p=\sum_{i=1}^{n+1}p_i\andeqn [p_i]=[p_1],\,\,\,i=1,2,...,n.
\eneq
Note that
$$
[p_{n+1}']\le [q]\andeqn [p_{n+1}']\le [p_1].
$$

\end{proof}

\section{Traces}

\begin{prop}\label{TRC}
Let $A$ be a unital separable simple \CA\, in ${\cal C}_1.$
For any positive numbers $\{R_n\}$ such that $\lim_{n\to\infty} R_n=\infty,$ there exists a sequence
of $C_n=P_nM_{r(n)}(C(X_n))P_n,$ where
$r(n)\ge 1$ is an integer, $X_n$ is a finite CW complex with dimension
$k(n)$ and $P_n\in M_{r(n)}(C(X_n))$ is a projection, a sequence of projections $p_n\in A,$ a sequence of \morp s $L_n: A\to C_n$ and a sequence
of unital \hm s $h_n: C_n\to p_nAp_n$ such that
 \beq\label{TRC-1}
&&\lim_{n\to\infty} \|a-[(1-p_n)a(1-p_n)+h_n\circ L_n(a)]\|=0\tforal a\in A\\
&&{k(n)+1\over{{\rm rank}(P_n(x))}}<{1\over{R_n}} \tforal x\in X,\\
&&\lim_{n\to\infty}\sup\{\tau(1-p_n): \tau\in QT(A)\}=0\andeqn\\\label{TRC-1+}
&&\lim_{n\to\infty}\sup_{\tau\in QT(A)}|\tau(h_n\circ L_n(a))-\tau(a)|=0\tforal a\in A.
\eneq

\end{prop}

\begin{proof}
Let $\{{\cal F}_n\}\subset A$ be an increasing sequence of finite subsets whose union is dense in $A.$
Since $A$ is in ${\cal C}_1,$ there exists a sequence of projections $e_n\in A$ and a sequence of $B_n\subset A$ with $1_{B_n}=e_n$ and
$B_n=P_nM_{r_n}(C(X_n))P_n,$ where $r_n\ge 1$ is an integer, $X_n$ is a compact subset of a finite CW complex with
dimension $k(n)$ and
$P_n\in M_{r_n}(C(X_n))$ is a projection, such that
\beq\label{apptrace-3}
\|e_nx-xe_n\|&<&1/2^{n+2}\tforal x\in {\cal F}_n,\\\label{apptrace-3+1}
{\rm dist}(e_nxe_n, B_n)&<&1/2^{n+2}\tforal x\in {\cal F}_n,\\\label{apptrace-3+1+0}
{k(n)+1\over{{\rm rank}(P_n(\xi))}}&<&1/R_n\tforal \xi\in X
\andeqn\\\label{apptrace-3+2}
\tau(1-e_n)&<&1/2^{n+1}\tforal \tau\in QT(A).
\eneq
For each $x\in {\cal F}_n,$ let $y(x)\in B_n$ such that
\beq\label{apptrace-4}
\|e_nxe_n-y(x)\|<1/2^{n+2}.
\eneq
Since $B_n$ is amenable,
it  follows from Theorem 2.3.13 of \cite{Lnbk} that there exists a unital \morp\, $\psi_n: e_nAe_n\to B_n$ such that
\beq\label{apptrace-4+}
\|\psi_n(y(x))-{\rm id}_{B_n}(y(x))\|<1/2^{n+2}\tforal x\in {\cal F}_n.
\eneq
Put ${\cal G}_n=\{y(x): x\in {\cal F}_n\}.$
By Corollary 6.8 of \cite{Lnnewapp}, there exists $C_n\in {\cal I}^{(k(n))}$
with the form $C_n=Q_nM_r(C(Y_n))Q_n,$ where $Y_n$ is a finite CW complex,
\beq\label{apptrace-4+1}
\min\{{\rm rank}\, Q_n(y): y\in Y_n\}=\min\{{\rm rank}\, P_n(x): x\in X_n\}.
\eneq
 and a unital
$1/2^{n+2}$-${\cal G}_n$-multiplicative \morp\, $\psi_n': B_n\to C_n$ and a surjective \hm\, $h_n: C_n\to B_n$ such that
\beq\label{apptrace-5}
h_n\circ \psi_n'={\rm id}_{B_n}.
\eneq
Now define
\beq\label{apptrace-6}
L_n=\psi_n'\circ \psi_n.
\eneq
It is readily checked that
\beq\label{apptrace-7}
\lim_{n\to\infty}\|L_n(ab)-L_n(a)L_n(b)\|=0\tforal a,\, b\in A.
\eneq
By (\ref{apptrace-3}) and (\ref{apptrace-3+1}),
\beq\label{apptrace-6+1}
\lim_{n\to\infty}\|x-(e_nxe_n+(1-e_n)x(1-e_n))\|=0\tforal x\in A.
\eneq
Since quasitraces are norm continuous (Corollary II 2.5 of \cite{BH}), by (\ref{apptrace-3+2}), it follows that
\beq\label{apptrace-6+2}
\tau(x)&=&\lim_{n\to\infty}\tau(e_nxe_n+(1-e_n)x(1-e_n))\\
&=& \lim_{n\to\infty}(\tau(e_nxe_n)+\tau((1-e_n)x(1-e_n)))\\
&=& \lim_{n\to\infty}(\tau(e_nxe_n))\\\label{apptrace-6+5}
&=&\lim_{n\to\infty}\tau(h_n\circ L_n(x))
\eneq
for all $x\in A_+$ and all quasitraces $\tau\in QT(A).$
Also by (\ref{apptrace-3+1+0}) and (\ref{apptrace-4+1}),
$$
{k(n)+1\over{{\rm rank}\, Q_n(y)}}<1/R_n\tforal y\in Y_n.
$$

\end{proof}

\begin{cor}\label{apptrace}
Let $A$ be a unital separable simple \CA\, in ${\cal C}_1.$ Then,
there exists a sequence of unital \CA\, $A_n\in {\cal I}^{k(n)},$ a unital sequence of
\morp s $L_n: A\to A_n$ and a sequence of unital \hm s $h_n: A_n\to A$ such that
\beq\label{apptrace-1}
\lim_{n\to\infty}\sup_{\tau\in QT(A)}|\tau(h_n\circ L_n(a))-\tau(a)|=0
\eneq
for all $a\in A,$  and  for each projection $p\in A,$ there exists a sequence of projection
$p_n\in A_n$ such that
\beq\label{apptrace-2}
\lim_{n\to\infty}\|L_n(p)-p_n\|=0.
\eneq
\end{cor}

 Note that the projections $p_n\in A_n$ can be easily constructed from the construction
in the proof of \ref{TRC}.

\begin{cor}\label{quasitrace}
Let $A$ be a unital simple \CA\, in ${\cal C}_1.$ Then every quasitrace extends a trace.
\end{cor}

\begin{proof}
Let $\tau: A_+\to \R_+$ be a quasitrace in $QT(A).$  It is known that every quasitrace on $A_n$ extends a trace.
We will use the notation in the proof of \ref{TRC}.
Thus $\tau\circ h_n$ is a trace.  If $x, y\in A_+,$ then
\beq\label{qtrace-1}
\tau\circ h_n(L_n(x+y))&=&\tau\circ h_n(L_n(x)+L_n(y))\\
&=& \tau\circ h_n(L_n(x))+\tau\circ h_n(L_n(y)).
\eneq
So $\tau\circ h_n\circ L_n$ extends a state. Let $t$ be a weak limit of $\{\tau\circ L_n\phi_n\}.$
By  (\ref{apptrace-6+5}),
$t(x)=\tau(x)$ for all $x\in A_+.$
\end{proof}

In the following, if $\Omega$ is a compact convex set, $\partial_e(\Omega)$ is the set of extremal points of $\Omega.$

\begin{thm}\label{extrace}
Let $A$ be a unital separable simple \CA\, in ${\cal C}_1.$ Then
$r_A(\partial_e(T(A))=\partial_e(S(K_0(A))).$
\end{thm}

\begin{proof}
Note that, by \ref{quasitrace}, $T(A)=QT(A).$ It follows from 6.1 of \cite{RrUHF} that
 $r_A$ is surjective and $\partial_e(S(K_0(A)))\subset
r_A(\partial_e(T(A))).$
We will prove that $r_A(\partial_e(T(A)))\subset \partial_e(S(K_0(A))).$

Suppose $\tau\in \partial_e(T(A))$ and there are $s_1, s_2\in
S(K_0(A))$ such that
$$
r_A(\tau)=ts_1+(1-t)s_2
$$
for some $t\in (0,1).$  Suppose that $s_1\not=s_2.$
Then, since $A\in {\cal C}_1,$ there is a projection
$p\in A$ such that
\beq\label{extrace-1}
s_1([p])\not=s_2([p]).
\eneq
%Let
%\beq\label{extrace-2}
%\dt=|s_1([p])-s_2([p])|>0.
%\eneq

Let $A_n\in {\cal I}^{k(n)},$ $L_n$ and  $h_n$ be as in
\ref{apptrace}.
In particular, there are projections $p_n\in A_n$ such that
\beq\label{extrace-3}
\lim_{n\to\infty}\|L_n(p)-p_n\|=0\andeqn\\\label{extrace-3+1}
\lim_{n\to\infty}|\tau'(h_n(p_n))-\tau'(p)|=0\tforal \tau'\in T(A).
\eneq
Moreover,
\beq\label{extrace-3+n}
\lim_{n\to\infty}|\tau'(h_n(1_{A_n}))-1|=0\tforal \tau'\in T(A).
\eneq

For each $n,$
\beq\label{extrace-4}
\tau(h_n(q))=ts_1([h_n(q)])+(1-t)s_2([h_n(q)])
\eneq
for all projections $q\in A_n\otimes {\cal K}.$
Write
$A_n=C_{n,1}\oplus C_{n,2}\oplus \cdots \oplus C_{n,k(n)},$
where each $C_{n,i}=P_{n,i}M_{r(n,i)}(C(X_{n,i}))P_{n,i}$ and
$X_{n,i}$ is connected.
Note that $\rho_{C_{n,i}}(K_0(C_{n,i}))=\Z.$
We may assume that $h_n(C_{n,i})\not=0$ (otherwise, we delete that summand).
Therefore there are
$0\le a_{n,i}, \bt_{n,i}$ such that
\beq\label{extrace-5}
a_{n,i}\tau\circ h_n|_{K_0(C_{n,i})}&=&(s_1\circ [h_n])|_{K_0(C_{n,i})}\andeqn\\
b_{n,i}\tau\circ h_n|_{K_0(C_{n,i})}&=& (s_2\circ [h_n])|_{K_0(C_{n,i})},
\eneq
$i=1,2,...,k(n)$ and $n=1,2,....$
Since $r_A(\tau)=ts_1+(1-t)s_2,$
\beq\label{extrace-5+1}
&&ta_{n,i}\tau\circ h_n(1_{C_{n,i}})+(1-t)b_{n,i}\tau\circ h_n(1_{C_{n,i}})\\
&&=ts_1\circ [h_n(1_{C_{n,i}}])+(1-t)s_2\circ [h_n(1_{C_{n,i}})]\\
&&=\tau(h_n(1_{C_{n,i}})).
\eneq

It follows that
\beq\label{extrace-5+2}
ta_{n,i}+(1-t)b_{n,i}=1.
\eneq
Note that
\beq\label{extrace-6}
\sum_{i=1}^{k(n)}a_{n,i}\tau(h_n(1_{C_{n,i}}))=s_1([h_n(1_{A_n})])
\andeqn\sum_{i=1}^{k(n)}b_{n,i}\tau(h_n(1_{C_{n,i}}))=s_2([h_n(1_{A_n})]).
\eneq
Put
$$
a_n=\sum_{i=1}^{k(n)}a_{n,i}\tau(h_n(1_{C_{n,i}}))\andeqn b_n=\sum_{i=1}^{k(n)}b_{n,i}\tau(h_n(1_{C_{n,i}})).
$$
By (\ref{extrace-3+n}),
\beq\label{extrace-7}
\lim_{n\to\infty}a_n=1\andeqn \lim_{n\to\infty}b_n=1.
\eneq
Since $s_1$ and $s_2$ are states on $K_0(A)$ and $A$ is simple, $a_n>0$ and $b_n>0.$
Let $\pi_{n,i}: A_n\to C_{n,i}$ be the projection map.
Define
\beq\label{extrace-8}
\tau^{(n,1)}&=&({1\over{a_n}})\sum_{i=1}^{k(n)}a_{n,i}\tau\circ
h_n|_{C_{n,i}}\circ \pi_{n,i}\andeqn\\\label{extrace-8+}
\tau^{(n,2)}&=&({1\over{b_n}})
\sum_{i=1}^{k(n)}b_{n,i}\tau\circ h_n|_{C_{n,i}}\circ \pi_{n,i}.
\eneq
Therefore  $\tau^{(n,1)}$ and $\tau^{(n,2)}$ are tracial states on $A_n.$
By (\ref{extrace-5+2}),
\beq\label{extrace-9}
\tau|_{h_n(A_n)}=t(\sum_{i=1}^{k(n)}a_{n,i}\tau\circ
h_n|_{C_{n,i}}\circ \pi_{n,i})+(1-t)(\sum_{i=1}^{k(n)}b_{n,i}\tau\circ h_n|_{C_{n,i}}\circ \pi_{n,i})
\eneq
By  the definition  of $L_n,$  (\ref{extrace-3+1}), (\ref{extrace-7}), (\ref{extrace-8}) and (\ref{extrace-8+}),
\beq\label{extrace-10}
\tau(a)&=&\lim_{n\to\infty} \tau(h_n\circ L_n(a))\\\label{extrace-10+1}
&=&\lim_{n\to\infty} [t\tau^{(n,1)}(L_n(a))+(1-t)\tau^{(n,2)}(L_n(a))]
\eneq
for all $a\in A.$
Note that $\tau^{(n,i)}\circ L_n$ is  a state on $A,$ $i=1,2.$ Let $\tau_1$ and $\tau_2$ be
limit points of $\{\tau^{(n,1)}\circ L_n\}$ and $\{\tau^{(n,2)}\circ L_n\},$ respectively.  One checks easily that both
are tracial states on $A.$ By (\ref{extrace-10+1}),
\beq\label{extrace-11}
\tau=t\tau_1+(1-t)\tau_2.
\eneq
Since $\tau\in \partial_e(T(A)),$ this implies that
\beq\label{extrace-12}
\tau=\tau_1=\tau_2.
\eneq
On the other hand,
\beq\label{extrace-13}
\tau_1(p)=\lim_{n\to\infty}\tau^{(n,1)}(L_n(p))&=&\lim_{n\to\infty}s_1([p_n])=s_1([p])\andeqn\\
\tau_2(p)=\lim_{n\to\infty}\tau^{(n,2)}(L_n(p))&=& \lim_{n\to\infty}s_2([p_n])=s_2([p]).
\eneq
This contradicts the assumption that $s_1([p])\not=s_2([p]).$ It follows that $r_A(\tau)\in \partial_e(K_0(A)).$
\end{proof}

The following is a variation of Theorem 5.3 of \cite{BPT}.

\begin{thm}\label{BA}
Let $A$ be a unital separable simple \CA\, in ${\cal C}_1.$ Then
$W(A)=V(A)\sqcup {\rm LAff}_b(T(A)).$
\end{thm}

\begin{proof}
Note that $QT(A)=T(A).$
By  Theorem 4.4 of \cite{PT} (see also Theorem 2.2 of \cite{BPT} ), it suffices to prove that the map
from $W(A)$ to $V(A)\sqcup {\rm LAff}_b(T(A))$ is surjective.
The proof of that is a slight modification of that of Theorem 5.3 of \cite{BPT}. We also apply Lemma 5.2 in \cite{BPT}.
The only difference is that,  in the proof of Theorem 5.3 of \cite{BPT}, at the point where   5.1 of \cite{BPT} is used, we use
\ref{TRC} instead.
\end{proof}

\begin{cor}\label{malmostdiv}
Let $A$ be a unital separable simple \CA\, in ${\cal C}_1.$ Then $A$ has $0$-almost divisible Cuntz semigroup
(see definition 2.5 of {\rm \cite{Winv}}).

\end{cor}

\begin{proof}
Let $a\in M_{\infty}(A)$ and $k\ge 1$ be an integer.  If $\langle a\rangle$ is represented by a projection $p\in M_m(A),$ then by
\ref{Propdiv},  since $M_m(A)$ is in ${\cal C}_1,$ there is a projection $p_1\in M_m(A)$ such that
$$
k [p_1]\le [p] \le (k+1)[p_1].
$$
Now suppose that $a$ can not be represented by a projection. Then, by \ref{BA}, $\langle a \rangle \in {\rm LAff}_b(T(A)).$
Note that the function $\langle a\rangle /k\in {\rm LAff}_b(T(A)).$ It follows from \ref{BA} that there is $x\in W(A)$ such that
$x=\langle a\rangle /k.$  Then
$$
kx \le \langle  a\rangle \le (k+1) x.
$$

\end{proof}

\begin{lem}\label{nLmeasure}
Let $C=\lim_{n\to\infty}(C_n,\psi_n)$ be a unital  \CA\, such that
each $C_n$ is a separable unital amenable \CA\, with $T(C_n)\not=\emptyset$ and each $\psi_n$ is unital \hm.
Let $X$ be a compact metric space and $\Delta: (0,1)\to (0,1)$ be a non-decreasing map. Suppose that $\phi: C(X)\to C$ is a \morp\, and
\beq\label{nLm-1}
\mu_{\tau\circ \phi}(O_s)\ge \Delta(s)\tforal \tau\in T(C)
\eneq
and for all open balls $O_s$ with radius $s\ge \eta$ for some $\eta\in (0,1/9).$

Then, for any $\ep>0,$ any finite subset ${\cal F}\subset C(X),$ any finite
subset ${\cal G}\subset C,$ there
exists an integer $n\ge 1,$ a unital \morp\, $L: C(X)\to C_n$  and
a unital \morp\, $r: C\to C_n$
such that $L=r\circ \phi,$
\beq\label{nLm-2}
\|\psi_{n, \infty}\circ L(f)-\phi(f)\|&<&\ep\tforal f\in {\cal F},\\\label{nLm-3}
\|\psi_{n, \infty}\circ r(g)-g\|&<&\ep\tforal g\in {\cal G}\andeqn\\\label{nLm-4}
\mu_{t\circ L}(O_r)&\ge & \Delta(r/3)/3\tforal t\in T(C_n)
\eneq
and for all $r\ge 17\eta/8.$  Furthermore, if $\phi$ is $\ep/2$-${\cal F}$-multiplicative, we may also require that $L$ is $\ep$-${\cal F}$-multiplicative.
\end{lem}

\begin{proof}
Since each $C_n$ is amenable, by applying 2.3.13 of \cite{Lnbk},
there exists a sequence of \morp\,
$r_n: C\to C_n$ such that
\beq\label{nLm-5}
\lim_{n\to\infty} \|\psi_{n, \infty}\circ r_n(x)-x\|=0\tforal x\in C.
\eneq
Define $L_n=r_n\circ \phi,$ $n=1,2,....$
Using the definition of inductive limits, for any sufficiently large $n,$
$L_n$ can be chosen as $L$ and $r_n$ can be chosen as $r$ to satisfy
(\ref{nLm-2}) and (\ref{nLm-3}) as well as the requirement  that $L$ is
$\ep$-${\cal F}$-multiplicative, provided that $\phi$ is $\ep/2$-${\cal F}$-multiplicative.  To see that we can also find $L$ so
(\ref{nLm-4}) holds,  we will prove the following: for any integer $k\ge 1,$ there is an integer
$n\ge k$ such that,
if $f\in C(X)$ with $0\le f\le 1$ and
$\{x: f(x)=1\}$ contains an open ball $O_s$ with radius $s\ge 17\eta/8,$ then \beq\label{nLm-6}
t\circ L_n(f)\ge \Delta(s/3)/3\tforal t\in T(C_n).
\eneq
This will imply (\ref{nLm-4}).

Otherwise, there is a sequence $\{k(n)\}$ with $\lim_{n\to\infty}k(n)=\infty,$
there  is $t_n\in T(C_{k(n)})$ and $s_n\in (17\eta/8, 1/2)$ such that
\beq\label{nLm-7}
t_n\circ L_{k(n)}(f_n)<\Delta(s_n/3)/3
\eneq
for all $n$ and for some $f_n\in C(X)$ with $0\le f\le 1$ and
 $\{x: f_n(x)=1\}$ contains an open ball with radius $s_n\ge 17\eta/8.$
 Note that $\{t_n\circ r_{k(n)}\}$ is  a sequence of states
of $C.$ Let $t_0$ be a weak limit point of $\{t_n\circ r_{k(n)}\}.$
We may assume, that
$t_0(a)=\lim_{m\to\infty} t_{n(m)}\circ r_{k(n(m))}(a)$ for all $a\in C.$
By passing to a subsequence, we may assume that $s_{n(m)}\to s,$
where $s\in [17\eta/8, 1/2].$

Let $x_1, x_2,...,x_l$ be a set of finite points of $X$ such that
$\cup_{i=1}^lO(x_i, 15s/32)\supset X,$ where $O(x_i, 15s/32)$ is the open ball
with center $x_i$ and radius $15s/32.$  We may assume that
$17s/16>s_{n(m)}>15s/16+s/2^8.$  Note that $30s/32\ge \eta.$
Let $f_i\in C(X)$ satisfy $0\le f_i\le 1$ and
$f_i(x)=1$ if $x\in O(x_i, 15s/32)$ and $f_i(x)=0$ if ${\rm dist}(x, x_i)\ge 15s/32+s/2^8,$ $i=1,2,...,l.$ It follows that there are infinitely many
$f_{n(m)}$ such that $f_{n(m)}\ge f_j$ for some $j\in \{1,2,...,l\}.$
To simplify notation, by passing to a subsequence, we may assume, for all $m,$ $f_{n(m)}\ge f_j.$
By (\ref{nLm-7}),
\beq\label{nLm-8}
t_{n(m)}\circ L_{k(n(m))}(f_j)<\Delta(s/3)/3.
\eneq
One verifies that (since $C=\lim_{n\to\infty}(C_n, \psi_n)$), by (\ref{nLm-3}),
$t_0$ is a tracial state. It follows from (\ref{nLm-1}) that
\beq\label{nLm-9}
t_0(f)\ge \Delta(15s/32)
\eneq
for all $f\in C(X)$ with $0\le f\le 1$ and $\{x: f(x)=1\}$ contains
an open ball $O_{15s/32}$ with radius $15s/32.$
However, by (\ref{nLm-8}),
\beq\label{nLm-10}
t_0(f_j)<\Delta(s/3)/3.
\eneq
A contradiction.

\end{proof}

\begin{NN}
{\rm Let $C(X)$ be a compact metric space.  Suppose that $\phi: C(X)\to A$ is a monomorphism, where $A$ is a unital simple \CA\, with $T(A)\not=\emptyset.$
Then there is a nondecreasing map $\Delta: (0,1)\to (0,1)$ such that
\beq\label{Delta}
\mu_{\tau\circ \phi}(O_r)\ge \Delta(r)
\eneq
for all $\tau\in T(A)$ and all open balls with radius $r>0$
(see 6.1 of \cite{Lnnewapp}).}
\end{NN}

The following statement can be easily proved by the argument used in the proof of \ref{nLmeasure}
and that of  (\ref{TRC-1+}) of \ref{TRC}.

\begin{lem}\label{Lmeas}
Let ${\cal B}$ be a class of unital  separable amenable \CA\, $B$ with $T(B)\not=\emptyset.$
Let $A$ be a unital separable simple \CA\, which is tracially ${\cal B}.$
Let $X$ be a compact metric space.
%$r\ge 1$ be an integer and
%$E\in M_r(C(X))$ be a projection such that $E(x)\not=0.$
Suppose
that $\phi: C(X)\to A$ is a unital monomorphism with
\beq\label{Lmeas-1}
\mu_{\tau\circ \phi}(O_r)\ge \Delta(r)\tforal \tau\in T(A)
\eneq
and for all open balls with radius $r>0.$
Then, for any $a\in A_+\setminus \{0\},$ any $\eta>0,$ $\dt>0$ and
any finite subset ${\cal G}\subset C(X),$ there exists a projection
$p\in A,$ a unital \SCA\, $B\in {\cal B}$ with $1_B=p$ and a unital $\dt$-${\cal G}$-multiplicative \morp\, $\Phi: C(X)\to B$ such that
\beq\label{Lmeas-2}
\|p\phi(x)-\phi(x)p\|<\dt\tforal x\in {\cal G},\\
\|p\phi(x)p-\Phi(x)\|<\dt\tforal x\in {\cal G}\andeqn\\
\mu_{\tau\circ \Phi}(O_r)\ge \Delta(r/3)/3 \tforal \tau\in T(B)
\eneq
and for all open balls $O_r$ with radius  $r\ge \eta.$
\end{lem}

\section{The unitary group}

\begin{df}\label{DKd}
{\rm For each integer $d\ge 1,$ let $K(d)$ be an integer associated with $d$ given by Lemma 3.4 of \cite{Ph1}.
It should be noted (from the proof of Lemma 3.4 of \cite{Ph1}), or by choosing even large $K(d),$  and applying \cite{Ph4}) that, if $K\ge K(d),$ then, for any compact metric space
$X$ with covering dimension $d,$ any projection $p\in M_N(C(X))$ (for some integer $N\ge K$) with rank $p$  at least $K$ at each point $x$ and any unitary
$u\in U_0(pM_N(C(X))p),$ there are selfadjoint elements $h_1, h_2, h_3\in (pM_N(C(X))p)$ such that
$$
\|u-\exp(ih_1)\exp(ih_2)\exp(ih_3)\|<1.
$$
Let $g(d)=K(d)$ for all $d\in \N.$  Put ${\cal C}_{1,1}={\cal C}_g. $
}
\end{df}

\begin{prop}\label{exprk}
Let $A$ be a unital simple \CA\, in ${\cal C}_{1,1}$ and let
 $u\in U_0(A).$ Then, for any $\ep>0,$
there are four unitaries $u_0, u_1, u_2, u_3\in A,$ such that
$u_1, u_2, u_3$ are exponentials and $u_0$ is a unitary with
${\rm cel}(u_0)\le 2\pi$ such that
\beq\label{exprk-1}
\|u-u_0u_1u_2u_3\|<\ep/2.
\eneq
Moreover,
${\rm cer}(A)\le 6+\ep.$
\end{prop}

\begin{proof}
Let $u\in U_0(A).$ Then, for any $\pi/4>\ep>0,$ there are unitaries
$v_1, v_2,...,v_n\in U(A)$ such  that
\beq\label{exprk-2}
v_1=u,\,\,\, v_n=1\andeqn \|v_{i+1}-v_i\|<\ep/16,\,\,\,i=1,2,...,n.
\eneq

Since $A$ is simple and has (SP), there are mutually orthogonal and
mutually equivalent non-zero projections
$q_1, q_2,...,q_{4(n+1)}\in A.$
For each integer $d\ge 1, $ let $K(d)$ be an integer
given by \ref{DKd}.
Since $A\in {\cal C}_{1,1},$ there exist a projection $p\in A$ and a
\SCA\, $C=PM_r(C(X))P,$ where $r\ge 1$ is an integer, $X$ is a compact subset of a $d(X)$ dimensional CW complex and $P\in M_r(C(X))$ is a projection such that
\beq\label{exprk-3}
\|pv_i-v_ip\|&<&\ep/128,\,\,\,i=1,2,...,n\\
{\rm dist}(pv_ip, C)&<&\ep/128,\,\,\,i=1,2,...,n\\
{d(X)+1\over{{\rm rank}(P(x))}}&<&{d\over{64K(d)
}}\tforal x\in X\andeqn\\
1-p &\lesssim & q_1.
\eneq
 There are unitaries $w_i\in (1-p)A(1-p)$ with $w_0=(1-p)$
 such that
\beq\label{exprk-4}
\|w_i-(1-p)v_i(1-p)\|<\ep/16,\,\,\,i=1,2,...,n.
\eneq
Furthermore, there is a unitary $z\in C$ such that
\beq\label{exprk-5}
\|z-pup\|<\ep/16
\eneq
Therefore
\beq\label{exprk-6}
\|u-(w_1\oplus z)\|<\ep/8.
\eneq
Put $z_1=w_1+p.$
Since $1-p\lesssim q_1,$ by Lemma 6.4 of \cite{Lncltr1},
${\rm cel}(z_1)\le 2\pi+\ep/4.$
%there is a unitary $u_0\in A$ such that
%\beq\label{exprk-7}
%\|z_1-u_0\|<\ep/4.
%\eneq
By the choice of $K(d)$ (see \ref{DKd}),
%It follows from Theorem 4.5 of \cite{Phamj} that
there are three self-adjoint elements $h_1, h_2, h_3\in C$ such that
\beq\label{exprk-8}
\|z-\exp(ih_1)\exp(ih_2)\exp(ih_3)\|<1.
\eneq
Let $u_j=\exp(ih_j)+(1-p),$ $j=1,2,3.$
Then,
\beq\label{exprk-9}
\|u-z_1u_1u_2u_3\|<1+\ep/4
\eneq
 Moreover, since ${\rm cel}(z_1)\le 2\pi+\ep/4,$ there is a unitary $u_0$ such that  ${\rm cel}(u_0)\le 2\pi$ and
 $$
 \|u_0-z_1\|<\ep/2.
 $$
 It follows that
 ${\rm cer}(A)\le 6+\ep.$

\end{proof}

\begin{thm}\label{culength}
Let $A$ be a unital simple \CA\, in ${\cal C}_{1,1}.$
Suppose that $u\in CU(A).$ Then, $u\in U_0(A)$ and
\beq\label{culength-1}
{\rm cel}(u)\le 8\pi.
\eneq
\end{thm}

\begin{proof}
Since $u\in CU(A),$ it is easy to show that $[u]=0$ in $K_1(A).$ It follows from \ref{srk1} that $A$ has stable rank one. Therefore
$u\in U_0(A).$
Now let $\Lambda={\rm cel}(u).$  Let $\ep>0.$ Choose an integer $K_1\ge 1$ such that
$(\Lambda+1)/K_1<\ep/4.$
There exists a unitary $v\in U_0(A)$ which
is a finite product of commutators such that
\beq\label{cul-2}
\|u-v\|<\ep/16.
\eneq

Fix a finite subset ${\cal F}\subset A$ which contains $u$ and $v$ and among other elements. Let $\dt>0.$
Since $A$ is in ${\cal C}_{1,1},$  there is a projection $p\in A$ and
\SCA\, $C\subset A$ with $1_C=p$ and
$C=PM_r(X)P,$ where $X$ is a compact subset of a  finite CW complex with dimension
$d,$ $r\ge 1$ is an integer, $P\in M_r(X)$ is a projection
with
\beq\label{cul-3}
{\rm rank}P(\xi)>K(d)\tforal \xi\in X,
\eneq
where $K(d)$ is the integer given by Lemma 3.4 of \cite{Ph1},
such that
\beq\label{cul-4}
\|px-xp\|<\dt\tforal x\in {\cal F},\\
{\rm dist}(pxp, C)<\ep/16\tforal x\in {\cal F}\andeqn
\eneq
$$
(K_1+1)[1-p]\le [p]
$$
By choosing sufficiently large ${\cal F}$ and sufficiently small
$\dt,$ we obtain unitary $u_1\in (1-p)A(1-p)$ and
a unitary $v_1\in U_0(C)$ which is a finite product of commutators of unitaries in $C$
such that
\beq\label{cul-5}
\|u-(u_1+v_1)\|<\ep/8\andeqn\\
{\rm cel}(u_1)\le \Lambda+1 \,\,\,{\rm ( in\,\,\,} (1-p)A(1-p){\rm )}.
\eneq
Write $w=u_1+p.$ It follows from 6.4 of \cite{Lncltr1}  that
\beq\label{cul-6}
{\rm cel}(w)<2\pi+\ep/4.
\eneq
Since $v_1\in U_0(C)$ is a finite product of commutators of unitaries  in $C,$ by the choice of $K(d)$ and by applying 3.4 of \cite{Ph1},
\beq\label{cul-7}
{\rm cel}(v_1)\le 6\pi+\ep/8\,\,\,\,\,\,
{\rm (in } C\,{\rm)}.
\eneq
Thus
\beq\label{cul-8}
{\rm cel}(v_1+(1-p))\le 6\pi+\ep/4.
\eneq
It follows from (\ref{cul-5}), (\ref{cul-6}) and
(\ref{cul-8}) that
\beq\label{cul-9}
{\rm cel}(u_1+v_1)\le 2\pi+\ep/4+6\pi+\ep/4.
\eneq
By (\ref{cul-5}),
$$
{\rm cel}(u)\le (\ep/16) \pi+8\pi+\ep/2\le 8\pi+\ep.
$$
\end{proof}

\begin{thm}\label{expdiv}
Let $A$ be a unital simple \CA\, in ${\cal C}_{1,1}$  and let $k\ge 1$ be an integer. Suppose that
$u, v\in U(A)$ such that
$[u]=[v]$ in $K_1(A),$
\beq\label{expdiv-1}
u^k,\,\,\,v^k\in U_0(A)\andeqn {\rm cel}((u^k)^*v^k)\le L
\eneq
for some $L>0.$ Then
\beq\label{expdiv-2}
{\rm cel}(u^*v)\le L/k+8\pi.
\eneq
\end{thm}

\begin{proof}
%Let $d\ge 0$ be an integer. Fix another integer $K(d)$ associated with $d$ given by Lemma 3.4 of \cite{Ph1}.
Suppose that
$$
u^*v=\prod_{j=1}^{R(u)} \exp( ia_j)\andeqn (u^*)^kv^k=\prod_m^{R(v)}\exp(ib_m),
$$
where $a_j, b_m\in A_{s.a.}.$ Since  ${\rm cel}((u^*)^k v^k)\le L,$ we may assume that
$\sum_m\|b_m\|\le L$ (see \cite{Rg}). Let $M=\sum_j \|a_j\|.$ In particular, ${\rm cel}(u^*v)\le M.$

Let $\ep>0.$
Let $\dt>0$ be such that ${\dt\over{1-\dt}}<{\ep\over{2(M+L+1)}}.$
Let $\eta>0.$
Since $A\in {\cal C}_{1,1},$ there exists a projection $p\in A$ and a \SCA\, $C\in {\cal I}_d,$ where $C=PM_r(C(X))P$ and $1_C=p$ such that
\beq\label{expdiv-1+}
&&\|u-u_0\oplus u_1\|<\eta\andeqn \|v-v_0\oplus v_1\|<\eta,\\
&&u_0, v_0\in U((1-p)A(1-p)),\,\,\, u_1, v_1\in U(C),\\
&&u_0^*v_0
=\prod_{j=1}^{R(u)}\exp(i(1-p)a_j(1-p)),\,\,\,(u_1^*)^kv_1^k=\prod_{m=1}^{R(v)}\exp(i b_m'),\\
&&{d+1\over{\rm rank}P(x)}\le \min\{2\pi^2/\ep, d/K(d)+1\}\tforal x\in X;\\
&&\tau(1-p)<\dt\tforal \tau\in T(A),
\eneq
where $b_m'\in C_{s.a.}$ and $\|b_m'\|\le L+\ep/4.$

%To simplify the notation, we may assume that ${\rm rank }P(x)\not=0.$
Note that ${\rm rank}P$ is a continuous function on $X.$

It follows from 3.3 (1) of \cite{Ph1} that there exists
$a\in C_{s.a.}$ with $\|a\|\le L+\ep/4$ such that
\beq\label{expdiv-2+}
{\rm det}(\exp(i a)(u_1^*)^k v_1^k)=1 \,\,\,(\tforal x\in X).
\eneq
 Therefore
 \beq\label{expdiv-3}
 {\rm det}((\exp(ia/k)u_1^*v_1)^k)=1\tforal x\in X.
 \eneq
 It follows that
 \beq\label{expdiv-4}
 {\rm det}(\exp(ia/k)u_1^*v_1)(x)=\exp(i2l(x)\pi/k) \tforal x\in X
 \eneq
for some continuous function $l$ on $X$ with $l(x)$ being an integer.
Since $X$ is compact, $l$ has only finitely many values. We may choose
these values among $0, 1,...,k-1.$ Let $f(x)=- 2l(x)\pi/k$ for $x\in X.$
Then $f\in C(X)_{s.a.}$ and $\|f\|\le 2\pi.$ Note that
$\exp(i f/{\rm rank}P)\cdot 1_C$ commutes with $\exp(i a/k)$ and
\beq\label{expdiv-5}
\exp(i(f(x)/{\rm rank}P(x)))\cdot 1_C)\exp(ia/k)=\exp(i((f(x){\rm rank}P(x))+a(x)/k))\tforal x\in X.
\eneq
We have that
\beq\label{expdiv-6}
{\rm det}(\exp(i(f/{\rm rank}P+a/k))u_1^*v_1)(x)=1\tforal x\in X.
\eneq
It follows from 3.4 of \cite{Ph1} that
\beq\label{expdiv-7}
{\rm cel}(u_1^*v_1)\le 2\pi/(2\pi^2/\ep)+(L+\ep/4)/k+6\pi.
\eneq
By applying Lemma 6.4 of \cite{Lncltr1},
\beq\label{expdiv-8}
{\rm cel}((u_0\oplus p)^*(v_0\oplus p))\le 2\pi+\ep/2.
\eneq
It follows that, with sufficiently small $\eta,$
\beq\label{expdiv-9}
{\rm cel}(u^*v)\le 2\pi+\ep/2+ \ep/\pi+L/k+ \ep/4k+6\pi.
\eneq
The lemma follows.

\end{proof}

\begin{thm}\label{torsionfree}
Let $A$ be a unital infinite dimensional simple \CA\, in ${\cal C}_1.$ Then $U_0(A)/CU(A)$ is a torsion free and divisible group.
\end{thm}

\begin{proof}
We have shown that $A$ has stable rank one. Therefore, by Thomsen's result (\cite{Th}),
$U_0(A)/CU(A)\cong \Aff(T(A))/\overline{\rho_A(K_0(A))}.$ In particular, $U_0(A)/CU(A)$ is divisible.
To show that it is torsion free, we assume that $x\in \Aff(T(A))/\overline{\rho_A(K_0(A))}$ such that
$kx=0$ for some integer $k\ge 1.$
Let $y\in \Aff(T(A))$ such that ${\bar y}=x$ in the quotient.
So $ky\in \overline{\rho_A(K_0(A))}.$

Let $\ep>0.$ There is an element $z\in K_0(A)$ such that
\beq\label{torsionfree-1}
\sup_{\tau\in T(A)}|ky(\tau)-\rho_A(z)(\tau)|<\ep/2.
\eneq
We may assume that $z=[p]-[q],$ where $p, q\in M_n(A)$ are two projections for some integer $n\ge 1.$
By (SP), there is a non-zero projection $e\in A$ such that
$\tau(e)<\ep/4k$ for all $\tau\in T(A).$ It follows from \ref{Propdiv} that there are mutually orthogonal projections
$p_1, p_2,...,p_{k+1}\in pM_n(A)p$ and mutually orthogonal projections $q_1, q_2,...,q_{k+1}\in qM_n(A)q$ such that
\beq\label{torsionfree-2}
&&[p_{k+1}]\le [e],\,\,\,[p_1]=[p_i],\,\,\, i=1,2,...,k,\andeqn\\
&& [q_{k+1}]\le [e],\,\,\,[q_1]=[q_i],\,\,\,i=1,2,...,k.
\eneq
Then, by (\ref{torsionfree-1}),
\beq\label{torsionfree-3}
\sup_{\tau\in T(A)}|y(\tau)-\rho_A([p_1]-[q_1])(\tau)|<\ep/2k+\ep/2k<\ep.
\eneq
It follows that $y\in \overline{\rho_A(K_0(A))}.$ This implies  that $x=0.$ Therefore  that $U_0(A)/CU(A)$ is torsion free.

\end{proof}

\begin{thm}\label{cornuc}
Let $A$ be a unital simple \CA\, with stable rank one and
let $e\in A$ be a non-zero projection. Then the map
$u\mapsto u+(1-e)$ induces an isomorphism from
$U(eAe)/CU(eAe)$ onto $U(A)/CU(A).$
\end{thm}

\begin{proof}
Note, by the assumption that $A$ has stable rank one, $CU(eAe)\subset U_0(eAe)$ and $CU(A)\subset U_0(A).$
We define  a map $\Aff(T(eAe))$ to  $\Aff(T(A))$ as follows.
Let $\Lambda_B: B_{s.a}\to \Aff(T(B))$ by
$\Lambda_B(b)(\tau)=\tau(b)$ for all $\tau\in T(B)$ and $b\in B_{s.a.},$  where $b\in B$ and $B$ is a \CA.
By \cite{CP}, we will identify $\Aff(T(A))$ with $\Lambda_A(A_{s.a.})$ and
$\Aff(T(eAe))$ with $\Lambda_{eAe}((eAe)_{s.a.}).$

Define $\gamma: \Aff(T(eAe))\to \Aff(T(A))$ by
$\gamma(a)(\tau)=\tau(a)$ for all $\tau\in T(A)$ and $a\in (eAe)_{s.a.}.$
If $\tau\in T(A),$ define $t(c)={\tau(c)\over{\tau(e)}}$ for all $c\in eAe.$ Then $t$ defines a tracial state in $T(eAe).$ Therefore
if $b\in (eAe)_{s.a.}$ such that $t(a)=t(b)$ for all $t\in T(eAe),$ $\tau(a)=\tau(b)$ for $\tau\in T(A).$ So $\gamma$ is well defined.
Clearly $\gamma$ is \hm. Since $A$ is simple, $\gamma$ maps
$\rho_{eAe}(K_0(eAe))$ into $\rho_A(K_0(A)).$ Hence it
induces a \hm\,
$$
{\bar \gamma}: \Aff(T(eAe))/\overline{\rho_{eAe}(K_0(eAe))}\to \Aff(T(A))/\overline{\rho_A(K_0(A))}.
$$
Since $A$ is simple, ${\bar \gamma}$ is injective.

To see it is also surjective, let $h\in A_{s.a.}.$ Since
$A$ is simple, there is an integer $K\ge 1$ such that
\beq\label{cornuc-1}
N[1_A]\ge K[e]\ge [1_A]
\eneq
for some integer $N.$
Then, there is a partial isometry $U\in M_N(A)$ such that
\beq\label{cornuc-2}
U^*U=1_A\andeqn UU^*\le {\rm diag}(\overbrace{e,e,...,e}^K).
\eneq
Let $E={\rm diag}(\overbrace{e,e,...,e}^K).$ We will identify $EM_N(A)E$ with $M_K(eAe).$
Write $UhU^*=(h_{i,j}),$ a $K\times K$ matrix in $M_K(eAe).$
Write $h=h_+-h_-$ and
$h_+=(h_{i,j}^+)$ and $h_-=(h_{i,j}^-).$

It is known (see 2.2.2 of \cite{Lnbk}, for example) that there are
$h_{+,k}=(h_{k,i,j}^+)$ and $h_{-,k}=(h_{k,i,j}^-),$ $k=1,2,...,n$ and
$x_1,x_2,...,x_n,y_1,y_2,..., y_n\in M_K(eAe)$ (for some integer $n\ge 1$) such that
\beq\label{cornuc-3}
h_+=\sum_{k=1}^n h_{+,k},\,\,\, h_-=\sum_{k=1}^n h_{-,k},\\
x_k^*x_k=h_{+,k},\,\,\, x_kx_k^*=\sum_{j=1}^n h_{k,j,j}^+\\
y_k^*y_k=h_{-,k}\andeqn y_ky_k^*=\sum_{j=1}^n h_{k,j,j}^-.
\eneq
Note $b=\sum_{k=1}^n x_kx_k^*-\sum_{k=1}^n y_ky_k^*\in eAe.$
We have
\beq\label{cornuc-4}
\tau(b)=\tau(h_+-h_-)=\tau(h)\tforal \tau\in T(A).
\eneq
It follows that $\overline{\gamma}$ is surjective.
From this and by a theorem of Thomsen (\cite{Th}),
the map $u\mapsto u+(1-e)$ is an isomorphism from $U_0(eAe)/CU(eAe)$ onto $U_0(A)/CU(A).$
Since $A$ has stable rank one, $U(eAe)/U_0(eAe)=U(A)/U_0(A).$ It follows
that $u\to u+(1-e)$ is an isomorphism from $U(eAe)/ CU(eAe)$ onto $U(A)/CU(A).$
\end{proof}

\begin{lem}\label{cornerlength}
Let $K>1$ be an integer. Let $A$ be a unital simple \CA, let $e\in A$ be a projection, let $u\in U_0(eAe)$
and let  $w=u+(1-e).$ Suppose that $\eta>0,$
\beq\label{cornth-1}
{\rm dist}({\bar w}, {\bar 1})<\eta
\eneq
and $1-e\lesssim {\rm diag}(\overbrace{e,e,...,e}^{K-1}).$
Then, if $\eta<2,$
\beq\label{cornth-2}
{\rm dist}_{eAe}({\bar u}, {\bar e})<K\eta;
\eneq
If furthermore, $A\in {\cal C}_{1,1},$ then
\beq\label{cornth-3}
{\rm cel}_{eAe}(u)\le K\eta+8\pi.
\eneq
If $\eta=2$ and $A\in {\cal C}_{1,1},$ then
\beq\label{cornth-4}
{\rm cel}_{eAe}(u)\le K{\rm cel}(w)+8\pi.
\eneq
\end{lem}

\begin{proof}
We first consider the case that $\eta<2.$
Let $\ep>0$ such that $\eta+\ep<2.$  There is $w_1\in U(A)$ such that
$\overline{w_1}=\overline{w}$ and
\beq\label{cornth-5}
\|w_1-1\|<\eta+\ep/2<2.
\eneq
Thus there is $h\in A_{s.a}$ such that
\beq\label{cornth-6}
w_1=\exp(ih) \andeqn \|h\|<2\arcsin ({\eta+\ep/2\over{2}})<\pi.
\eneq
It follows that
\beq\label{cornth-7}
\|w_1-1\|=\|\exp(ih)-1\|=|\exp(i\|h\|)-1|.
\eneq
Since
$$
1-e\lesssim {\rm diag}(\overbrace{e,e,...,e}^{K-1}),
$$
we may view  $h$ as an element in $M_{K}(eAe)$ (see the proof of \ref{expdiv}).
It follows from \ref{cornuc} that
there is $f\in \Aff(T(eAe))$ such that
$f(\tau)=\tau(h)$ for all $\tau\in T(eAe).$  Note
that $\tau(h)=(\tau\otimes Tr_K)(h)$ for $\tau\in T(eAe).$
Since $A$ is simple, by \cite{CP} (see also 9.3 of \cite{Lncltr1}),
there exists $a\in (eAe)_{s.a.}$ such that
\beq\label{cornth-8}
\tau(a)=f(\tau)\tforal \tau\in T(eAe)\andeqn
\|a\|<\|f\|+\dt
\eneq
for some $\dt>0$ such that
\beq\label{cirnth-8+}
2K\arcsin({\eta+\ep/2\over{2}})+\dt<2K\arcsin ({\eta+\ep\over{2}}).
\eneq
It follows that
\beq\label{cornth-9}
{\overline{\Delta}}_A(\exp(ia)(u^*+(1-e))=0.
\eneq
Note that
\beq\label{cornth-9+}
&&(1-e)(\exp(ia)u^*)=(\exp(ia)u^*)(1-e)=1-e\andeqn\\
&&e\exp(ia)=\exp(ia)e=e+\sum_{i=1}^{\infty}{(ia)^k\over{k!}}.
\eneq
Therefore
\beq\label{cornth-9+1}
{\overline{\Delta}}_{eAe}(e\exp(ia)u^*)=0.
\eneq
Thus, by \cite{Th},
\beq\label{cornth-10}
\overline{e\exp(ia)}=\overline{u}\,\,\,{\rm in}\,\,\, U(eAe)/CU(eAe).
\eneq
Note that
\beq\label{cornth-11}
\|a\|&\le & \|f\|+\dt\le \sup\{{|\tau(h)|\over{\tau(e)}}: \tau\in T(A)\}
+\ep/2\\
&\le & K\|h\|+\dt\le K(2\arcsin({\eta+\ep\over{2}})).
\eneq

If
$$
K(2\arcsin({\eta\over{2}}))<\pi,
$$
$$
K(2\arcsin({\eta+\ep\over{2}})<\pi
$$
for some $\ep>0.$
In this case, we compute that
\beq\label{cornth-12}
\|\exp(ia)-1\|=\|\exp(i\|a\|)-1\|\le K(\eta+\ep).
\eneq
It follows that
\beq\label{cornth-13}
{\rm dist}(\overline{u}, \overline{e})<K(\eta+\ep)
\eneq
for all $\ep>0.$ Therefore
\beq\label{cornth-14}
{\rm dist}(\overline{u}, \overline{e})<K\eta.
\eneq

Otherwise, if $K\eta\ge 2,$
we certainly have
\beq\label{conrth-15}
{\rm dist}(\overline{u}, \overline{e})\le 2\le K\eta.
\eneq

Now we assume that $A\in {\cal C}_{1,1}.$ If $\eta<2,$ by (\ref{cornth-11}) and by \ref{culength},
\beq\label{cornth-16}
{\rm cel}_{eAe}(u)\le K(\eta+\ep)+8\pi
\eneq
for all $\ep>0.$ It follows that
\beq\label{cornth-17}
{\rm cel}_{eAe}(u)\le K\eta+8\pi
\eneq

If $\eta=2,$ choose $R=[{\rm cel}(w)]+1.$
Thus ${{\rm cel}(w)\over{R}}<1.$
There is a projection $e'\in M_{R+1}(A)$ such that
\beq\label{corn-18}
[(1-e)+e']=(K+RK)[e].
\eneq
It follows from 3.1 of \cite{Lnhomp} that
\beq\label{cornth-19}
{\rm dist}(\overline{u+(1-e)+e'}, \overline{1+e'})<{{\rm cel}(w)\over{R+1}}.
\eneq
Put $K_1=K(R+1).$  Then, from what we have proved,
\beq\label{cornth-20}
{\rm cel}(u) &\le & K_1({{\rm cel}(w)\over{R+1}})+8\pi\\
& = & K{\rm cel}(w)+8\pi.
\eneq
\end{proof}

\section{${\cal Z}$-stability}

\begin{df}\label{Dtrwnuc}
{\rm
Let $A$ be a unital separable \CA. We say $A$ has finite weak tracial nuclear dimension if the following holds:
For any $\ep>0,$ $\sigma>0,$ and any finite subset ${\cal F}\subset A,$ there exists a projection $p\in A$ and a unital \SCA\, $B$ with $1_B=p$ and
with ${\dim}_{nuc} B=m<\infty$ satisfying the following:
\beq\label{Dtrnuc-1}
\|px-xp\|&<&\ep\tforal x\in {\cal F},\\
{\rm dist}(pxp, B)&<&\ep\tforal x\in {\cal F},\\
\tau(1-p)&<&\sigma\tforal \tau\in T(A)\andeqn
T(A)\not=\emptyset.
\eneq
}
\end{df}

%\begin{df}\label{Dtrnuc}
%
%Let $A$ be a unital simple separable \CA. We say $A$ has finite tracial nuclear dimension if the following hold:
%For any $\ep>0,$ any $a\in A_+\setminus \{0\}$  and any finite subset ${\cal F}\subset A,$ there exists a projection $p\in A$ and a unital \SCA\, $B$ with $1_B=p$ and
%with ${\dim}_{nuc} B=m<\infty$ satisfying the following:
%\beq\label{Dtrnuc-1}
%\|px-xp\|<\ep\tforal x\in {\cal F},\\
%{\rm dim}(pxp, B)<\ep\tforal x\in {\cal F},\\
%1-p\lessim a.
%}
%\end{df}

The following is a based on a result of Winter (\cite{Winv}).

\begin{lem}\label{LZstable1}
Let $A$ be a unital separable simple infinite dimensional \CA\, which
has finite weak tracial nuclear dimension.
Suppose that each unital hereditary \SCA\, of $A$ has the property of tracial $l$-almost divisible
for some integer $l\ge 0$.
Then, for any integer $k\ge 1,$ there is a sequence of order zero \morp s $L_n: M_k\to A$  such that
$\{L_n(e)\}$ is central sequence of $A$ for a minimal projection $e\in M_k$
and such that, for every integer $m\ge 1,$
\beq\label{LZs1-1}
\lim_{n\to\infty}\max_{\tau\in T(A)}\{|\tau(L_n(e)^m)-1/k|\}=0.
\eneq
\end{lem}

\begin{proof}
Let $\{x_1, x_2,...,\}$ be a dense sequence in the unit ball of $A.$
Let $k\ge 1$ be an integer.
Fix an integer $n\ge 1.$
There is $1>\gamma_n>0$  such that
$$
\|a^{m/n}x-xa^{m/n}\|<1/4n,\,\,\, m=1,2,...,n
$$
for all $0\le a\le 1$ and those $x$ such that $\|x\|\le 1$ and
$$
\|ax-xa\|<\gamma_n.
$$

Since $A$ has finite weak tracial nuclear dimension, there is a projection $p_n\in A$ and a unital \SCA\, $B_n$ with $1_B=p_n$ and
with ${\dim}_{nuc} B=d_n<\infty$ satisfying the following:
\beq\label{LZs1-2}
\|p_nx_i-x_ip_n\|&<&1/4n,\,\,\,i=1,2,...,n,\\
{\rm dist}(p_nx_ip_n, B_n)&<&1/3n,\,\,\, i=1,2,...,n,\andeqn\\
\tau(1-p_n)&<&1/4n\tforal \tau\in T(A).\\
\eneq
Let $y_{i,n}\in B_n$ such that $\|y_{i,n}\|\le 1$  and
\beq\label{LZs1-2+1}
\|p_nx_ip_n-y_{i,n}\|<1/2n,\,\,\,i=1,2,...,n.
\eneq

It follows from Lemma 4.11 of \cite{Winv}, since $p_nAp_n$ has the tracial $m$-almost divisible property,  that there is an order zero
\morp\, $\Phi_n: M_k\to p_nAp_n$ such that
\beq\label{LZs1-3}
\|\Phi_n(a)y_i-y_i\Phi_n(a)\|&<&\min\{1/4n,\gamma_n/4 \} \|a\|\rforal a\in M_k,\,\,\,i=1,2,...,n,\\
\andeqn \tau({\Phi_n(\rm id}_{M_k}))&\ge& (1-1/n)\tforal \tau\in T(p_nAp_n).
\eneq
By Proposition 1.1 of \cite{Winv}, there is a \hm\, $\psi_n: C_0((0,1])\otimes M_k\to A$ such that
$\Phi_n(a)=\psi_n(\imath\otimes a)$ for all $a\in M_k,$ where $\imath(t)=t$ for $t\in (0,1].$
Let $c_{n}=\imath^{1/n}.$ Define
$L_n(a)=\psi_n(c_n\otimes a).$ Clearly $L_n$ is a zero order \morp.
%For $n,m\ge 2,$ put $a_{m,n}=\psi_n(c_n^m\otimes {\rm id}_{M_k}).$ So $a_{m,n}=\Phi_n({\rm id}_{M_k})^{m/n}.$
It follows that
$$
(L_n(e_1))^m=\psi_n(c_n^m\otimes e_1)=\psi_n((\imath\otimes e_1)^{m/n})=(\Phi_n(e_1))^{m/n}
$$
for all projections $e_1\in M_k.$
Let $\{e_{i,j}\}$ be a matrix unit for $M_k$ and denote by $e_i=e_{ii},$ $i=1,2,...,k.$
Let $x_{i,m}=\psi_n(c_n^{m/2}\otimes e_{i1}).$
Then, since  $\psi_n$ is a \hm,
$$
x_{i,m}x_{i,m}^*=L_n(e_i)^m\andeqn x_{i,m}^*x_{i,m}=L_n(e_1)^m,\,\,\,i=1,2,...,k.
$$
It follows that
\beq\label{LZs1-4}
\tau(L_n(e_1)^m)&=& \tau(\psi_n(c_n^m\otimes e_1))\\
\ge {1\over{k}}\tau(\Phi_n({\rm id}_{M_k}))&\ge& {1-1/n\over{k}}\tforal \tau\in T(p_nAp_n),
\eneq
$m=1,2,...,n.$
It follows that, for any $m\ge 1,$
\beq\label{LZs1-5}
\lim_{n\to\infty}\max_{\tau\in T(A)}|\tau(L_n(e_1)^m)-1/k|=0.
\eneq
We also have,
\beq\label{LZs1-6}
\|L_n(e_1)y_i-y_iL_n(e_1)\|&=&\|\Phi_n(e_1)^{1/n}y_i-y_i\Phi_n(e_1)^{1/n}\|\\
&\le &1/4n,\,\,\,i=1,2,...,n.
\eneq
It follows that
\beq\label{LZs1-7}
\|L_n(e_1)x_i-x_iL_n(e_1)\|<1/n,\,\,\,i=1,2,...,n.
\eneq
Since $\{x_1,x_2,...,\}$ is dense in the unit ball of $A,$ we conclude that
\beq\label{LZs1-8}
\lim_{n\to\infty} \|L_n(e_1)x-xL_n(e_1)\|=0\tforal x\in A.
\eneq
So $\{L_n(e_1)\}$ is a central sequence of $A.$
\end{proof}

We now apply the argument established in \cite{MS} to prove the following.

\begin{thm}\label{TZs-1}
Let $A$ be a unital amenable separable simple \CA\, with finite weak tracial nuclear dimension. Suppose that $A$ has the strict comparison property for positive elements and every unital hereditary \SCA\, of $A$ has the the property of
$m$-almost divisible Cuntz semigroup. Then $A$ is ${\cal Z}$-stable.
\end{thm}

\begin{proof}
By exactly the same argument as in the proof that shows that(ii) implies (iii) in \cite{MS}, using \ref{LZstable1} instead of Lemma 3.3 of \cite{MS}, one concludes that any completely positive linear map from $A$ into $A$ can be excised in small central sequence.  As in \cite{MS}, this implies that
$A$ has property (SI). Using \ref{LZstable1} instead of Lemma 3.3 of \cite{MS}, the same proof that shows that (iv) implies (i)
in \cite{MS} shows that $A$ is ${\cal Z}$-stable.
\end{proof}

\begin{cor}\label{CZs}
Let $A$ be a unital amenable separable simple \CA\, in ${\cal C}_1.$ Then
$A$ is ${\cal Z}$-stable.
\end{cor}

\begin{proof}
Note, by \ref{malmostdiv} and \ref{herd}, that
every unital hereditary \SCA\, of $A$ has the property of 0-almost divisible
Cuntz semgroup.
The lemma then  follows  from \ref{TZs-1} and the fact that each \CA\, in ${\cal I}_d$ has nuclear dimension no more than $d$ for all $d\ge 0.$

\end{proof}

\section{ General Existence Theorems}

\begin{lem}\label{Ad0}
Let $X$ be a connected finite CW complex with ${\rm dim}X=3,$ let $Y=X\setminus \{\xi\},$ where
$\xi\in X$ is a point, let $\Omega$ be a connected CW complex and let
$\Omega'=\Omega\setminus \{\omega\}.$ Let $P\in M_k(C(\Omega))$ (for some integer $k\ge 3({\rm dim}Y+1)$) be a projection
with ${\rm rank}(P)\ge 3({\rm dim}Y+1).$ Let
$\kappa\in KK(C_0(Y), C_0(\Omega'))$ such that $\kappa(F_3K_*(C_0(Y)))\subset F_3K_*(C_0(\Omega)).$
Then there is a unital \hm\,
$\phi: C(X)\to PM_k(C(\Omega))P$ such that
\beq\label{Ad0-1}
[\phi|_{C_0(Y)}]=\kappa.
\eneq

\end{lem}

\begin{proof}
This is a combination of Proposition 3.16  and Theorem 3.10 of \cite{EG}.
\end{proof}

\begin{cor}\label{Ad0+}

Let $X$ be a connected finite CW complex such that $X^{(3)}$ (see \ref{DFilt2}) has $r$ components,  let $Y=X\setminus \{\xi\},$ where
$\xi\in X$ is a point. Let $\Omega$ be a connected CW complex and let
$\Omega'=\Omega\setminus \{\omega\}$ for some point $\omega\in \Omega.$  Let $P\in M_k(C(\Omega))$ (for some integer $k\ge 3({\rm dim}Y+1)$) be a projection
with ${\rm rank}(P)\ge 3({\rm dim}Y+1).$ Let
$\kappa\in KK(C_0(Y), C_0(\Omega'))$ such that
\beq\label{nAd0+-1}
\kappa(F_3K_*(C_0(Y)))\subset F_3K_*(C_0(\Omega))\andeqn \kappa=\kappa_1\circ s_3,
\eneq
where $s_3: C(X)\to C(X^{(3)})$ is the restriction and $\kappa_1\in KK(C(X^{(3)}), C_0(\Omega'))$
satisfies
$\kappa_1|_{\rho_{K_0(C(X^{(3)}))}(C(X^{(3)})}=0$ with the composition
$$
K_0(C(X^{(3)})={\rm ker}\rho_{K_0(C(X^{(3)}))}\oplus \rho_{C(X^{(3)})}(C(X^{(3)}).
$$
Then there is a unital \hm\,
$\phi: C(X)\to PM_k(C(\Omega))P$ such that
\beq\label{nAd0-1}
[\phi|_{C_0(Y)}]=\kappa.
\eneq

\end{cor}

\begin{cor}\label{Ad0++}

Let $X$ be a connected finite CW complex,  let $Y=X\setminus \{\xi\},$ where
$\xi\in X$ is a point.  Suppose that $X^{(3)}$ has $r$ components.
 Let  $\Omega$ be a connected CW complex and let
$\Omega'=\Omega\setminus \{\omega\}$ for some point $\omega\in \Omega.$  Let $P\in M_k(C(\Omega))$ (for some integer $k\ge 3r({\rm dim}Y+1)$) be a projection
with ${\rm rank}(P)\ge 3r({\rm dim}Y+1).$ Let
$\kappa\in KK(C_0(Y), C_0(\Omega'))$ such that
\beq\label{nnAd0+-1}
\kappa(F_3K_*(C_0(Y))) &\subset & F_3K_*(C_0(\Omega)),\\
\kappa|_{F_mK_*(C_0(Y))}&=&0\tforal m\ge 4\andeqn\\
\kappa|_{K_*(C_0(Y),\Z/k\Z)}&=&0\tforal k>1.
\eneq
Then there is a unital \hm\,
$\phi: C(X)\to PM_k(C(\Omega))P$ such that
\beq\label{nnAd0-1}
[\phi|_{C_0(Y)}]=\kappa.
\eneq

\end{cor}

\begin{proof}
Let $s_3$ be as in \ref{Ad0+}. Note that ${\rm ker}(s_3)_{*}=F_4K_*(C_0(Y))$(see \ref{DFilt2}).  There is $\gamma\in Hom(K_*(C(X^{(3)})), K_*(C_0(\Omega')))$ such that
$\gamma\circ (s_3)_*|_{K_*(C_0(Y))}=\kappa|_{K_*(C_0(Y))}$ and
$\gamma|_{\rho_{C(X^{(3)})} (K_0(C(X^{(3)})))}=0.$
Define $\kappa_1|_{K_*(C(X^{(3)}))}=\gamma$ and
$\kappa_1|_{K_*(C(X^{(3)}),\Z/k\Z))}=0$ for $k>1.$  Thus \ref{Ad0+} applies.
\end{proof}

%\begin{lem}\label{FK1tq}
%Let $X$ be a connected finite CW complex, let $\ep>0$ and let ${\cal
%F}\subset C(X)$ be a finite subset. Let $Y$ be a connected finite CW
%complex with dimension 3 for which $K_0(C_0(Y))=\Z$ and $K_1(C(Y))$
%is finite.
%There exists an integer $L\ge 1$ satisfying the following: For any
%$\af\in KK(C_0(X), C_0(Y)),$  there exists a unital  ${\cal
%F}$-$\ep$-multiplicative \morp\, $\phi: C(X)\to  M_L(C(Y))$ such
%that
%\beq\label{FK1tq-1}
%[\phi|_{C_0(Y)}]=\af.
%\eneq
%\end{lem}

\begin{lem}\label{Ad0k1}
Let $X$ be a connected finite CW complex with dimension d and
with $K_1(C(X))=\Z^r\oplus {\rm Tor}(K_1(C(X)),$
 let $\dt>0,$ let  ${\cal G}\subset C(X)$ be a finite subset and let ${\cal P}\subset \underline{K}(C_0(Z))$ (where $Z=X\setminus \{\xi\}$ for some
point $\xi\in X$) be a finite subset.
There exists an integer $N_1(\dt, {\cal G}, {\cal P})\ge 1$ satisfying the following:

Let $Y$ be a connected finite CW complex,
and let $\alpha\in KK(C_0(Z), C_0(Y_0))),$
where $Y_0=Y\setminus \{y_0\}$ for some point $y_0\in Y.$
%such that
%$\alpha|_{{\rm ker}\rho_{C(X)}}\subset {\rm ker}\rho_{C(Y)}.$
%Then there exists an integer $N_1(\dt, {\cal G}, {\cal P})\ge 1,$ f
For any projection
$P\in M_{\infty}(C(Y))$ with ${\rm rank}P\ge N_1(\dt, {\cal G}, {\cal P})\cdot 3({\rm dim}Y+1),$ there exists a
unital $\dt$-${\cal G}$-multiplicative \morp\,
$L: C(X)\to C=PM_R(C(Y))P $(for some integer $R\ge {\rm rank}P$)
such that
$
[L]|_{\cal P}=\alpha|_{\cal P}.
$

Moreover,
%for a fixed \hm\,
%$$J: \Z^r\subset K_1(C(X))\to
%U(M_{d}((X)))/CU(M_{d}(C(X)))$$
%which splits the short exact sequence
%$$
%0\to \Aff(T(M_d(C(X))))/\overline{\rho_{M_d(C(X))}(K_0(C(X)))}
%\to U(M_{d}(C(X)))/CU(M_{d}(C(X)))\to K_1(C(X))\to 0
%$$
%and
suppose that
 $\lambda:J_c(\Z^r)\to
U(M_{d}(PM_R(C(Y))P))/CU(M_{d}(PM_R(C(Y))P))$ is a \hm\,
with $\Pi_c\circ \lambda\circ J_c=\af|_{K_1(C(X))}$ (see \ref{DCU}),
 one may require
that
\beq\label{Ad0k1-1}
L^{\ddag}|_{J(\Z^r)}=\lambda.
\eneq

\end{lem}

\begin{proof}
 Note that $Tor(K_1(C(X)))$ is a finite group.
%We first assume that $\alpha|_{Tor(K_1(C(X)))}\not=0

There is a connected finite CW complex $T$ with ${\rm dim}T=3$ such that
$K_0(C_0(T_0))=K_0(C_0(Z))$ and $K_1(C(T))=Tor(K_1(C(X))),$ where
$T_0=T\setminus \{\xi\}$ and where $\xi\in T$ is a point.

Let $B=\overbrace{C(\T)\oplus C(\T)\oplus\cdots\oplus C(\T)}^{r}.$
We identify $\Z^r$ with $K_1(B).$
Let
$$B_0=\overbrace{C_0(\T_0)\oplus C_0(\T_0)\oplus \cdots \oplus
C_0(\T_0)}^{r},$$
where $\T_0$ is the circle minus a single point.
%Note that $KK(B_0, PM_r(C(Y))P)=Hom(K_1(B_0), K_1(PM_r(C(Y))P)).$
%Then there is a unital \hm\, $\phi_2: B\to M_{3({\rm dim}Y+1)}(C(Y))$
%such that $[\phi_2]|_{\Z^k}=\alpha|_{\Z^k}.$
Since $K_1(C_0(\T_0))=K_1(C(\T)),$ we obtain a isomorphism
$\bt_1: K_*(C_0(Z))\to K_*(B_0\oplus C_0(T_0)).$
There is an invertible element $\bt\in KK(C_0(Z), B_0\oplus C_0(T_0))$ such that
$\bt|_{K_i(C_0(Z))}=\bt_1|_{K_i(C_0(Z))},$ $i=0,1.$
There is, by \cite{DL}, a unital
$\dt$-${\cal G}$-multiplicative \morp\, $L_1: C(X)\to M_K(B\oplus C(T))$
such that
$$
[L_1]|_{\cal P}=\bt|_{\cal P},
$$
where $K$ is an integer depending only on
$X,$ $\dt,$ ${\cal G}$ and ${\cal P}.$
We may also assume that $L_1^{\ddag}$ is defined and
$L_1^{\ddag}$ is invertible on $J_c(\Z^r).$
Put $N(\dt, {\cal G}, {\cal P})=K+(rK).$
%and
%put $\af_0=\af|_{\underline{K}(C_0(Z))}.$

%It suffices to show the case  that $Y$ is connected.
%Let $Y_0=Y\setminus\{y_0\},$ where $y_0\in Y$ is a point.
Note that $\bt^{-1}\times \af\in KK(B_0\oplus C_0(T_0), C_0(Y_0)).$
There is $\af_1\in KK(B_0, C_0(Y_0))$ and  $\af_2
\in KK(C_0(T_0),C_0(Y_0))$ such that
$\af_1\oplus \af_2=\bt^{-1}\times \af_0.$

Note that $\alpha_2(Tor(K_1(C(T))))\subset Tor(K_1(C(Y)))$ (or it is zero).
Note that $F_3(K_1(C(T)))=Tor(K_1(C(T)))$ and
$F_m(K_1(C(T)))=\{0\}$ if $m>3$ (see Lemma 3.3 of \cite{EL}).
It follows from \ref{torF3} that $Tor(K_1(C(Y)))\subset F_3(K_1(C(Y))).$ Therefore, by Theorem 3.10 and Proposition 3.16 of \cite{EG}, there is a unital \hm\, $\phi_1: C(T)\to M_s(C(Y)),$ where $s=3({\rm dim}Y+1)$  such that
$$
[\phi_1|_{C_0(T_0)}]=\alpha_2.
$$
Note that $KK(B_0, C_0(Y_0))=Hom(K_1(B_0), K_1(C_0(Y_0))).$ Let
$z_j\in C(\T)$ be the standard unitary generator of $C(\T)$ for the
$j$-th copy of $C(\T)$ in $B.$ Note that
$K_1(B)=K_1(B_0)$ is generated by $[z_1],[z_2],...,[z_r].$ We may assume
that $[L_1]\circ J(\Z^r)$ is generated by $[z_j],$ $j=1,2,...,r.$

Let $P\in M_R(C(Y))$ be a projection (for some large $R>0$) such that
$$
{\rm rank}P\ge 6K({\rm dim}Y+1).
$$
Let $Q$ be a trivial projection with rank $Ks.$
We may assume that $P\ge Q.$ Note also
the map
$$
\imath: U(M_{s}(C(Y)))/CU(M_{s}(C(Y)))
\to U(PM_R(C(Y))P)/CU(PM_R(C(Y))P)
$$
is an isomorphism.
% (since
%$U((P-Q)M_r(C(Y))(P-Q))/U_0((P-Q)M_r(C(Y))(P-Q))$ to
%$U(PM_R(C(Y))P)/U_0(PM_R(C(Y))P)$ ($\cong K_1(C(Y))$) is an isomorphism).
Let $\lambda_1=\imath^{-1}\circ \lambda\circ (L_1^{\ddag})^{-1}|_{J_c(\Z^r)}.$
Put $u_j\in U(M_s(C(Y)))$ such that
\beq\label{Ad0k1-10}
[u_j]=\alpha_1([z_j])\andeqn
\overline{u_j}=\lambda_1\circ J([z_j]),\,\,\,j=1,2,...,r.
\eneq
Define $\psi_j: C(\T)\to M_s(C(Y))$ by sending $z_j$ to $u_j,$ $j=1,...,r.$
Put $\psi: B\to M_{rs}(C(Y))$ by
$\psi={\rm diag}(\psi_1,\psi_2,...,\psi_r).$
Define
$$
L=({\rm diag}(\psi,\phi_1)\otimes {\rm id}_{M_K})\circ L_1.
$$
we check that $L$ meets all the requirements.

\end{proof}

\begin{lem}\label{homhom}
Let $X,$
 %be a connected finite CW complex with the covering dimension $d,$ let
 $\dt>0,$
 ${\cal G}\subset C(X)$ and ${\cal P}\subset \underline{K}(C_0(Z))$ be as
 in \ref{Ad0k1}.
 % (where $Z=X\setminus \{\xi\}$ for some
%point $\xi\in Z$) be a finite subset.
%There exists an integer $N_1(\dt, {\cal G}, {\cal P})\ge 1$ satisfying the following:
Let $Y$ be a connected finite CW complex
and let $\alpha\in KK(C_0(Z), C_0(Y_0)))$
such that
\beq\label{homh-1}
\af({F_3K_*(C_0(Z)}) &\subset&  F_3K_*(C_0(Y_0)),\\
\af|_{F_mK_*(C_0(Z))}&=& 0\tforal m\ge 4\andeqn\\
\af|_{K_*(C_0(Z), \Z/k\Z)}&=&0\tforal k>1,
\eneq
where $Y_0=Y\setminus \{y_0\}$ for some point $y_0\in Y.$
Suppose, in addition,  that $$\lambda(J(\Z^r)\cap SU_{d}(C(X))/CU(M_{d}(C(X)))\subset
SU_{d}(C)/CU(M_{d}(C)),$$
where $C=PM_R(C(Y))P.$
%such that
%$\alpha|_{{\rm ker}\rho_{C(X)}}\subset {\rm ker}\rho_{C(Y)}.$
%Then there exists an integer $N_1(\dt, {\cal G}, {\cal P})\ge 1,$ f
Then one may require $L$ to be a \hm\,as  in the conclusion of Lemma \ref{Ad0k1} . Moreover,
one can make ${\rm rank}P= 6r_0r(d+1),$ where $r_0$ is the number of connected components of $X^{(3)}.$

\end{lem}

\begin{proof}
Let $r_0$ be the number of connected components of $X^{(3)}.$
By \ref{Ad0++}, there is a unital \hm\, $h_0: C(X)\to M_{3r_0(d+1)}(C(Y))$ such that
$$
[h_0]=\af|_{\underline{K}(C_0(Z))}.
$$
We may write, as in \ref{Skeleton},
$J(\Z^r)=G_1\oplus G_2,$ where $G_1$ is free and $G_2=J(\Z^r)\cap SU_{d}(C(X)).$
Let $\lambda_1=\lambda-h_0^{\ddag}.$ Since $\Pi_c\circ \lambda\circ J_c=\af|_{\Z^r},$
${\rm Im}\,\lambda_1|_{J_c(\Z^r)}\subset U_0(M_{d}(C(Y))/CU(M_{d}(C(Y))).$
Thus, by the assumption,
\beq\label{homh-3}
\lambda_1(G_2)\subset (SU_{d}(C)\cap U_0(M_{d}(C)))/CU(M_{d}(C(Y))).
\eneq
It follows from \ref{SUCU} that $\lambda_1|_{G_2}=0.$ By \ref{Skeleton}, write
$K_1(C(X^{(1)}))=S\oplus K_1(C(X))/F_3K_1(C(X)).$ We may assume that $\Pi(G_1)=K_1(C(X))/F_3K_1(C(X)),$
where $\Pi: U(M_{d}(C(X)))/CU(M_{d}(C(X)))\to K_1(C(X))$ is the quotient map.
Let $G_1=\Z^k$ be generated by  the free generators $g_1,g_2,...,g_k$ ($k\le r$). Let
$z_j\in M_{d}(C)$ be unitary such that ${\bar z_j}=\lambda_1(g_j),$ $j=1,2,...,k.$
Let $v_1, v_2,...,v_k\in C(X)$ be unitaries representing $g_1,g_2,...,g_k,$ respectively.
Since $X^{(1)}$ is $1$-dimensional, it is easy and well known that there is a unital \hm\,
$h_1: C(X^{(1)})\to M_{3k(d+1)}(C)$ such that $h_1(v_j)=z_j,$ $j=1,2,...,k.$
Define $L=h_0\oplus h_1.$
\end{proof}

\begin{lem}\label{Ad1}
Let $X$ be a connected finite CW complex and let $Y=X\setminus \{\xi\},$ where $\xi\in X$ is a point. Let $K_0(C_0(Y))=\Z^k\oplus Tor(K_0(C_0(Y)))$
and $K_1(C(X))=\Z^r\oplus {\rm Tor}(K_1(C(X))).$ For any $\dt>0,$ any finite subset ${\cal G}\subset C(X)$ and any finite subset ${\cal P}\subset \underline{K}(C_0(Y)),$
there exist integers $N_1, N_2\ge 1$ satisfying the following:

Let $\Omega$ be a finite CW complex and let $\kappa\in Hom_{\Lambda}(\underline{K}(C_0(Y)), \underline{K}(C(\Omega)))$ and let
$$
K=\max\{|\rho_{C(\Omega)}(\kappa(g_i))|: g_i=(\overbrace{0,...,0}^{i-1},1,0,...,0)\in \Z^k\}.
$$
For  any projection
$P\in M_{\infty}(C(\Omega))$ with
${\rm rank}P\ge (N_2K+ N_1({\rm dim}Y+1)\},$ there is  a unital $\dt$-${\cal G}$-multiplicative \morp\, $L: C(X)\to PM_{N}(C(\Omega))P$ (with some integer
$N\ge {\rm rank}P$) such that
\beq\label{Ad1-1}
[L]|_{\cal P}=\kappa|_{\cal P}.
\eneq
Moreover, if $\lambda: J(\Z^r)\subset J(K_1(C(X)))\to U(PM_N(C(\Omega))P)/CU(PM_N(C(\Omega))P)$ is a \hm, then one may further require that
\beq\label{Ad1-1+1}
L^{\ddag}|_{J(\Z^r)}=\lambda.
\eneq

\end{lem}

\begin{proof}
%There is a connected finite CW complex $Z$ with ${\rm dim}Z=3$ such that
%$K_*(C(Z))=K_*(C(X)).$ Therefore there is an invertible
%element $\zeta\in KK(C(X), C(Z))$ such that $\zeta|_{K_*(C(X))}$ is the identity.
%There is an integer $N_0(\dt, {\cal G}, {\cal P})\ge 1$ such that
%for any projection $P_0\in M_r(C(Z))$ with rank at least $N_0(\dt, {\cal G}, {\cal P})$ there exists a unital \morp\,
%$\Phi: C(X)\to P_0M_r(C(Z))P_0$ such that
%$$
%[\Phi]|_{\cal P}=\zeta|_{\cal P}.
%$$

%Let $\af_1\in Hom(K_*(C_0(Y), K_*(C(\Omega))))$ such that
%$\af_1|_{K_0(C(Y))}=0$ and $\af_1|_{K_1(C_0(Y))}=\kappa|_{K_1(C_0(Y))}.$
%Let $N_1(\dt/4, {\cal G}, {\cal P})\ge 1$ be the integer given
%by \ref{Ad0k1}.
It suffices to show the case that $\Omega$ is connected.
 So in what follows we assume that $\Omega$ is a connected finite CW complex. Let $\omega\in \Omega$ be a point and let $\Omega_0=\Omega\setminus \{\omega\}.$
There is a splitting exact sequence
\beq\label{Ad-1-10}
0\to C_0(\Omega_0)\to C(\Omega)\to \C\to 0.
\eneq
%with splitting map $r.$
Thus
$$
KK(C_0(Y), C(\Omega))=KK(C_0(Y), \C)\oplus KK(C_0(Y), C_0(\Omega_0)).
$$
Write $\kappa=\kappa_0\oplus \kappa_1,$ where
$\kappa_0\in KK(C_0(Y), \C)$ and $\kappa_1\in KK(C_0(Y), C_0(\Omega_0)).$

Let $N_1(\dt/2, {\cal G}, {\cal P})\ge 1$ be given by
\ref{Ad0k1} for $X,$ $\dt/2,$ ${\cal G}$ and ${\cal P}.$ Let
$N_2=N_2(\dt/2, {\cal G}, {\cal P})$ be the integer given by
 Lemma 10.2 of \cite{LnApp} for $X,$ $\dt/2,$ ${\cal G}$ and
 ${\cal P}.$
 %that there is an integer $N_1(\dt/4, {\cal G}, {\cal P})\ge 1$  satisfying the following:
%Let $N_2(\dt/4, {\cal G}, {\cal P})$ be the integer given by
There exists, by Lemma 10.2 of \cite{LnApp},  a unital $\dt/2$-${\cal G}$-multiplicative \morp\, $L_1: C(X)\to M_{N_2K}$ such that
\beq\label{Ad1-2}
[L_1](g_i)=\rho_{C(\Omega)}(\kappa_0(g_i))\in \Z,\,\,\,[L_1]|_{\cal P}=\kappa_0|_{\cal P}.
\eneq
%and
%\beq\label{Ad1-3}
%{R_1\over{\max\{K,1\}}}\le N_2(\dt/2, {\cal G}, {\cal P}).
%\eneq

Let $\imath: M_{N_2K}\to M_{N_2K}(C(\Omega))$ be a unital embedding.
%Let
%\beq\label{Ad1-3+1}
%N(\dt, {\cal G}, {\cal P})=[N_2(\dt/2, {\cal G}, {\cal P})/({\rm dim}Y+1)]+1+3N_1(\dt/2, {\cal G},{\cal P}).
%\eneq
Put $N_1=3N_1(\dt/2, {\cal G}, {\cal P})+1$ and
put $R_1=N_1({\rm dim}Y+1).$
Let $P\in M_{\infty}(C(\Omega))$ be a projection
whose rank is at least $N_2K+R_1.$
Let $Q\le P$ be a trivial projection of rank $N_2K.$
Let $P_1=P-Q.$  $P_1$ is  a projection with at least rank $R_1.$
Note that the embedding (for some large $R\ge 1$)
$$
\gamma_1: U(P_1M_R(C(\Omega))P_1)/CU(P_1M_R(C(\Omega))P_1)
\to U(PM_R(C(\Omega))P)/CU(PM_R(C(\Omega))P)
$$ is an isomorphism.
Let
$$\gamma_2: U(M_{N_2K}(C(\Omega)))/CU(M_{N_2K}(C(\Omega)))
\to U(PM_R(C(\Omega))P)/CU(PM_R(C(\Omega))P).$$
be the \hm\, defined by $u\to P-Q+u$ for unitaries $u\in QM_R(C(\Omega))Q=M_{N_2K}(C(\Omega)).$
Define $\lambda_1=\gamma_1^{-1}\circ \lambda-\gamma_2\circ \imath\circ L_1^{\ddag}|_{J(\Z^r)}.$
It follows from \ref{Ad0k1} that there is a
$\dt/2$-${\cal G}$-multiplicative \morp\,
$L_2: C(X)\to P_1M_R(C(\Omega))P_1$ such that
\beq\label{Ad1-11}
[L_2]|_{\cal P}=\kappa_1|_{\cal P}\andeqn
L_2^{\ddag}|_{J(\Z^r)}=\lambda_1.
\eneq
We then define
$$
L=L_1+L_2.
$$

It is ready to verify that $L,$ $N_1$ and $N_2$
meet all the requirements.

\end{proof}

The following is a variation of a result of L. Li of \cite{Li}.

\begin{lem}\label{Ad2}
Let $X$ be  a path connected compact metric space, let $\ep>0$  and let ${\cal F}\subset C(X)_{s.a.}$ be a finite subset, there exist a unital \hm\,
$\phi_1: C(X)\to C([0,1])$ and
an integer $N\ge 1$ satisfying the following, if $P\in M_r(C(Y))$ is a projection with
\beq\label{Ad2-0}
{\rm rank} P(y)\ge N ({\rm dim}Y+1)\tforal y\in Y
\eneq
and $\lambda: \Aff(T(C(X)))\to \Aff(T(PM_r(C(Y))P))$ is a  unital positive linear map, where $Y$ is a compact metric space,
then there is  unital \hm\,
$\phi_2: C([0,1])\to PM_r(C(Y))P$ such that
\beq\label{Ad2-1}
|\tau\circ \phi_2\circ \phi_1(f)-\lambda(f)(\tau)|<\ep
\tforal f\in {\cal F}
\eneq
and for all $ \tau\in T(PM_r(C(Y))P).$

\end{lem}

\begin{proof}

It follows from Lemma 2.9 of \cite{Li} that there exist a continuous map
$\af: [0,1]\to X$ and a unital positive linear map $\gamma: C([0,1])\to C(X)$ such that
\beq\label{Ad2-2}
|\tau(\gamma(f\circ \af))-\tau(f)|<\ep/2\tforal f\in {\cal F}
\eneq
and for all $\tau\in T(C(X)).$ Define $\phi_1:C(X)\to C([0,1])$ by
$\phi_1(f)=f\circ \af$ for all $f\in C(X).$
Let $N$ be given by Corollary 2.6  of \cite{Li} for $X=[0,1],$ $\ep/2$
and for finite subset ${\cal G}=\{f\circ \af: f\in {\cal F}\}.$
We will identify $C([0,1])_{s.a.}$ with $\Aff(T(C([0,1])))$ and
$C(X)_{s.a.}$ with $\Aff(T(C(X))).$
By applying Corollary 2.6 of \cite{Li}, one obtains a unital \hm\, $\phi_2: C([0,1])\to PM_r(C(X))P$ such that
\beq\label{Ad2-3}
|\tau\circ \phi_2(g)-\lambda\circ \gamma(g)(\tau)|<\ep/2\tforal g\in {\cal G}
\eneq
and all $\tau\in T(PM_r(C(Y))P).$
Then
\beq\label{Ad2-4}
|\tau(\phi_2\circ \phi_1(f))-\lambda(f)(\tau)|\le
|\tau\circ \phi_2(f\circ \af)-\lambda\circ \gamma(f\circ \af)(\tau)|\\
+|\lambda\circ \gamma(f\circ \af)(\tau)-\lambda(f)(\tau)|\\
<\ep/2+\ep/2=\ep
\eneq
for all $f\in {\cal F}$ and for all $\tau\in T(PM_r(C(Y))P).$

\end{proof}

\begin{lem}\label{Almtorfree}
Let  $1>\ep>0,$ $Y$ be a finite CW complex, $r\ge 1$ be an integer and $C=PM_m(C(Y))P$
for some projection
$P\in M_m(C(Y))$ such that
${\rm rank}\, P(y)\ge (6\pi /\ep) ({\rm dim }Y+1)$ and $m\ge {\rm rank} P(y)$  for all $y\in Y.$
Suppose that $u\in U(M_r(C))$ with $[u]=0$ in $K_1(C)$
such that
$u^k\in CU(M_r(C))$ for some integer $k\ge 1,$ then
$$
{\rm dist}(\overline{u}, \overline{1_{M_r(C)}})<\ep/r\,\,\,\,{\rm in (} U(M_r(C))/CU(M_r(C)) {\rm )}.
$$
\end{lem}

\begin{proof}
Since ${\rm rank}P\ge  6\pi ({\rm dim} Y+1),$ $u\in U_0(M_r(C))$ (see \cite{Rff}).
Write $u=\prod_{j=1}^s \exp(\sqrt{-1} h_j),$ where $h_j\in M_r(C)_{s.a.},$ $j=1,2,...,s.$
Since $u^k\in CU(M_r(C)),$
$$
{\rm det}(u^k(y))=1\tforal y\in Y.
$$
It follows that
\beq\label{almtorfree-1}
({1\over{2\pi \sqrt{-1}}})\sum_{j=1}^s Tr(h_j)(y)=I/k\,\,\, {\rm for}\,\,\,y\in Y,
\eneq
where $I(y)$ is an integer for all $y\in Y.$
Note that $I/k$ is a continuous function on $Y.$
Let $Y=Y_1\sqcup Y_2\sqcup\cdots \sqcup Y_{l},$ where each $Y_i$ is connected, $i=1,2,...,l.$
It follows that $I(y)/k$ is a constant integer on each $Y_i,$ $i=1,2,...,l.$
For each $i,$ define $f_i$ to be a number in  $[-\pi, \pi]$ so that
$\exp(\sqrt{-1} f_i)=\exp(\sqrt{-1} l(y)\pi /k)$ for $y\in Y.$
Define
\beq\label{almtorfree-2}
v(y)=\exp(-\sqrt{-1}f_i/r{\rm rank }P(y))\tforal y\in Y.
\eneq
It is  a unitary in $U_0(M_r(C)).$
Then
\beq\label{almtorfree-3}
\|v-1_{M_r(C)}\|<  \ep/r.
\eneq
On the other hand,
\beq\label{almtorfree-4}
{\rm det}((vu(y)))=1\tforal y\in Y.
\eneq
Therefore, by \ref{SUCU},  $vu\in SU_r(C)\cap U_0(M_r(C))\subset CU(M_r(C)).$
It follows from this and (\ref{almtorfree-3}) that
\beq\label{almtorfree-5}
{\rm dist}({\overline{u}}, \overline{1_{M_r(C)}})\le {\rm dist}(\overline{v},{\overline{1_{M_r(C)}}})<\ep/r.
\eneq
\end{proof}

\begin{thm}\label{Ext1}
Let $X=X_1\sqcup X_2\sqcup\cdots \sqcup X_{s}$ be a finite CW complex with dimension $d\ge 0,$
where each $X_i$ is a connected finite CW complex and
 let
 ${\rm ker}\rho_{C(X)}=\Z^k\oplus Tor(K_0(C(X))).$
 Suppose that $\{g_1,g_2,...,g_k\}$ is the standard generators for
 $\Z^k.$

 For any $\ep>0,$ any finite subset ${\cal G}\subset C(X),$ any finite subset
 ${\cal P}\subset \underline{K}(C(X)),$ any finite subset ${\cal H}\subset C(X)_{s.a.},$ any
$\sigma_1, \sigma_2>0,$ any
finite subset of ${\cal U}\subset U(M_{d}(C(X))),$
%and any finite subset ${\cal H}\subset C(X)_+\setminus \{0\}$ any
%faithful tracial state $\tau_0\in  T(C(X)),$  there is an integer
and any integer $S\ge 1,$  there exists an integer $N$ satisfying the following:

For any  finite CW complex $Y,$
any $\kappa\in Hom_{\Lambda}(\underline{K}(C(X)), \underline{K}(C(Y)))$ with $\kappa([1_{C(X_i)}])=[P_i]$ for
some projection $P_i\in M_m(C(Y))$ (and for some integer $m\ge 1$), such that
$P=P_1+P_2+\cdots +P_s\in M_m(C(Y))$ is a projection, and
${\rm rank}P_i(y)\ge \max\{NK , \, N({\rm dim}Y+1)\}$ for all $y\in Y,$
where
$$
K= \max_{1\le i\le k}\{\sup\{|\rho_{C(Y)}(\kappa(g_i))(\tau)|: \tau\in T(C(Y))\}\},
$$
for any continuous \hm\,
\beq\label{Ext1-1}
&&\hspace{-0.6in}\gamma: U(M_{d}(C(X)))/CU(M_{d}(C(X)))\\
&&\to
U(M_{d}(PM_m(C(Y))P))/CU(M_{d}(PM_m(C(Y))P))
\eneq
and
for any continuous affine map $\lambda: T(PM_m(C(Y)P))\to T(C(X))$
such that $\kappa,$ $\gamma$ and $\lambda$  are compatible,
then there exists a unital $\ep$-${\cal G}$-multiplicative
\morp\, $\Phi: C(X)\to PM_m(C(Y))P$ such that
\beq\label{Exit1-1}
[\Phi]|_{\cal P}&=&\kappa|_{\cal P},\\
{\rm dist}(\Phi^{\ddag}(z), \gamma(z))&<&\sigma_1\tforal z\in \overline{{\cal U}}\andeqn\\
|\tau\circ \Phi(a)-\lambda(\tau)(a)|&<&\sigma_2\tforal a\in {\cal H}.
\eneq
If $u_1, u_2,..., u_m\in {\cal U}$ so that $[u_j]\not=0$ in $K_1(C(X))$ and $\{[u_1], [u_2],...,[u_m]\}$ generates a free group, we may require that
\beq\label{Exit1-1++}
\Phi^{\ddag}(\overline{u_j})=\gamma(\overline{u_j}),\,\,\, j=1,2,...,m.
\eneq

Moreover, one may require that
$$
P=Q_0\oplus {\rm diag}(\overbrace{Q_1,Q_1,...,Q_1}^{S_1})\oplus Q_2
$$
 for some integer $S_1\ge S,$  where $Q_0, Q_1$ and $Q_2$  are projections in $PM_m(C(Y))P,$
$Q_0$ is unitarily equivalent to $Q_1,$
and  $\Phi=\Phi_0\oplus \overbrace{\Phi_1\oplus\Phi_1\oplus...\oplus\Phi_1}^{S}\oplus
\Phi_2,$
 $\Phi_0: C(X)\to Q_0M_m(C(Y))Q_0,$
$\Phi_1=\psi_1\circ h,$ $\psi_1: C(J)\to Q_1M_m(C(Y))Q_1$ is a unital \hm, $h: C(X)\to C(J)$ is a unital \hm, $\Phi_2=\psi_2\circ h$  and
$\psi_2: C(J)\to Q_2M_m(C(Y))Q_2$ is a unital \hm, where
$J$ is a disjoint union of $s$ many unit intervals.

\end{thm}

\begin{proof}
Without loss of generality, we may assume that
${\cal G}$ and ${\cal H}$ are in the unit ball of $C(X)$ and may
assume that ${\cal H}\subset {\cal G}.$ Moreover, to simplify notation, without loss of generality,
we may also assume that $X$ is connected.
Write $K_1(C(X))=\Z^{k_1}\oplus Tor(K_1(C(X))).$
Furthermore, we may assume  that
${\cal U}={\cal U}_0\sqcup {\cal U}_1\sqcup {\cal U}_2,$
where $\overline{{\cal U}_0}\subset \Aff(T(C(X)))/{\overline{\rho_{C(X)}(K_0(C(X)))}},$ $\overline{{\cal U}_1}\subset J_c(\Z^{k_1})$ and $\overline{{\cal U}_2}\subset J_c(Tor(K_1(C(X)))).$
Let ${\cal H}_0\subset C(X)_{s.a.}$ such that
$$
\overline{{\cal U}_0}\subset \overline{{\cal H}_0}.
$$
Put ${\cal H}_1={\cal H}\cup {\cal H}_0.$ Let $0<\dt<\ep.$
We may also assume that
$(2\dt, {\cal G}, {\cal P})$ is a $KK$-triple.
%and $(2\dt, {\cal G},{\cal U})$ is a ${\cal U}$-triple.
%We also assume that $(2\dt, {\cal G})$ is  a $(\sigma_1/2, \overline{{\cal U}_0})$ pair.
Choose an integer $R\ge 1$ such that
\beq\label{Ext1-2}
{1\over{RS}}<\min\{\sigma_1/8\pi, \sigma_2/8\pi\}.
\eneq
Let $N_0$ (in place of $N$) be the integer given
by \ref{Ad2} for $\ep/4S$ (in place of $\ep$) and
${\cal H}_1$ (in place of ${\cal F}$).
Let $N_1$ and $N_2$  be  integers given by \ref{Ad1}
(for $\dt=\ep/2$).
Put $$N_1'=\max\{N_0, 2N_1,  2N_2,48\pi/\sigma_1\}.$$
Define
$$
N=2(N_1'+1)(2RS+1).
$$
Let $\kappa\in KK(C(X), C(Y))$ such that
$\kappa([1_{C(X)}])=P$ for some projection
$P\in M_m(C(Y))$  (for some large $m\ge 1$)
with ${\rm rank}P(y)\ge \max\{NK, N({\rm dim}Y+1)\},$
where
$$
K=\max_{1\le i\le k}\sup\{|\rho_{C(Y)}(\kappa(g_i))(\tau)|: \tau\in T(C(Y))\}.
$$

To simplify the proof, by considering each connected component separately,
we may assume that $Y$ is connected.
Let $Q_0\in PM_m(C(Y))P$ be a projection
with
$${\rm rank}Q_0=2\max\{N_1K, N_1({\rm dim}Y+1)\}\ge N_1(K+{\rm dim} Y+1).$$
This is possible because that ${\rm rank} P\ge \max\{NK, N({\rm dim}Y+1)\}.$
Note that $P-Q$ has rank
larger than
$
2N_1'RS({\rm dim}Y+1).
$
Let $P_1=P-Q_0.$
There is $R_1<N_1'({\rm dim}Y+1)$ such that
\beq\label{Ext1-3}
{\rm rank}P_1-({\rm dim} Y+1)=S_1N_1'({\rm dim}Y+1)+R_1
\eneq
for some integer $S_1\ge RS.$
Then there is a projection $P_1'\le P_1$ such that $P_1'$ is unitarily equivalent to
$$
{\rm diag}(\overbrace{Q_0,Q_0,...,Q_0}^{S_1}).
$$
Write
$$
P_1=\overbrace{(Q_1, Q_1,...,Q_1)}^{S_1}\oplus Q_2,
$$
where $Q_1$ is unitarily equivalent to $Q_0$ and $Q_2\lesssim Q_1,$
where  $Q_2$ has rank $R_1+{\rm dim}Y+1.$
Note that
\beq\label{Ext1-4}
{{\rm rank}Q_0+{\rm rank}(Q_2)\over{{\rm rank}P}}<\min\{\sigma_1/4, \sigma_2/4\}.
\eneq
For any integer $r\ge 1,$ let $\Gamma: T(M_r(C(Y))\to T(PM_m(C(Y))P)$ be the map defined by
\beq\label{Ext1-5}
\Gamma(\tau)(a)=\int_Y tr(a)d\mu_{\tau}
\eneq
for all $\tau\in T(C(Y))$ and all $a\in PM_m(C(Y))P,$ where
$\mu_\tau$ is the probability measure induced by $\tau$ and
where $tr$ is the normalized trace on $M_{{\rm rank}P}.$
Since the rank $Q_1$ is at least
$N_0({\rm dim}Y+1),$ it follows from \ref{Ad2} that there are unital \hm s
$h: C(X)\to C([0,1]),$ $\psi_1: C([0,1])\to Q_1M_m(C(Y))Q_1$ and $\psi_2: C([0,1])\to Q_2M_m(C(Y))Q_2$ such that
\beq\label{Ext1-6}
|\tau\circ \psi_1\circ h(a)-\lambda\circ \Gamma(\tau)(a)|<\min\{\sigma_1, \sigma_2\}/2S
\eneq
for all $a\in {\cal H}_1$ and all $\tau\in T(Q_1M_m(C(Y))Q_1),$ where
$\lambda$ is given by the theorem.
Define $\Phi_1=\psi_1\circ h$ and $\Phi_2=\psi_2\circ h.$
Define
$\Psi: C(X)\to P_1M_m(C(Y))P_1$ by
$\Psi=\overbrace{(\Phi_1, \Phi_1,...,\Phi_1)}^{S_1}\oplus \Phi_2.$
Let $\kappa_1=\kappa-[\Psi].$
Let
$$\Lambda: U(Q_0M_m(C(Y))Q_0)/CU(Q_0M_m(C(Y))Q_0)\to
  U(PM_m(C(Y))P)/CU(PM_m(C(Y))P)
  $$
  be the isomorphism induced  by $u\mapsto (P-Q_0)+u$ for all unitaries
  $u\in U(Q_0M_m(C(Y))Q_0).$ Define
  $\Lambda': U(P_1M_m(C(Y))P_1)/CU(P_1M_m(C(Y))P_1)\to
  U(PM_m(C(Y))P)/CU(PM_m(C(Y))P)$ similarly.
  Define $\gamma_1= \Lambda^{-1}\circ (\gamma-\Lambda'\circ \Psi^{\ddag}).$ Then, by applying \ref{Ad1}, we obtain a unital
  $\ep/2$-${\cal G}$-multiplicative \morp\, $\Psi_1: C(X)
  \to Q_0M_m(C(Y))Q_0$ such that
  \beq\label{Ext1-10}
  [\Psi_1]|_{\cal P}=\kappa_1|_{\cal P}\andeqn
  \Psi_1^{\ddag}|_{J(\Z^{k_1})}=\gamma_1|_{J(\Z^{k_1})}.
  \eneq
Put $C=PM_m(C(Y))P.$
  Define $\Phi=\Psi_1\oplus \Psi.$ It is clear that
  \beq\label{Ext-11}
  [\Phi]|_{\cal P}&=&[\Psi_1]|_{\cal P}+[\Psi]|_{\cal P}\\
  &=&\kappa_1|_{\cal P}+[\Psi]|_{\cal P}=\kappa|_{\cal P}.
   \eneq
It follows from (\ref{Ext1-6}) that
\beq\label{Ext-12}
|\tau\circ \Phi(a)-\lambda(\tau)(a)|<\min\{\sigma_1/2,\sigma_2\}
\eneq
for all $a\in {\cal H}_1$ and $\tau\in T(C).$
It follows from (\ref{Ext1-10}) that
\beq\label{Ext1-12}
\Phi^{\ddag}(z)=\gamma(z)\tforal z\in \overline{{\cal U}_1}.
\eneq
Let $J_C: K_1(C)\to U(C)/CU(C)$ be the \hm\, defined in \ref{DCU}
which splits the following short exact sequence:
\beq\label{Ext1-12+}
0\to \Aff(T(C))/\overline{\rho_{C(Y)}(K_0(C))}
\to U(C)/CU(C)
\to K_1(C)\to 0.
\eneq
(It should be noted that ${\rm rank}P\ge N_1'({\rm dim}Y+1)$.)
%Since $\Aff(T(PM_m(C(Y))P))/\overline{\rho_{C(Y)}(K_0(C(Y)))}$ is torsion free,
%$$\Phi^{\ddag}(z)-J_Y(\Phi^{\ddag}(z))=0$$
%for all $z\in \overline{{\cal U}_2}.$
Let $z\in \overline{{\cal U}_2}$ and let $v_0, v_1\in U(C)$ such that
$\overline{v_0}=\lambda(z)$ and $\overline{v_1}=\Phi^{\ddag}(z).$
Since $\gamma$ is compatible with
$\kappa,$
we have
\beq\label{Ext1-13-1}
v_0^*v_1\in U_0(C)\andeqn
(\lambda(z^*)\Phi^{\ddag}(z))^k=\overline{1_C} \,\,\, {\rm in}\,\,\, U(C)/ CU(C)\,\,\,\rforal z\in \overline{{\cal U}_1}.
\eneq
Since $N_1\ge 48\pi/\sigma_1,$ by \ref{Almtorfree},
\beq\label{Ext1-13}
{\rm dist}(\lambda(z), \Phi^{\ddag}(z))<\sigma_1/2\tforal z\in \overline{{\cal U}_2}.
\eneq
Now for $z\in \overline{{\cal U}_0},$ by (\ref{Ext1-12}),
\beq\label{Ext1-14}
{\rm dist}(\Phi^{\ddag}(z), \gamma(z))<\sigma_1\tforal z\in \overline{{\cal U}_0}.
\eneq
The lemma follows.
\end{proof}

\begin{cor}\label{Cexthom}
In the statement of \ref{Ext1}, let $\xi_i\in X_i$ be a point and $X_i'=X_i\setminus \{\xi_i\},$ $i=1,2,...,s.$
Let $C=PM_m(C(Y))P.$
Suppose that, in addition,  $Y$ is connected, $Y_0=Y\setminus\{y_0\}$ for some point $y_0\in Y,$
\beq\label{Cexth-1}
&&\kappa|_{\underline{K}(C(X_i'))}\in KK(C_0(X_i'), C_0(Y_0)),\\
&&\kappa(F_3K_*(C(X)))\subset F_3K_*(C),\,\,\,\kappa|_{F_mK_*(C(X))}=0\tforal m\ge 4,\\
&&\kappa|_{K_*(C(X), \Z/k\Z)}=0\tforal k\ge 1\andeqn\\
&&\lambda(SU_{d}(C(X))/CU(M_{d}(C(X))))\subset SU_{d}(C)/CU(M_{d}(C)).
\eneq
Then $\Phi$ and $\Phi_0$ can be required to be \hm s.
\end{cor}

\begin{proof}
The proof is exactly same but applying \ref{homhom} instead of \ref{Ad0k1}.

\end{proof}

\begin{cor}\label{Cext1}
Let $\Omega=\Omega_1\sqcup \Omega_2\sqcup\cdots \sqcup\Omega_s$ be a disjoint union of connected
finite CW complexes. Suppose that
 $X=\Omega\times \T$ has dimension $d+1$ and
$X_j=\Omega_j \times \T,$ $j=1,2,...,s.$ Then Theorem \ref{Ext1} holds for this $X$ with $d$ in the statement
replaced by $d+1$ and  with the following additional requirements.
 Suppose  that ${\cal P}={\cal P}_0\sqcup \boldsymbol{\bt}({\cal P}_1),$
where ${\cal P}_0, {\cal P}_1\subset \underline{K}(C(\Omega))$ are finite subsets and suppose ${\cal U}_b\subset J_c({\boldsymbol{\bt}}(F_2K_0(C(\Omega))))\cap \overline{ {\cal U}}$
is a finite subset such that, in addition,
\beq\label{Cext1-1}
\kappa|_{\boldsymbol{\bt}({\cal P}_1)}=0\andeqn
\gamma|_{{\cal U}_b}=0.
\eneq
Then, one can further require that
\beq\label{Cext1-2}
P_0=P_{00}\oplus P_{01}\andeqn
\Phi_0=\Phi_{00}\oplus \Phi_{01},
\eneq
where $P_{00}, P_{0,1}\in M_m(C(Y))$ are projections,
$\Phi_{00}(f\otimes g)=\sum_{j=1}^s f(\xi_j)q_{0,j}\cdot h_j(g)$ for all $f\in C(\Omega)$ and $g\in C(\T),$
$\xi_j\in \Omega_j$ is a point, $q_{0,j}\in M_m(C(Y))$ is a projection with
$\sum_{j=1}^sq_{0,j}=P_{00},$  $h_j: C(\T)\to q_{0,j}M_m(C(Y))q_{0,j}$ is a \hm\, ($j=1,2,...,s$) with
$h_j(z)\in U_0(q_{0,j}M_m(C(Y))q_{0,j})$ ($j=1,2,...,s$),
$\Phi_{01}(f\otimes g)=\sum_{j=1}^s L_j(f)\cdot g(1)q_{1,j}$ for all $f\in C(\Omega)$ and $g\in C(\T),$
$1\in \T$ is the point, $q_{1,j}\in M_m(C(Y))$ is a projection with $\sum_{j=1}^sq_{1,j}=P_{01},$ and
$L_j: C(\Omega)\to q_{1,j}M_m(C(Y))q_{1,j}$ is a unital \morp.
\end{cor}

\begin{proof}
Note that, as in the proof of \ref{Ext1}, one may assume that $\Omega$ is connected.
Write $K_1(C(X))=\Z^{k_1}\oplus Tor(K_1(C(X))).$
Note that
$$
\boldsymbol{\bt}(F_2K_0(C(\Omega)))\subset F_3K_1(C(X)).
$$
In that case, one may assume that
$${\cal U}={\cal U}_{00}\sqcup {\cal U}_{01}\sqcup{\cal U}_{01t}\sqcup {\cal U}_{11}\sqcup {\cal U}_b',$$
where
\beq
&&\hspace{-0.4in}\overline{{\cal U}_{00}}\subset  \Aff (T(C(X)))/\overline{\rho_{C(X)}(K_0(C(X))},\\
&&\hspace{-0.4in}{\cal U}_{01}\subset \{u\otimes 1_{C(\T)}: u\in U(M_{3(d+1)}(C(\Omega)))\andeqn [u]\in \Z^k \subset K_1(C(X))\},\\
&&\hspace{-0.4in}{\cal U}_{01t} \subset \{u\otimes 1_{C(\T)}: u\in U(M_{3(d+1)}(C(\Omega)))\andeqn [u]\in Tor(K_1(C(X)) \} \\
 &&
{\cal U}_{11} =\{1\otimes z\}\andeqn
\overline{{\cal U}_b'}={\cal U}_b.
%\andeqn\\
%{\cal U}_{1,0}&=&\{u\in U(M_{3(d+1)}(C(X)): [u]\not=0\andeqn [u]\in \boldsymbol{\bt}({\rm ker}\rho_{C(\Omega)}(K_0(C(\Omega))))\}.
\eneq
To simplify notation further, we may assume that ${\cal U}_b=\{x_1,x_2,...,x_{n(b)}\},$
where $x_i\in U(M_{6(d+1)}(C(X)))$ such that
there exists a unital \hm\,
$$H_i: M_2(C(S^3))\to U(M_{6(d+1)}(C(X)))$$ such that
$H_i(u_b)=x_i,$ $i=1,2,..., n(b).$
By the assumption, $\kappa([x_i])=0,$ $i=1,2,...,n(b).$
In the proof of \ref{Ext1}, we choose  $R$ so that
$$
1/RS<\min\{\sigma_1/16, \sigma_2/16\}.
$$
We also choose $Q_0$ so that ${\rm rank}Q_0=N_1K({\rm dim}Y+1)+1.$
Let $Q_{00}\le Q_0$ have rank $N_1K({\rm dim}Y)$ and
$Q_{01}$ have rank one.
We proceed with the proof and construct $\Phi_1,$ $\Phi_2$  and $\Psi.$
Since both $\Phi_1$ and $\Phi_2$ are \hm s which factor through $C(J),$  by \ref{CSU},
\beq\label{Cext-10}
\Psi^{\ddag}(x)=0\tforal x\in {\cal U}_b.
\eneq

Let $\gamma_1=\Lambda^{-1}(\gamma-\Lambda'\circ \Psi^{\ddag})$ (as $\gamma_1$ in the proof of \ref{Ext1}).
We then proceed to construct $L: C(\Omega)\to Q_{00}M_m(C(Y))Q_{00}$
the same way as $\Psi_1$ in the proof of \ref{Ext1} so that
\beq\label{Cext-11}
[L]|_{{\cal P}_0}=\kappa_1|_{{\cal P}_0}\andeqn
(L)^{\ddag}|_{\overline{{\cal U}_{01}}}=\gamma_1|_{\overline{{\cal U}_{01}}}.
\eneq
Define $\Phi_{01}: C(X)\to Q_{00}M_m(C(Y))Q_{00}$ by
$\Phi_{01}(f\otimes g)=L(f)\cdot g(1)Q_{00}$ for all $f\in C(\Omega)$ and  $g\in C(\T),$ where $1\in \T$ is a point.
Note that
\beq\label{Cext-11+}
\Phi_{01}^{\ddag}(x)=0\tforal x\in  {\cal U}_b.
\eneq
Now define $h: C(\T)\to Q_{01}M_m(C(Y))Q_{01}$ by
$h(g)=g(x)$ for all $g\in C(\T),$  where $x\in U_0(Q_{01}M_m(C(Y))Q_{01})$
is a unitary so that
\beq\label{Cext-12}
{\bar x}=\gamma_1(\overline{1_{C(\Omega)}\otimes z}).
\eneq
Define $\Phi_{00}: C(X)\to Q_{01}M_m(C(Y))Q_{01}$ by
$\Phi_{00}(f\otimes g)=f(\xi)Q_{01}\cdot h(g)$ for all $f\in C(\Omega)$ and for all $g\in C(\T),$ where $\xi\in \Omega$ is a point.
We also have that
\beq\label{Cext-12+}
\Phi_{00}^{\ddag}|_{{\cal U}_b}=0.
\eneq

We then define $\Phi_0=\Phi_{00}\oplus \Phi_{01}$ and $\Phi=\Psi\oplus \Phi_0\oplus \Psi\oplus \Phi_2.$
As in the proof of \ref{Ext1}, we have
\beq\label{Cext-13}
{\rm dist}(\lambda(x), \Phi^{\ddag}(x))<\sigma_1\tforal x\in \overline{{\cal U}_{01t}}.
\eneq
We have, for all $x\in {\cal U},$  as in the proof of \ref{Ext1},
\beq\label{Cext-14}
{\rm dist}(\lambda(x), \Phi^{\ddag}(x))<\sigma_1\tforal x\in {\cal U}.
\eneq
The rest of the requirements are now are readily checked.

\end{proof}

\begin{thm}\label{sec1MT}
Let $X$ be a compact metric space such that
$C(X)=\lim_{n\to\infty} (C(X_n),\psi_n),$ where  each $X_n$ is a finite CW complex and
$\psi_n: C(X_n)\to C(X_{n+1})$ is a unital \hm\,
and let $\phi_{n,\infty}: C(X_n)\to  C(X)$ be the
unital \hm\, induced by the inductive limit system.
%and
 %let $G_0\subset {\rm ker}\rho_{C(X)}$ be a finitely generated subgroup with
% $G_0=\Z^k\oplus Tor(G_0).$ Suppose that $\Z^k$ is
 %generated by $g_1, g_2,...,g_k.$
For any $\ep>0,$ any finite subset ${\cal G}\subset C(X),$ any finite subset
 ${\cal P}\subset \underline{K}(C(X)),$ any finite subset ${\cal H}\subset C(X)_{s.a.},$ any
$\sigma_1, \sigma_2>0,$ any
finite subset ${\cal U}\subset U(M_{r}(C(X)))$ (for some integer $r\ge 1$)
%, a finitely generated free subgroup $S\subset K_1(C(X))$ such that $J_c(S)\subset G(\overline{{\cal U}}),$ the subgroup generated by
%$\overline{{\cal U}}$
and any integer $L_1\ge 1,$  there exists  an integer $n\ge 1$  such that
 ${\cal P}\subset [\phi_{n, \infty}](\underline{K}(C(X_n)),$ a finite set of mutually orthogonal projections
$q_1, q_2,...,q_s\in C(X_n)$ with $1_{C(X_n)}=q_1+q_2+\cdots +q_s,$
a finite subset $g_1,g_2,...,g_k$ which generates $\Z^k$ such that
${\rm ker}\rho_{K_0(C(X))}\cap {\cal P}$ is contained in a finitely generated subgroup
$G_0=\Z^k\oplus Tor(G_0),$
an integer $N\ge 1$  and a finitely generated subgroup
$G_1\subset U(M_l(C(X)))/CU(M_l(C(X)))$ {\rm (}for some  $l${\rm )} with
 $\overline{{\cal U}}\subset G_1,$ $\Pi|_{G_1}$ is injective and
 $\Pi(G_1)\subset (\psi_{n, \infty})_{*1}(K_1(C(X_n)))$ satisfying the following:

For any  finite CW complex $Y,$
any $\kappa\in Hom_{\Lambda}(\underline{K}(C(X_n)), \underline{K}(C))$ with $\kappa([q_i])=[P_i]$ for
some projection $P_i\in M_m(C(Y))$ (and for some integer $m\ge 1$),
$P=P_1+P_2+\cdots + P_s \in M_m(C(Y))$ is a projection,
${\rm rank}P_i(y)\ge \max\{NK, N({\rm dim}Y+1)\}$ for all $y\in Y,$ where $C=PM_m(C(Y))P,$ and
where
$$
K=\max_{1\le i\le k}\{\sup\{|\rho_{C}(\kappa(g_i'))(\tau)|: \tau\in T(C)\}\},
$$
$g_i=(\psi_{n, \infty})_{*0}(g_i')$ for some $g_i'\in K_0(C(X_n)),$  $i=1,2,...,k,$
for any continuous \hm\,
\beq\label{sec1MT-1}
&&\hspace{-0.6in}\gamma: G_1+\Aff(T(C(X)))/\rho(K_0(C(X)))\\
&&\to
U(M_l(C))/CU(M_l(C))
\eneq
and
for any continuous affine map $\lambda: T(C)\to T(C(X))$
such that $\rho_{C(X)}([\psi_{n, \infty}(q_i)])(\lambda(\tau))=\rho_C([P_i])(\tau)$ for all $\tau\in T(C),$
$\kappa([\psi_{n, \infty}](\xi))=\Pi( \gamma(g))$ for all $g\in G_1$ and   $\xi\in K_1(C(X_n))$ with
$[\psi_{n,\infty}](\xi)=g,$ and that $\lambda$ and $\gamma$ are compatible,
then there exists a unital $\ep$-${\cal G}$-multiplicative
\morp\, $\Phi: C(X)\to PM_m(C(Y))P$ such that
\beq\label{sec1MT-2}
[\Phi\circ \psi_{n, \infty}]&=&\kappa,\\
{\rm dist}(\Phi^{\ddag}(x), \gamma(x))&<&\sigma_1\tforal x\in \overline{{\cal U}}\andeqn\\
%\Phi^{\ddag}|_{J_c(S)}=\lambda|_{J_c(S)}\andeqn\\
|\tau\circ \Phi(a)-\lambda(\tau)(a)|&<&\sigma_2\tforal a\in {\cal H}.
\eneq
Moreover,  one may require that
$$
P=Q_0\oplus {\rm diag}(\overbrace{Q_1,Q_1,...,Q_1}^{L_1})\oplus Q_2,
$$
where $Q_0, Q_1$ and $Q_2$  are projections in $PM_m(C(Y))P,$
$Q_0$ is unitarily equivalent to $Q_1,$
and  $\Phi=\Phi_0\oplus \overbrace{\Phi_1\oplus\Phi_1\oplus...\oplus\Phi_1}^{L_1}\oplus
\Phi_2,$
where $L_1\ge L$ is an integer, $\Phi_0: C(X)\to Q_0M_m(C(Y))Q_0,$
$\Phi_1=\psi_1\circ \phi_0,$ $\psi_1: C(J)\to Q_1M_m(C(Y))Q_1$ is a unital \hm, $\phi_0: C(X)\to C(J)$ is a unital $\ep$-${\cal G}$-multiplicative \morp, $\Phi_2=\psi_2\circ \phi_0$  and
$\psi_2: C(J)\to Q_2M_m(C(Y))Q_2$ is a unital \hm, where $J$ is a disjoint union of $s$ many unit intervals.

\end{thm}

\begin{proof}
Fix $\ep>0$ and a finite subset ${\cal G}\subset C(X).$ Without loss of generality, we may assume
that ${\cal H}\subset {\cal G}$ and ${\cal G}$ are  in the unit ball of $C(X).$
%To simplify notation, we may also assume
%that $g_1, g_2,...,g_k\in {\cal P}.$

We first prove the case that $X$ is a compact subset of a finite CW complex of dimension $d. $
Since we assume that, in this case, ${\rm dim}X_n=d,$
the embedding from
$$U(M_d(C(X)))/CU(M_d(C(X)))$$ into
$U(M_n(C(X))/CU(M_n(C(X)))$ is an isomorphism for all $n\ge d.$ We may assume
that ${\cal U}\subset U(M_{d}(C(X))),$ without loss of generality.

There is a sequence of decreasing finite CW complexes $\{X_n\}$ of dimension $d$ such that
$\cap_n X_n=X.$  Write $C(X)=\lim_{n\to\infty}(C(X_n), r_n),$ where
$r_n: C(X_n)\to C(X)$ is defined by $r_n(f)=f|_X$ for all $f\in C(X_n),$ $n=1,2,....$
Let $(\dt, {\cal G}, {\cal P})$ be a $KL$-triple. We may assume that
$$
\dt<
\min\{\ep, \sigma_1/18(d+1)^2, \sigma_2/18(d+1)^2\}.
$$

By  2.6 of \cite{Lnnewapp}, there exists an integer $n_0\ge 1$ such that
there is a unital  $\dt/4$-${\cal G}$-multiplicative \morp\, $\Psi: C(X)\to C(X_{n_0})$ such that
\beq\label{sec1MT-3}
\|r_{n_0}\circ \Psi(g)-g\|<\dt/4\rforal g\in {\cal G}.
\eneq

We may assume  that  ${\cal P}'\subset \underline{K}(C(X_{n_0}))$
is a finite subset such that
\beq\label{sec1MT-5}
 [r_{n_0}]({\cal P}')={\cal P}.
\eneq

Suppose that $p,q\in M_m(C(X))$ are two projections such that
$\tau\otimes Tr(p)=\tau\otimes Tr(q)$ for all $\tau\in T(C(X)),$ where
$Tr$ is the standard trace on $M_m.$ Then,
$Tr(p(x))=Tr(q(x))$ for all $x\in X.$ With sufficiently large $n_0,$ we may assume that there are projections $p', q'\in M_m(C(X_{n_0}))$ such that
$p'|_X=p$ and $q'|_X=q.$ Since $Tr(p)$ and $Tr(q)$ are integer valued continuous functions on $X_{n_0},$ there is $n_0'\ge n_0$ such that
\beq\label{sec1MT-5+1}
Tr(p'(x))=Tr(q'(x))\tforal x\in X_{n_0'}.
\eneq
Therefore, by choosing larger $n_0,$ we may assume that
\beq\label{sec1MT-5+2}
Tr(p'(x))=Tr(q'(x))\tforal x\in X_{n_0}.
\eneq
Let $G_0'$ be the subgroup generated by ${\rm ker}\rho_{C(X)}\cap {\cal P}.$
From the above (see (\ref{sec1MT-5+2})), we may assume, by choosing larger $n_0,$ that
\beq\label{sec1MT-5+3}
G_0'\subset (r_{n_0})_{*0}({\rm ker}\rho_{C(X_{n_0})}).
\eneq
We may write
\beq\label{sec1MT-5+4}
{\rm ker}\rho_{C(X_{n_0})}=\Z^{k_1}\oplus {\rm Tor}(K_0(C(X_{n_0}))).
\eneq
Let $s_1, s_2,...,s_{k_1}$ be free generators of $\Z^{k_1}.$ We may write
that
\beq\label{sec1MT-5+5}
(r_{n_0})_{*0}(s_i)=g_i,\,\,\,i=1,2,...,k\andeqn (r_{n_0})_{*0}(s_j)=0\tforal j\ge k+1.
\eneq
Thus we may write
\beq\label{sec1MT-5+6}
G_0=(r_{n_0})_{*0}({\rm ker}\rho_{C(X_{n_0})})=\Z^k\oplus {\rm Tor}(G_0).
\eneq
We may also assume that there exists
${\cal U}'\subset  U(M_{d}(C(X_{n_0})))$ such that $r_{n_0}({\cal U}')={\cal U}.$
%$S'\subset K_1(C(X_{n_0})$ which is a free subgroup such that $J_c(S')\subset \overline{{\cal U}'}.$
Let ${\cal G}'\subset C(X_{n_0})$ be a finite subset such that $r_{n_0}({\cal G}')={\cal G}$ and let ${\cal H}'\subset
C(X_{n_0})_{s.a.}$ be a finite subset such that $r_{n_0}({\cal H}')={\cal H}$ and they are all in the unit ball
of $C(X_{n_0}).$
We may assume that,  by (\ref{sec1MT-3}),
\beq\label{sec1MT-4}
[r_{n_0}\circ \Psi]|_{\cal P}=[{\rm id}_{C(X)}]|_{\cal P}\andeqn
(r_{n_0}\circ \Psi)^{\ddag}|_{\cal U}=({\rm id}_{C(X)})^{\ddag}|_{\cal U}.
\eneq
 Suppose that $X_{n_0}=X_{n_0,1}\sqcup X_{n_0,2}\sqcup\cdots \sqcup X_{n_0,s}$ is
a finite disjoint union of clopen subsets. Let $q_j=1_{C(X_{n_0, j})},$ $j=1,2,...,s.$
Let $L_1\ge 1.$   Let  $N\ge 1$ be given by \ref{Ext1} for $X_{n_0}$ (in place of $X$),
$\dt/4$ (in place of $\ep$), ${\cal G}'$ (in place of ${\cal G}$),  ${\cal P}'$ (in place of ${\cal P}$),
${\cal U}'$ (in place of ${\cal U}$), $\sigma_1/4$ (in place of $\sigma_1$) and $\sigma_2/4$ (in place of
$\sigma_2$) and $L_1.$ We choose $G_1= J_c((r_{n_0})_{*1}(K_1(C(X_{n_0})))),$ where $J_c$ may be chosen to be as in \ref{DUb}. Note that $l$ can be chosen to be ${\rm dim}X_{n_0}+1.$

Now suppose that $\kappa$ is given as in the theorem (for the above $g_i,$ $i=1,2,...,k,$ and $q_j,$ $j=1,2,...,s$) and $L_1\ge 1$
is given.

By applying \ref{Ext1}, there is a  unital $\dt/4$-${\cal G}'$-multiplicative \morp\,
$F: C(X_{n_0})\to PM_m (C(Y))P$ such that
\beq\label{sec1MT-6}
[F]|_{\cal P'}&=&\kappa |_{\cal P'},\\
{\rm dist}(F^{\ddag}(z), \gamma\circ r_{n_0}^{\ddag}(z))&<&\sigma_1/4\rforal z\in \overline{{\cal U}'}\andeqn\\
%L^{\ddag}\andeqn\\
|\tau\circ F(a)-\lambda (\tau)(r_{n_0}(a))|&<&\sigma_2/4 \tforal a\in {\cal H}',
\eneq
where $P=P_1+P_2+\cdots P_s,$ $[P_j]=\kappa([q_j])$ and
${\rm rank}P_i(y)\ge \max\{NK, N({\rm dim}Y+1)\}$ for all $y\in Y,$ $j=1,2,...,s.$
Moreover,  as in the proof of \ref{Ext1}, $P=Q_0\oplus {\rm diag}(\overbrace{Q_1,Q_1,...,Q}^{L_1}\oplus Q_2,$
as required, and
$$
F=\Psi_0\oplus \overbrace{\Psi_1,\Psi_1,...,\Psi_1}^{L_1}\oplus \Psi_2,
$$
where $\Psi_0: C(X_{n_0})\to Q_0M_m(C(Y))Q_0$ is a unital  $\dt/4$-${\cal G}'$-multiplicative \morp\, $\Psi_1=\psi_1'\circ h$ and $\Psi_2=\psi_2'\circ h,$ where $h: C(X_{n_0})\to C(J)$ is a unital \hm\, and
$J$ is a disjoint union of finitely many intervals,
$\psi_1': C(J)\to Q_1M_m(C(Y))Q_1$ and $\psi_2': C(J)\to Q_2M_m(C(Y))Q_2$
are unital \hm s.

Define $\Phi=F\circ \Psi.$ It follows that
\beq\label{sec1MT-7}
[\Phi]|_{\cal P}&=& \kappa|_{\cal P},\\
{\rm dist}(\Phi^{\ddag}(z), \gamma(z))&<&\sigma_1\rforal z\in \overline{{\cal U}}\andeqn\\
|\tau\circ F(a)-\lambda(\tau)(a)| &<& \sigma_2\tforal a\in {\cal H}.
\eneq

For the general case,  we may write that
$C(X)=\overline{\cup_{n=1}^{\infty} C(X_n)},$ where each $X_n$ is a compact subset of a finite CW complex.
For any $\eta>0$ and any finite subset ${\cal F}\subset C(X),$ we may assume that ${\cal F}\subset C(X_{n_1})$ for some $n_1\ge 1$ with an error within $\eta/2.$
Then, by 2.3.13 of \cite{Lnbk}, there is an integer $n_2\ge 1$ and a unital $\eta/4$-multiplicative \morp\,
$\Psi': C(X)\to C(X_{n_2})$ such that
\beq\label{sec1MT-8}
\|r_{n_2}\circ \Psi'(f)-f\|<\eta/4\tforal f\in {\cal F}.
\eneq
With sufficiently small $\eta$ and large ${\cal F},$ by considering maps from $C(X_{n_2}),$ one sees that the general case follows from
the case that $X$ is a compact subset of a finite CW complex.

\end{proof}

\begin{cor}\label{sec1MC}
Let $\Omega$ be a compact metric space
and let $X=\Omega\times \T.$
 Then \ref{sec1MT} holds
for this $X.$ Suppose also that ${\cal P}={\cal P}_0\sqcup \boldsymbol{\bt}({\cal P}_1),$ where ${\cal P}_0, \, {\cal P}_1\subset
[\psi_{n, \infty}](\underline{K}(C(X_n)))$ are finite subsets and suppose that ${\cal U}_b\subset J_c({\boldsymbol{\bt}}(F_2K_0(C(\Omega))))\cap \overline{{\cal U}}$ is a finite subset such that, in addition,
\beq\label{secMC-1}
\kappa|_{\boldsymbol{\bt}({\cal P}_1)}=0\andeqn
\gamma|_{{\cal U}_b}=0.
\eneq
Then, one may further requite that
\beq\label{nCext1-2}
P_0=P_{00}\oplus P_{01}\andeqn
\Phi_0=\Phi_{00}\oplus \Phi_{01},
\eneq
where $P_{00}, P_{0,1}\in M_m(C(Y))$ are projections,
$\Phi_{00}(f\otimes g)=\sum_{j=1}^s f(\xi_j)q_{0,j}\cdot h_j(g)$ for all $f\in C(\Omega)$ and $g\in C(\T),$
$\xi_j\in \Omega_j$ is a point, $q_{0,j}\in M_m(C(Y))$ is a projection with
$\sum_{j=1}^sq_{0,j}=P_{00},$  $h_j: C(\T)\to q_{0,j}M_m(C(Y))q_{0,j}$ is a \hm\,  with
$h_j(z)\in U_0(q_{0,j}M_m(C(Y))q_{0,j})$ ($j=1,2,...,s$),
$\Phi_{01}(f\otimes g)=\sum_{j=1}^s L_j(f)\cdot g(1)q_{1,j}$ for all $f\in C(\Omega)$ and $g\in C(\T),$
$1\in \T$ is a point, $q_{1,j}\in M_m(C(Y))$ is a projection with $\sum_{j=1}^sq_{1,j}=P_{01},$ and
$L_j: C(\Omega)\to q_{1,j}M_m(C(Y))q_{1,j}$ is a unital \morp.

\end{cor}

\section{The uniqueness statement and the existence theorem for Bott map}

The following is taken from  2.11 of \cite{EGL}.

\begin{thm}\label{INV}
Let $\ep>0.$ Let $\Delta: (0,1)\to (0,1)$ be an increasing map and let $d\ge 0$ be an integer.
There exists $\eta>0,$ $\gamma_1, \gamma_2>0$ and a finite subset ${\cal H}\subset
C(\T)_{s.a.}$ and an integer $N\ge 1$ satisfying the following:

  Let $\phi, \psi: C(\T)\to C=PM_r(C(Y))P$ be two unital \hm s for some connected finite CW complex with ${\rm dim} Y\le d$ and ${\rm rank}P\ge N$ such that
\beq\label{INV-1}
|\tau\circ \phi(g)-\tau\circ \psi(g)|&<&\gamma_1\rforal g\in {\cal H}\andeqn \rforal \tau\in T(C),\\
{\rm dist}(\phi^{\ddag}({\bar z}), \psi^{\ddag}({\bar z}))&<&\gamma_2,\\
\mu_{\tau\circ \phi}(I_r)&\ge& \Delta(r)\tforal \tau\in T(C)
\eneq
and for all open arcs $I_r$ with length $r\ge \eta.$  Then there exists a unitary $u\in C$ such that
\beq\label{INV-2}
\|u^*\phi(z)u-\psi(z)\|<\ep.
\eneq

\end{thm}
(Here $z\in C(\T)$ is the identity map on the unit circle.)

\begin{proof}
The proof of 2.11 of \cite{EGL} does not need the assumption that ${\rm dim}Y\le 3.$ The main technical lemma used in the proof was 4.47' of \cite{G1} which is a restatement of 4.47 which stated without assuming ${\rm dim}Y\le 3.$
Perhaps, a quick way to see this is to refer to the proof of Theorem 3.2 of \cite{Lninv} which is a modification of that of 2.11 of \cite{EGL}.
Again, note that  Lemma 3.1 of \cite{Lninv} is another restatement of
4.47' of \cite{G1} which, as mentioned above, is a restatement of 4.47 of
\cite{G1}. So Lemma 3.1  of \cite{Lninv} holds without assuming $d={\rm dim}Y\le 3.$ However the integer $L$ in Lemma 3.1 depends on $d.$
There are two occasions that ``since ${\rm dim}Y\le 3"$ appears in the proof of 3.2 of \cite{Lninv}. In both cases, we can simply replace $3$ by $d,$ in
the next line (i.e., replace $3k_0m_1$ by $dk_0m_1$ and replace
$3k_0l_1$ by $dk_0l_1$).  Note also, since $X=\T,$ $K_i(C(X))$ has no torsion. Therefore, we do not need $D_j$ ($j\ge 2$) in the modification.
The same simple modification of proof of 2.11 of \cite{EGL} also leads
to this lemma.

\end{proof}

\begin{rem}\label{ReuniT}

Note that the above lemma also holds if $\phi$ and $\psi$ are assumed to be unital $\dt$-${\cal G}$-multiplicative \morp s, where $\dt>0$ and finite subset ${\cal G}\subset C(\T)$ depend on $\ep,$ since
$C(\T)$ is weakly semi-projective.
\end{rem}

\begin{cor}\label{Uniunt}
Let $X$ be a compact metric space, let ${\cal F}\subset C(X)$ be a finite subset, let $\ep>0$ be a positive number and let $d\ge 1.$
Let $\Delta: (0,1)\to (0,1)$ be a nondecreasing map.
Let ${\cal U}\subset M_{m(X)}(C(X))$ be a finite subset of unitaries
which represent non-zero elements in $K_1(C(X)).$

There exist $\eta>0,$  $\gamma_1>0,$ $\gamma_2>0,$
$\dt>0,$ a finite subset ${\cal G}\subset C(X)$  and a finite subset ${\cal P}\subset \underline{K}(C(X)),$
 a finite subset ${\cal H}\subset C(X)_{s.a.},$  a finite subset ${\cal V}\subset K_1(C(X))\cap {\cal P},$ an integer
 $N\ge 1$
  satisfying the following:
For any finite CW complex  $Y$ with ${\rm dim}Y\le d,$   any projection $P\in M_m(C(Y))$
with ${\rm rank}P(y)\ge N$ for all $y\in Y$ and two unital $\dt$-${\cal G}$-multiplicative \morp s
$\phi, \psi: C(X)\to C=PM_m(C(Y))P$  such that
\beq\label{NT-1-1}
[\phi]|_{\cal P}=[\psi]|_{\cal P},
\eneq
\beq\label{NT-1}
\mu_{\tau\circ \phi}(O_r)\ge \Delta(r),\,\,\,
\mu_{\tau\circ \psi}(O_r)\ge \Delta(r),
\eneq
for all $\tau\in T(M_n(C(Y)))$ and for all $r\ge \eta,$
\beq\label{NT-2}
|\tau\circ \phi(g)-\tau\circ \psi(g)|<\gamma_1
\tforal g\in {\cal H}\tand\\
{\rm dist}(\overline{\langle \phi(u)\rangle},\overline{\langle \psi(u)\rangle})<\gamma_2\tforal u\in J_{c(G({\cal V}))}({\cal V}),
\eneq
there exists, for each $v\in {\cal U},$  a unitary $w\in M_{m(X)}(C)$ such that
\beq\label{NT-3}
\|w(\phi\otimes {\rm id}_{m(X)}(v))w^*-(\psi\otimes {\rm id}_{m(X)})(v)\|<\ep.
\eneq
\end{cor}

\begin{proof}

Fix $v\in {\cal U}.$
Let $r\in (0,1).$ Choose a $r/2$-dense set $\{s_1, s_2,...,s_n\}$ in $\T.$
Let $f_j$ be in $C(\T)_+$ such that $0\le f_j\le 1,$ $f_j(s)=1$ if $|s_j-s|\le r/5$ and $f_j(s)=0$ if $|s_j-s|\ge r/2,$
$j=1,2,...,n.$

Let $T_r=\{\tau\in T(C(X)): \mu_{\tau}(O_r)\ge \Delta(r)\}.$
It is easy to see that $T_r$ is a compact subset of $T(C(X))$ (in the weak*-topology).
Let
$$
I_r=\{f\in C(X): \tau(f^*f)=0\tforal \tau\in T_r\}.
$$
Then $I_r$ is a closed two-sided ideal of $C(X).$
Put
$$
J_r=\{f\in M_{m(X)}(C(X)): (\tau\otimes Tr_{m(X)})(f^*f)=0\tforal \tau\in T_r\}.
$$
Note that $J_r\subset J_{r'}$ if $0<r'<r$ and it is easy to check
that $\cap_{1>r>0} J_r=\{0\}.$
There is $1>F_1(r)>0$ such that
\beq\label{Uniuni-10}
\pi(f_j(v))\not=0,\,\,\,j=1,2,...,n,
\eneq
where $\pi: M_{m(X)}(C(X))\to M_{m(X)}(C(X))/J_{F_1(r)}$ is the quotient map.
Therefore $\tau(f_j(v))>0$ for all $\tau\in T_{F_1(r)}.$
Define
\beq\label{Uniuni-11}
\Delta_v(r)=\inf\{\tau(f_j(v)): \tau\in T_{F_1(r)},j=1,2,...,n\}.
\eneq
Since $T_{F_1(r)}$ is compact, $\Delta_v(r)>0.$
We note that
\beq\label{Uniuni-12}
\mu_{\tau}(O_r)\ge \Delta_v(r)
\eneq
for all open arcs $O_r.$
Define $\Delta_1: (0, 1)\to (0,1)$ by
$\Delta_1(r)=\inf_{\{v\in {\cal U}\}}\Delta_v(r).$ Then $\Delta_1$ is an increasing map.

Now let $\eta_1>0$ (in place of $\eta$), $\gamma_1'>0$ (in place of
$\gamma_1$), $\gamma_2>0,$ ${\cal H}_1$ (in place of ${\cal H}$) be a finite subset  and $N\ge 1$ be an integer required by \ref{INV} for
$\Delta_1$ and $\ep$ given.
Also let $\dt_1>0$ (in place of $\dt$) and ${\cal G}_1$ (in place of ${\cal G}$) be finite subset given by \ref{ReuniT} for $\ep.$

Choose $1>\eta>0$ so that $\eta<F_1(\eta_1).$ Choose
$\gamma_1=\gamma_1'/m(X)$ and
$$
{\cal H}=\{h_{i,i}\in C(X)_{s.a}: (h_{i,j})=h(v)\,\,\,\text{ for \,\,\,some}\,\,\, h\in {\cal H}\andeqn v\in {\cal U}\}.
$$
Choose $\dt=\dt_1/m(X)^2$  and
$$
{\cal G}=\{g_{i,j}: (g_{i,j})=g(v) \,\,\,\text{for\,\,\,some}\,\,\,
g\in {\cal G}\andeqn v\in {\cal U}\}.
$$
Choose a finite subset ${\cal P}\subset \underline{K}(C(X))$ so that
it contains $[{\rm id}_{C(X)}]\in K_0(C(X))$ and
$\{[v]\in K_1(C(X)): v\in {\cal U}\}.$ Choose ${\cal V}={\cal U}.$

 Now, if $\phi, \psi: C(X)\to C=PM_m(C(Y))P$ are as described in the lemma, where
${\rm rank}P\ge N,$ and ${\rm dim}Y=d,$ which satisfy the assumption
for the above chosen $\eta,$ $\gamma_1,\gamma_2,$ $\dt,$ ${\cal G},$
${\cal H},$ ${\cal P},$ ${\cal V}$ and $N,$
define $\lambda_v: C(\T)\to M_{m(X)}(C(X))$  by
$\lambda_v(f)=f(v)$ for all $f\in C(\T).$ Define $\phi_{v}=(\phi\otimes {\rm id}_{M_{m(X)}})\circ \lambda_v$ and $\psi_v: (\psi\otimes {\rm id}_{M_{m(X)}})\circ \lambda_v.$ Then we apply \ref{ReuniT} to $\phi_v$ and
$\psi_v.$ The lemma follows.

\end{proof}

%Let ${\cal Y}$ be a class of finite CW complexes.

\begin{lem}\label{Sd1uni} (Uniqueness statement for ${\rm dim}Y\le d$)
Let $X$ be a compact metric space, let ${\cal F}\subset C(X)$ be a finite subset and let $\ep>0$ be a positive number. Let $d\ge 0$ be an integer and let
$\Delta: (0,1)\to (0,1)$ be a nondecreasing map.
%There exists $\eta_1>0$ satisfying the following:
%for any $\sigma_1>0,$ there exists $\eta_2>0$ satisfying the following:
%for any $\sigma_2>0,$ there exists $\eta_3>0$ satisfying the following:
%for any $\sigma_3>0,$ there exists $\eta_4>0$ satisfying the following:
%For any  $\sigma_4>0,$
There exist $\eta>0,$  $\gamma_1>0,$ $\gamma_2>0,$
$\dt>0,$ a finite subset ${\cal G}\subset C(X)$  and a finite subset ${\cal P}\subset \underline{K}(C(X)),$
 a finite subset ${\cal H}\subset C(X)_{s.a.},$  a finite subset ${\cal V}\subset K_1(C(X))\cap {\cal P},$ an integer
 $N\ge 1$
 and an integer $K\ge 1$ satisfying the following:
For any finite CW complex  $Y$ with ${\rm dim}Y\le d,$  any projection $P\in M_m(C(Y))$
with ${\rm rank}P(y)\ge N$ for all $y\in Y$ and two unital $\dt$-${\cal G}$-multiplicative \morp s
$\phi, \psi: C(X)\to C=PM_m(C(Y))P$  such that
\beq\label{nNT-1-1}
[\phi]|_{\cal P}=[\psi]|_{\cal P},
%=[h]|_{\cal P}
%\andeqn [h]|_{F_mK_*(C(X))}=0\tforal m\ge 4
\eneq
%for some unital \hm\, $h: C(X)\to PM_n(C(Y))P,$
\beq\label{nNT-1}
\mu_{\tau\circ \phi}(O_r)\ge \Delta(r),\,\,\,
\mu_{\tau\circ \psi}(O_r)\ge \Delta(r),
\eneq
for all $\tau\in T(M_m(C(Y)))$ and for all $r\ge \eta,$
\beq\label{nNT-2}
|\tau\circ \phi(g)-\tau\circ \psi(g)|<\gamma_1
\tforal g\in {\cal H}\andeqn\\
{\rm dist}(\overline{\langle \phi(u)\rangle},\overline{\langle \psi(u)\rangle})
<\gamma_2\tforal u\in J_{c(G({\cal V})}({\cal V})\\
%{\rm dist}(\overline{\langle \phi(u)\rangle}(x), SU_s(C)/CU(M_s(C)))&<&\gamma_2
\eneq
%for all $ x\in J_{c(G({\cal V})}({\cal V})\cap SU_{s}(C(X))/CU(M_{s}(C(X))),$
%where $s\ge \max\{R(G({\cal V})),d\},$
%where we assume that elements in $J_c({\cal V})$ are represented by unitaries in $M_{s'}(C(X)),$
then there exists a unitary $W\in M_K(C)$ such that
\beq\label{NNT-3}
\|W\phi^{(K)}(f)W^*-\psi^{(K)}(f)\|<\ep\tforal f\in {\cal F},
\eneq
where
$$
\phi^{(K)}(f)={\rm diag}(\overbrace{\phi(f),\phi(f),...,\phi(f)}^K)\andeqn
\psi^{(K)}(f)={\rm diag}(\overbrace{\psi(f), \psi(f),...,\psi(f)}^K)
$$
for all $f\in C(X).$
%provided the following also hold:
%$\psi=h$ is a unital \hm, or
%\beq\label{MT-3+1}
%[h]|_{F_mK_*(C(X))}=0\tforal m\ge 4\andeqn\\\label{MT-3+2}
%{\rm dist}(\overline{\langle \phi(u)\rangle}(x), SU_s(C)/CU(M_s(C)))<\gamma_2
%\eneq
%for all $ x\in J_{c(G({\cal V})}({\cal V})\cap SU_{s}(C(X))/CU(M_{s}(C(X))),$
%where $s\ge \max\{R(G({\cal V})),d\}.$

%Moreover,  $\eta>0,$  $\gamma_1>0,$ $\gamma_2>0,$
%$\dt>0,$  ${\cal G}$  ${\cal P},$
%${\cal H},$   ${\cal V},$
 %$N$
 %and  $K$  do not depend on ${\cal Y}.$

\end{lem}

Lemma \ref{Sd1uni} will be proved in Section 12.

We actually will use a revised version of the above statement:

\begin{lem}\label{Sd1uniR}
The same statement \ref{Sd1uni} holds if
(\ref{nNT-1}) is replaced by
\beq\label{NT-1+1}
\mu_{\tau\circ \phi}(O_r)\ge \Delta(r)\tforal \tau\in T(PM_m(C(Y))P).
\eneq
\end{lem}

By the virtue of 3.4 of \cite{Lnnewapp}, \ref{Sd1uniR} and \ref{Sd1uni} are equivalent.

\begin{lem} \label{Sd1ext}(Existence statement for ${\rm dim} Y=d$)
Let $d\ge 0$ be an integer. Let $X$ be a compact metric space with $C(X)=\lim_{n\to\infty} (C(X_n), \psi_n),$ where
$X_n$ is a finite CW complex, let ${\cal F}\subset C(X)$ be a finite subset, let $\ep>0$ and  $\dt_0>0$ be positive numbers.
Let $\Delta: (0,1)\to (0,1)$ be a nondecreasing map.
For any finite subset ${\cal P}\subset \underline{K}(C(X))$ and any finite subset ${\cal Q}\subset K_0(C(X)),$
there exists $\eta>0,$ $\dt>0$ and a finite subset ${\cal G}\subset C(X),$
%satisfying the following:
%for any $\sigma_1>0,$ there exists $\eta_2>0$ satisfying the following:
%for any $\sigma_2>0,$ there exists $\eta_3>0$ satisfying the following:
%for any $\sigma_3>0,$ there exists $\eta_4>0$ satisfying the following:
%For any  $\sigma_4>0,$
%for any finite subset ${\cal P}\subset \underline{K}(C(X)),$
%there exists
%$\sigma>0,$
an integer $n_0\ge 1,$ an integer $N\ge 1$ and an integer $K\ge 1$
%and
% a finite subset ${\cal H}\subset C(X)$
satisfying the following:
 %and a finite subset ${\cal U}\subset U_c(K_1(C(X)))$
 %for which $[{\cal U}]\subset {\cal P}$ satisfying the following:
For any finite CW complex $Y$ with ${\rm dim}Y\le d,$
any
unital $\dt$-${\cal G}$-multiplicative \morp\,  $\phi: C(X)\to C=PM_R(C(Y))P$ (for some integer $R\ge 1$ and
a projection $P\in M_R(C(X))$) such that
${\rm rank}P(y)\ge \max\{N,NK_1\}$ for all $y\in Y$ (for some integer
$K_1\ge 0$ so that (\ref{Sd1ext-2}) holds),
\beq\label{Sd1ext-1}
%(h_*)|_{F_mK_*(C(X))}=0\tforal m\ge 4\andeqn
\mu_{\tau\circ \phi}(O_r)\ge \Delta(r)
\eneq
for all balls $O_r$ of $X$ with radius $r\ge \eta$ and for all $\tau\in T(C),$
and for any
$\kappa\in KK(C(X_{n_0})\otimes C(\T), C(Y))$ with
\beq\label{Sd1ext-2}
\max\{ |\rho_{C}(\kappa(\boldsymbol{\bt}(g_i')))(t)|: 1\le i\le k \andeqn t\in T(C)\}\le K_1,
\eneq
where
$g_1,g_2,...,g_k$ generates a subgroup which contains the subgroup generated by $K_1(C(X))\cap {\cal P},$
$[\psi_{n_0, \infty}](g_i')=g_i$ for some $g_i'\in K_1(C(X_{n_0})),$
and,
%if  elements $\boldsymbol{\bt}(Q)$  can be represented by unitaries in $M_n(C(X)\otimes C(\T)),$
for any \hm\,
$$
\Gamma: G({\cal Q})\to U(M_n(C))/CU(M_n(C)),
%\Aff (T(PM_m(C(Y))P)/\overline{\rho_{C(Y)}(K_0(C(Y)))}
$$
where $n=\max\{R({\boldsymbol{\bt}}(G({\cal Q}))),d+1\}$ (see {\rm \ref{DUb}}) and ${\cal Q}\subset [\psi_{n_0, \infty}](K_0(C(X_{n_0})))$  with
$\kappa\circ \boldsymbol{\bt}|_{{\cal Q}}=\Pi_C\circ \Gamma|_{{\cal Q}},$
 there exists a
 unitary $u\in M_K(PM_R(C(Y))P)$ such that
 \beq\label{Sd1ext-3}
 \|[\phi^{(K)}(f), u]\|&<&\ep\tforal f\in {\cal F},\\
 {\rm Bott}(\phi^{(K)}\circ \psi_{n_0, \infty}, u)&=& K\kappa\circ \boldsymbol{\bt}\tand\\
{\rm dist}({\rm Bu}(\phi^{(K)},u)(x),K\Gamma(x)) &<&\dt_0 \tforal
x\in {\cal Q}.
 \eneq

\end{lem}

%First, we note that \ref{Sd1uni} has been proved the family ${\cal Y}$ which contains $[0,1]$ and finitely many points.

\vspace{0.2in}

{\bf The proof  of \ref{Sd1ext} holds for ${\rm dim}Y\le d$  under assumption
that \ref{Sd1uni} holds for  ${\rm dim }Y\le d.$}

 \begin{proof}
 We will apply \ref{sec1MT} and the assumption that \ref{Sd1uni} holds for all finite CW complexes $Y$
 with ${\rm dim}Y\le d.$
 To simplify notation, we may assume, without loss of generality, that $Y$ is connected.

We assume that $\ep<\dt_0/2.$
 We may assume that $1_{C(X)}\in {\cal F}.$
 Let $\Delta_1(r)=\Delta(r/3)/3$ for all $r\in (0,1).$
 Let ${\cal P}\subset \underline{K}(C(X))$ be a finite subset.
 Let $(\ep', {\cal G}', \boldsymbol{\bt}({\cal P}))$ be
 a $KL$-triple for $C(X)\otimes C(\T).$ To simplify notation, by choosing
 a smaller $\ep$ and larger ${\cal F},$ we may assume
 that $(\ep, {\cal F}',\boldsymbol{\bt}({\cal P}))$ is
 a $KL$-triple for $C(X)\otimes C(\T),$ where
 ${\cal F}'=\{f\otimes g: f\in {\cal F}\andeqn g\in \{1_{C(\T)}, z, z^*\}\}.$ To simplify notation, without loss of generality, we may also assume that
${\cal Q}=K_0(C(X))\cap {\cal P}.$

There is $n_0'\ge 1$ such that ${\cal P}\subset [\psi_{n_0', \infty}](\underline{K}(C(X_{n_0'}))).$
Let $\eta_1>0$ (in place of $\eta$) be as in
\ref{Sd1uni} for $\ep/16,$ ${\cal F}$ and $\Delta_1.$
%Fix $1>\sigma_1>0.$ Let $\eta_2>0$ for $\sigma_1'=\sigma_1/4$ (in place of $\sigma_1$)
%Let $\sigma_1'>\sigma_2>0.$ Let $\eta_3$ for $\sigma_2'=\sigma_2/4$ (in place $\sigma_2$). Let $\sigma_2'>\sigma_3>0.$ Let
%$\eta_4>0$ for $\sigma_3/4.$ Let $0<\sigma_4'<\sigma_3'.$
%Let $\sigma_4=\sigma_4'/4.$
 It follows from 3.4 of \cite{Lnnewapp} that there is  finite subset
${\cal H}_0\subset C(X)_{s.a.}$ and $\sigma_0>0$ such that, for any unital positive
linear maps $\psi_1, \psi_2: C(X)\to C$ (for any unital stably finite
\CA\, $C$),
\beq\label{Sd1extp-1}
|\tau(\psi_1(a))-\tau(\psi_2(a))|<\sigma_0\tforal a\in {\cal H}
\eneq
implies that
\beq\label{Sd1extp-2}
\mu_{\tau\circ \psi_2}(O_r)\ge \Delta_1(r)\tforal \tau\in T(C)
\eneq
and  all open balls $O_r$ of $X$ with $r\ge \eta_1,$
provided that
\beq\label{Sd1extp-2+}
\mu_{\tau\circ \psi_1}(O_r)\ge \Delta(r)
\eneq
for all open balls $O_r$ with radius $r\ge \eta_1/4$ and for any
tracial state $\tau$ of $C.$

Let $\eta=\eta_1/4.$
Let  $\gamma_1'>0$ (in place $\gamma_1$) and $\gamma_2'>0$ (in place
$\gamma_2$) be as given  by
\ref{Sd1uni} for $\ep/16$ (in place of $\ep$), $X,$ $\Delta_1.$
Let $\sigma=\min\{\gamma_1'/4, \gamma_2'/4, \sigma_0, \dt_0/2\}.$
Let $\dt_1>0$ (in place of $\dt$), ${\cal G}_1\subset C(X)$ (in place of ${\cal G}$) be a finite subset,
${\cal P}_1\subset \underline{K}(C(X))$ be a finite subset,  ${\cal H}\subset C(X)_{s.a.}$
be a finite subset, ${\cal V}\subset K_1(C(X))\cap {\cal P}$
be a finite subset, $N_1$ (in place of $N$) be an integer and $K\ge 1$ be an integer as given by \ref{Sd1uni} for $\ep/16$ (in place of $\ep$), ${\cal F},$
$\Delta_1$ (and for $X$).
We may assume that ${\cal F}\subset {\cal G}'.$
Put ${\cal H}_1={\cal H}\cup{\cal H}_0.$
Let ${\cal U}$ be a finite subset of $U(M_L(C(X)))$ such that
$\overline{{\cal U}}=J_{c(G({\cal V})}({\cal V}).$
Let $G_{\cal U}$ be the subgroup of $K_1(C(X)\otimes C(\T))$ generated by
$\{[u\otimes 1]:u\in {\cal U}\}.$  Let $R(G_{\cal U})$ be as defined in \ref{DUb}.
We may assume that $L\ge \max\{n, R(G({\cal U}))\}\ge {\rm dim}Y+1.$

%Let $s_1, s_2,..., s_{l}\in \rho_{C(X)}(K_0(C(X)).$

  Let $n_0$ (in place of $n$) be an integer,\,$q_1, q_2,...,q_s\in C(X_{n_0}\times \T)$ be a finite set of mutually orthogonal projections
with $1_{C(X_{n_0}\times \T)}=\sum_{j=1}^sq_j$ which represent each connected component of $X_{n_0}\times \T,$
$s_1,s_2,...,s_l$ (in place of $g_1, g_2,...,g_k$), $G_0$ and $N_2\ge 1$ (in place of $N$) be the integer given  by \ref{sec1MT} for $X\times \T$ (in place of $X$),
 $\ep/16$ (in place of $\ep$),
${\cal G}_2={\cal G'}\otimes \{1_{C(\T)}, z, z^*\}$
(in place of ${\cal G}$), ${\cal P}_1\oplus \boldsymbol{\bt}({\cal P})$ (in place of ${\cal P}$), ${\cal H}_1\otimes \{1_{C(\T)}\}$
(in place of ${\cal H}$), ${\cal U}\otimes \{1_{C(\T)}\}$ (in place of ${\cal U}$), $\sigma/L$ (in place of $\sigma_1$ and $\sigma_2$) and integer $L_1.$ We may assume that $n_0\ge n_0'.$
Note that there are mutually orthogonal projections $q_1', q_2',...,q_s'\in C(X_{n_0})$ such that
$q_i=q_i'\otimes 1_{C(\T)},$ $ i=1,2,...,s.$
We may assume that
\beq\label{sd1ext-4+1}
G_0=\Z^{k_1}\oplus \Z^{k_2}\oplus {\rm Tor}(G_0),
\eneq
where $\Z^{k_1}\subset {\rm ker}\rho_{C(X)}$ and $\Z^{k_2}\in \boldsymbol{\bt}(K_1(C(X))).$
We may further write that ${\rm Tor}(G_0)=G_{00}\oplus G_{01},$ where
$G_{00}\subset {\rm ker}\rho_{C(X)}$ and $G_{01}\subset \boldsymbol{\bt}(K_1(C(X))).$
We assume that $\Z^{k_1}$ is generated by $s_1,s_2,...,s_{k_1}$ and
$\Z^{k_2}$ is generated by $s_{k_1+1}, s_{k_1+2},...,s_{k_2}.$ Write $s_{k_1+i}=\boldsymbol{\bt}(g_i),$  where $g_i\in K_1(C(X)),$
$i=1,2,..., {k_2}.$
%and $g_{k_1+k_2+1}, g_{k_1+k_2+2},..., g_l$ generates $G_{01}.$
Without loss of generality, to simplify notation,
we may assume that $G_0\subset [\psi_{n_0,\infty}](K_0(C(X_{n_0})))\otimes C(\T)),$  $\Z^{k_1}\in [\psi_{n_0, \infty}](K_0(C(X_{n_0})))$
and $\Z^{k_2}\subset \boldsymbol{\bt}([\psi_{n_0, \infty}](K_1(C(X_{n_0})))).$ Let $s_1',s_2',...,s_{k_1}'\in
K_0(C(X_{n_0}))$ and $g_{k_1+1}', g_{k_1+2}',...,g_{k_1+k_2}'\in K_1(C(X_{n_0}))$ such that $[\psi_{n_0,\infty}](s_i')=s_i,$
$i=1,2,...,k_1$ and $[\psi_{n_0,\infty}](g_{i}')=g_i,$
$i=1,2,...,k_2.$

Let $N=\max\{N_1, N_2({\rm dim}Y+1)\}.$
Let $\dt_2>0$ and ${\cal G}_3\subset C(X)$ be a finite subset such that for any unital $\dt_2$-${\cal G}_3$-multiplicative
\morp\, $L': C(X)\to C'$ (for any unital \CA\, with $T(C')\not=\emptyset$),
\beq\label{sd1ext-4+2}
\tau([L'](s_j))<1/2N\tforal \tau\in T(C'),\,\,\,j=1,2,...k_1,
\eneq
(see 10.3 of \cite{LnApp}).
We may assume that ${\cal F}\cup {\cal G}_1\subset {\cal G}_3.$
Let $G_4=\{f\otimes g: f\in {\cal  G}_2, g=1_{C(\T)}, z, z^*\}.$ With even smaller $\dt_2,$ we may assume that
for any unital $\dt_2$-${\cal G}_4$-multiplicative \morp\, $L'': C(X)\otimes C(\T)\to C',$
$[L''\circ \psi_{n_0, \infty}]$ is well defined on $\underline{K}(C(X_{n_0})\otimes C(\T)).$
%\beq\label{sd1ext-4-3}
%\tau([L'']
%
%(\boldsymbol{\bt}(g_j)))=0\tforal \tau\in T(C'), \,\,\, j=k_1+k_2+1, %k_1+k_2+2,...,k,
%\eneq
%since $G_{01}$ is a torsion group.
%Choose $\gamma>0$ so that
%\beq\label{Sd1extp-3}
%\gamma< 1/N({\rm dim}Y+1).
%\eneq

Let $\dt=\min\{\dt_1, \dt_2\}$ and ${\cal G}={\cal F}\cup{\cal G}_1\cup {\cal G}'\cup {\cal G}_3\subset C(X).$
Let $K_1\ge 0$ be an integer.
Let  $Y$ be a finite CW complex with dimension at most $d,$ let   $\phi: C(X)\to C=PM_m(C(Y))P$ be a unital
$\dt$-${\cal G}$-multiplicative \morp\,  with
${\rm rank}P\ge \max\{N, NK_1\}$ and
\beq\label{Sd1extp-4}
\mu_{\tau\circ \phi}(O_r)\ge \Delta(r)\tforal \tau\in T(C)
\eneq
and for all open balls $O_r$ of $X$ with radius $r\ge \eta,$  let $\kappa\in  KK(C(X_{n_0})\otimes C(\T), C(Y))$ be
such that
\beq\label{Sd1extp-5}
|\rho_{C}(\kappa\circ \boldsymbol{\bt}(g_i'))(t)|\le K_1\tforal t\in T(C),
\eneq
$i=1,2,...,k_2.$ Note that, by (\ref{sd1ext-4+2}), since ${\rm rank} P\ge NK_1,$
\beq\label{Sd1extp-5+1}
|\rho_{C}(\kappa(s_i'))(t)|\le K_1\tforal t\in T(C),
\eneq
$i=1,2,...,k_1.$
Without loss of generality, we may assume that $\phi(\psi_{n_0, \infty}(q_j))=P_j$ is a projection, $j=1,2,...,s,$ and
$\sum_{j=1}^s P_j={\rm id}_C.$

%we may assume that $Y$ is connected.
%Set $C=PM_m(C(Y))P.$
%Let
%\beq\label{Sd1extp-6}
%K'=\max\{1, {|\tau(\kappa(g_i))|{{\rm rank}P}}: 1\le i\le k\}.
%\eneq
%Since $K'/{\rm rank}P<\gamma\le 1/2N$ and $1\ge 2N\gamma,$
%we have
%Note that
%$$
%{\rm rank}P\ge NK_1.
%$$
Define $\kappa_1\in Hom_{\Lambda}(\underline{K}(C(X_{n_0}\times \T)),\underline{K}(C(Y)))$ by
$$
\kappa_1|_{\underline{K}(C(X_{n_0}))}=[\phi\circ \psi_{n_0,\infty}]\andeqn
\kappa_1|_{\boldsymbol{\bt}(\underline{K}(C(X_{n_0})))}
=\kappa|_{\boldsymbol{\bt}(\underline{K}(C(X_{n_0})))}.
$$
Let $\Gamma$ be given as in the statement so that $\Pi_C\circ \Gamma|_{\cal Q}=\kappa_1\circ \boldsymbol{\bt}|_{\cal Q}.$
Define $\lambda: T(C)\to T(C(X\times \T))$ by
$$
\lambda(\tau)(f\otimes g)=\tau(\phi(f))\cdot t_m(g)
$$
for all $\tau\in T(C)$ and for all $f\in C(X)$ and
$g\in C(\T),$ where $t_m$ is the tracial state of $C(\T)$ induced by
the normalized Lesbegue measure. One checks that $\lambda$ is compatible
with $\kappa_1.$
%First, we note, since ${\rm dim}Y=1,$ that
%$$
%\cup_{n=1}^{\infty}U(M_n(C))/CU(M_n(C))
%=U(C)/CU(C).
%$$
Fix a splitting map $J_{M_L(C)}: K_1(C)\to U(M_L(C))/CU(M_L(C))$ (note that $L\ge {\rm dim}Y+1$), we write
$$
U(M_L(C))/CU(M_L(C))=
\Aff(T(M_L(C)))/\overline{\rho_{M_L(C)}(K_0(C))}\oplus J_{M_L(C)}(K_1(C)).
$$
Let $G_1=G_{\cal U}+\boldsymbol{\bt}(G({\cal Q})).$ Note since
$K_1(C(X)\otimes C(\T))=K_1(C(X))\oplus \boldsymbol{\bt}(K_0(X)),$ we may write $G_1=G_{\cal U}\oplus \boldsymbol{\bt}(G({\cal Q})).$ Note also that we have assumed that
$$
G_1\subset [\psi_{n_0, \infty}](K_1(C(X_{n_0})))\oplus \boldsymbol{\bt}\circ [\psi_{n_0, \infty}](K_0(C(X_{n_0}))).
$$
To simplify notation, we may assume that
$G_1=[\psi_{n_0, \infty}](K_1(C(X_{n_0})))\oplus \boldsymbol{\bt}\circ [\psi_{n_0, \infty}](K_0(C(X_{n_0}))).$
We note that there is an injective \hm\, $J_c: G_1\to U(M_L(C(X)\otimes C(\T)))/CU(M_L(C(X)\otimes C(\T)))$ such that $\Pi\circ J_c={\rm id}_{G_1}$
(see \ref{DUb}),
where $\Pi$
%: U(M_L(C(X)\otimes C(\T)))/CU(M_L(C(X)\otimes C(\T)))\to U(M_L(C(X\times \T))/U_0(M_L(C(X\times \T)))$$
is  defined in \ref{DUb} for $C(X)\otimes C(\T).$
Let $\imath: U(M_n(C))/CU(M_n(C))\to U(M_L(C))/CU(M_L(C))$ denote the embedding
given by $x\to 1_{L-n}\oplus x.$
%Write $J(K_1(C(X)\otimes C(\T)))=J(K_1(C(X)))\oplus
%J(\boldsymbol{\bt}(K_0(C(X)))).$
Define
$$
{\widetilde\Gamma}: J(G_1)+
\Aff(T(M_L(C(X\times \T))))/\overline{\rho_{M_L(C(X\times \T))}(K_0(C(X\times \T)))}
\to  U(M_L(C))/CU(M_L(C))
$$
as follows.
Let  $\Gamma_1: G({\cal Q})\to U(M_L(C))/CU(M_L(C))$ be defined by
$$
\Gamma_1=\imath\circ \Gamma-J_{M_L(C)}\circ \kappa_1\circ {\boldsymbol{\bt}}|_{G({\cal Q})}.
$$
Since $\Gamma$ and $\kappa_1$ are compatible, the image of $\Gamma_1$
is in $\Aff(T(M_L(C)))/\overline{\rho_{M_L(C)}(K_1(C))}.$
%Since $\Aff(T(M_L(C)))/\overline{\rho_{M_L(C)}(K_1(C))}$ is divisible, there exists a
%\hm\, ${\widetilde{\Gamma_1}}: J(\boldsymbol{\bf}(K_0(C(X)))\to
%\Aff(T(M_L(C)))/\overline{\rho_{M_L(C)}(K_0(C))}$ which extends $\Gamma_1.$
Define
\beq\label{Sd1extp-7}
{\widetilde{\Gamma}}(x)&=&\phi^{\ddag}(x)\tforal x\in J(U_{\cal U}),\\
\widetilde{\Gamma}(x)&=& \Gamma_1(\Pi(x))\tforal x\in J(\boldsymbol{\bt}(G({\cal Q})))\andeqn\\
\widetilde{\Gamma}(x)&=&{\overline{\lambda}(x)}
\eneq
for all $x\in \Aff(T(M_L(C(X\times \T))))/\overline{\rho_{M_L(C(X\times \T))}(K_0(C(X\times \T)))},$
where
\beq
&&\hspace{-0.3in}\overline{\lambda}: \Aff(T(M_L(C(X\times \T))))/\overline{\rho_{M_L(C(X\times \T))}(K_0(C(X\times \T)))}\to\\
&&\hspace{0.8in}\Aff(T(M_L(C)))/\overline{\rho_{M_L(C)}(K_0(C))}
\eneq
is the map induced by $\lambda.$ Then $\widetilde{\Gamma}$ is continuous, moreover,
$$
\kappa_1\circ \Pi|_{J([\psi_{n_0, \infty}](K_1(C(X_{n_0}\otimes C(\T))))}=
\Pi_C\circ {\tilde{\Gamma}}|_{J([\psi_{n_0, \infty}](K_1(C(X_{n_0}\otimes C(\T)))))},
$$
$\lambda(\tau)(\psi_{n_0, \infty}(q_j))=\tau(P_j),$ $j=1,2,...,s.$

It follows from \ref{sec1MT} that there is
$\dt/4$-${\cal G}$-multiplicative \morp\, $\Phi: C(X)\to C$ such that
\beq\label{Sd1extp-8}
[\Phi\circ \psi_{n_0,\infty}] &=&
\kappa_1\\\label{Sd1exp-8+}
{\rm dist}(\overline{\langle \Phi(x)\rangle}, {\widetilde{\Gamma}}(x))&<&\sigma/L\tforal x\in
\overline{{\cal U}\otimes 1_{C(\T)}}\cup J(\boldsymbol{\bt}( {\cal Q}))\andeqn\\\label{sd1ext-8+2}
|\tau\circ \Phi(a)-\lambda(\tau)(a)|&<&\sigma/L\tforal a\in {\cal H}_1
\eneq
and for $\tau\in T(C).$
From (\ref{Sd1exp-8+}), we also have
\beq\label{Sd1exp-8+3}
{\rm dist}(\overline{\langle \Phi(x)\rangle}, \Gamma(x))&<&\sigma\tforal x\in
J_{c(\boldsymbol{\bt}( G({\cal Q})))}(\boldsymbol{\bt}( {\cal Q}))).
\eneq

Let $\psi=\Phi|_{C(X)}.$
Then
\beq\label{Sd1extp-9}
[\psi]|_{{\cal P}_1}=[\phi]|_{{\cal P}_1}.
\eneq
By  (\ref{sd1ext-8+2}) and the choices of ${\cal H}_0$ and $\sigma_0,$  we have
\beq\label{Sd1extp-11}
\mu_{\tau\circ \psi}(O_r)\ge \Delta_1(r)
\eneq
for all $\tau\in T(C)$ and all open balls $O_r$ with radius $r\ge \eta.$
We also have
\beq\label{Sd1extp-12}
|\tau\circ \phi(a)-\tau\circ \psi(a)|<\gamma_1'\tforal a\in {\cal H}\andeqn\\
{\rm dist}(\phi^{\ddag}({\bar x}), \psi^{\ddag}({\bar x}))<\gamma_2'
\tforal x\in {\cal U}.
\eneq
It follows from \ref{Sd1uni} that there exists a unitary
$W\in M_K(C)$ such that
\beq\label{Sd1extp-13}
\|\phi^{(K)}(f)-W^*\psi^{(K)}(f)W\|<\ep/16\tforal f\in {\cal F}.
\eneq
Let $v$ be a unitary in $C$ such that
\beq\label{Sd1extp-14}
\|v-\Phi(1_{C(\T)}\otimes z)\|<\ep/4.
\eneq
Take $u=W^*{\rm diag}(\overbrace{v,v,...,v}^K)W.$ Since we have assumed that $(\ep, {\cal F}', {\boldsymbol{\bt}}({\cal P}))$ is a $KL$-triple,
this implies that
\beq\label{Sd1extp-15}
{\rm Bott}(\phi^{(K)}\circ \psi_{n_0,\infty},\, u)|_{\cal P}=K\kappa\circ {\boldsymbol{\bt}}.
\eneq
Moreover, by (\ref{Sd1exp-8+3}),
\beq
{\rm Bu}(\phi^{(K)}, \,u)(x),K\Gamma(J_c(\boldsymbol{\bt}(x)))<\sigma<\dt_0\tforal x\in {\cal Q}.
\eneq
 \end{proof}

\section{The Basic Homotopy Lemma}

In this section we will prove the following statement holds {\it under assumption that \ref{Sd1uni} holds
for all finite CW complexes with ${\rm dim}Y\le d.$}

\begin{NN}\label{Sd1hom} (Homotopy Lemma for ${\rm dim}Y=d$)
Let $d\ge 0$ be an integer, let $X$ be a compact metric space, let $\ep>0,$ let ${\cal F}\subset C(X)$ be a finite subset and let
$\Delta: (0,1)\to (0,1)$ be a nondecreasing map.
Then there exist $\eta>0,\dt>0,$ $\gamma>0,$  a finite subset ${\cal G}\subset C(X)$ and a  finite subset ${\cal P}\subset \underline{K}(C(X)),$
a finite subset ${\cal Q}\subset {\rm ker}\rho_{C(X)},$ an integer $N\ge 1$
and an integer $K\ge 1$   satisfying the following:

Suppose that $Y$ is a finite CW complex with ${\rm dim}Y\le d,$
$P\in M_m(C(Y))$ is a projection such that ${\rm rank} P\ge N$ and
$\phi: C(X)\to C=PM_m(C(Y))P$ is a unital $\dt$-${\cal G}$-multiplicative \morp\,
%\hm\, with
%$h_{*}|_{F_mK_*(C(X))}=0$ for all $m\ge 4$
and  $u\in C$ is a unitary such that
\beq\label{Sd1hom-1}
\|[\phi(g),\, u]\|&<&\dt\tforal g\in {\cal G},\\
{\rm Bott}(\phi,\, u)|_{\cal P}&=&0,\\
{\rm dist}({\rm Bu}(\phi,\,u)(x), {\bar 1})&<&\gamma \tforal x\in {\cal Q}\tand\\
\mu_{\tau\circ \phi}(O_r)&\ge& \Delta(r)\tforal \tau\in T(PM_m(C(Y))P)
\eneq
and for all open balls $O_r$ of $X$ with radius $r\ge \eta.$ Then
there exists a continuous path $\{u_t: t\in [0,1]\}\subset U(M_K(C))$
such that
\beq\label{Sd1hom-2}
u_0={\rm diag}(\overbrace{u,u,...,u}^K),\,\,\, u_1=1_{M_K(C)}\\\tand
\|[\phi^{(K)}(f),\, u_t]\|<\ep\tforal f\in {\cal F}.
\eneq

\end{NN}

\begin{NN}\label{Sd1homF}
Let $X$ be a compact metric space, let $\ep>0,$ let ${\cal F}\subset C(X)$ be a finite subset and let
$\Delta: (0,1)\to (0,1)$ be a nondecreasing map.
There exists $\eta >0,\, \dt>0,$ $\gamma>0,$ a finite subset ${\cal G}\subset C(X)$ and a  finite subset ${\cal P}\subset \underline{K}(C(X)),$
 a finite subset ${\cal Q}\subset {\rm ker}\rho_{C(X)},$ and an integer $N\ge 1$   and an integer $K\ge 1$ satisfying the following:

Suppose that $Y$ is a finite CW complex with ${\rm dim}Y\le d,$  $P\in M_m(C(Y))$ is a projection with ${\rm rank}P\ge N$  and $\phi: C(X)\to C=PM_m(C(Y))P$ is a unital $\dt$-${\cal G}$-multiplicative \morp\,
%\hm\, with
%$h_{*}|_{F_mK_*(C(X))}=0$ for all $m\ge 4$
and  $u\in C$ is a unitary such that
\beq\label{Sd1homf-1}
\|[\phi(c),\, u]\|&<&\dt\tforal c\in {\cal G},\\
{\rm Bott}(\phi,\, u)|_{\cal P}&=&0\andeqn\\\label{Sd1homf-1+1}
{\rm dist}({\rm Bu}(\phi,\, u)(x), {\bar 1})&<&\gamma\tforal x\in {\cal Q}.
%\mu_{\tau\circ \phi}(O_r)\ge \Delta(r)\tforal \tau\in T(PM_m(C(Y))P)
\eneq
Suppose that there exists a \morp\, $L: C(X)\otimes C(\T)\to
C$ such that
\beq\label{Sd1homf-2}
\|L(c\otimes 1)-\phi(c)\|<\dt,\,\,\,
\|L(c\otimes z)-\phi(c)u\|<\dt\tforal c\in {\cal G}\\\label{Sd1homf-2+}
\tand \mu_{\tau\circ L}(O_r)\ge \Delta(r)\tforal \tau\in T(C)
\eneq
and for all open balls $O_r$ of $X\times \T$ with radius $r\ge \eta,$
where $z\in C(\T)$ is the identity function on the unit circle.
Then
there exists a continuous path $\{u_t: t\in [0,1]\}\subset U(M_K(C))$
such that
\beq\label{nSd1hom-2}
u_0=u^{(K)}\,\,\, u_1=1_{M_K(C)}\\\andeqn
\|[\phi^{(K)}(f),\, u_t]\|<\ep\tforal f\in {\cal F}.
\eneq
 \end{NN}

\begin{proof}

Note we  assume that \ref{Sd1uni} holds for ${\rm dim}Y\le d.$

Fix $1/2>\ep>0$ and ${\cal F}$ as stated in the \ref{Sd1homF}.
Without loss of generality, to simplify notation, we may assume that ${\cal F}$ is in the unit ball
of $C(X).$
Let
$${\cal F}'=\{f\otimes g: f\in {\cal F}\andeqn g\in \{1, z, z^*\}\}.$$

Let $\eta>0,$ $\gamma_1,\gamma_2>0,$ $\dt_1>0$ (in place of $\dt$), ${\cal G}'\subset  C(X\times \T)$ (in place of ${\cal G}$) be a finite subset, ${\cal P}_1\subset \underline{K}(C(X)\otimes C(\T))$ (in place of ${\cal P}$) be a finite subset, ${\cal H}\subset (C(X)\otimes C(\T))_{s.a.}$ be a finite subset and let ${\cal V}\subset K_1(C(X)\otimes C(T))\cap {\cal P}_1$
 be a finite subset, $N\ge 1$ be an integer and $K\ge 1$ be another integer
given by \ref{Sd1uni} (in fact by \ref{Sd1uniR}) for
$\ep/16$ (in place of $\ep$), ${\cal F}'$ (in place of ${\cal F}$)
and for $X\otimes \T$ (in place of $X$).

To simplify notation, without loss of generality, we may assume that
$$
{\cal G}'=\{f\otimes g: f\in {\cal G}''\andeqn g\in \{1, z, z^*\} \},
$$
where ${\cal G}''\subset C(X)$ is a finite subset,
$
{\cal P}_1={\cal P}_0\oplus \boldsymbol{\bt}({\cal P}),
$
where ${\cal P}_0,\, {\cal P}\subset \underline{K}(C(X))$ are finite subsets and
$$
{\cal H}=\{f\otimes g: f\in {\cal H}_0\andeqn g\in {\cal H}_1\},
$$
where $1_{C(X)}\in {\cal H}_0\subset C(X)_{s.a.}$ and
$1_{C(\T)}\in {\cal H}_1\subset C(\T)_{s.a.}$ are finite subsets.
We may further assume that $\ep/16<\dt_1$ and ${\cal F}\subset {\cal G}_1$ and $(2\dt_1, {\cal G}', {\cal P}_1)$ is a $KL$-triple.

Write $C(X)=\lim_{n\to\infty}(C(X_n), \psi_n),$ where each $X_n$ is a finite CW complex and $\psi_n$ is a unital \hm.
We may assume that $n_0\ge 1$ is an integer such that
${\cal F}\subset \psi_{n, \infty}(C(X_{n_0}))$ and  ${\cal P}_0, {\cal P}\subset [\psi_{n_0, \infty}](\underline{K}(C(X_{n_0}))).$

Suppose that ${\cal U}\subset U(M_R(C(X\otimes \T)))$ is a finite subset
such that $\overline{\cal U}=J_{c(G({\cal V}))}({\cal V})$ and
$R=\max\{R(G({\cal V})),d\}.$
To simplify notation further, we may assume
that ${\cal U}={\cal U}_0\sqcup {\cal U}_1,$
where
\beq
{\cal U}_0&\subset &\{u\otimes 1_{M_R}: u\in U(M_R(C(X))),\,\,[u]\not=0 \,\,\,{\rm in}\,\,\,K_1(C(X))\},\andeqn\\
{\cal U}_1&\subset &\{x: [x]\in \boldsymbol{\bt}(K_0(C(X)))\,\,\, \andeqn
[x]\not=0\,\,\,{\rm in}\,\,\,K_1(C(X)\otimes C(\T))\}.
\eneq
Furthermore, we may write
$$
G_{{\cal U}_1}=G_{00}\oplus G_{\bf b},
$$
where $G_{{\cal U}_1}$ is the subgroup generated by $\{[u]: u\in {\cal U}_1\},$ $G_{00}$ is a finitely generated free group
such that $G_{00}\cap \boldsymbol{\bt}({\rm ker}\rho_{C(X)})=\{0\}$ and $G_{\bf b}\subset
\boldsymbol{\bt}({\rm ker}\rho_{C(X)})$ is also a finitely generated subgroup.

Let $G_{{\cal U}_0}$ be the subgroup of $K_1(C(X)\otimes C(\T))$ generated
by $\{[u]: u\in {\cal U}_0\}.$ Let $G({\cal U})=G_{{\cal U}_0}\oplus G_{{\cal U}_1}.$
%Let $J_c: G({\cal U})\to U(M_R(C(X)))/UM_R(C(X)),$ where $R=\max\{R(G({\cal U})), d\}$ be as defined in \ref{DUb}.
Accordingly, we may assume that
$$
\overline{{\cal U}_1}=\overline{{\cal U}_{10}\sqcup {\cal U}_{11}},
$$
where $\overline{{\cal U}_{10}}$ is a set of generators of $J_c(G_{00})$ and
$\overline{{\cal U}_{11}}$ is the set of generators of $J_c(G_{\bf b}).$  Put
${\cal U}_{\bf b}=\overline{{\cal U}_{11}}.$ Note that ${\cal U}_{\bf b}=J_c(\boldsymbol{\bt}(F_2K_0(C(X)))\cap \overline{\cal U}$
and $J_c(G_{00})\cap SU_R(C(X))/CU(M_R(C(X)))=\{0\}.$
Let ${\cal Q}\subset {\rm ker}\rho_{C(X)}$ be  a finite subset
such that
$$
{\cal Q}\supset \{x\in {\rm ker}\rho_{C(X)}: \boldsymbol{\bt}(x)=[u],\,\,\, u\in \Pi(\overline{{\cal U}_{11}})\}.
$$
Let ${\cal G}_u\subset C(X)$ be a finite subset such that
${\cal U}_0=\{(x_{ij})_{R\times R}: x_{ij}\in {\cal G}_u\}.$
Let $0<\dt_2=\min\{\ep/16, \dt_1/16R^3, \gamma_1/16R^3, \gamma_1/16R^3\}.$
Let $\dt_2>dt_3>0$ be a
positive number and ${\cal G}\supset {\cal G}''\cup {\cal G}_u$ be a finite subset satisfying
following:
If $\phi': C(X)\to C'$ is a unital $\dt$-${\cal G}$-multiplicative \morp, $ u'\in C'$ is a unitary and
$L': C(X)\otimes C(\T)$ is a unital \morp, where $C'$ is any unital \CA\, such that $\|[\phi'(g),\,u']\|<\dt$ for all $g\in {\cal G},$
$$
\|\phi'(g)-L'(g\otimes 1)\|<\dt\andeqn \|\phi'(g)u'-L'(g\otimes z)\|<\dt
$$
for all $g\in {\cal G},$
% then
%$$
%\|L'(c\otimes 1)-\phi'(c)\|<\dt_2\tforal c\in {\cal G}_u.
%$$
then  $[L'\circ \psi_{n_0, \infty}]$ is well defined on
$\underline{K}(C(X_{n_0})\otimes C(\T)).$  In particular, we also assume that
$[\psi'\circ \psi_{n_0, \infty}]$ is well defined on $\underline{K}(C(X_{n_0})).$ Let
${\cal G}_1=\{f\otimes g: f\in {\cal G}, \andeqn g\in \{1, z, z^*\}\}.$
% and let ${\cal G}={\cal G}_1\cup {\cal G}_2.$
Without loss of generality, we may also assume
that
$(\Psi')^{\ddag}$ is defined on ${\cal U}$ for any $2\dt_3$-${\cal G}_1$-multiplicative unital \morp\, $\Psi': C(X)\otimes C(\T)\to C'$ for any
unital \CA\, $C'.$

Let $\gamma=\min\{\gamma_1/16R^3,\gamma_2/16R^3\}$ and let
$L_1\ge 1.$
Let $n\ge n_0$ be an integer,
 $q_1, q_2,...,q_s\in C(X_n\times \T)$ be mutually orthogonal projections with
$1_{C(X\times \T)}=\sum_{j=1}^sq_j$ which represent each connected component of $X_n,$
 $N\ge 1$ be an integer and $G_1$ be a finitely generated subgroup
 and ${\cal U}_{\bf b}$ be
as required
by \ref{sec1MC} for
$\ep/16$ (in place of $\ep$), ${\cal G},$   $G_0,$ ${\cal P}_1$ (in place of ${\cal P}$), ${\cal H}$ and $\gamma$ (in place of
 $\sigma_1$ and $\sigma_2$).
 %We may write $q_j=q_j'\otimes 1_{C(\T)},$ where
% $q'_j\in C(X)$ is a projection.
%Here we choose $K=0.$
%By Lemma \ref{rsmall}, choose $\dt_3>0$  and  a finite subset ${\cal G}_2\subset C(X)$ such that
%such that, for any unital \morp\, $\phi': C(X)\to C'$ ( for any unital \CA\, $C'$),
%\beq\label{tr}
%|\tau ( {\rm bott}_0(\phi',\, w)(g_j))|<1/N\tforal \tau\in T(C'),\,\,\, j=1,2,...,
%\eneq
%provided that $\|[\phi'(g),\, w]\|<\dt_3\tforal g\in {\cal G}_2.$
We may write $q_j=q_j'\otimes 1_{C(\T)},$ where $q_j'$ is a projection,
$j=1,2,...,s.$
Let $\dt=\min\{\dt_1/2, \dt_2/2, \dt_3/2, \gamma/4\}.$
Let $Y$ be a finite CW complex with ${\rm dim}Y\le d.$ Without loss of generality, to simplify
notation, we may assume that $Y$ is connected. Let $P\in M_m(C(Y))$ be a projection with ${\rm rank} P\ge N$ (so ${\rm rank}\,P\ge \max\{NK', N\},$ if $K'=0$).

Suppose that $\phi: C(X)\to C=PM_m(C(Y))P$ is a unital $\dt$-${\cal G}$-multiplicative \morp\, and
$L: C(X)\otimes C(\T)\to C$ is a unital $\dt$-${\cal G}_1$-multiplicative
\morp\, satisfying the assumption.
Without loss of generality, we may assume that $L\circ \psi_{n, \infty}(q_j)=P_j=\psi\circ \psi_{n, \infty}(q_j)$ is a projection, $j=1,2,...,s$ and
$1_C=\sum_{j=1}^sP_j.$
Note, by the assumption above, $[\phi\circ \psi_{n, \infty}]$ is well defined
on $\underline{K}(C(X_{n_0}))$ and
$[L\circ \psi_{n, \infty}]$ is well defined on $\underline{K}(C(X_{n_0})\otimes C(\T)).$

Let $H: C(X_n\times \T)\to C$  be defined
by $H(f\otimes g)=\phi(f)\cdot g(1)\cdot P$ for all $f\in C(X_n)$ and $g\in C(\T),$ where $1\in \T$ is a point. Denote $\kappa=[H].$
Note that, by the assumption,
\beq\label{Sd1homF-10}
[L\circ \psi_{n, \infty}]=[H].
\eneq
Note also that $[H]|_{\boldsymbol{\bt}(\underline{K}(C(X_n)))}=0.$
%One has $F_m(K_i(C(X)\otimes C(\T)))=
%F_m(K_i(C(X)))\oplus F_m(\boldsymbol{\bt}(K_{i+1}(C(X)))).$ Therefore
%\beq\label{Sd1homF-10+}
%[H]|_{F_mK_*(C(X)\otimes C(T))}=0.
%\eneq
Let $\lambda:T(C)\to T(C(X\times \T))$ be defined by
\beq\label{Sd1homF-10+}
\lambda(\tau)(g)=\tau\circ L(g)\tforal g\in C(X\times \T)
\eneq
and for all $\tau\in T(C).$
%In the following definition of $\Gamma_1,$ we write
%$K_1(C(X)\otimes C(\T))=K_1(C(X))\oplus \boldsymbol{\bt}(K_0(C(X))).$
%Let $S$ be the subgroup generated by ${\cal V}.$
Note that
$[L]|_{\boldsymbol{\bf}(\boldsymbol{\bt}(K_0(C(X)))\cap {\cal P})}=0.$
%It follows that $L^{\ddag}$ maps $\overline{{\cal V}}$ into
%$\Aff (T(M_R(C)))/\overline{\rho_{M_R(C))}(K_0(C)))}.$
%However, $$\Aff (T(PM_m(C(Y))P))/\overline{\rho_{K_0(C(Y))}(K_0(C(Y)))}$$
%is a divisible group. Therefore there is a \hm\,
%$$\Gamma: J_c(\boldsymbol{\bt}(K_0(C(X)))\to \Aff(T(PM_m(C(Y))P))/\overline{\rho_{K_0(C(Y))}(K_0(C(Y)))}$$ such that
%$$
%\Gamma|_{\overline{{\cal V}}}=L^{\ddag}|_{\overline{{\cal V}}}.
%$$
%Moreover,  by (\ref{Sd1homf-1+1}),
%\beq\label{Sd1homF-10++}
%\Gamma|_{{\cal U}_{\bf b}}=0.
%\ene
Define
$$
\Gamma_1: J_c(G({\cal V}))+\Aff(T(M_R(C(X\times\T))))/
\overline{\rho_{M_R(C(X\times \T))}(K_0(C(X\times \T)))}
\to U(M_R(C))/CU(M_R(C))
$$
by
\beq\label{Sd1homF-11}
\Gamma_1(z) &=&\phi^{\ddag}(z)\tforal z\in J_c(G({\cal U}_0)),\\
\Gamma_1(z)&=& L^{\ddag}(z)\tforal z\in J_c(G_{00}),\\
\Gamma_1(z)&=& {\bar 1}\tforal z\in
J_c(G_{\bf b})\andeqn\\
\Gamma_1(z)&=&\overline{\lambda}(z)
\eneq
for all $z\in \Aff(T(M_R(C(X)\otimes C(\T))))/\overline{\rho_{M_R(C(X)\otimes C(\T))}(K_0(C(X)\otimes C(\T)))},$
where $\overline{\lambda}$ is induced by $\lambda.$
%Since $\phi$ is \hm\, $\phi(SU_R(C(X)))\subset SU_R(C).$
Note that $\Gamma_1|_{{\cal U}_{\bf b}}=0.$
It is also clear that $\lambda(\tau)(q_j)=\tau(P_j),$ $j=1,2,...,s.$

%It follows that
%\beq\label{Sd1homF-11+}
%{\rm dist}(\Gamma_1(x), SU_R(C)/CU(M_R(C)))<\gamma_2/2\tforal x\in J_c({\cal V}).
%\eneq

 It follows from
 \ref{sec1MC}  that there exist three unital
 $\dt$-${\cal G}$-multiplicative \morp s
 $\Phi_{00}: C(X)\otimes C(\T)\to P_{00}M_m(C(Y))P_{00},$
 $\Phi_{01}: C(X)\otimes C(\T)\to P_{01}M_m(C(Y))P_{01}$ and $\Phi_1: C(X)\otimes C(\T)\to P_1M_m(C(Y))P_1$ with $P=P_{00}\oplus P_{01}\oplus P_1,$ $\Phi_{00}(f\otimes g)=\sum_{j=1}^s f(\xi_j)e_{0,j}\cdot h_j(g)$
 for all $f\in C(X)$ and $g\in C(\T),$
 where $\xi_j\in X,$ $e_{0,j}\in P_{00}M_m(C(Y))P_{00}$ is a projection so that
 $\sum_{j=1}^s e_{0,j}=P_{00}, $
 $h_j: C(\T)\to e_{0,j}M_m(C(Y))e_{0,j}$ is a unital \hm\, with $h_j(z)\in U_0( e_{0,j}M_m(C(Y))e_{0,j})$ ($j=1,2,...,s$),
 $\Phi_{01}(f\otimes g)=\sum_{j=1}^s\Psi_j(f)\cdot g(1)\cdot e_{1,j}$ for all $f\in C(X)$ and  $g\in C(\T),$
 where $1\in \T$ is the point on the unit circle,
 where $\Psi_j: C(X)\to e_{1,j}M_m(C(Y))e_{1,j}$ is a unital \morp\, and $e_{1,j}\in P_{01}M_m(C(Y))P_{01}$ is a projection so that $\sum_{j=1}^se_{1,j}=P_{01},$
$ \Phi_1=\psi_1\circ \psi_0,$ where $\psi_0: C(X\otimes \T)\to C(J)$ is a $\dt$-${\cal G}$-multiplicative \morp, $J$ is a finite disjoint union
 of intervals, and $\psi_1: C(J)\to P_1M_m(C(Y))P_1$ is a unital \hm,
 such that
 \beq\label{Sd1hmF-12}
 [(\Phi_0\oplus \Phi_1)\circ \psi_{n,\infty}]&=&[H],\\\label{Sd1homF-12+}
 {\rm dist}((\Phi_0\oplus \Phi_1)^{\ddag}(z), \Gamma_1(z))&<&\gamma \tforal z\in {\cal U}\andeqn\\\label{Sd1homF-12++}
 |\tau\circ (\Phi_0\oplus \Phi_1)(a)-\lambda(\tau)(a)|&<&\gamma\tforal a\in {\cal H} \andeqn\,\rforal \tau\in T(C),
 \eneq
 where $\Phi_0=\Phi_{00}\oplus \Phi_{01}.$
% By (\ref{Sd1homF-11+}),
% \beq\label{Sd1hmF-12+}
% {\rm dist}((\Phi_0\oplus \Phi_1)^{\ddag}(x), SU_R(C)/CU(M_R(C)))<\gamma_1
 %\eneq
% for all $x\in J_c({\cal V})\cap SU_R(C(X)\otimes C(\T))/CU(M_R(C(X)\otimes C(\T))).$
 Let $z_{00}=\Phi_{00}(1\otimes z)=\sum_{j=1}^sh_j(z).$ Since $h_j(z)\in U_0(e_{0,j}M_m(C(Y))e_{0,j}),$
 there is a continuous path of unitaries $\{z_{00,j}(t): t\in [1/2,1]\}\subset U_0(e_{0,j}M_m(C(Y))e_{0,j})$ such that
 \beq\label{Sd1homF-12+1}
 z_{00,j}(1/2)=h_j(z)\andeqn z_{00,j}(1)=e_{0,j},\,\,\, j=1,2,...,s.
 \eneq
 Define $z_{00}(t)=\sum_{j=1}^sz_{00,j}(t)$ for $t\in [0,1/2].$ Then
 \beq\label{Sd1homF-12+2}
 \Phi_{00}(f\otimes 1)z_{00}(t)=z_{00}(t) \Phi_{00}(f\otimes 1)\, \tforal f\in C(\T)\andeqn t\in [1/2,1].
 \eneq
 Let $z_0=\psi_0(1\otimes z).$ There is a unitary $v_0\in C(J)$
 such that
\beq\label{Sd1homF-13}
\|z_0-v_0\|<\dt.
\eneq
There exists a continuous path of unitaries $\{w(t): t\in [1/2,1]\}\subset C(J)$ such that
\beq\label{Sd1homF-14}
w(1/2)=v_0,\,\,\,w(1)=1.
\eneq
Note that
\beq\label{Sd1homF-15}
w_t\phi_0(f)=\phi_0(f)w_t\tforal f\in C(X)\andeqn \tforal t\in [1/2,1].
\eneq
On the other hand,  we also have, by (\ref{Sd1homF-10}) , (\ref{Sd1hmF-12}), the assumption
(\ref{Sd1homf-1+1}) and (\ref{Sd1homF-10+}),
\beq\label{Sd1homF-15+1}
[L]|_{{\cal P}_1}&=&[\Phi_0\oplus \Phi_1]|_{{\cal P}_1},\\
{\rm dist}(L^{\ddag}(x), \Gamma_1(x))&<& 2\gamma \tforal x\in {\cal U}\andeqn\\
|\tau(L(a))-\lambda(\tau)(a)|&=&0\tforal a\in C(X)\otimes C(\T)\andeqn \tau\in T(C).
\eneq
With these and (\ref{Sd1homf-2+}) as well as (\ref{Sd1hmF-12}), (\ref{Sd1homF-12+}) and (\ref{Sd1homF-12++}),
we conclude, by applying  \ref{Sd1uniR},  that there exists a unitary $W\in M_K(PM_m(C(Y))P)$ such that
\beq\label{Sd1homF-16}
\|L^{(K)}(f\otimes 1)-W^*(\Phi_0^{(K)}\oplus \Phi_1^{(K)})(f\otimes 1)W\|&<&\ep/16
\tforal f\in {\cal F}\andeqn\\
\|u^{(K)}-W^*(z_{00}^{(K)}\oplus P_{01}^{(K)}\oplus v_0^{(K)})W\|&<&\ep/16+2\dt<3\ep/16,
\eneq
%where
%\beq
%u^{(K)}={\rm diag}(\overbrace{u,u,...,u}^{K}),
%z_{00}^{(K)}={\rm diag}(\overbrace{z_{00}, z_{00},...,z_{00}}^{K}),\\
%P_{01}^{(K)}={\rm diag}(\overbrace{P_{01},P_{01},...,P_{01}}^{K})\andeqn
%v_0^{(K)}={\rm diag}(\overbrace{v_0,v_0,...,v_0}^{K}).
%\eneq
There exists a continuous path of unitaries $\{u_t: t\in [0,1/2]\}\subset M_K(PM_m(C(Y))P)$ such that
\beq\label{nSd1homF-16}
&&u_0=u^{(K)},\,\,\, u_{1/2}=W^*(z_{00}^{(K)}\oplus P_{01}^{(K)}\oplus v_0^{(K)})W\andeqn\\
&&\|u^{(K)}-u_t\|<\ep/2\tforal t\in [0,1/2].
\eneq
Define
\beq\label{Sd1homF-17}
u_t=W^*(z_{00}(t)^{(K)}\oplus P_{01}^{(K)}\oplus \psi_1(w(t)^{(K)})W\tforal t\in [1/2,1].
\eneq
Then $\{u_t: t\in [0,1]\}$ is a continuous path of unitaries, $u_0=u,\,\,\, u_1=1$ and
\beq\label{Sd1homF-18}
\hspace{-0.6in}\|\phi^{(K)}(f)u_t-u_t\phi^{(K)}(f)\|
&\le & 2\dt+2\|\phi(f)-L(f\otimes 1)\|\\
&+&\|[L^{(K)}(f\otimes 1)u_t-u_tL^{(K)}(f\otimes 1)]\|\\
&<& 6\dt<3\ep/8\,\,\,\tforal f\in {\cal F}\andeqn \tforal t\in [0,1].
\eneq

\end{proof}

\begin{lem}\label{oldmatrix}
Let $n\ge 64 $ be an integer. Let $\ep>0$ and $1/2>\ep_1>0.$ There
exists ${\ep\over{2n}}>\dt>0$
%$\eta>0,$
and a finite subset ${\cal G}\subset D\cong M_n$ satisfying the
following:

Let $\Delta, \Delta_1, \Delta_2:(0,1)\to (0,1)$ be  increasing maps.
Suppose that $X$ is a compact metric space,  ${\cal F}\subset C(X)$
is a finite subset,  $1>b>a>0$ and $1>c>4\pi/n>0.$
Then there exists a finite subset
${\cal F}_1\subset C(X)$ satisfying the following:

Suppose that $A$ is a unital \CA\, with $\text{T}(A)\not=\emptyset,$
$D\subset A$ is a \SCA\, with $1_D=1_A,$ $\phi: C(X)\to A$ is a
unital \morp\, and suppose that $u \in U(A)$  such that
\beq\label{matr-0}
%&&\|[w,\, u]\|<\dt_1,\,\,\,
\|[x,\,u]\|<\dt \tand \|[x,\,\phi(f) ]\|<\dt \tforal  x\in {\cal
G} \tand f\in {\cal F}_1.
%\andeqn\\
%\|[x, \, u]\|<\dt_\tforal x\in {\cal G}_2.
\eneq
Suppose also that
\beq\label{matr-1}
\tau(\phi(f))\ge \Delta(r) \tforal \tau\in \text{T}(A)\tand
\eneq
for all $f\in C(X)$ with $0\le f\le 1$ so that $\{x\in X: f(x)=1\}$
 contains an
open ball of $X$ with radius $r\ge a,$ and suppose
that
\beq\label{matr-1+}
\tau(\phi(f)^{1/2}g(u)\phi(f)^{1/2})\ge \Delta_1(r)\Delta_2(s)\tforal \tau\in \text{T}(A),
\eneq
for all $f\in C(X)$ with $0\le f\le 1$ so that $\{x\in X: f(x)=1\}$ contains an open ball of $X$ with radius
$r\ge b$ and for all $g\in C(\T)$ with $0\le g\le 1$ so that $\{t\in  \T: g(t)=1\}$ contains
an open arc of $\T$ with length $s\ge c.$
 Then, there exists a unitary $v\in D$ and  a
continuous path of unitaries $\{v(t): t\in [0,1]\}\subset D$ such
that
\beq\label{matr-2}
&&\|[u, \, v(t)]\|<n\dt<\ep,\,\,\,\|[\phi(f),\, v(t) ]\|<n\dt<\ep\\
&&\,\,\, \tforal f\in {\cal F}\tand t\in [0,1],\\
&&v(0)=1,\,\,\, v(1)=v\tand\\
&&\hspace{-0.4in}\tau(\phi(f)^{1/2}g(vu)\phi(f)^{1/2})\ge (1-1/2^{n+2}){\Delta((1-1/2^{n+1})r)\over{n^2}}\tforal \tau\in \text{T}(A)
\eneq
for any pair of $f\in C(X)$ with $0\le f\le 1$ so that the set
$\{x: f(x)=1\}$
contains an open ball with radius $r\ge (1+1/2^{n+1})a$ and $g\in C(\T)$ with $0\le
g\le 1$ so that $\{t\in \T: g(t)=1\}$  contains an  open arc of $\T$ with length $s\ge 4\pi/n+\pi/n2^{n+1}.$ Moreover,
\beq\label{matr-2+}
\hspace{-0.2in}\tau(\phi(f)^{1/2}g(vu)\phi(f)^{1/2})\ge \Delta_1((1-1/2^{n+1})r)\Delta_2((1-1/2^{n+1})s)-{\Delta_1(b/2)\Delta_2(c/2)\over{2^{n+5}}}
\eneq
for all $\tau\in \text{T}(A),$
for all  $f\in C(X)$ with $0\le f\le 1$ so that the set
$\{x: f(x)=1\}$
contains an open ball with radius $r\ge (1+2^{-n-1})b$
%$i=0,1,2,...,m_0-1$
and $g\in C(\T)$ with $0\le
g\le 1$ so that $\{t\in \T: g(t)=1\}$ contains an  open arc of $\T$ with length at
least $s\ge (1+1/n^22^{n+1})c.$

Furthermore,
\beq\label{matr-n2}
{\rm length}(\{v(t)\})\le \pi.
\eneq

\end{lem}

\begin{proof}
Let $r_1>r_2>\cdots >r_{l-1}>r_l$ and
$r_l=b$ such that $r_j-r_{j+1}<a/2^{n+1},$ $j=1,2,...,l-1.$

%We may assume that $0<\ep<1.$  Set $\eta=1/n.$
Let $0<\dt_0={\ep_1\Delta(a/2)\Delta_1(b/2)\Delta_2(c/2)\over{128n^2}}.$
%Choose $0<\eta<1/n.$
Let $\{e_{i,j}\}$ be a matrix unit for $D$ and let ${\cal
G}=\{e_{i,j}\}.$ Define
\beq\label{pmatr-2}
 v=\sum_{j=1}^n
e^{2\sqrt{-1}j\pi/n}e_{j,j}.
\eneq

Let $g_j\in C(\T)$ with $g_j(t)=1$ for
$|t-e^{2\sqrt{-1}j\pi/n}|<\pi/n$ and $g_j(t)=0$ if
$|t-e^{2\sqrt{-1}j\pi/n}|\ge \pi/n+\pi/n2^{n+2}$ and $1\ge g_j(t)\ge 0,$
$j=1,2,...,n.$   As in the proof Lemma 5.1 of \cite{Lnhomp}, we may also assume
that
\beq\label{matr-3}
g_i(e^{2\sqrt{-1}j\pi/n}t)=g_{i+j}(t)\tforal t\in \T
\eneq
where $i, j\in \Z/n\Z.$
Let $1=c_0>c_1>c_2>\cdots >c_{m_1}=c$ so that $c_j-c_{j+1}<c/n^22^{n+1},$ $j=0,1,...,m_1-1.$

Define $g_{i,j,c}\in  C(\T)$ with $0\le g_{i,j,c}\le 1,$ $g_{i,j,c}(t)=1$ for
$|t-e^{2\sqrt{-1}j\pi/n^32^{n+1}}|<c_i$ and $g_{i,j,c}(t)=0$ if
$|t-e^{2\sqrt{-1}j\pi/n^32^{n+1}}|\ge c_i+c/n^22^{n+2},$ $i=1,2..., m_1,$
$j=1,2,...,n^32^{n+1}.$  We may also assume that
\beq\label{matr-3+}
g_{i,j,c}(e^{2\sqrt{-1}k\pi/n} t)=g_{i,j+k',c}(t)\tforal t\in \T,
\eneq
where  $k'=kn^22^{n+1}$ and $j,k,k'\in \Z/n^32^{n+1}\Z.$

Let $1=b_0>b_1>\cdots >b_{m_0}=b$ such that $b_j-b_{j+1}<b/2^{n+1},$ $j=0,1,...,m_0-1.$
Let $\{x_1, x_2,...,x_N\}$ be an $a/2^{n+2}$-dense subset of $X.$ Define
$f_{i,m}\in C(X)$ with $f_{i,m}(x)=1$ for $x\in B(x_i, r_m)$ and $f_{i,m}(x)=0$
if $x\not\in B(x_i, r_m+a/2^{n+2})$ and $0\le f_{i,m}\le 1,$ $i=1,2,...,N$ and $m=1,2,..., l+1.$
Define $f_{i,j,b}\in C(X)$ with $0\le f_{i,j,b}\le 1,$ $f_{i,j,b}(x)=1$ for $x\in B(x_i, b_j)$ and
$f_{i,j,b}(x)=0$ if $x\not\in  B(x_i, b_j+b/2^{n+1}),$ $i=1,2,...,N,$ $j=1,2,...,m_0.$
% Suppose that $\{x_1, x_2,...,x_m\}$ is a
%$\ep_1/64$-dense subset of $X.$ Suppose that  $g_k\in C(X)$ such
%that $0\le g_k\le 1,$ $g_k(x)=1$ if ${\rm dist}(x, x_k)<1/2n$ and
%$g_k(x)=0$ if ${\rm dist}(x,x_k)\ge 1/n.$
Note that
\beq\label{matr-4-1}
\tau(\phi(f_{i,m}))\ge \Delta(r_m)\tforal \tau\in \text{T}(A),\,\,\,i=1,2,...,N, j=1,2,...,s+1.
\eneq
Fix a finite subset ${\cal F}_0\subset C(\T)$ which at least
contains
$$
\{g_1,g_2,..., g_n\}\cup\{ g_{i,j,c} : 1\le i\le m_1, 1\le j\le n  \}
$$
and ${\cal F}_1\subset C(X)$ which
at least contains ${\cal F},$  $\{f_{i,m}: 1\le i\le N, 1\le m\le s+1\}$ and
$\{f_{i,j,b}: 1\le i\le N, 1\le j\le m_0\}.$

Choose  $\dt$ so small in (\ref{matr-0}) that the following hold:
\begin{enumerate}
\item there exists a unitary $u_i\in e_{i,i}Ae_{i,i}$ such that
$\|e^{2\sqrt{-1}i\pi/n}e_{i,i}ue_{i,i}-u_i\|< \dt_0^2/16n^42^{n+6}$,
$i=1,2,...,n;$
\item  $\|e_{i,j}g(u)-g(u)e_{i,j}\|<\dt_0^2/16n^42^{n+6},$
$\|e_{i,j}\phi(f)-\phi(f)e_{i,j}\|<\dt_0^2/16n^42^{n+6}$ for $f\in {\cal
F}_1$ and $g\in {\cal F}_0,$ $i,\,j,\,k=1,2,...,n;$

\item
$\|e_{i,i}g(vu)-e_{i,i}g(e^{2\sqrt{-1}i\pi/n}u)\|<\dt_0^2/16n^42^{n+6}$
for all $g\in {\cal F}_0$ and

\item $\|e_{i,j}^*g(u)e_{i,j}-e_{j,j}g(u)e_{j,j}\|<\dt_0^2/16n^42^{n+6},$
$\|e_{i,j}^*\phi(f)e_{i,j}-e_{j,j}\phi(f)e_{j,j}\|<\dt_0^2/16n^42^{n+6}$
for all $f\in {\cal F}_1$ and $g\in {\cal F}_0,$  $i,\,j=1,2,...,n.$

\end{enumerate}
%\item there is a unitary $W_i\in A$ such that
%$W_i^*u_iW=u_1,$ $i=1,2,...,n.$
It follows from (4) that, for any $k_0\in \{1,2,...,N\}$ and $m'\in\{1,2,...,l+1\},$
\beq\label{matr-4n1}
\tau(\phi(f_{k_0,m'})e_{j,j})\ge \Delta(r_{m'})/n-n\dt_0^2/16n^42^{n+6}.
\eneq
Fix $k_0, m'$ and $k.$
 For each $\tau\in \text{T}(A),$ there is at least one $i$ such
that
\beq\label{matr-4n2}
\tau(\phi(f_{k_0,m'})e_{j,j}g_i(u))\ge \Delta(r_{m'})/n^2-\dt_0^2/16n^42^{n+6}.
\eneq
Choose $j$ so that $k+j=i\hspace{-0.05in}\mod (n).$ Then, by (\ref{matr-3}),
\beq\label{matr-4}
\tau(\phi(f_{k_0,m'})g_k(vu))&\ge &
\tau(\phi(f_{k_0,m'})e_{j,j}g_k(e^{2\sqrt{-1}j\pi/n}u))-{8\dt_0^2\over{16n^42^{n+6}}}\\
&= &\tau(\phi(f_{k_0,m'})e_{j,j}g_{i}(u))-{\dt_0^2\over{2n^42^{n+6}}}
\\
&\ge & {\Delta(r_{m'})\over{n^2}}-{9\dt_0^2\over{16n^42^{n+6}}}\tforal \tau\in
\text{T}(A).
\eneq
%For any $0<\ep_1<\ep,$  with sufficiently small $\dt_1$ and
%sufficiently large ${\cal G}_1,$ by ? of \cite{Lnhmp}, there
%exists a unital \morp\, $L: C(X\times \T)\to A$ such that
%\beq\label{matr-5-1}
%\|L(f\otimes g)-h(f)g(vu)\|<\ep_1<\ep
%\eneq
%for all $f\in \{1_{C(X)}\}\cup {\cal F}\cup\{g_s: s=1,2,...,m\}$
%and $g\in \{z, 1_{C(\T)}, f_k: k=1,2,...,n\}.$
%So if $\ep_1$ (and $\dt<\sigma$ ) is  sufficiently small one has
%It follows that
%\beq\label{matr-5}
%\tau(f_i(w)f_j(u))
%\mu_{\tau\circ L}(B(x_{k_0}, 1/n)\times B(e^{2\sqrt{-1} k\pi/n},
%\pi/n)) \ge {\sigma\over{2n^2}}
%\eneq
%for $k_0=1,2,...,m$ and $k=1,2,...,n.$
Note again $\tau(xy)=\tau(yx)$ for all $x, y\in A.$
It is then  easy to compute that
\beq\label{matr-6}
%\mu_{\tau\circ L}(O_a)
\tau(\phi(f)^{1/2}g(vu)\phi(f)^{1/2})\ge {\Delta((1-1/2^{n+1})r) \over{n^2}}-{9\dt_0^2\over{16n^42^{n+5}}}\rforal \tau\in
\text{T}(A)
\eneq
and for any  pair of $f\in C(X)$ with $0\le f\le 1$
such that $\{x\in X:  f(x)=1\}$
contains an open ball with radius $r\ge a+a/2^{n+1}$ and $g\in C(\T)$ with $0\le
g\le 1$ such that $\{t\in \T: g(t)=1\}$ contains open arc of length at least
$4\pi/n+\pi/n^22^{n+1}.$
One then concludes that
%\beq\label{matr-6}
%\mu_{\tau\circ L}(O_a)
%\tau(\phi(f)g(vu))\ge {j\Delta(r/2) \over{2n^2}}-{2j\dt_0^2\over{16n^6}}\rforal \tau\in
%\text{T}(A)
%\eneq
%and for any  pair of $f\in C(X)$ with $0\le f\le 1$
%such that $\{x\in X:  f(x)=1\}$
%contains an open ball with radius $r\ge 2b$ and $g\in C(\T)$ with $0\le
%g\le 1$ such that $\{t\in \T: g(t)=1\}$ contains open arc of length at least
%$4j\pi/n.$
\beq\label{nmatr-6}
%\mu_{\tau\circ L}(O_a)
\tau(\phi(f)g(vu))
\ge (1-1/2^{n+2}){\Delta((1-1/2^{n+1})r)\over{n^2}}\rforal \tau\in
\text{T}(A)
\eneq
and for any  pair of $f\in C(X)$ with $0\le f\le 1$
such that $\{x\in X:  f(x)=1\}$
contains an open ball with radius $r\ge (1+1/2^{n+1})a$ and $g\in C(\T)$ with $0\le
g\le 1$ such that $\{t\in \T: g(t)=1\}$ contains open arc of length at least
$s\ge 4\pi/n+\pi/n^22^{n+1}.$

On the other hand, by (2), (3) and (4) above,
\beq\label{matr-7}
&&\tau(\phi(f_{i,j,b})^{1/2} e_{k,k}g_{i',j',c}(vu)\phi(f_{i,j,b})^{1/2})\\
&&\hspace{0.4in}\ge
\tau(\phi(f_{i,j,b})^{1/2} e_{k,k}g_{i',j',c}(e^{2\pi\sqrt{-1} k/n}u)\phi(f_{i,j,b})^{1/2})-{\dt_0^2\over{16n^42^{n+6}}}\\
&&\hspace{0.4in}=\tau(\phi(f_{i,j,b})^{1/2} e_{k,k}g_{i',j'+k',c}(u)\phi(f_{i,j,b})^{1/2})-{\dt_0^2\over{16n^42^{n+6}}}\\
&&{\rm (} k'=kn^22^{n+1} {\rm )}\\
&&\hspace{0.4in}\ge
{\Delta_1(b_j)\Delta_2(c_{i'})\over{n}}-4n{\dt_0^2\over{16n^42^{n+6}}}\\
\eneq
for all $\tau\in T(A),$ $k=1,2,...,n,$ $i=1,2,..., N,$ $j=1,2,...,m_0,$
$i'=1,2,...,m_1$ and $j'=1,...,n^22^{n+1}.$
Thus
\beq\label{matr-8}
\tau(\phi(f_{i,j,b})^{1/2}g_{i',j',a}(vu)\phi(f_{i,j,b})^{1/2})\\
\ge
\Delta_1(b_j)\Delta_2(a_{j'})-4n^2{\dt_0^2\over{16n^42^{n+6}}}\\
=\Delta_1(b_j)\Delta_2(c_{i'})-{\dt_0^2\over{8n^22^{n+5}}}
\eneq
for all $\tau\in T(A).$ It then follows
\beq\label{matr-9}
\tau(\phi(f)^{1/2}g(vu)\phi(f)^{1/2})
\ge \Delta_1(b_i)\Delta_2(c_{j})-{\dt_0^2\over{n^22^{n+5}}}
\eneq
for all $\tau\in T(A),$ for any $f\in C(X)$ with $0\le f\le 1$ so that
$\{x\in X: f(x)=1\}$ contains an open ball with radius $r\ge (1+1/2^{n+1})b_i$ and
for any $g\in C(\T)$ with $0\le g\le 1$ so that
$\{t\in \T: g(t)=1\}$ contains an open arc with length $s\ge (1+1/n^22^{n+1})c_j,$
$i=1,2,...,m_0$ and $j=1,2,...,m_1.$
From this, one concludes that
\beq\label{matr-9+1}
\tau(\phi(f)^{1/2}g(vu)\phi(f)^{1/2})
\ge \Delta_1((1-1/2^{n+1})r)\Delta_2((1-1/2^{n+1})s)-{\dt_0^2\over{n^22^{n+5}}}
\eneq
for all $\tau\in T(A),$ for any $f\in C(X)$ with $0\le f\le 1$ so that
$\{x\in X: f(x)=1\}$ contains an open ball with radius $r\ge (1+1/2^{n+1})b$ and
for any $g\in C(\T)$ with $0\le g\le 1$ so that
$\{t\in \T: g(t)=1\}$ contains an open arc with length $s\ge (1+1/n^22^{n+1})c.$

Note that if $\|[\phi(f),\, e_{i,i}]\|<\dt,$ then
$$
\|[\phi(f), \sum_{i=1}^n\lambda_ie_{i,i}]\|<n\dt<\ep
$$
for any $\lambda_i\in \T$ and $f\in {\cal F}_1.$  We then also
require that $\dt<\ep/2n.$ Thus, one obtains a continuous path
$\{v(t): t\in [0,1]\}\subset D$ with ${\rm length}(\{v(t)\})\le \pi$
and with $v(0)=1$ and $v(1)=v$ so that the second part of
(\ref{matr-2}) holds.

\end{proof}

\begin{NN}\label{matricx}
{\rm
Let $X$ be a metric space with metric $d_0.$ Define a metric
$d$ on $X\times \T$ as follows:
$$
d((x,t), (y,r))=\sqrt{d_0(x, y)^2+|t-r|^2}
$$
for all $x, y\in X$ and $t,r\in \T.$
}
\end{NN}
\begin{NN}\label{Delta0}
{\rm
Define $\Delta_{00}: (0,1)\to (0,1)$ as follows:
\beq
\hspace{-0.4in}\Delta_{00}(r)&=&{1\over{2n^2}},\,\, \, \text{if}\,\,\, r\in [4\pi/n+\pi/2n^2, 4\pi/(n-1)+\pi/2(n-1)^2)\andeqn n\ge 64;\\
\hspace{-0.2in}\Delta_{00}(r)&=&{1\over{2(64)^2}},\,\,\,\text{if}\,\,\, r\ge 4\pi/63+\pi/2(63)^2.
\eneq
Define $\Delta_{00}^n: (0,1)\to (0,1)$ as follows.
\beq
\Delta_{00}^n(r)&=&{\prod_{j=1}^n(1-1/2^{j+1})\over{k^2}},\\
&& \text{if}\,\,\, r\in [4\pi/k+\sum_{j=k}^n\pi/k^22^{j+1},\,
  4\pi/(k-1)+\sum_{j=k}^n \pi/(k-1)^22^{j+1}),\\
  &&k=65,...,n;\\
\Delta_{00}^n(r)&=&{\prod_{j=1}^n(1-1/2^{j+1})\over{(64)^2}},\,\,\,\text{if}\,\,\,r\ge 4\pi/64+\pi/2;\\
\Delta_{00}^n(r)&=&r\Delta_{00}^n(4\pi/n+\pi/n^22^{n+1}),\,\,\,
\text{if}\,\,\,r\in (0, 4\pi/n+\pi/n^22^{n+1}).
\eneq

Let $\Delta: (0,1)\to (0,1)$ be an increasing map.
Define
$$
\Delta_0(\Delta)(r)={\Delta(r/2)\Delta_{00}(r/2)\over{4}}\andeqn \Delta_1(\Delta)=3\Delta_0(3r/4)/4\tforal r\in (0,1).
$$

}
\end{NN}

\begin{lem}\label{changing}
Let $X$ be a compact metric space, let $\ep>0,$ let ${\cal F}\subset C(X)$
be a finite subset and let $\Delta: (0,1)\to (0,1)$ be an increasing map. Let $\eta\in (0, 1/2).$
%There exists
%$\dt>0,$
%a finite subset ${\cal G}\subset C(X)$ and an integer $K\ge 1$ satisfying the following:
Suppose that $\phi: C(X)\to A$ is a unital \morp\, for some unital \CA\, $A$ with $T(A)\not=\emptyset$ and $u\in U(A)$ is a unitary such that
\beq\label{chang-1}
%\|[\phi(g),\, u]\|<\dt\andeqn
\tau(\phi(g))\ge \Delta(r)
\eneq
for all $g\in C(X)$ with $0\le g\le 1$ so that
$\{x\in X: g(x)=1\}$ contains an open ball with radius $r\ge \eta/2.$
Then there are a unitary $v\in M_K\subset M_K(A)$ and a continuous path of unitaries
$\{v_t : t\in [0,1]\}\subset M_K$ such that $v_0=1,$ $v_1=v$ and
\beq\label{chang-2}
\tau(\phi^{(K)}(f)g(vu^{(K)}))\ge \Delta(r/2)\Delta_{00}(s/2)/4\tforal \tau\in T(M_K(A)),
\eneq
for all $f\in C(X)$ with $0\le f\le 1$ so that $\{x\in  X: f(x)=1\}$ contains an open ball
of radius $r\ge 4\eta/3$ and for all $g\in C(\T)$ with
$0\le g\le 1$ so that $\{t\in \T: g(x)=1\}$ contains an open arc with length $s\ge 5\eta/2,$
where $\Delta_{00}$ is defined in \ref{Delta0}.
\end{lem}

\begin{proof}

Let  $\ep,$ ${\cal F},$ $\Delta$ and $\eta$ be as given.
Choose an integer $K_1\ge 1$ such that
$1/K_1<\eta/16.$  Let $K=K_1!/16!.$  We will use induction to prove the following:

%There exists
%$\dt>0,$
% a finite subset ${\cal G}\subset C(X)$
%satisfying the following:
Suppose that $\phi: C(X)\to A$ is a unital \morp\, for some unital \CA\, $A$ with $T(A)\not=\emptyset$ and $u\in U(A)$ is a unitary such that
\beq\label{nchang-2}
%\|[\phi(g),\, u]\|<\dt\andeqn
\tau(\phi(g))\ge \Delta(r)
\eneq
for all $g\in C(X)$ with $0\le g\le 1$ so that
$\{x\in X: g(x)=1\}$ contains an open ball with radius $r\ge \eta,$
then there is a unitary $v\in M_{n!/32!}\subset M_{n!/32!}(A)$ and a continuous path of unitaries
$\{v_t : t\in [0,1]\}\subset M_{n!/32!}$ such that $v_0=1,$ $v_1=v$ and
\beq\label{nnchang-2}
\tau(\phi^{(n!/32!)}(f)g(vu^{(n!/32!)}))\ge \prod_{k=32}^n(1-1/2^{k+2})\Delta((\prod_{k=32}^n(1-1/2^{k+1})r))\Delta_{00}^n(r)
\eneq
for all $ \tau\in T(M_{n!/32!}(A)),$
for all $f\in C(X)$ with $0\le f\le 1$ so that $\{x\in  X: f(x)=1\}$ contains an open ball
of radius $r\ge \prod_{k=32}^n(1+1/2^{k+2})\eta$ and for all $g\in C(\T)$ with
$0\le g\le 1$ so that $\{t\in \T: g(x)=1\}$ contains an open arc with length $s\ge 4\pi/(n-1)+\sum_{k=32!}^n\pi/k^22^{k+1},$
where $\Delta_0$ is defined in \ref{Delta0}.
We use induction on $n.$

Let $n=64.$ Consider $\phi^{(64)}$ and $u^{(64)}$ and $D=M_{64}\subset M_{64}(A).$
Note that $x^{(64)}d=dx^{(64)}$ for all $x\in A$ and $d\in D.$  We choose $a=\eta/2$ and
ignore $b,\, c$  and the last part of statement of \ref{oldmatrix} after ``Moreover."
We also use $\dt=0.$   Lemma \ref{matrix} implies that the above statement
holds for $n=64.$   Denote by $v_{64}\in M_{64}$ the unitary provided by \ref{oldmatrix} ($n=64$).

We now assume that the above statement holds for some $n\ge 64.$
Denote by $v_n$ the unitary $v\in M_{n!/32!}$ provided by the above statement for $n.$
Let $D=M_{n+1}.$ We write $M_{(n+1)!/32!}=M_{n!/32!}\otimes D$ and  consider  $A\otimes M_{n!/64}\otimes D$
instead of $A.$
Put $\Delta_1=(\prod_{k=64}^n(1-1/2^{k+2}))\Delta(\prod_{k=64}^n (1-1/2^{k+1}))r)$ and
$\Delta_2(s)=\Delta_{00}^n(s).$  Choose $a=\eta/2,$ $b=\eta/2$ and $c=2\pi/n.$
Consider $\phi^{(n!/32!)}$ and
$$
U_n=v_nv_{n-1}^{(n)}\cdots v_{64}^{(n!/64!)} u^{(n!/32!)}
$$
(in place of $u$).
 We then apply Lemma \ref{oldmatrix} again.
 It follows that there is a unitary $v_{n+1}\in D=M_{n+1}$ and a continuous path of unitaries $\{v_{n+1}(t): t\in [0,1]\}\subset M_{n+1}$ such that $v_{n+1}(0)=v_{n+1},$ $v_{n+1}(1)=1,$
 \beq\label{change-4}
 \tau(\phi^{(n+1)!/32!)}(f)g(v_{n+1}U_n^{(n+1)}))\ge (1-2^{n+3})\Delta((1-1/2^{n+2})r)/(n+1)^2
 \eneq
 for all $\tau\in T(M_{(n+1)!/32!}(A)),$
 for all $f\in C(X)$ with $0\le f\le 1$ so that $\{x\in X; f(x)=1\}$ contains an open ball of radius $r\ge (1-2^{n+1})\eta/2$
 and for all $g\in C(\T)$ with $0\le g\le 1$ so that $\{t\in \T: g(x)=1\}$ contains an open arc
 of length $s\ge 4\pi/(n+1)+\pi/(n+1)2^{n+1},$ and
  \beq\label{change-5}
\hspace{-0.6in}&& \tau(\phi^{(n+1)!/32!)}(f)g(v_{n+1}U_n^{(n+1)}))\\
 &&\ge \Delta_1((1-2^{n+2})r)\Delta_2( (1-1/2^{n+2})s)-{\Delta_1(\eta/2)\Delta_{00}^n
 (\pi/n))\over{2^{n+6}}}\\
 &&\ge \prod_{k=64}^{n+1}(1-1/2^{k+1})\Delta(\prod_{k=64}^n(1-1/2^{k+1})r)\Delta_{00}^{n+1}(s)
\eneq
 for all $ \tau\in T(M_{(n+1)!/32!}(A)),$
 for all $f\in C(X)$ with $0\le f\le 1$ so that $\{x\in X; f(x)=1\}$ contains an open ball of radius $r\ge (1-2^{n+1})\eta/2$
 and for all $g\in C(\T)$ with $0\le g\le 1$ so that $\{t\in \T: g(x)=1\}$ contains an open arc
 of length $s\ge 4\pi/n+\pi/(n+1)2^{n+1}.$

 This proves the above statement
for $n+1$ and ends the induction.  Then lemma follows.

\end{proof}

\begin{lem}\label{ChT}
Let $X$ be a compact metric space, let $\ep>0,$ let ${\cal F}\subset C(X)$ be a finite subset and let $\Delta: (0,1)\to (0,1)$ be  an increasing map. Let $\eta\in (0, 1/2).$
There exists
$\dt>0,$
a finite subset ${\cal G}\subset C(X)$ and an integer $K\ge 1$ satisfying the following:
Suppose that $\phi: C(X)\to A$ is a unital \morp\, for some unital \CA\, $A$ with $T(A)\not=\emptyset$ and $u\in U(A)$ is a unitary such that
\beq\label{ChT-1}
\|[\phi(g),\, u]\|<\dt\tforal g\in {\cal G}\andeqn
\tau(\phi(g))\ge \Delta(r)
\eneq
for all $g\in C(X)$ with $0\le g\le 1$ so that
$\{x\in X: g(x)=1\}$ contains an open ball with radius $r\ge \eta/2,$
then there is a  continuous path of unitaries
$\{u_t : t\in [0,1]\}\subset M_K(A)$ such that $u_0=u^{(K)}$  and $u_1=U$ for some unitary $U\in M_K(A),$
and
\beq\label{ChT-2}
\|[\phi^{(K)}(f),\, u_t]\|<\ep\tforal f\in {\cal F},
\eneq
and  there exists a unital \morp\, $\Phi: C(X)\times C(\T)\to M_K(A)$ such that
\beq\label{ChT-3}
&&\|\Phi(f\otimes 1)-\phi^{(K)}(f)\|<\ep\tforal f\in {\cal F},\,\,\, \|\Phi(1\otimes z)-U\|<\ep\andeqn\\
&&\mu_{\tau\circ L}(O_r)\ge \Delta_1(\Delta)(r)
\eneq
for all $ \tau\in T(M_K(A))$ and for  open balls of $X\otimes \T$
of radius $r\ge 5/2\eta.$
\end{lem}

\begin{proof}
Let $\ep>0,$ ${\cal F}\subset C(X)$ be a finite subset,
$\Delta: (0,1)\to (0,1)$ be an increasing map and let
$\eta\in (0, 1/2).$  To simplify the notation, without loss of generality, we may assume
that ${\cal F}$ is a subset of the unit ball of $C(X).$

Let $0<\dt_0<\min\{\ep/2, \Delta_0(\Delta)(\eta/8)/16\}.$
Let ${\cal F}_1$ be a finite subset.  There exists $\ep/2>\dt>0$ and a finite subset ${\cal G}\subset C(X)$
containing ${\cal F}$ satisfying the following:
For any unital \morp\, $\psi: C(X)\to C$ (for any unital \CA) and any unitary $W\in C$ with
$$
\|[\psi(g), W]\|<\dt\tforal g\in {\cal G}_1,
$$
there exists a unital \morp\, $L: C(X)\otimes C(\T)\to C$ such that
$$
\|L(f\otimes 1)-\psi(f)\|<\dt_0/2 \tforal f\in {\cal F}_1\andeqn \|L(1\otimes z)-W\|<\dt_0/2.
$$
Let $K_1\ge 1$ be an integer such that $1/K_1<\eta/16.$
Let $K=K_1!/32!.$
Suppose that $\phi$ and $u$ satisfy the assumption of the lemma for the above $\dt$ and ${\cal G}.$
By applying \ref{changing}, we obtain a unitary $v\in M_K\subset M_K(A)$ and a continuous path of unitaries
$\{u_t: t\in [0,1]\}\subset M_K$ such that $u_0=1,$ \,\, $u_1=v$ and
\beq\label{ChT-4}
\tau(\phi^{(K)}(f)g(vu^{(K)}))\ge \Delta(r/2)\Delta_{00}(s/2)/4\tforal \tau\in T(M_K(A)),
\eneq
for all $f\in C(X)$ with $0\le f\le 1$ so that $\{x\in X: f(x)=1\}$ contains an open ball of radius $r\ge 4\eta/3$ and
for all $g\in C(\T)$ with $0\le g\le 1$ so that $\{t\in \T: g(t)=1\}$ contains an open arc with length $s\ge 5\eta/2.$

Note that
\beq\label{ChT-5}
\|[\phi^{(K)}(g),\, vu^{(K)}]\|<\dt<\ep\tforal g\in {\cal G}.
\eneq
It follows that there exists a unital \morp\, $\Phi: C(X)\otimes C(\T)\to M_K(A)$ such that
\beq\label{ChT-6}
\|\Phi(f\otimes 1)-\phi^{(K)}(f)\|&<&\dt_0/2<\ep\tforal f\in {\cal F}_1\andeqn\\
\|\Phi(1\otimes z)-vu^{(K)}\|&<&\dt_0/2<\ep.
\eneq
Define $U=vu^{(K)}.$ With sufficiently large ${\cal F}_1$ which can be determined before choosing
$\dt$ and ${\cal G},$ we have
\beq\label{ChT-7}
\mu_{\tau\circ L}(O_r)\ge \Delta_1(\Delta)(r)\tforal \tau\in T(M_K(A))
\eneq
and for all open ball $O_r$ with radius $r\ge 2\eta.$
\end{proof}

{\bf The proof of \ref{Sd1hom} under the assumption that \ref{Sd1uni} (and
\ref{Sd1uniR}) holds for finite CW complexes $Y$ with
${\rm dim}Y\le d$}

{\rm
Let $\ep>0,$ ${\cal F}\subset C(X)$ and $\Delta$ be given as in \ref{Sd1hom}.
Define $\Delta': (0,1)\to (0,1)$ by
$\Delta'(r)=\Delta(15r/16).$
Let $\Delta_1(\Delta')$ be  as defined in \ref{Delta0}.
Let $\eta,$ $\dt_1$ (in place of $\dt$), ${\cal G}_1$ (in place of ${\cal G}$)  ${\cal P},$ ${\cal Q},$ $N\ge 1$ and $K_1$ (in place of $K$) be as required by \ref{Sd1homF}
for $\ep,$ ${\cal F}$ and $\Delta_1(\Delta').$

Let $\dt_2>0$ (in place of $\dt$), let ${\cal G}_2\subset C(X)$ (in place of ${\cal G}$) be a finite subset and let
$K_2\ge 1$ (in place of $K$) be an integer required by \ref{ChT} for
$\min\{\ep/2, \dt_1\}$ (in place of $\ep$), ${\cal G}_1\cup {\cal F}$ (in place of ${\cal F}$), $\Delta_1(\Delta')$
(in place of $\Delta$) and $\eta/16$ (in place of $\eta$).

Let $\dt=\min\{\dt_2, \dt_1/2, \ep/4\}$ and let
${\cal G}={\cal F}\cup {\cal G}_1\cup {\cal G}_2.$ Put $K=K_1K_2!.$

Now let $Y$ be a finite CW complex with ${\rm dim}Y\le d$ and $C=PM_m(C(Y))P$ for some projection
$P\in M_m(C(Y))$ with ${\rm rank} P(y)\ge N\ge 1,$
let $\phi: C(X)\to C$ be a unital $\dt$-${\cal G}$-multiplicative \morp\, and let $u\in U(C)$ be a
unitary which satisfies the assumption of \ref{Sd1hom} for the above $\eta,$ $\gamma,$ $\dt,$
${\cal G},$ ${\cal P}$ and  ${\cal Q}.$
%so that 1----2 hold.
It follows from \ref{ChT} that there is a continuous path of unitaries
$\{w_t:t\in [0,1/2]\}\subset M_{K_2!}(C)$ such that $w_0=u^{(K_2!)},$
\beq\label{Phom-1}
\|[\phi^{(K_2!)}(g),\, w_t]\|<\min\{\ep/2,\dt_1\}\tforal t\in [0,1/2],\,\,\,g\in {\cal G}_1\cup {\cal F}
\eneq
and there exists a unital \morp\, $\Phi: C(X)\otimes C(\T)\to M_{K_2!}(C)$ such that
\beq\label{Phom-2}
\|\Phi(f\otimes 1)-\phi^{(K_2!)}(f)\|<\min\{\ep/2, \dt_1\}\tforal f\in {\cal G}_1\cup {\cal F},\\
\|\Phi(1\otimes z)-w_{1/2}\|<\min\{\ep/2,\dt_1\}\andeqn\\
\mu_{\tau\circ \Phi}(O_r)\ge \Delta_1(\Delta')(r)\tforal \tau\in M_{K_2!}(C)
\eneq
and for all open balls $O_r$ with radius $r\ge 5\eta/32.$

Then, by \ref{Sd1homF}, there is a continuous path of unitaries $\{v_t: t\in [1/2,1]\}\subset M_{K_1K_2!}(C)$ with $v_{1/2}=w_1$ and $v_1=1$ such that
\beq\label{Phom-3}
\|[\phi^{(K)}(f),\, v_t]\|<\ep\tforal t\in [1/2, 1],\,\,\,\andeqn \tforal f\in {\cal F}.
\eneq
Now define
$$
u_t=\begin{cases} w_t^{(K_1)}\,\,\,\text{if}\,\,\, t\in [0,1/2];\\
               v_t,\,\,\,\text{if}\,\,\, t\in (1/2,1].
\end{cases}
$$
Note that $u_0=u^{(K)}={\rm diag}(\overbrace{u,u,...,u}^K).$
This path meets the requirements.

}

\section{The proof of the uniqueness theorem \ref{Sd1uni}}

{\bf Proof of \ref{Sd1uni} }

{\rm
The case that $Y$ is a single point is well known. A reference can be found in Theorem 2.10 of \cite{Lnnewapp}.
 The case that $Y$ is a set of finitely many points follows. The case that $Y=[0,1]$ has been proved in Theorem 3.6 of \cite{Lnnewapp}.

We now assume that \ref{Sd1uni} holds for the case that $Y$ is  any  finite CW complex
of dimension no more than $d\ge 0.$ We will use it to show that \ref{Sd1uni} holds for the case
that $Y$ is any  finite CW complex of dimension no more than  $d+1.$  Then induction
implies that \ref{Sd1uni} holds for  any integer $d\ge 0.$ Note that now we can apply \ref{Sd1hom} and \ref{Sd1ext} for
$Y$ being a finite CW complex with ${\rm dim}Y\le d.$

Let $\ep>0$ and let ${\cal F}\subset C(X)$ be a finite subset. To simplify notation, without loss of generality,
we may assume that ${\cal F}$ is in the unit ball of $C(X).$
Let $\eta_1'>0$ (in place of $\eta$),  $\dt_1>0$ (in place of $\dt$), $\gamma_0>0$ (in place of $\gamma$), ${\cal G}_1\subset C(X)$
(in place of ${\cal G}$) be a finite subset,  let ${\cal P}_0\subset \underline{K}(C(X))$ (in place of ${\cal P}$) be a finite subset and ${\cal Q}\subset {\rm ker}_{\rho_{C(X)}}$ be a finite subset, let
 $N_1$ (in place of $N$) and an integer $K_1$ (in place $K$) required by \ref{Sd1hom} for $\ep/32$ (in place of $\ep$), ${\cal F}$  and $\Delta(r/3)/3$ (in place of
$\Delta$).
We may assume that ${\cal Q}\subset K_0(C(X))\cap {\cal P}_0.$
Let $\eta_1=\eta_1'/3.$
We may assume that $\dt_1<\ep/32K_1.$

Write $C(X)=\lim_{n\to\infty}(C(X_n), \imath_n),$ where
each $X_n$ is a finite CW complex and $\imath_n$ is a unital \hm.
%We may assume  that
Let $K_2\ge 1$ be an integer such that $K_2x=0$ for any $x\in {\rm Tor}(K_i(C(X)))\cap {\cal P}_0$ ($i=0,1$).
Let $K_3=K_1\cdot K_2!.$
We may also assume that, for any $\dt_2$-$\{z,1\}\times {\cal G}_2$-multiplicative
\morp\, $\Lambda: C(\T)\otimes C(X)\to C$ (for any unital \CA\, $C$
with $T(C)\not=\emptyset$), $[\Lambda]$ is well defined and
$$
\tau([\Lambda(\boldsymbol{\bf}(g))])=0
$$
for all $g\in Tor(K_1(C(X)))\cap {\cal P}_0.$
Furthermore, we may assume that $\dt_2$ is so small and ${\cal G}_2$ is so large   that
%if
%$\|uv-vu\|<3\dt_2,$ then the Exel formula
%$$
%\tau({\rm bott}_1(u,v))={1\over{2\pi\sqrt{-1}}}(\tau(\log(u^*vuv^*))
%$$
%holds in any unital \CA\, $C$ with tracial rank zero
%and any $\tau\in T(C)$ (see Theorem 3.6 of \cite{Lnamj}). Moreover
${\rm Bott}(\psi', v)|_{{\cal P}_0}$ is well defined for any
unital \hm\, $\psi'$ from $C(X)$ and any unitary $v$ in the target algebra
such that $\|[\phi'(g),\, v]\|<3\dt_2$ for all $g\in {\cal G}_2.$ Moreover
if $\|v_1-v_2\|<3\dt_2,$ then
$$
{\rm Bott}(\phi',v_1)|_{{\cal P}_0}={\rm Bott}_1(\phi',v_2)|_{{\cal P}_0}.
$$
We also assume that, if there are unitaries $u_1, u_2, v_1, v_2, v_3, v_4$ and projections $e_1, e_2$ such that
\beq\label{ADDD-1}
\|u_1-u_2\|<3\dt_2,\,\,\, \|e_1-e_2\|<3\dt_2\,\,\,\|[e_i, v_j]\|<3\dt_2 \andeqn \|[u_i, v_j]\|<3\dt_2,
\eneq
$ i=1,2$ and $j=1,2,3,4,$
then
\beq\label{ADDD-2}
{\rm bott}_0(e_1, v_j)={\rm bott}_0(e_2, v_j),\,\,\, {\rm bott}_1(u_1,v_j)={\rm bott}_1(u_2, v_j),\\
{\rm bott}_1(e_1, v_1v_2v_3v_4)=\sum_{j=1}^4{\rm bott}_0(e_1, v_j)\andeqn\\\label{ADDD-3}
{\rm bott}_1(u_1, v_1v_2v_3v_4)=\sum_{j=1}^4{\rm bott}_1(u_1, v_j).
\eneq
We assume that $m(X)\ge 1$ is an integer and
$g_i\in U(M_{m(X)}(C(X)))$ so that $\{[g_1], [g_2],...,[g_{k(X)}]\}$ forms  a set of generators for $K_1(C(X))\cap {\cal P}_0.$ We also assume that
$[g_j]\not=0,$ $j=1,2,...,k(X).$ Let ${\cal U}=\{g_i, g_2,...,g_{k(X)}\}.$
We may also assume that $m(X)\ge R(G({\cal U})).$

%To simplify the notation, we may assume that $[g_1], [g_2],...,[g_{k(X)}]\subset {\cal P}.$
Let ${\cal S}_1\subset C(X)$ be a finite subset
such that
$$
{\cal U}=\{(a_{i,j})\in {\cal U}: a_{i,j}\in {\cal S}_1\}.
$$
We may assume that  ${\cal P}_1=\{p_1, p_2,...,p_{k_0}\}\subset M_{m(X)}(C(X))$ is  a finite subset of projections
such that ${\cal P}_1=K_0(C(X))\cap {\cal P}_0.$ Let ${\cal  S}_0\subset C(X)$ be a finite subset such that
$$
{\cal P}_1=\{(b_{ij}): b_{i,j}\in  {\cal S}_0\}.
$$
%We may also assume that $m(X)\ge R(G({\cal P}_1))
%We may also assume that $\{[p_j]:j=1,2,...,k_0(X)\}$ forms a set of generators for $K_0(C(X)).$
Moreover,  we may assume, without loss of generality, that
${\cal Q}\subset \{[p_i]-[p_j]: 1\le i,j\le k_0\}.$
We may also assume that $m(X)\ge R(G({\cal P}_1)).$

Let
$$\dt_u'=({1\over{K_3(d+1+m(X))^2}})\min\{1/256, \dt_1/16, \dt_2/16, \gamma_0/16\}
$$
and ${\cal G}_u'={\cal F}\cup{\cal G}_1\cup{\cal G}_2\cup {\cal S}_0\cup {\cal S}_1.$
Let $\eta_2>0$ (in place of $\eta$),
$\dt'>0$ (in place of $\dt$), ${\cal G}'\subset C(X)$ (in place of ${\cal G}$) be a finite subset, $n_0\ge 1,$ $N_2\ge 1$ (in place of $N$) be an integer, $K_4\ge 1$ (in place of $K$) be an integer given by \ref{Sd1ext} for $\dt_u'/2$ (in place of $\ep$), $\gamma_0/(d+1+m(X))$ (in place of $\dt_0$),
 ${\cal G}_u'$ (in place of ${\cal G}$), ${\cal P}_0$ (in place of ${\cal P}$)
and ${\cal Q}.$ We may assume that ${\cal P}_0\subset
[\psi_{n_0,\infty}](\underline{K}(C(X_{n_0}))).$ Furthermore,
we may assume, without loss generality, that there are unitaries $g_i'\in M_{m(X)}(C(X_{n_0}))$ such that
$\psi_{n_0,\infty}\otimes {\rm id}_{M_{m(X)}}(g_i')=g_i,$ $i=1,2,...,k(X),$
and there are projections $p_j'\in M_{m(X)}(C(X_{n_0}))$ such that
$\psi_{n_0,\infty}\otimes {\rm id}_{M_{m(X)}}(p_j')=p_j,$ $j=1,2,...,k_0.$
Without loss of generality, we may assume that $K_3|K_4.$
We may also assume that $K_4x=0$ for all $x\in {\rm Tor}(K_i(C(X_{n_0})),$
$i=0,1.$

To simplify notation, without loss of generality, by adding more projections, we may further assume that
$\{p_1',p_2',...,p_{k_0}'\}$ generates $K_0(C(X_{n_0})),$ and by adding more unitaries, we may assume that
$\{g_1',g_2',...,g_{k(X)}'\}$ generates $K_1(C(X_{n_0})).$

Let $\dt_u=\min\{\dt_u', \dt'/2\}$ and
${\cal G}_u={\cal G}_u'\cup {\cal G}'.$
Let $\dt_3>0$ (in place of $\dt$) and let ${\cal G}_3'\subset C(\T)\otimes C(\T)$ (in place of ${\cal G}$) be given by Lemma 10.3 of \cite{LnApp} for $1/4NK_4m(X)$ (in place of $\sigma$) and $\T\times \T$ (in place of $X$).
 Without loss of generality, we may assume
that
${\cal  G}_3'=\{1\otimes 1, 1\otimes z, z\otimes 1\}.$
Let
$$
{\cal G}_3=\{z\otimes g: g\in {\cal G}_u\}\cup\{1\otimes g: g\in {\cal G}_u\}.
$$
Let $\ep_1'>0$ (in place of $\dt$) and let ${\cal G}_4\subset C(X)$ (in place of ${\cal G}$) be a finite subset
given by 3.4 of \cite{Lnnewapp} for $\min\{\eta_1/2,\eta_2/2\},$ $3/4$ (in place of $\lambda_1$) and
$1/4$ (in place of $\lambda_2$).

Let $\ep_1''=\min\{1/27K_3K_4(d+1+m(X))^2, \dt_u/K_3K_4(2d+2+m(X))^2,  \dt_3/2K_3K_4(d+1+m(X))^4, \ep_1'/2K_3K_4(d+1+m(X))^2,
  \gamma_0/16K_3K_4(d+1+m(X))^2\}$ and let ${\bar \ep}_1>0$ (in place of $\dt$) and ${\cal G}_5\subset C(X)$ (in place of ${\cal F}_1$) be a finite subset
 given by 2.8 of \cite{Lnmem} for $\ep_1''$ (in place of $\ep$) and ${\cal G}_u\cup {\cal G}_4$
(and $C(X)$ in place of $B$).
Put
$$
\ep_1=\min\{\ep_1', \ep_1'',{\bar \ep_1}\}.
$$
Let $\eta_3>0$ (in place of $\eta$),
%be required by \ref{Sd1uni} for
%$\ep_1/4$ (in place $\ep$) and ${\cal G}_5$ (in place of ${\cal F}$) for ${\cal Y}$ being the class of
%finite CW complexes with dimension no more than $d.$
%Let $\eta_3''>0$ (in place of $\eta$) be required by XXXX(2.12) of \cite{ChomC1}
%for $\ep_1/4$ (in place of $\ep$) and ${\cal G}_5$
%(in place of ${\cal F}$).
%Let $\eta_3=\min\{\eta_3', \eta_3''\}.$
%Let $\sigma_3=\Delta(\eta_3'').$
 $1/4>\gamma_1>0,$  $1/4>\gamma_2'>0$ (in place of $\gamma_2$), $\dt_4>0$(in place of $\dt$), ${\cal G}_6\subset C(X)$ (in place of ${\cal G}$), ${\cal H}\subset C(X)$ be a finite subset, let ${\cal P}_2\subset \underline{K}(C(X))$ (in place of ${\cal P}$), let $N_3\ge 1$ (in place of $N$) and let $K_5\ge 1$(in place of $K$) be given by \ref{Sd1uni}  for $\ep_1/2^8(m(X)+d+1)^2$ (in place of $\ep$), ${\cal G}_u\cup {\cal G}_4\cup {\cal G}_5$ (in place of ${\cal F}$),
  $\Delta$ and for  ${\rm dim Y}\le d.$
%Let $\eta_4>0$ (in place of $\eta_2$) be required by XXXX2.11 of \cite{ChomC1} for
%$\ep_1/4$ (in place of $\ep$), ${\cal G}_5$ (in place of ${\cal F}$),
%$\eta_3$ (in place of $\eta_1$), $\sigma_3$ (in place of $\sigma_1$).
%Let $\sigma_4=\Delta(\eta_4).$ Let $\dt_5>0,$ ${\cal G}_7\subset C(X)$
%(in place of ${\cal G}$),${\cal P}_2\subset \underline{K}(C(X))$ (in place of ${\cal P}$) be required by 2.11 of
%\cite{ChomC1}.
Let $\eta=\min\{\eta_1/4, \eta_2/4, \eta_3/4\}.$
Let $\dt=\min\{\ep_1/4, \dt_u/4, \dt_3/4m(X)^2,\,\dt_4/4, \dt_5/4\},$
${\cal G}={\cal G}_u\cup {\cal G}_4\cup {\cal G}_5\cup {\cal G}_6\cup {\cal G}_7\cup {\cal H}$ and
${\cal P}={\cal P}_0\cup  {\cal P}_2.$
Let
$\gamma_2<\min\{ \gamma_2'/16(d+1+m(X))^2,\dt_u/9(d+1+m(X))^2,
1/256N_1(d+1+m(X))^2\}.$
%holds in any unital \CA\, $C$ with tracial rank zero and any $\tau\in T(C)$ (see Theorem 3.6 of \cite{Lnamj}).
We may assume that $(\dt, {\cal G}, {\cal P})$ is a $KL$-triple.
Denote $\eta=\min\{\eta_3, \eta_4\}.$
Let $N=4(k(X)+k_0(X)+1)\max\{N_1, N_2, N_3\}.$
We also assume that $\eta$ is smaller than the one required by
\ref{Uniunt} for $\ep_1/4$ and ${\cal U}.$ We also assume that
$\gamma_1,\gamma_2$ and $\dt$ are smaller, and ${\cal H},$ ${\cal G},$ ${\cal P},$ ${\cal V}$ and $N$ are larger than required by \ref{Uniunt} for $\ep_1/4$ and ${\cal U}$ as well as $\Delta.$

Now suppose that $\phi, \psi: C(X)\to C=PM_r(C(Y))P$ are unital $\dt$-${\cal G}$-multiplicative \morp s, where $Y$ is a
finite CW complex of dimension $d+1$  and ${\rm rank}P(y)\ge N$
for all $y\in Y,$
% for some integer $n$
 which satisfy the assumption for the above
$\eta,$ $\dt$, $\gamma_i$ ($i=1,2$), ${\cal P},$ ${\cal V}=K_1(C(X))\cap {\cal P}$ and ${\cal H}.$
To simplify notation, without loss of generality, we may write
$\phi$ and $\psi$ instead of $\phi^{(K_5)}$ and $\psi^{(K_5)},$ respectively,
and we also write   $C$  instead of $M_{K_5}(C).$

Let ${\bar p}_{j,(1)},\, {\bar p}_{j,(2)}\in C\otimes M_{m(X)}$ be a projection such that
\beq
\|{\bar p}_{j,(1)}-\phi\otimes {\rm id}_{M_{m(X)}}(p_j)\|<\ep_1/16\andeqn\\
\|{\bar p}_{j,(2)}-\psi\otimes {\rm id}_{M_{m(X)}}(p_j))\|<\ep_1/16
\,\,\,j=1,2,...,k_0(X).
\eneq
Let ${\bar g}_{j,(1)}, {\bar g}_{j,(2)}\in C\otimes M_{m(X)}$ be a unitary such that
\beq
\|{\bar g}_{j,(1)}-\phi\otimes {\rm id}_{M_{m(X)}}(g_j)\|<\ep_1/16\andeqn\\
\|{\bar g}_{j,(2)}-\psi\otimes {\rm id}_{M_{m(X)}}(g_j)\|<\ep_1/16,\,\,\,j=1,2,...,k(X).
\eneq
Since $\phi_{*0}([p_i])=\psi_{*0}([p_i]),$  there is a unitary $X_{0,i}\in M_{m(X)+d+1}(C)$
 such that
 \beq\label{PM-n12}
 X_{0,i}({\bar p}_{i,(1)}\oplus {\rm id}_{C}^{(d+1)})X_{0,i}^*={\bar p}_i^{(2)}\oplus {\rm id}_{C}^{(d+1)},\,\,\,i=1,2,...,k_0(X).
 \eneq
It follows from \ref{Uniunt} that there is a unitary $X_{1,j}\in M_{m(X)}(C)$ such that
\beq\label{PM-n12+}
\|X_{1,j}{\bar g}_{j,(1)}X_{1,j}^*-
{\bar g_j}^{(2)}\|<\ep_1/4,
\eneq
$j=1,2,...,k(X).$

To simplify notation, without loss of generality, we may assume that $Y$ is connected. Let $n={\rm rank}\,P.$
Let $Y^{(d)}$ be the $d$-skeleton of $Y.$ There is a
compact subset $Y_d'$ of $Y$ which contains $Y^{(d)}$ and which is a $d$-dimensional finite CW complex and satisfies
the following:
\begin{enumerate}
\item $Y\setminus Y_d'$ is a finitely many disjoint union of open $(d+1)$-cells: $D_1,D_2,...,D_R;$
    %\item The closure of the $D_j$ in $Y$ is homeomorphic to the closed $d+1$-ball
    \item $\|X_{i,j}(y)-X_{i,j}(\xi)\|<\ep_1/16,$ \,$i=1,2$ and $j=1,2,..., k_0(X),$
    \item $\|{\bar p}_{j,(i)}(y)-{\bar p}_{j,(i)}(\xi)\|<\ep_1/16,$
    $j=1,2,...,k_0(X),$ $i=1,2,$
    \item $\|{\bar g}_{j,(i)}(y)-{\bar g}_{j,(i)}(\xi)\|<\ep_1/16,$ $j=1,2,...,k_0(X),$ $i=1,2,$
    \item $\|\phi(g)(y)-\phi(g)(\xi)\|<\ep_1/16$ and
    \item $\|\psi(g)(y)-\psi(g)(\xi)\|<\ep_1/16$ for all $g\in {\cal G},$ where $\xi$ is in one of the $(d+1)$-cells and
        $y$ is in the boundary (in $Y$) of the $(d+1)$-cell.

        \end{enumerate}
Denote by $\xi_j\in D_j$ the center of $D_j,$ where we view $D_j$ as an open $(d+1)$-ball.   Let $Y_d=Y_d'\sqcup \{\xi_1,\xi_2,...,\xi_R\}.$
Let $B=\pi(C),$ where $\pi(f)=f|_{Y_d}.$  We may write
$B=P'M_r(C(Y_d))P',$ where $P'=P|_{Y_d}.$

By applying \ref{Sd1uni}  for  finite CW complex with dimension no more than $d,$  for each $i,$ there exists a unitary
$w\in B$ such that
\beq\label{PM-n1}
\|w\phi(g)w^*-\psi(g)\|_{Y_d}<\ep_1/2^8(m(X)+d+1)^2\tforal g\in {\cal G}_5.
\eneq
%and there are unital \hm s $h_1, h_2: C(X)\to B$ such that
%\beq\label{PM-n2}
%\|\phi(g)-h_{1}(g)\|_{Y_d}<\ep_1/4\andeqn
%\|\psi(g)-h_2(g)\|_{Y_d}<\ep_1/4
%\eneq
%for all $g\in {\cal G}_5$
%and for all $y\in Y_d.$
Recall,  to simplify notation,  that we write
$\phi$ and $\psi$ instead of $\phi^{(K_5)}$ and $\psi^{(K_5)},$ respectively,
and we also write   $C$  instead of $M_{K_5}(C).$
%Moreover (by also applying 3.4 of \cite{ChomC1}),
Note that we have that
\beq\label{PM-n3}
\mu_{\tau\circ \phi}(O_r)\ge \Delta(r) \rforal r\ge \eta \andeqn \rforal \tau\in T(B).
\eneq

Let $W=w\otimes {\rm id}_{M_{m(X)}}.$
Then
\beq\label{PM-n4}
\|[{\bar g}_{j,(1)}, X_{1,j}|_{Y_d}^*W]\|<\ep_1/4,\,\,\,j=1,2,...,k(X).
\eneq
So ${\rm bott}_1(\phi\otimes {\rm id}_{M_{m(X)}}(g_j), X_{1,j}|_{Y_d}^*W)$ is well defined.
Let $\af_1: K_1(C(X_{n_0}))\to K_0(B)$ be defined by
$\af_1(g_j')={\rm bott}_1(\phi\otimes {\rm id}_{M_{m(X)}}(g_j), X_{1,j}|_{Y_d}^*W),$ $j=1,2,...,k(X).$
Let $B'=B|_{Y_d'}$ and let $\pi': B\to B'$ be the quotient map induced by the restriction.
  Let $Y_i$ be the boundary of $D_i$ in $Y_d,$ $i=1,2,...,R.$
Let $B_i=C|_{Y_i}$ and $\pi_i: B\to B_i$ be the surjective map induced by the restriction, $i=1,2,...,R.$
Let $\af_{1,i}: K_1(C(X_{n_0}))\to K_0(B_i)$ by
$\af_{1,i}=(\pi_i)_{*1}\circ \af_1.$  Let $B_i'=C|_{\{\xi_i\}}$ and  $\pi_i': B\to B_i'$
 be the quotient map, $i=1,2,...,R.$
Define $\af_{1,i}'=(\pi_i')_{*1}\circ \af_1.$
Note that
\beq\label{PM-n5}
\af_{1}'(g_j')&=&{\rm bott}_1((\phi\otimes {\rm id}_{M_{m(X)}})|_{Y_d'}(g_j), X_{1,j}^*|_{Y_d'}W|_{Y_d'}),\\
\af_{1,i}(g_j') &=&{\rm bott}_1((\phi\otimes {\rm id}_{M_{m(X)}})(g_j')|_{Y_i}, X_{1,j}^*|_{Y_i}W|_{Y_i})\andeqn\\
\af_{1,i}'(g_j')&=&{\rm bott}_1((\phi\otimes {\rm id}_{M_{m(X)}})(g_j)(\xi_i), X_{1,j}^*(\xi_i)W(\xi_i))\
\eneq
$j=1,2,...,k(X)$ and $i=1,2,...,R.$
Note that,  by (\ref{PM-n4}), 10.3 of \cite{LnApp} (in the connection of 2.8 of \cite{Lnmem}) and by the choice of $\dt_3,$
$\ep_1''$  and ${\cal G}_u',$  we have
\beq\label{PM-n5+}
&&\hspace{-0.6in}|\rho_{B}(\af_1(g_j))(\tau)|<1/4N_2K_4m(X),\, \hspace{0.2in}|\rho_{B'}(\af_1'(g_j))(\tau)|<1/4N_2K_4m(X),\\\label{PM-n5++}
&&\hspace{-0.6in} |\rho_{B_i}(\af_{1,i}(g_j))(\tau)|<1/4N_2K_4m(X)\andeqn |\rho_{B_i'}(\af_{1,i}'([g_j]))(\tau)|<1/4N_2K_4m(X),
\eneq
$j=1,2,...,k(X)$ and $i=1,2,...,R,$ and for all $\tau\in T(B),$
$\tau\in T(B'),$ $\tau\in T(B_i)$ and $\tau\in T(B_i'),$ respectively.

Denote by $q_{0,j}={\bar p}_{j,(1)}\oplus {\rm id}_C^{(d+1)}$ and
$q_{0,j}'={\bar p}_{j,(2)}\oplus {\rm id}_C^{(d+1)},$  $j=1,2,...,k_0(X).$
  Let ${\bar W}=w\otimes {\rm id}_{M_{m(X)+d+1}}=w^{(d+1+m(X))}\in M_{m(X)+d+1}(B).$ Then we have
 \beq\label{PM-n13}
 \|{\bar W}q_{0,i}|_{Y_d}{\bar W}^*-q_{0,j}'|_{Y_d}\|<(m(X)+d)^2\ep_1/4<\min\{\dt_u/8, \gamma_0/64\}.
 \eneq
 There is a unitary $\Theta_i\in M_{m(X)+d}(B)$ such that
 \beq\label{PM-n14}
 \|\Theta_i-1\|<\dt_u/4\andeqn (\Theta_i{\bar W})q_{0,i}|_{Y_d}(\Theta_i{\bar W})^*=q_{0,i}'|_{Y_d},
 \eneq
 $i=1,2,...,k_0(X).$

 Define
 $$
 z_j=({\rm id}_{M_{m(X)+d+1}(B)}-q_{0,1}|_{Y_d})\oplus X_{0,j}^*|_{Y_d}(\Theta_i{\bar W})q_{0,1}|_{Y_d},\,\,\,j=1,2,...,k_0(X).
 $$
 It follows that
 \beq\label{PM-n14+}
 {\rm bott}_0(\phi,\, X_{0,j}^*|_{Y_d}{\bar W})([p_j])=[z_j],
 \eneq
 $j=1,2,...,k_0(X).$

 We obtain a \hm\, $\af_0: K_0(C(X_{n_0}))\to K_1(B)$ by
 $\af_0([p_j'])=[z_j],$ $j=1,2,...,k_0(X).$
 Let $\af_0'=(\pi')_{*0}\circ \af_0.$ Let $\af_{0,i}=(\pi_i)_{*0}\circ \af_0$ and
 $\af_{0,i}'=(\pi_i')_{*0}\circ \af_0,$  $i=1,2,...,R.$ Note that
 \beq\label{PM-n15}
 \af_0'([p_j'])=[z_j|_{Y_d'}],\,\,\,
\af_{0,i}([p_j'])=[z_j|_{Y_i}] \andeqn \af_{0,i}'([p_i'])=[z_j(\xi_i)]
\eneq
$j=1,2,...,k_0(X)$ and $i=1,2,...,R.$
Let $G_1=[\psi_{n_0, \infty}](K_0(C(X_{n_0}))).$

 Define $\Gamma: G_1\to U(M_{m(X)+d+1}(B))/CU(M_{m(X)+d+1}(B))$ by
 $\Gamma([p_j])={\bar z^*}_j,$ $j=1,2,...,k_0(X).$
 Define $\Gamma': G_1\to U(M_{m(X)+d+1}(B'))/CU(M_{m(X)+d+1}(B'))$ by
 $\Gamma'([p_j])=\overline{z_j^*|_{Y_d'}},$ $j=1,2,...,k_0(X).$
 Define $\Gamma_i: G_1\to  U(M_{m(X)+d+1}(B_i))/CU(M_{m(X)+d+1}(B_i))$ by
 $\Gamma_i=\pi_i^{\ddag}\circ \Gamma$ and
 $\Gamma_i': G_1\to  U(M_{m(X)+d+1}(B_i'))/CU(M_{m(X)+d+1}(B_i'))$ by
 $\Gamma_i'=(\pi_i')^{\ddag}\circ \Gamma,$
$i=1,2,...,R.$ Note that $\Gamma,$ $\Gamma',$ $\Gamma_i$ and $\Gamma_i'$ are compatible
 with $-\af_0,$ $-\af_0',$ $-\af_{0,i}$ and $-\af_{0,i}',$ respectively.
 We note that
 \beq\label{PM-n16-}
 \Gamma'([p_j])=\overline{z_j^*|_{Y_d'}}, \, \Gamma_i([p_j])=\overline{z_j^*|_{Y_i}}\andeqn
 \Gamma_i'([p_j])=\overline{z_j^*(\xi_i)},
 \eneq
 $j=1,2,...,k_0(X)$ and $i=1,2,...,R.$

 By the Universal Coefficient Theorem, there is $\af\in KK(C(X_{n_0}), B)$ such that $\af|_{K_i(C(X_{n_0}))}=\af_i,$ $i=0,1.$
 %By (\ref{PM-n4}) and 10.3 of \cite{LnApp},
 %\beq\label{PM-n16--}
 %|\tau(\boldsymbol{\bt}(g_j))|<1/4N_2\tforal \tau\in T(C).
 %\eneq
  It follows (using  (\ref{PM-n5+})) from  \ref{Sd1ext} (for ${\rm dim}Y\le d$) that
 there  is a unitary $U\in M_{K_4}(B)$ such that
 \beq\label{PM-n16}
 \|[\phi^{(K_4)}(g), \, U]\| &<& \dt_u/2\tforal g\in {\cal G}_u,\\
 {\rm Bott}(\phi^{(K_4)}\circ \psi_{n_0, \infty}, U)&=&- K_4\af\andeqn\\\label{PM-n16++}
  {\rm dist}({\rm Bu}(\phi, \, U)(x),\,K_4\Gamma(x))&<&\gamma_0/64(d+1+m(X))\tforal
  x\in {\cal Q}.
 \eneq

Denote by $Z_{i,j}=(\pi_i({\bar g}_{j,(1)}))^{(K_4)},$
$i=1,2,...,R$ and $j=1,2,...,k(X).$
In the following computation, we will identify $U(\xi_i), $ $W(\xi_i),$ $X_{1,j}(\xi_i)$  and
$Z_{i,j}(\xi_i)$ with constant unitaries in $M_{m(X)}(B'),$ when it makes sense.
We also will use (2) and (4) above, as well as (\ref{ADDD-2})-(\ref{ADDD-3}) in the following computation.
%Since $\|X_{1,j}(y)-X_{1,j}(\xi_i)\|<\ep_1/4$ for all $y\in Y_i,$
%AND
We have
\beq\label{PM-n17}
&&\hspace{-0.6in}{\rm bott}_1((\pi_i\circ \phi)^{(K_4)}, ((W(\xi_i)^{(K_4)}U(\xi_i))^*W|_{Y_i}^{(K_4)}U|_{Y_i}))(g_j)\\
&=&{\rm bott}_1(Z_{i,j}, U(\xi_i)^*(W(\xi_i)^*X_{1,j}|_{Y_j}X_{1,j}|_{Y_j}^*W|_{Y_i})^{(K_4)}U|_{Y_i})\\
&=& {\rm bott}_1(Z_{i,j}, U(\xi_i)^*(W(\xi_i)^*X_{i,j}(\xi_i)X_{i,j}^*|_{Y_i}W|_{Y_i})^{(K_4)}U|_{Y_i})\\
&=& {\rm bott}_1(Z_{i,j},U(\xi_i)^*(W(\xi_i)^*X_{i,j}(\xi_i))^{(K_4)})+
{\rm bott}_1(Z_{i,j},(X_{i,j}^*W)|_{Y_i}^{(K_4)}U|_{Y_i})\\
&=& {\rm bott}_1(Z_{i,j}(\xi_i),U(\xi_i)^*(W(\xi_i)^*X_{i,j}(\xi_i))^{(K_4)})\\
&& +
{\rm bott}_1(Z_{i,j}, (X_{i,j}^*|_{Y_i}W|_{Y_i})^{(K_4)})+{\rm bott}_1(Z_{i,j},
U|_{Y_i})\\
&=& {\rm bott}_1(Z_{i,j}(\xi_i),U(\xi_i)^*)+{\rm bott}_1(Z_{i,j}(\xi_i),
(W(\xi_i)^*X_{i,j}(\xi_i))^{(K_4)})\\
&&+K_4\af_{1,i}(g_j')-K_4\af_{1,i}(g_j')\\
&=& K_4\af'_{1,i}(g_j')-K_4\af'_{1,i}(g_j')=0.
\eneq

Similarly, (put $Q_{0,j}=q_{0,j}^{(K_4)}$),
\beq\label{PM-n18}
&& \hspace{-0.8in} {\rm bott}_0((\pi_i\circ \phi)^{(K_4)}, (W(\xi_i)^{(K_4)}U(\xi_i))^*W|_{Y_i}^{(K_4)}U|_{Y_i})([p_j])\\
&=& {\rm bott}_0(Q_{0,j}|_{Y_i},U(\xi_i)^*(W(\xi_i)^*W|_{Y_i})^{(K_4)}U|_{Y_i})\\
&=&{\rm bott}_0(Q_{0,j}|_{Y_i}, U(\xi_i)^*(W(\xi_i)^*X_{0,j}|_{Y_i}X_{0,j}^*|_{Y_i}W|_{Y_i})^{(K_4)}U|_{Y_i})\\
&=&{\rm bott}_0(Q_{0,j}|_{Y_i}, U(\xi_i)^*(W(\xi_i)^*X_{0,j}(\xi_i)X_{0,j}^*|_{Y_i}W|_{Y_i'})^{(K_4)}U|_{Y_i})\\
&=&{\rm bott}_0(Q_{0,j}|_{Y_i}, U(\xi_i)^*(W(\xi_i)^*X_{0,j}(\xi_i))^{(K_4)})\\
&&\hspace{0.4in}+{\rm bott}_0(Q_{0,j}|_{Y_i},(X_{0,j}^*|_{Y_i}W|_{Y_i})^{(K_4)}U|_{Y_i})\\
&=& {\rm bott}_0(Q_{0,j}|_{Y_i}, U(\xi_1)^*)+{\rm bott}_0(Q_{0,j}|_{Y_i},
 (W(\xi_i)^*X_{0,j}(\xi_i))^{(K_4)})\\
 &&\hspace{0.4in}+{\rm bott}_0(Q_{0,j}|_{Y_i},(X_{0,j}^*|_{Y_i}W|_{Y_i})^{(K_4)})
 +{\rm bott}_0(Q_{0,j}|_{Y_i}, U|_{Y_i})\\
 &=&{\rm bott}_0(Q_{0,j}(\xi_i), U(\xi_1)^*)+{\rm bott}_0(Q_{0,j}(\xi_i),(W(\xi_i)^*X_{0,j}(\xi_i))^{(K_4)})\\
 &&\hspace{0.6in}+K_4\af_{0,j}([p_j'])-K_4\af_{0,j}([p_j'])\\\label{PM-n18+}
 &=& K_4\af_{0,j}'([p_j'])-K_4\af_{0,j}'([p_j'])=0.
 \eneq
Since $K_4x=0$ for all $x\in {\rm Tor}(K_i(C(X_{n_0})),$ $i=0,1,$ we  have
\beq\label{PM-n19-1}
{\rm Bott}((\pi_i\circ \phi\circ \psi_{n_0, \infty})^{(K_4)},\, (W(\xi_i)^{(K_4)}U(\xi_i))^*W|_{Y_i}^{(K_4)}U|_{Y_i})=0,
\eneq
$i=1,2,...,R.$
It follows that
\beq\label{PM-n19}
{\rm Bott}((\pi_i\circ \phi)^{(K_4)},\, (W(\xi_i)^{(K_4)}U(\xi_i))^*W|_{Y_i}^{(K_4)}U|_{Y_i})|_{{\cal P}_0}=0,
\eneq
$i=1,2,...,R.$

We also estimate on $Y_i,$ using (\ref{PM-n1}), (\ref{PM-n16}) and   (4),
\beq\label{PM-n19+1}
&&\hspace{-0.4in}\|[\pi\circ \phi^{(K_4)}(g), \, (W(\xi_i)^{(K_4)}U(\xi_i))^*W|_{Y_i}^{(K_4)}U|_{Y_i}]\|\\
&<&\ep_1/4+\dt_u/2+\ep_1/4+\ep_1/4+\ep_1/4+\ep_1/4+\dt_u/2<\dt_1
\eneq
for all $g\in {\cal G}_u.$

For each $i,$ there is $\Xi_j\in U(M_{K_4}(B))$ such that
\beq\label{PM-n19+}
\|\Xi_j-1\|<\dt_u/2\andeqn \Xi_jUQ_{0,j}U^*\Xi_j^*=Q_{0,j},\,\,\, j=1,2,...,k_0(X).
\eneq
Denote
$$
P_{i,j}=(1-Q_{0,j})|_{Y_i}\andeqn  Q_{i,j}=Q_{0,j}|_{Y_i}.
$$
$j=1,2,...,k_0(X)$ and $i=1,2,...,R.$
By identifying $(\Xi_jU)(\xi_i),$ $(\Theta{\bar W})(\xi_i)$ and $X_{0,j}(\xi_i)$ with constant unitaries on $Y_i,$  by (3)
and (4) above,
there is a unitary $\Xi_{i,j}\in M_{K_4}(B_i)$ such that
$$
\|\Xi_{i,j}-1\|<\ep_1/4\andeqn
\|[Q_{i,j},\, \Xi_{ij}(\Xi_jU)(\xi_i)^* ((\Theta_i{\bar W}(\xi_i))^*\Theta_i|_{Y_i} {\bar W}|_{Y_i})^{(K_4)}(\Xi_jU)|_{Y_i}]\|=0,
$$
$j=1,2,...,k_0(X)$ and $i=1,2,...,R.$
Similarly, there is a unitary $\Xi_{i,j}'\in M_{K_4}(B_i)$ such that
$$
\|\Xi_{i,j}'-1\|<\ep_1/4\andeqn \|[Q_{i,j},\,\Xi_{ij}(\Xi_jU)(\xi_i)^* ((\Theta_i{\bar W}(\xi_i))^*X_{0,j}(\xi_i))^{(K_4)}\Xi_{i,j}']\|=0,
$$
$j=1,2,...,k_0(X)$ and $i=1,2,...,R.$
Set
$$
P_{i,j}'=(1-Q_{0,j})(\xi_i)\andeqn Q_{i,j}'=Q_{0,j}(\xi_i)
$$
as constant projections. Define
\beq
\Omega_{i,j}&=&P_{i,j}+Q_{i,j}(X_{0,j}^*\Theta_i{\bar W})^{(K_4)}\Xi_jU)|_{Y_i}\andeqn\\
\Omega_{i,j}'&=& P_{i,j}'+Q_{i,j}'(\Xi_jU)(\xi_i)^* ((\Theta_i{\bar W}(\xi_i))^*X_{0,j}(\xi_i))^{(K_4)}
\eneq
Then (see also (3) above)
\beq\label{PM-n20-1}
\|\Omega_{i,j}'-(P_{i,j}+Q_{i,j}\Xi_{ij}(\Xi_jU)(\xi_i)^* ((\Theta_i{\bar W}(\xi_i))^*X_{0,j}(\xi_i))^{(K_4)}\Xi_{i,j}')\|<\ep_1/2
\eneq
Thus, from above,
\beq\label{PM-n20-2}
\hspace{-0.2in}\|\Omega_{i,j}'\Omega_{i,j}-(P_{i,j}+Q_{i,j} \Xi_{ij}(\Xi_jU)(\xi_i)^* ((\Theta_i{\bar W}(\xi_i))^*)^{(K_4)}(\Theta_i {\bar W}|_{Y_i})^{(K_4)}(\Xi_jU)|_{Y_i})\|<\ep_1
\eneq

We also have (see \ref{DUb} and \ref{NBott}), by (\ref{PM-n16++}),
\beq\label{PM-n20}
&&\hspace{-1in}{\rm dist}({\rm Bu}((\pi_i\circ \phi)^{(K_4)}, ({\bar W}(\xi_i)^{(K_4)}U(\xi_i))^*{\bar W}|_{Y_i}^{(K_4)}U|_{Y_i})([p_j]), \,{\bar 1})\\
&&\hspace{-0.75in}<{\rm dist}(\overline{P_{i,j}+{\bar Q_{i,j}}(\Xi_iU)(\xi_i)^* ((\Theta_i{\bar W}(\xi_i))^*\Theta_i|_{Y_i} {\bar W}|_{Y_i})^{(K_4)}(\Xi_iU)|_{Y_i}}, \,{\bar 1}) +2\dt_u\\
&&\hspace{-0.75in}<{\rm dist}(\overline{\Omega_{i,j}'\Omega_{i,j}},\,{\bar 1})
+\ep_1+2\dt_u\\
&&\hspace{-0.75in}\le {\rm dist}(\overline{\Omega_{i,j}'},{\bar 1})+
{\rm dist}(\overline{\Omega_{i,j}}, {\bar 1})+\ep_1+2\dt_u\\
%&&\hspace{-0.75in}=  {\rm dist}(\overline{(P_{i,j}+{\bar Q_{i,j}}(U\Xi_i)(\xi_i)^* (\Theta_i{\bar W})^*(\xi_i)(X_{0,j}X_{0,j}^*)|_{Y_i}\Theta_i|_{Y_i} {\bar W}|_{Y_i})^{(K_4)}(U\Xi_i)|_{Y_i}}, \,{\bar 1})+2\dt_u\\
%&&\hspace{-0.75in}\le {\rm dist}(\overline{(P_{i,j}+{\bar Q_{i,j}}(U\Xi_i)(\xi_i)^* (\Theta_i{\bar W})^*(\xi_i)X_{0,j}(\xi_i)X_{0,j}|_{Y_i}^*\Theta_i|_{Y_i} {\bar W}|_{Y_i})^{(K_4)}(U\Xi_i)|_{Y_i}}, \,{\bar 1})\\
 % &&\hspace{0.4in}+\ep_1/4+2\dt_u\\
&&\hspace{-0.75in}= {\rm dist}(\overline{P_{i,j}+{\bar Q_{i,j}}(\Xi_iU)(\xi_i)^* (\Theta_i{\bar W}(\xi_i)X_{0,j}(\xi_i))^{(K_4)}},\, {\bar 1})\\
&&+{\rm dist}(\overline{P_{i,j}+{\bar Q_{i,j}}(X_{0,j}^*\Theta_i|_{Y_i} {\bar W}|_{Y_i})^{(K_4)}(\Xi_iU)|_{Y_i}}, \,{\bar 1}) +\ep_1+2\dt_u\\
&&\hspace{-0.75in}={\rm dist}(\overline{P_{i,j}+{\bar Q_{i,j}}(\Theta_i{\bar W}(\xi_i)X_{0,j}(\xi_i))^{(K_4)}}, \overline{P_{i,j}+{\bar Q_{i,j}}(\Xi_iU)(\xi_i)})\\
&&\hspace{-0.5in}+{\rm dist}(\overline{P_{i,j}+{\bar Q_{i,j}}(X_{0,j}^*\Theta_i|_{Y_i} {\bar W}|_{Y_i})^{(K_4)}}, \overline{P_{i,j}+{\bar Q_{i,j}}(\Xi_iU)^*|_{Y_i}})+\ep_1+2\dt_u\\
&&\hspace{-0.75in}={\rm dist}(\overline{({\overline{ z_j^*}}(\xi_i)^{(K_4)}}, \overline{P_{i,j}+{\bar Q_{i,j}}(\Xi_iU)(\xi_i)})\\
&&\hspace{-0.5in}+{\rm dist}({\overline{z_j}}|_{Y_i}
^{(K_4)}, \overline{P_{i,j}+{\bar Q_{i,j}}(\Xi_iU)^*|_{Y_i}})+\ep_1+2\dt_u\\
&&\hspace{-0.75in}<\gamma_0/64(d+1+m(X))+\gamma_0/64(d+1+m(X))+\ep_1+2\dt_u\\\label{PM-n20+}
&&\hspace{-0.75in}<\gamma_0/(d+1+m(X)),
\,\,\,j=1,2,...,k_0(X)\andeqn\,\,\,i=1,2,...,R.
\eneq

Now we are ready to apply \ref{Sd1hom} (for ${\rm dim} Y\le d$) using (\ref{PM-n19}), (\ref{PM-n20})-(\ref{PM-n20+}),(\ref{PM-n19+1}) and (\ref{PM-n3}).
By \ref{Sd1hom}, there is a continuous path of unitaries
$\{V_i(t): t\in [0,1]\}\subset M_{K_1K_4}(B_i)$ such that
\beq\label{PM-n21}
&&V_i(0)=(W(\xi_i)^{(K_1K_4)}U(\xi_i)^{(K_1)})^*W^{(K_1K_4)}|_{Y_i}U^{(K_1)}|_{Y_i},\,\,\,
V(1)=1,\\\label{PM-n21+}
&&\andeqn \|[(\pi_i\circ \phi)^{(K_1K_4)}(f),\,V_i(t)]\|<\ep/32\tforal t\in [0,1]\andeqn f\in {\cal F},
\eneq
$i=1,2,...,R.$

Define $u\in M_{K_1K_4}(C)$ (in fact it should be in  $M_{K_1K_4K_5}(C)$ but we replace $M_{K_5}(C)$ by $C$ early on) as follows:
$u(y)=W(y)^{(K_1K_4)}U^{(K_1)}(y)$ for $y\in Y_d'.$
Note that $D_i$ is homeomorphic to the $d+1$-dimensional open ball of radius $1.$
Each point of $D_i$ is identified by a pair $(x, t),$ where
$x$ is on  $\partial{D_i}\cong S^d,$ the boundary of $D_j,$ and $t$ is distance from the point to the center $\xi_i.$  Let $f_i: \partial{D_i}\to Y_i$ be the continuous map given
by $Y.$ Now define (note that $V_i(t)\in M_{K_1K_4}(B_i)$)
\beq\label{PM-n22}
u_i(x, t)=W(\xi_i)^{(K_1K_4)}U(\xi_i)^{(K_1)}V_i(1-t)(f_i(x))
\eneq
Note that $u(x,1)=W^{(K_1K_4)}(f_i(x))U^{(K_1)}(f_i(x))$ for $x\in \partial{D_i}$ and
$u_i(x,0)=W(\xi_i)^{(K_1K_4)}U(\xi_i)^{(K_1)}.$ Define $u$ on $D_i$ as
$u_i(x,t).$ Then $u\in M_{K_1K_4}(C),$ $u|_{Y_d}=(W^{(K_1K_4)}U^{(K_1)})|_{Y_d}.$
Let $K=K_1K_4.$
We have
\beq\label{PM-n23}
\hspace{-0.6in}\|u\phi^{(K)}(f)u-\psi^{(K)}(f)\|_{Y_d}
&= & \|W^{(K)}U^{(K_1)}\phi^{(K)}(f)(U^*)^{(K_1)}(W^*)^{(K)}
-\psi^{(K)}(f)\|_{Y_d}\\
&<& \dt_u/2+\|W^{(K)}\phi^{(K)}(f)W^{(K)}-\psi^{(K)}(f)\||_{Y_d}\\
&<& \dt_u/2+\ep_1/4<\ep\tforal f\in {\cal F}.
\eneq
Moreover, for $y\in D_i$ and any $y'\in Y_i,$  (5),(6) and   (\ref{PM-n22}) and (\ref{PM-n21+}),
we have
\beq\label{PM-n24}
&&\hspace{-0.8in}\|u(y)\phi^{(K)}(f)(y)u^*(y)-\psi^{(K)}(f)(y)\|\\
%&<& \|u(y)\phi^{(K)}(f)(y')u^*(y)-\psi^{(K)}(f)(\xi_i)\|+2\ep_1/4\\
&&\hspace{-0.2in}<\|u(y)\phi^{(K)}(f)(y')u^*(y)-\psi^{(K)}(f)(\xi_i)\| +3\ep_1/4\\
&&\hspace{-0.2in}< \|W(\xi_i)^{(K)}U(\xi_i)^{(K_4)}\phi^{(K)}(f)(y')(U(\xi_i)^*)^{(K_4)}
(W(\xi_i)^*)^{(K)}-\psi^{(K)}(f)(\xi_i)\|\\
&&\hspace{0.6in}+\ep/32+3\ep_1/4\\
&&\hspace{-0.2in}< \|W(\xi_i)^{(K)}\phi^{(K)}(f)(y')(W(\xi_i)^*)^{(K)}-\psi^{(K)}(f)(\xi_i)\|
+\dt_u/2+\ep/32+3\ep_1/4\\
&& \hspace{-0.2in}< \|W(\xi_i)^{(K)}\phi^{(K)}(f)(\xi_1)(W(\xi_i)^*)^{(K)}-\psi^{(K)}(f)(\xi_i)\|\\
 && \hspace{0.6in} +\ep_1/4 +\dt_u/2+\ep/32+3\ep_1/4\\
&&\hspace{-0.2in}<\|W(\xi_i)^{(K)}\phi^{(K)}(f)(\xi_i)(W(\xi_i)^*)^{(K)}-\psi^{(K)}(f)(\xi_i)\|\\
&&\hspace{0.6in}
+\ep_1/4+\dt_u/2+\ep/32+\ep_1\\
&&\hspace{-0.2in}< \ep_1/4+\dt_u/2+\ep/32+5\ep_1/4<\ep\,\,\,\,\tforal f\in {\cal F}.
\eneq
It follows that
\beq\label{PM-n25}
\|u\phi^{(K)}(f)u^*-\psi^{(K)}(f)\|<\ep\tforal f\in {\cal F}.
\eneq

 }

\section{The reduction}

%\begin{NN}\label{measure}
%{\rm
%Let $X$ be a compact metric space, $r\ge 1$ be an integer and $E\in M_r(C(X))$ be a projection. Put $C=PM_r(C(X))P.$%
%Suppose $\tau\in T(C).$ It is known  that there exists a probability measure $\mu_\tau$ on $X$ such that
%$$
%\tau(f)=\int_X  t_x(f(x))d\mu_\tau(x)\rforal f\in C
%$$
%where  $t_x$ is the normalized trace on $E(x)M_rE(x)$ for all $x\in X$ (see 2.17 of \cite{Lncrell}).
%}%
%
%\end{NN}

\begin{thm}\label{UNI}
 The statement of {\rm \ref{Sd1uni} } holds for  all those compact subsets $Y$ of a finite CW complex with dimension
 no more than $d,$ where $d$ is a non-negative integer.
\end{thm}

\begin{proof}
Let $\ep>0$ and ${\cal F}\subset C(X)$ be a finite subset given.
Let $\Delta_1(r)=\Delta(r/3)/3$ for all $r\in (0,1).$
Let $\eta_1>0$ (in place of $\eta$), $\dt_1>0$ (in place of $\dt$) $\gamma_1'>0$ (in place of $\gamma_1$), $\gamma_2'>0$ (in place of  $\gamma_2$),
${\cal G}\subset C(X)$  be a finite subset, ${\cal P}\subset \underline{K}(C(X))$ be a finite subset, ${\cal H}\subset C(X)_{s.a.}$ be a finite subset, ${\cal V}\subset K_1(C(X))\cap {\cal P},$ $N\ge 1$ be an integer and $K\ge 1$ be an integer
given by \ref{Sd1uni} for $\ep,$ ${\cal F},$  $\Delta_1$ and $d.$

Let $\eta=\eta_1/3,$ $\dt=\dt_1/2,$ $\gamma_1=\gamma_1'/2.$  Suppose that $\phi,$ $\psi$  and
$C=PM_m(C(Y))P$ satisfy the assumptions for the above $\eta,$ $\dt,$ $\gamma_1,$ $\gamma_2,$
${\cal G},$ ${\cal P},$ ${\cal H},$ ${\cal V},$ $N$ and $K.$

Suppose that $C=\lim_{n\to\infty}(C_n, \psi_n),$ where $C_n=P_nM_m(C(Y_n))P_n,$ where $Y_n$ is a finite CW complex of dimension no more than $d.$
Let $\dt>\dt_0>0$ and ${\cal G}_0\subset C$ be a finite subset.
It follows from \ref{nLmeasure} that there exist an integer $n\ge 1,$  a unital \morp\, $r: C\to C_n$ and  unital \morp s  $\Phi,
\Psi: C(X)\to C_n$ such that
$\Phi=r\circ \phi,$  $\Psi=r\circ \psi,$
\beq\label{UNI-2}
\|\psi_{n, \infty}\circ \Phi(f)-\phi(f)\|&<&\dt_0\rforal f\in {\cal G},\\\label{UNI-3}
\|\psi_{n, \infty}\circ \Psi(f)-\phi(f)\|&<&\dt_0\rforal f\in {\cal G},\\\label{UNI-3+}
\|\psi_{n, \infty}\circ r(g)-g\|&<&\dt_0\tforal f\in {\cal G}_0\andeqn\\\label{UNI-4}
\mu_{t\circ \Phi}(O_s),\,\,\, \mu_{t\circ \Psi}(O_s)&\ge & \Delta(s/3)/3\tforal t\in T(C_n)
\eneq
for all $s\ge 17\eta_1/8.$

By choosing small $\dt_0$ and large ${\cal G}_0,$ we see that we reduce the general case to the case that $Y$ is a finite
CW complex and \ref{Sd1uni} applies.

\end{proof}

\begin{thm}\label{MT1}
Let $A$ be a unital separable simple \CA\, which is tracially ${\cal I}_d$ for some integer $d\ge 0.$ Then $A\otimes Q$ has tracial rank at most one.

\end{thm}

\begin{proof}
Let $\ep>0,$ $a\in A\otimes Q_+\setminus \{0\}$ and let ${\cal F}\subset A\otimes Q$ be a finite subset. We may assume that $1_A\in {\cal F}.$  Note that $A$ and $A\otimes Q$ has the strict comparison for positive elements.
Let $b=\inf \{d_{\tau}(a): \tau\in T(A)\}.$  Then $b>0.$

We write $Q=\lim_{n\to\infty} (M_{n!},\imath_n),$ where
$\imath_n: M_{n!}\to M_{(n+1)!}$ is a unital embedding defined by
$\imath_n(x)=x\otimes 1_{M_{n+1}}$ for all $x\in M_{n!}.$
To simplify notation, without loss of generality, we may assume that
${\cal F}\subset A\otimes M_{n!}$ for some integer $n\ge 1.$
Denote $A_0=A\otimes M_{n!}.$
Since $A_0$ is tracially ${\cal I}_d,$ there is a projection $e_0\in A_0$ and a unital \SCA\, $B_0=EM_{r_0}(C(X))E$ with $1_{B_0}=e_0,$
where $X$ is a compact subset of a finite CW complex with dimension at most $d,$ $r_0\ge 1$ is an integer and
$E\in M_{r_0}(C(X))$ such that
\beq\label{MT-2}
\|e_0x-xe_0\|&<&\ep/8\tforal x\in {\cal F},\\
{\rm dist}(e_0xe_0, B_0)&<&\ep/8\tforal x\in {\cal F}\andeqn\\
\tau(1-e_0)&<&b/8\tforal \tau\in T(A_0).
\eneq
We may assume that $E(x)\not=0$ for all $x\in X.$
Let ${\cal F}_1\subset B_0$ be a finite subset such that
\beq\label{MT-3}
{\rm dist}(e_0xe_0, {\cal F}_1)<\ep/8\tforal x\in {\cal F}.
\eneq
We may assume that $1_{B_0}\in {\cal F}_1.$
We may assume that $X$ is an infinite set (in fact, if $B_0$ can always be chosen to finite dimensional,
then $A$ is an AF algebra).
To simplify notation, without loss of generality,  we may assume that ${\cal F}$ and ${\cal F}_1$ are in the unit ball.

Put $A_1=e_0A_0e_0.$  Let $\imath': B_0\to A_1$ be the unital embedding.
By \ref{trivialpro}, there is $r'\ge 1,$ a projection $E'\in M_{r'}(B_0)$ and  a unitary $W'\in M_{r'}(B_0)$ such that
$E'M_{r'}(B_0)E'\cong M_{k_0}(C(X))=C(X)\otimes M_{k_0}$ for some $k_0\ge 1$ and $(W')^*{1_{B_0}}W'\le E'.$
Let $E=(W')E'(W')^*.$ Then  $EM_{r'}(B_0)E\cong C(X)\otimes M_{k_0}$
and $e_0=1_{B_0}\in EM_{r'}(B_0)E.$ Let $A_2=E((e_0Ae_0)\otimes M_{r'})E.$
%There is $n_1\ge n$ such that $r'|n_1.$
%Put $A_2= (e_0A_0e_0)\otimes M_{r'}$
%So we may identify $E'$ with a projection in $A_2$ and $W'$
%with a unitary in $W'.$
Let $e\in EM_r(B_0)E$ be a projection which may be identified with $1_{C(X)}\otimes e'\in C(X)\otimes M_{k_0},$
where $e'\in M_{k_0}$ is a minimum rank one projection. We also identify $e$ with the projection in $A_2.$
Put $B_1=eM_{r'}(B_0)e.$ Note that $B_1\cong C(X).$
We will identify $B_1$ with $C(X)$ when it is convenient. Denote by $\imath: B_1\to  eA_2e$ be the embedding.
Denote $A_3=eA_2e.$ We will identify $EM_{r'}(B_0)E$ with $M_{k_0}(B_1)$
and $M_{k_0}(eAe)$ with $A_2.$
There exists a nondecreasing map $\Delta: (0,1)\to (0,1)$ such that
\beq\label{MT-4}
\mu_{\tau\circ \imath}(O_r)\ge \Delta(r)\tforal \tau\in T(A_3)
\eneq
and for all open balls $O_r$ with radius $r>0.$
It follows from  \ref{herd} that $A_3$ is also tracially ${\cal I}_d.$
There exists a finite subset ${\cal F}_2$ in the unit ball of $B_1$ such that
\beq\label{MT-5}
 \{(a_{i,j})_{k_0\times k_0} \in M_{k_0}(B_1): a_{i,j}\in {\cal F}_2\} \supset {\cal F}_1
\eneq
(Here, again, we identify $EM_{r'}(B_0)E$ with $M_{k_0}(B_1)$).

Define $\Delta_1(r)=\Delta(r/3)/3$ for all $r\in (0,1).$
Let $\eta>0,$ $\gamma_1, \gamma_2>0,$  ${\cal G}\subset B_1$ be a finite subset,
${\cal P}\subset \underline{K}(B_1)$ be a finite subset,
${\cal H}\subset (B_1)_{s.a.}$ be a  finite subset, ${\cal V}\subset K_1(B_1)\cap {\cal P}$ be a finite
subset, $N\ge 1$ and $K\ge 1$ be integers given by \ref{UNI} for $\ep/2^8(r')^2$ (in place of $\ep$),
${\cal F}_2$ (in place of ${\cal F}$)  and $\Delta_1$ (in place of $\Delta$).
Let ${\cal U}\in U(M_r(B_1))$ be a finite subset (for some integer $r\ge 1$) such that
the image of ${\cal U}$ in $K_1(B_1)$ is ${\cal V}.$

Let $\{X_n\}$ be a decreasing sequence of finite CW complexes such that $X=\cap_{n=1}^{\infty} X_n$ and
let $s_n: C(X_n)\to C(X_{n+1})$ be the map defined by $s_n(f)=f|_{X_{n+1}}$ for $f\in C(X).$
Write $B_1=\lim_{n\to\infty} (C(X_n), s_n).$
Choose an integer $L_1\ge 1$ such that $1/L_1<b/8.$
Let $n_1\ge 1$ (in place of $n$) be an integer, $q_1,q_2,...,q_s\in C(X_{n_1})$  be mutually orthogonal projections, let
$g_1,g_2,...,g_k$  be a finite subset, $G_0\subset K_0(C(X))$ be a finite subset, $N\ge 1$ be an integer  and  $G_1\subset
U(M_l(B_1))/CU(M_l(B_1))$ (for some integer $l\ge 1$)  a finite subset with $\overline{{\cal U}}\subset G_1$ be
as given by \ref{sec1MT} for  $\ep/2^8(r')^2$ (in place of $\ep$), $\gamma_1/4$ (in place of $\sigma_1$),
$\gamma_2/4$ (in place $\sigma_2$), ${\cal G},$ ${\cal P},$ ${\cal H},$ ${\cal U},$ and $L_1.$
We may assume that $r=l$ without loss of generality (by choosing the larger among them).

Choose $\dt_1>0$ and a finite subset ${\cal G}_1\subset B_1$ such that, for any
$\dt_1$-${\cal G}_1$-multiplicative \morp\, $L$ from $B_1,$  $[L\circ s_{n_1,\infty}]$ is well defined
on $\underline{K}(C(X_{n_1})).$ Let $g'_1,g_2',...g_k'\in K_0(C(X_{n_1}))$ such that $(s_{n_1, \infty})_{*0}(g_i')=g_i,$
$i=1,2,...,k.$ We may also assume, by applying 10.3 of \cite{LnApp}, that
\beq\label{MT-5-1}
|\rho_{C'}([L\circ s_{n_1,\infty}](g_i'))(\tau)|<1/2N\tforal \tau\in T(C'),
\eneq
$i=1,2,...,k,$ for any $\dt_1$-${\cal G}_1$-multiplicative \morp\, $L: B_1\to C'$ for any unital \CA\, $C'$ with $T(C')\not=\emptyset.$ By choosing smaller $\dt_1$ and larger ${\cal G}_1,$ we may also  assume
that $[L]$ induces a well-defined  \hm\, $\Lambda'$ on $G_1.$ Furthermore, we may assume that
\beq\label{MT-5-2}
[L\circ s_{n_1, \infty}](\xi)=\Pi(\Lambda'(g))
\eneq
for all $g\in G_1$ and $\xi\in K_1(C(X_{n_1}))$ such that $g=(s_{n_1, \infty})_{*0}(\xi),$
provided $L$ is a $\dt_1$-${\cal G}_1$-multiplicative \morp.

Choose a set ${\cal F}_3\subset B_1$ of $(2N+1)(d+1)$ mutually orthogonal positive elements.
Since $A_3$ is simple  and unital, there are $x_{f,1},x_{f,2},...,x_{f,f(n)}\in A_3$ such that
\beq\label{nMT-5}
\sum_{j=1}^{f(n)} x_{f,j}^* f x_{f,j}=e\tforal f\in {\cal F}_3.
\eneq
Let $N_0=\max \{f(n): f\in {\cal F}_3\}\max\{\|x_{f,j}\|: f\in {\cal F}_3,\,\,\, 1\le j\le f(n)\}.$
Let $\dt_2=\min\{\dt/2, \dt_1/2, \ep/2^8(r')^2\}$ and let
${\cal G}_2={\cal F}_2\cup {\cal G}\cup {\cal F}_3\cup {\cal G}_1.$

Since $A_3$ is tracially ${\cal I}_d,$  by applying \ref{Lmeas},  there are a projection $e_1\in A_3,$ a unital \SCA\,
$C=PM_R(C(Y))P\in {\cal I}_d$ with $1_C=e_1$ such that
\beq\label{MT-6}
\|e_1x-xe_1\|<\dt_1\tforal x\in {\cal G}_1,\\
{\rm dist}(e_1xe_1, C)<\dt_1\tforal x\in {\cal G}_1,\\
\tau(1_{A_3}-e_1)<b/8\tforal \tau\in T(A_3),
\eneq
and there exists a unital $\dt_1$-${\cal G}_1$-multiplicative \morp\, $\Phi: B_1\to C$ such that
\beq\label{MT-7}
\|\Phi(x)-e_1xe_1\|<\dt_1\tforal  x\in {\cal G}_1\andeqn\\
\mu_{\tau\circ \Phi}(O_r)\ge \Delta_1(r)\tforal \tau\in T(e_1A_3e_1)
\eneq
for all open balls $O_r$ with radius $r\ge \eta.$
We may also assume that there is a projection $E^{(0)}\in M_{k_0}(e_1A_3e_1)$ such that $E^{(0)}\le e_0$ and
\beq\label{MT-8}
\|e_0(e_1\otimes 1_{M_{k_0}})e_0-E^{(0)}\|<\ep/2^8
\eneq
(Note that $e_0=1_{B_0}\in EM_{r'}(B_1)E=M_{k_0}(C(X))$).
Note also that we identify $(e_1\otimes 1_{M_{k_0}})$ with a projection $E'\le E$ in
$A_2.$

We also have, for $x\in {\cal F}_1,$ by (\ref{MT-5}) and above,
\beq\label{MT-8+1}
\|E^{(0)}x-xE^{(0)}\|&<& \ep/2^7+\|e_0(e_1\otimes 1_{M_{k_0}})e_0x-xe_0(e_1\otimes 1_{M_{k_0}})e_0\|\\\label{MT-8+2}
&=&\ep/2^7+\|e_0(e_1\otimes 1_{M_{k_0}})x-x(e_1\otimes 1_{M_{k_0}})e_0\|\\\label{MT-8+3}
&<&\ep/2^7+(k_0)^2\dt_1<3\ep/2^8
\eneq
Similarly,
\beq\label{nMT-8+2}
\|E^{(0)}xE^{(0)}-(\Phi\otimes {\rm id}_{M_{k_0}})(x)\|&<&\ep/2^6\tforal x\in {\cal F}_1.
\eneq
Since ${\cal G}_1\supset {\cal F}_3,$ we conclude that, for each $y\in Y,$ ${\rm rank} P(y)> 2N(d+1).$
Let $\kappa=[\Phi\circ s_{n_1, \infty}].$  It follows from the above construction, by \ref{sec1MT}, that there are
a unital $\ep/2^8(r')^2$-${\cal G}$-multiplicative \morp\, $\Psi: B_1\to C,$ mutually orthogonal projections
$Q_0, Q_1,...,Q_{L_1}, Q_{L_1+1}$ such that $Q_0, Q_1,...,Q_{L_1}$ are mutually equivalent,
$P=\sum_{i=1}^{L_1}Q_i,$
\beq\label{nMT-8}
[\Psi\circ s_{n_1, \infty}]&=&[\Phi\circ s_{n_1, \infty}],\\
{\rm dist}(\Psi^{\ddag}(x), L^{\ddag}(x))&<&\gamma_1/2\tforal x\in \overline{\cal U}\andeqn\\
|\tau\circ \Psi(a)-\tau\circ \Phi(a)|&<&\gamma_2\tforal a\in {\cal H} \andeqn \tforal \tau\in T(C),
\eneq
and $\Psi=\Psi_0\oplus\overbrace{\Psi_1\oplus \Psi_1\oplus \cdots \oplus\Psi_1}^{L_1}\oplus \Psi_2,$
where $\Psi_0: B_1\to Q_0CQ_0,$ $\Psi_1=\psi_1\circ \phi_0,$ $\psi_1: C(J)\to Q_1CQ_1$ is a unital \hm,
$\phi_0: B_1\to C(J)$ is a unital $\ep/2^8(r')^2$-${\cal G}$-multiplicative \morp\, $\Psi_2=\psi_2\circ \phi_0,$
$\psi_2: C(J)\to Q_{L_1+1}CQ_{L_1+1}$ is a unital \hm, and where $J$ is a finite disjoint union of intervals.

It follows from \ref{UNI} that there is an integer $K\ge 1$ and a unitary $U\in A_3\otimes M_{K}$ such that
\beq\label{MT-9}
\|U^*\Phi^{(K)}(f)U-\Psi^{(K)}(f)\|<\ep/2^8(r')^2\tforal f\in {\cal F}_2.
\eneq
Choose $n_1'\ge 1$ such that $K|n_1'.$
Note that $e_1A_3e_1\subset e_1A_3e_1\otimes M_{K}\otimes M_{(n_1')!/(n_1)!K}. $
With that, we may write $\Phi^{(K)}(f)=\Phi(f)\otimes 1_{M_K}$ and
$\Psi^{(K)}(f)=\Psi(f)\otimes 1_{M_K}.$
It follows that (working in $A_2\otimes M_{K}$)
\beq\label{MT-10}
\|(\Phi\otimes {\rm id}_{M_{k_0}})(x)\otimes 1_{M_K}-{\bar U}((\Psi\otimes {\rm id}_{M_{k_0}})(x)\otimes 1_{M_K}){\bar U}^*\|<\ep/2^8\tforal x\in {\cal F}_1,
\eneq
where ${\bar U}=U\otimes 1_{M_{k_0}}.$
Put
\beq
C_1=(\overbrace{\psi_1\oplus\psi_1\oplus\cdots\oplus \psi_1}^{L_1}\oplus \psi_2)(C(J))\andeqn
\Psi'=\overbrace{\Psi_1\oplus\Psi_1\oplus \cdots\oplus \Psi_1}^{L_1}\oplus \Psi_2.
\eneq
There is a projection $E^{(1)}\in M_{k_0}(C_1)$ such that
\beq\label{MT-11}
\|(\Psi'\otimes {\rm id}_{M_{k_0}})(1_{B_0})-E^{(1)}\|<\ep/2^7.
\eneq
Put $E_2={\bar U}(E^{(1)}\otimes 1_{M_K}){\bar U}^*.$
Thus, by (\ref{MT-10}) and (\ref{MT-8}),
\beq\label{MT-11+}
\|(E^{(0)}\otimes 1_{M_K})E_2-E_2\|<3\ep/2^8.
\eneq
There is a projection  $E_3\le (E^{(0)}\otimes 1_{M_K})$ and
\beq\label{MT-12}
\|E_3-E_2\|<3\ep/2^7.
\eneq
It follows that there is a unitary $U_1\in M_{k_0}(e_1A_3e_1\otimes M_K)=E'A_2E'\otimes M_K$
with $\|U_1-1\|<3\ep/2^7$ such that
$U_1^*E_2U_1=E_3.$ Put $W={\bar U}^*U_1.$
Let
$$
C_2=W^*(E^{(1)}\otimes 1_{M_K})(M_{k_0}(C_1)\otimes M_K)(E^{(1)}\otimes 1_{M_K})W.
$$
Since $E_3\in M_{k_0}(C_1)\otimes M_K,$ $C_2\in {\cal I}^{(1)}.$
We estimate that, for $x\in {\cal F}_1,$ by identifying $x$ with $x\otimes 1_{M_K},$
\beq\label{MT-13}
\hspace{-0.4in}\|E_3x-xE_3\| &<&\|E_2x-xE_2\|+3\ep/2^6\\
&=&\|E_2(E'\otimes 1_{M_K})x-x(E'\otimes 1_{M_{K}})E_2\|+3\ep/2^6\\
&<&\|E_2(E'xE'\otimes 1_{M_K})-(E'xE'\otimes 1_{M_K})E_2\|+\ep/2^7+3\ep/2^6\\
&<&\|E_2(\Phi\otimes {\rm id}_{M_K})(x)-(\Phi\otimes {\rm id}_{M_K})(x)E_2\|+2(k_0)^2\dt_1+5\ep/2^6\\
&<&\|E_2{\bar U}((\Psi\otimes {\rm id}_{M_k})(x)-(\Psi\otimes {\rm id}_{M_k})(x)){\bar U}^*E_2\|+2\ep/2^8+6\ep/2^6\\\label{MT-13+1}
&<& 8\ep/2^6=\ep/2^3.
\eneq
Similarly, for $x\in {\cal F}_1,$
\beq\label{MT-14}
&&\hspace{-0.8in}\|E_3xE_3-(\Psi'\otimes {\rm id}_{k_0})(x)\otimes 1_{M_K}\|\\
&<& 3\ep/2^6+\|E_2xE_2-E_2e_1xe_1E_2\|+\\
&&\|E_2e_1xe_1E_2-W^*((\Psi'\otimes {\rm id}_{k_0})(x)\otimes 1_{M_K})W\|\\
&<& 3\ep/2^6+(k_0)^2\dt_1\\
&&+
\|E_2((\Phi\otimes {\rm id}_{M_{k_0}})(x)\otimes 1_{M_{K}})E_2-W^*((\Psi'\otimes {\rm id}_{k_0})(x)\otimes 1_{M_K})W\|\\
&<& 7\ep/2^6+6\ep/2^7\\
&&+\|E_2((\Phi\otimes {\rm id}_{M_{k_0}})(x)\otimes 1_{M_{K}})E_2-
{\bar U}((\Psi'\otimes {\rm id}_{k_0})(x)\otimes 1_{M_K}){\bar U}^*\|\\
&<& 19\ep/2^6+\ep/2^6\\
&&+\|E_2((\Phi\otimes {\rm id}_{M_{k_0}})(x)\otimes 1_{M_{K}}-
{\bar U}((\Psi\otimes {\rm id}_{k_0})(x)\otimes 1_{M_K}){\bar U}^*)E_2\|\\
&<& 20\ep/2^6+\ep/2^8=24\ep/2^6=3\ep/2^3.
\eneq
Therefore
\beq\label{MT-15}
{\rm dist}(E_3xE_3,\, C_2)<3\ep/2^3\tforal x\in {\cal F}_1.
\eneq
Note that $E_3\le E^{(0)}\otimes 1_{M_K}.$
It follows that $E_3\otimes 1_{M_K}\le e_0\otimes 1_{M_K}.$
Note that we identify $x$ with $x\otimes 1_{M_K}.$ So
It follows that, for all $y\in {\cal F},$
$E_3(y\otimes 1_{M_K})E_3=E_3(e_0\otimes 1_{M_K})(y\otimes 1_{M_K})(e_0\otimes 1_{M_K})E_3,$
\beq\label{MT-16}
{\rm dist}(E_3(y\otimes 1_{M_K})E_3, C_2)<3\ep/2^3+2\ep/8<\ep.
\eneq
Since ${\cal F}\subset A_0\subset A_0\otimes M_{K}\otimes M_{n_1!/n!K},$ $C_2\subset A_0\otimes M_{K}\otimes M_{n_1!/n!K}\subset A,$
one may write (\ref{MT-16})  as
\beq\label{MT-17}
{\rm dist}(E_3yE_3, C_2)<\ep\tforal x\in {\cal F}
\eneq
and by (\ref{MT-13})-(\ref{MT-13+1}), one may write  that
\beq\label{MT-18}
\|E_3y-yE_3\|<\ep/8+\ep/2^3+\ep/8<\ep\tforal y\in {\cal F}.
\eneq
By identify $e_0$ with $e_0\otimes 1_{M_K}$ in $A,$ we also have,
by the choice of $L_1,$  that
\beq\label{MT-19}
\tau((1-e_0)+(e_0-E_3))<b/8+b/8=b/4\tforal \tau\in T(A).
\eneq
Thus
\beq\label{MT-20}
\max\{\tau(1-E_3): \tau\in T(A)\}<b=\inf\{d_\tau(a): \tau\in T(A)\}.
\eneq
It follows from  \ref{scomp} that
\beq\label{MT-21}
1-E_3\lesssim a.
\eneq
Therefore, from (\ref{MT-18}), (\ref{MT-17}) and (\ref{MT-21}),
$A$ has tracial rank at most one.

%Let $C_1=U^*{\abr U}^*C_1{\bar U}U.$

\end{proof}

\begin{thm}\label{MT2}
Let $A$ be a unital separable simple \CA\, which is tracially ${\cal I}_d$ for some integer $d\ge 0.$ Suppose that $A$ satisfies the UCT. Then $A$ has tracial rank at most one and is isomorphic to a unital simple AH-algebra with no dimension growth.
\end{thm}

\begin{proof}
It follows \ref{MT1} that $A$ is rationally tracial rank at most one.
It follows from \ref{weak}  that $K_0(A)$ is  weakly unperforated Riesz  group.  Moreover,
by \ref{extrace}, the map from $T(A)$ to $S_1(K_0(A))$ maps the $\partial_e(T(A))$ onto $\partial_e(S_1(K_0(A))).$  Thus, by \cite{V1}, there is a unital simple AH-algebra $C$
with no dimension growth that has the same Elliott invariant as that of $A.$ Since $A$ is assumed to satisfy the UCT, by the classification theorem in \cite{Lninv}, $A\otimes {\cal Z}\cong C.$  But, by \ref{CZs}, $A$ is ${\cal Z}$-stable.
Therefore $A\cong C.$ This proves, in particular, $A$ has tracial rank at most one.
\end{proof}

{\bf Proof of Theorem \ref{MMT2}}

\begin{proof}
This is an immediate consequence of \ref{MT2}.
\end{proof}
%\begin{thm}\label{MT3}
%Let $A$ be a unital separable simple \CA\, which is locally AH with no dimension growth. Then $A$ is isomorphic to a unital simple AH-algebra
%with no dimension growth.
%\end{thm}

{\bf The proof of  Theorem \ref{MMT1}}
\begin{proof}
This is an immediate corollary of \ref{MT2}.  There is $d\ge 0$ such that  $A$ is tracially ${\cal I}_d.$ It follows from \ref{MT1}
that $A$ has tracial rank at most one.
Note that, since $A$ is locally AH, $A$ also satisfies the UCT.

\end{proof}

\section{Appendix}

In the definition of \ref{Dlocal} and \ref{Drr}, we use ${\cal
I}^{(k)}$ and ${\cal I}_k$ as model classes of \CA s of rank $k.$ In
general, however, one could have more general \CA s  as defined
below.
\begin{df}\label{Ibk}
{\rm
Denote by $\overline{{\cal I}_k}$ the class of \CA s with the form
$PM_r(C(X))P,$ where $X$ is a compact metric space with covering dimension $k,$ $r\ge 1$ and $P\in M_r(C(X))$ is a projection.
}
\end{df}

However the following proposition shows that, by replacing ${\cal I}_k$ by $\overline{{\cal I}}_k,$ one will not make any gain.

\begin{prop}\label{removeIbk}
Let $C=PM_r(C(X))P\in \overline{{\cal I}_k}.$ Then, for any $\ep>0$ and any finite subset ${\cal F}\subset C,$ there exists a \SCA\, $C_1\subset C$ such that $C_1\in {\cal I}_k,$ with $C_1=QM_r(C(Y))Q,$ where $Y$ is a compact subset of a finite CW complex of dimension at most $k,$ $Q\in M_r(C(Y))$ is a projection such that
\beq\label{remove-1}
{\rm dist}(a, C_1)<\ep\tforal a\in {\cal F}\andeqn\\\label{remove-1-1}
\inf\{{\rm rank}\, Q(y): y\in Y\}=\inf\{{\rm rank}\, P(x): x\in X\}
\eneq

\end{prop}

\begin{proof}
There is a sequence of finite CW complexes $\{X_n\}$ with covering dimension $k$ such that
$C(X)=\lim_{n\to\infty} (C(X_n),\phi_n),$ where $\phi_n: C(X_n)\to C(X_{n+1})$ is a unital \hm.
Let $\phi_{n, \infty}: C(X_n)\to C(X)$ be the unital \hm\, induced by the inductive limit system. There is a compact
subset $Y_n\subset X_n$ such that $\phi_{n, \infty}(C(X_n))=C(Y_n).$  Note that $C(Y_n)\subset C(Y_{n+1})\subset C(X).$
Moreover, $C(X)=\overline{\cup_{n=1}^{\infty} C(Y_n)}.$
Let $\imath_n: C(Y_n)\to C(X)$ be the imbedding. Denote by $s_n: X\to Y_n$ the  surjective continuous map such that
$\imath_n(f)(x)=f(s_n(x))$ for all $f\in C(Y_n)$ (and $x\in X$).
Denote again by $\imath_n$ the extension from $M_r(C(Y_n))$ to $M_r(C(X)).$

Now let $1>\ep>0$ and let ${\cal F}\subset PM_r(C(X))P\subset M_r(C(X))$ be a finite subset.
There is an integer $n\ge 1,$ a projection $Q\in M_r(C(Y_n))$ and a finite subset ${\cal G}\subset QM_r(C(Y_n))Q$ such that
$$
\|P-Q\|<\ep/2\andeqn {\rm dist}(x, {\cal G})<\ep/2
$$
for all $x\in {\cal F}.$    It follows that $\imath_n(Q)$ has the same rank as that of $P$ at every point of $x\in X.$
But for each $y\in Y_n,$ there exists $x\in X$ such that $y=s_n(x).$  Therefore (\ref{remove-1-1}) also holds.

\end{proof}

%As \ref{Extremslow} indicated, there may not be any difference
%between ${\cal C}_1$ and ${\cal C}_{1,1}.$  In particular, given
%what have been proved in \ref{FT1}, a unital separable simple \CA\,
%$A$ in ${\cal C}_1$ which satisfies the UCT is ${\cal Z}$-stably
%isomorphic to a unital simple AH-algebra with no dimension growth.
%The following shows that \CA s in ${\cal C}_1$ are in fact ${\cal
%Z}$-stable without UCT, if we know that they also have locally
%finite nuclear dimension.

%\begin{cor}\label{Zstable}
%Let $A$ be a unital separable infinite dimensional simple \CA\,  in ${\cal C}_1.$
%Suppose that $A$ has locally finite nuclear dimension.
%Then $A$ is ${\cal Z}$-stable.
%\end{cor}

%\begin{proof}
%It follows from \ref{scomp}, \ref{BA} and \ref{Propdiv} that
%every unital simple \CA\, $A$ in ${\cal C}_1$ has the strict
%comparison for positive elements and $W(A)$ is almost divisible.
%Then, by \cite{Winv}, $A$ is ${\cal Z}$-stable.
%\end{proof}

\begin{lem}\label{commutators}
Let $G$ be a group, let $a, b\in G$ and let $k\ge 2$ be an integer.
Then there are $2(k-1)$ commutators $c_1, c_2,...,c_{2(k-1)}\in G$ such that
$$
(ab)^k=a^kb^k(\prod_{j=1}^{2(k-1)}c_{2(k-1)-(j-1)}).
$$
\end{lem}

\begin{proof}
Note that
\beq\label{commutators-1}
abab=ab^2a(a^{-1}b^{-1}ab)\tforal a, b\in G.
\eneq
Let $c_1=(a^{-1}b^{-1}ab).$
Therefore
\beq\label{commutators-2}
ab^2ac_1=a(ab^2)((b^2)^{-1}a^{-1}b^2a)c_1.
\eneq
Let $c_2=((b^2)^{-1}a^{-1}b^2a).$
Thus
$$
abab=a^2b^2(c_2c_1).
$$
This proves the lemma for $k=2.$
Suppose that the lemma holds for $1,2,...,k-1.$
Then
\beq\label{commutators-3}
(ab)^k=ab(a^{k-1}b^{k-1})(c_{2(k-2)}c_{2(k-2)-1}\cdots c_1),
\eneq
where $c_1,c_2,...,c_{2(k-2)}$ are commutators.
As in (\ref{commutators-1}),
\beq\label{commutators-4}
ab(a^{k-1}b^{k-1})=ab(b^{k-1})(a^{k-1})(a^{-(k-1)}b^{-(k-1)}a^{k-1}b^{k-1})
\eneq
Let $c_{2(k-1)-1}=(a^{-(k-1)}b^{-(k-1)}a^{k-1}b^{k-1}).$
Further,
\beq\label{commutators-5}
ab^ka^{k-1}=aa^{k-1}b^k(b^{-k}a^{-(k-1)}b^ka^{k-1}).
\eneq
Let $c_{2(k-1)}=(b^{-k}a^{-(k-1)}b^ka^{k-1}).$
Then
\beq\label{commutators-6}
(ab)^k=a^kb^k (c_{2(k-1)}c_{2(k-1)-1}\cdots c_1).
\eneq
This completes the induction.
\end{proof}

\begin{lem}\label{nlem}
{\rm (1)} Let $A$ be a unital \CA\, and let $u\in U(A).$
If $1/2>\ep>0$ and $\|u^k-v\|<\ep$ for some integer $k\ge 1$ and $v\in U_0(A)\cup CU(A),$ then there exist
$v_1\in CU(A)$ and $u_1\in U(A)$ such that
$$
\|u-u_1\|<\ep/k\andeqn u_1^k\in CU(A).
$$
Moreover, there are $2(k-1)$ commutators $c_1, c_2,...,c_{2(k-1)}$
such that
$$
u_1^k=v(\prod_{j=1}^{2(k-1)}c_j).
$$

 {\rm (2)} If $A$ is a unital infinite dimensional simple
\CA\, with (SP), then there is $u\in CU(A)$ such that $sp(u)=\T.$
\end{lem}

\begin{proof}
For (1), there is  $h\in A_{s.a.}$ such that
\beq\label{nlem-1}
\exp(ih)=u^kv^*\andeqn \|h\|<2\arcsin(\ep/2).
\eneq
Let $u_1=u\exp(-ih/k).$ Then
\beq\label{nlem-2}
u_1^k\in CU(A).
\eneq
 One also has
that
$$
\|u-u_1\|<\ep/k.
$$

By \ref{commutators},
$$
u_1^k=vc_{2(k-1)}c_{2(k-1)-1}\cdots c_1
$$
for some commutators $c_1, c_2,..., c_{2(k-1)}$ of $U(A).$

For (2), let $e_1, e_2$ be two non-zero mutually orthogonal and
mutually equivalent projections. Since $e_1Ae_1$ is an infinite
dimensional simple \SCA, one obtains a unitary $u_1\in e_1Ae_1$ such
that $sp(u_1)=\T.$ Let $w\in A$ such that $w^*w=e_1$ and $ww^*=e_2.$
Put $z=(1-e_1-e_2)+w+w^*.$ Then $z\in U(A).$ Define
$u=(1-e_1-e_2)+wu_1w^*+u_1.$ One verifies that $sp(z)=\T$ and $z\in
CU(A).$

\end{proof}

The following is known and has been used a number of times.
We include here for convenience and completeness.

\begin{prop}\label{LIFTCU}
Let $A$ be a separable unital \CA\, and let ${\cal U}\subset U(A)$
be a finite subset. Then, for any $\ep>0,$ there exists $\dt>0$ and a finite subset
${\cal G}\subset A$ satisfying the following: for any unital \CA\,
$B$ and any $\dt$-${\cal G}$-multiplicative \morp\, $L: A\to B,$
there exists a \hm\, $\lambda: G_{\cal U}\to U(B)/CU(B)$ such that
$$
{\rm dist}(\overline{\langle L_n(u)\rangle}, \lambda({\bar u}))<\ep
$$
for all $u\in {\cal U},$ where $G_{\cal U}$ is the subgroup generated
by $\{{\bar u}: u\in {\cal U}\}.$
\end{prop}

\begin{proof}
Suppose that the proposition is false.
Then, there are $\ep_0>0,$ a finite subset ${\cal U}\subset U(A),$
and sequence of decreasing numbers $\dt_n>0$ with $\lim_{n\to\infty}\dt_n=0$ and a sequence of  increasing finite subsets
${\cal G}_n\subset A$ with $\cup_{n=1}^{\infty}{\cal G}_n$ being dense in $A$, a sequence of unital \CA s $B_n$ and a sequence of $\dt_n$-${\cal G}_n$-multiplicative \morp s $L_n: A\to B_n$ such that
\beq\label{LIFT-1}
\inf\{\sup_{u\in {\cal U}}{\rm dist}(\overline{\langle L_n(u)\rangle}, \lambda({\bar u}))\ge \ep_0,
\eneq
where the infimum is taken among all \hm s $\lambda: G_{\cal U}\to
U(B_n)/CU(B_n).$ Define $\Psi: A\to \prod_{n=1}^{\infty}B_n$ by
$\Psi(x)=\{L_n(x)\}$ for all $x\in A.$ Let
$Q=\prod_{n=1}^{\infty}B_n/\oplus_{n=1}^{\infty}B_n$ and let $\Pi:
\prod_{n=1}^{\infty}B_n\to Q$ be the quotient map. Then $\Pi\circ
\Psi: A\to Q$ is a \hm. Therefore it induces a \hm\, $(\Pi\circ
\Psi)^{\ddag}: U(A)/CU(A)\to U(Q)/CU(Q).$ Fix an integer $k\ge 1.$
%Let $z\in G_{\cal U}$ such that $z^k=0.$
Suppose $z\in U(Q)/CU(Q)$ such that $z^k=0.$
There exists a unitary $u_z\in U(Q)$ such that ${\bar u}_z=z.$
There are $w_1, w_2,...,w_N\in U(Q)$ which are commutators such that
$u_z^k=\prod_{j=1}^Nw_j.$  Suppose $w_j=a_jb_ja_j^*b_j^*,$ $a_j, b_j\in U(Q).$
There are unitaries $x_n, a_{j(n)}, b_{j(n)}\in B_n$ such that
\beq\label{LIFT-3}
\Pi(\{x_n\})=u_z,\,\,\, \Pi(\{a_{j(n)}\})=a_j\andeqn \Pi(\{b_{j(n)}\})=b_j,
\eneq
$j=1,2,...,N.$
It follows that
\beq\label{LIFT-4}
\lim_{n\to\infty}\|x_n^k-\prod_{j=1}^N a_{j(n)}b_{j(n)}a_{j(n)}^*b_{j(n)}^*\|=0.
\eneq
It follows from \ref{nlem} that there exists a sequence of unitaries $y_n\in B_n$ such that
$$
y_n^k=\prod_{j=1}^N a_{j(n)}b_{j(n)}a_{j(n)}^*b_{j(n)}^*(\prod_{i=1}^{2(k-1)} v_{i(n)}),
$$
where $v_{i(n)}$ are unitaries in $B_n$ which are commutators, and
\beq\label{LIFT-5}
\lim_{n\to\infty}\|y_n-x_n\|=0.
\eneq
In particular, $\Pi(\{y_n\})=u_z.$
Note that
$$
\{\prod_{j=1}^N a_{j(n)}b_{j(n)}a_{j(n)}^*b_{j(n)}^*(\prod_{i=1}^k v_{i(n)})\}\in CU(\prod_{n=1}^{\infty}B_n).
$$

We have just shown that every finite subgroup of $U(Q)/CU(Q)$ lifts to a finite
subgroup of the same order. This implies that there exists a \hm\, $\gamma: G_{\cal U}\to U(\prod_{n=1}^{\infty}B_n)/CU(\oplus_{n=1}^{\infty} B_n)$ such that $\Pi^{\ddag}\circ \gamma=(\Pi\circ \Psi)^{\ddag}|_{G_{\cal U}}.$
Let $\pi_n: \prod_{n=1}^{\infty}B_n\to B_n$ be the projection onto the $n$-coordinate.
Define $\lambda_n: G_{\cal U}\to U(B_n)/CU(B_n)$ by $\lambda_n=\pi_n^{\ddag}\circ \gamma,$ $n=1,2,....$
There exists $n_0\ge 1$ such that, for all $n\ge n_0,$
\beq\label{LIFT-6}
{\rm dist}(\overline{\langle L_n(u)\rangle}, \pi_n^{\ddag}\circ \gamma({\bar u}))<\ep_0/2
\eneq
for all $u\in {\cal U}.$  This is a contradiction.

\end{proof}

\end{document}